\documentclass[1 leqno,11pt]{amsart}
\usepackage{amssymb, amsmath,amsmath,latexsym,amssymb,amsfonts,amsbsy, amsthm,mathtools,graphicx,CJKutf8,CJKnumb,CJKulem,color}

\numberwithin{equation}{section}

\allowdisplaybreaks

\let\al=\alpha
\let\b=\beta
\let\g=\gamma
\let\d=\delta

\let\s=\sigma
\let\f=\frac

\let\om=\omega

\let\Om=\Omega

\let\na=\nabla
\let\th=\theta
\let\pa=\partial
\let\ep=\epsilon
\let\lto=\longrightarrow

\def\Int{\mathrm{Int}}
\def\In{\mathrm{In}}
\def\rmA{{\mathrm{A}}}
\def\rmB{\mathrm{B}}
\def\rmC{\mathrm{C}}
\def\rmD{\mathrm{D}}
\def\rmE{\mathrm{E}}

\def\cF{\mathcal{F}}
\def\cG{\mathcal{G}}

\def\cL{{\mathcal L}}

\def\bbT{\mathbb{T}}
\def\bbD{\mathbb{D}}

\def\R{\mathbf R}

\def\no{\noindent}

\def\eqdef{\buildrel\hbox{\footnotesize def}\over =}

\newcommand{\beq}{\begin{equation}}
\newcommand{\eeq}{\end{equation}}
\newcommand{\ben}{\begin{eqnarray}}
\newcommand{\een}{\end{eqnarray}}
\newcommand{\beno}{\begin{eqnarray*}}
\newcommand{\eeno}{\end{eqnarray*}}

\newtheorem{theorem}{Theorem}[section]
\newtheorem{definition}[theorem]{Definition}
\newtheorem{lemma}[theorem]{Lemma}
\newtheorem{proposition}[theorem]{Proposition}

\newtheorem{remark}[theorem]{Remark}

\begin{document}
\begin{CJK*}{UTF8}{gkai}
\title[Linear inviscid damping and enhanced dissipation]{Linear inviscid damping and enhanced dissipation for the Kolmogorov flow}

\author{Dongyi Wei}
\address{School of Mathematical Science, Peking University, 100871, Beijing, P. R. China}
\email{jnwdyi@163.com}

\author{Zhifei Zhang}
\address{School of Mathematical Science, Peking University, 100871, Beijing, P. R. China}
\email{zfzhang@math.pku.edu.cn}

\author{Weiren Zhao}
\address{School of Mathematical Science, Peking University, 100871, Beijing, P. R. China}
\email{zjzjzwr@126.com}

\date{\today}

\maketitle

\begin{abstract}
In this paper, we prove the linear inviscid damping and voticity depletion phenomena for the linearized Euler equations around the Kolmogorov flow. These results confirm Bouchet and Morita's predictions based on numerical analysis. By using the wave operator method introduced by Li, Wei and Zhang, we solve Beck and Wayne's conjecture on the optimal enhanced dissipation rate for the 2-D linearized Navier-Stokes equations around the bar state called Kolmogorov flow. The same dissipation rate is proved for the Navier-Stokes equations if the initial velocity is included in a basin of attraction of the Kolmogorov flow with the size of $\nu^{\f 23+}$, here $\nu$ is the viscosity coefficient.

\end{abstract}
\section{Introduction}

In this paper, we study the incompressible Navier-Stokes equations on the torus $\mathbb{T}_{\d}^2=\big\{(x,y): x\in \bbT_{{2\pi}{\d}}, y\in \bbT_{2\pi}\big\}$ with $0<\d\le 1$:
\beq
\label{eq:NS}
\left\{
\begin{array}{l}
\pa_tV-\nu \Delta V+V\cdot\nabla V+\nabla P=0,\\
\na\cdot V=0,\\
V|_{t=0}=V_0(x,y).
\end{array}\right.
\eeq
Here $V=\big(V^1(t,x,y),V^2(t,x,y)\big), P(t,x,y)$ denote the velocity and pressure of the fluid respectively, and $\nu\ge 0$ is the viscosity
coefficient. When $\nu=0$, \eqref{eq:NS} is reduced to the Euler equations.

It is easy to see that $V_{NS}=(-a_0e^{-\nu t}\cos y,0)$ for any constant $a_0$ is a special solution of \eqref{eq:NS} called the bar state or Kolmogorov flow.
The bar states are physically relevant due to the fact that they can be viewed as maximum entropy solutions of the Euler equations,
as well as the fact that they dominate the flow when the aspect ratio of the torus $\delta<1$ \cite{BS}.

The first part of  this paper is devoted to studying the linear inviscid damping for the linearized Euler equations around the Kolmogorov flow $(-\cos y,0)$, which take in
the vorticity formulation:
\ben\label{eq:Euler-L}
\left\{
\begin{array}{l}
\pa_t\omega+\mathcal{L}\omega=0,\\
\om|_{t=0}=\om_0(x,y),
\end{array}\right.
\een
where $\omega=\pa_xV^2-\pa_y V^1$ is the vorticity and
\ben
\mathcal{L}=u(y)\pa_x+u''(y)\pa_x(-\Delta)^{-1},\quad u(y)=-\cos y.
\een

Since the breakthrough work on Landau damping by Mouhot and Villani \cite{MV}, the inviscid damping has been a very
active field as an analogues of Landau damping in hydrodynamics.
Bedrossian and Masmoudi \cite{BM1} proved nonlinear inviscid damping for the 2-D Euler equations around the Couette flow $(y,0) $ for the perturbation in Gevrey class.
On the other hand, Lin and Zeng \cite{LZ} proved that nonlinear inviscid damping is not true for the perturbation in $H^s$ for $s<\f32$. For general shear flows, the linear inviscid damping is also a difficult problem due to the presence of nonlocal part $u''(y)\pa_x(-\Delta)^{-1}$.
In such case, the linear dynamics is associated with the singularities at the critical layers $u=c$ of the solution for the Rayleigh equation
\beno
(u-c)(\phi''-\al^2\phi)-u''\phi=f.
\eeno
Case \cite{Case} gave a first prediction of linear damping for monotone shear flow. His prediction was confirmed by a series of works
\cite{Ros, Ste, Zill1, Zill2}, and finally by our work \cite{WZZ1}. Roughly speaking, if $u(y)\in C^4$ is monotone and $\mathcal{L}$ has no embedding eigenvalues, then
\beno
\|V(t)\|_{L^2}\leq \frac{C}{\langle t\rangle}\|\omega_0\|_{H^{-1}_xH_y^1},\quad\|V^2(t)\|_{L^2}\leq \frac{C}{\langle t\rangle^2}\|\omega_0\|_{H^{-1}_xH_y^2},
\eeno
for the initial vorticity $\om_0$ satisfying $\int_{\mathbb{T}}\om_0(x,y)dx=0$ and $P_{\mathcal{L}}\om_0=0$,
where $P_{\mathcal{L}}$ is the spectral projection to $\sigma_d\big(\mathcal{L}\big)$.

For non-monotone flows such as Poiseuille flow $u(y)=y^2$ and Kolmogorov flow $u(y)=\cos y$, it is even difficult to predict the inviscid damping due to strong
degeneration of the Rayleigh equation at critical points. Based on Laplace tools and numerical computations, Bouchet and Morita \cite{BM} gave a first prediction of
the inviscid damping for non-monotone flows, especially Kolmogorov flow. Instead of vorticity mixing mechanism in the monotone case, they found that the depletion phenomena of the vorticity at the stationary streamlines plays the key mechanism leading to the inviscid damping for non-monotone flows. In our work \cite{WZZ2}, we confirmed the predictions of Bouchet and Morita about the linear damping and the vorticity depletion phenomena for a class of shear flows $(u(y),0)$ in a finite channel denoted by $\mathcal{K}$, which consists of the function $u(y)$ satisfying $u(y)\in H^3(-1,1)$, and  $u''(y)\neq 0$ for critical points(i.e., $u'(y)=0$) and $u'(\pm 1)\neq 0$. Moreover,  for a class of symmetric flows including Poiseuille flow, we obtained the same decay rate of the velocity as in the monotone flow.

In this paper, we extend the linear inviscid damping and the vorticity depletion phenomena to the Kolmogorov flow on the torus.
These results confirm Bouchet and Morita's predictions based on numerical computations.
The first result is the linear inviscid damping.

\begin{theorem}\label{Thm:LinInviDam}
Let $\delta\in (0,1)$ and $\om(t,x,y)$ be the solution of \eqref{eq:Euler-L} with $\int_{\mathbb{T}_{2\pi\d}}\om_0(x,y)dx=0$.
Then it holds that
\begin{itemize}
 \item[1.] if $\om_0(x,y)\in H^{-\f12}_xH^1_y$, then
\beno
\|V(t)\|_{L^2}\leq \frac{C}{\langle t\rangle}\|\omega_0\|_{H^{-\f12}_xH_y^1};
\eeno
\item[2.] if $\om_0(x,y)\in H^{\f12}_xH^2_y$, then
\beno
\|V^2(t)\|_{L^2}\leq \frac{C}{\langle t\rangle^2}\|\omega_0\|_{H^{\f12}_xH_y^2};
\eeno
\item[3.] if $\om_0(x,y)\in H^{\f12}_xH^k_y$ for $k=0,1$, there exists $\om_\infty(x,y)\in H^{\f12}_xH^k_y$ such that
\beno
\|\om(t,x+tu(y),y)-\om_\infty\|_{L^2}\longrightarrow 0\quad \textrm{as}\quad t\rightarrow +\infty.
\eeno
\end{itemize}
\end{theorem}

\begin{remark}\label{Rmk:LinInviodd}
When $\delta=1$, Proposition \ref{Prop: A^2+B^2}, \ref{Prop: A_1^2+B_1^2}, \ref{prop:spectral} still hold for $|\al|\ge 2$. Thus, we can obtain the following decay result:
\begin{itemize}
 \item[1.] if $\om_0(x,y)\in H^{-\f12}_xH^1_y$, then
\beno
\|P_{\geq 2}V(t)\|_{L^2}\leq \frac{C}{\langle t\rangle}\|\omega_0\|_{H^{-\f12}_xH_y^1};
\eeno
\item[2.] if $\om_0(x,y)\in H^{\f12}_xH^2_y$, then
\beno
\|P_{\geq 2}V^2(t)\|_{L^2}\leq \frac{C}{\langle t\rangle^2}\|\omega_0\|_{H^{\f12}_xH_y^2}.
\eeno
\end{itemize}
Here and in what follows, $P_{\geq 2}$ denotes the orthogonal projection of $L^2(\mathbb{T}^2)$ to the subspace $\big\{f\in L^2(\mathbb{T}^2): f(x,y)=\sum\limits_{|\al|\geq 2}\widehat{f}(\al,y)e^{i\al x}\big\}$.
\end{remark}

The second result is the behavior of the solution at critical points $y=k \pi$ for $k\in \mathbb{Z}$, which in particular confirms
the vorticity depletion phenomena found by Bouchet and Morita \cite{BM}.

\begin{theorem}\label{Thm: critical}
Let $\delta\in (0,1)$ and $\om(t,x,y)$ be  the solution of \eqref{eq:Euler-L} with $\int_{\mathbb{T}_{2\pi \d}}\om_0(x,y)dx=0$.
Then for $s>\f12$, it holds that for all $k\in \mathbb{Z}$,
\begin{itemize}
 \item[1.] if $\om_e(x,y)\in H^{1+s}_xH^3_y$, then for $x\in\mathbb{T}_{2\pi\d},\ |t|>1$,
\beno
|\om(t,x,k\pi)|\leq \frac{C}{| t|}\|\omega_e\|_{H^{1+s}_xH_y^3};
\eeno
\item[2.] if $\om_e(x,y)\in H^{1+s}_xH^3_y$, then for $x\in\mathbb{T}_{2\pi\d},\ |t|>1$,
\beno
|V^2(t,x,k\pi)|\leq \frac{C}{|t|^2}\|\omega_e\|_{H^{1+s}_xH_y^3};
\eeno
\item[3.] if $\om_o(x,y)\in H^{s-\f12}_xH^3_y$, then for $x\in\mathbb{T}_{2\pi\d},\ |t|>1$,
\beno
|V^1(t,x,k\pi)|\leq \frac{C}{| t|^{\frac{3}{2}}}\|\omega_o\|_{H^{s-\f12}_xH_y^3}.
\eeno
\end{itemize}
Here $\om_e(x,y)=\f12(\om_0(x,y)+\om_0(x,-y))$ and $\om_o(x,y)=\f12(\om_0(x,y)-\om_0(x,-y))$.
\end{theorem}

Recently, Lin and Xu \cite{LX} proved the linear inviscid damping for shear flows including Kolmogorov flow in
the following sense:
\beno
\lim_{T\to \infty}\f 1 T\int_0^T\|V(t)\|_{L^2}^2dt=0.
\eeno
Their proof is based on the Hamiltonian structure of the linearized Euler equation and an instability index theory recently developed by Lin and Zeng \cite{LZ-R}.
\medskip

The second part of this paper is devoted to studying the 2-D incompressible Navier-Stokes equation with small viscosity on the torus. The vorticity formulation takes
\ben\label{eq:NS-vorticity}
\left\{
\begin{array}{l}
\omega_t-\nu\Delta\omega+V\cdot\nabla \omega=0,\\
\om|_{t=0}=\om_0(x,y).
\end{array}\right.
\een
Experiments and numerical studies \cite{Cou, MS} have shown that the solutions of the nearly inviscid 2-D incompressible Navier-Stokes equations will approach to
some quasi-stationary(metastable) states such as the bar stats and dipole states in a much shorter time than the diffusive time $O(\f 1\nu)$.
Then they dominate the dynamics for very long time intervals. Afterwards they decay on the diffusive time scale.
See \cite{BW-B} for a similar phenomena for Burgers equation with small viscosity.

To explain the metastability phenomena,
Beck and Wayne \cite{BW} considered the linearized Navier-Stokes equations around the bar states, which are relatively easy to analyse from a mathematical view of point. Moreover, the bar states are more physically relevant in the case when the aspect ratio of the torus $\delta<1$ \cite{BS}. The linearized Navier-Stokes equations around the Kolmogorov flow $V_{NS}=(-a_0e^{-\nu t}\cos y,0)$ take as follows
\ben\label{eq:NS-L}
\left\{
\begin{array}{l}
\pa_t\omega+\mathcal{L}_{\nu}(t)\omega=0,\\
\om|_{t=0}=\om_0(x,y),
\end{array}\right.
\een
where
\ben
\mathcal{L}_{\nu}(t)=-\nu \Delta-a_0e^{-\nu t}\cos y\pa_x(1+\Delta^{-1}).
\een

To study the dynamic of \eqref{eq:NS-L}, the main difficulty is that the linearized operator $\mathcal{L}_{\nu}(t)$ is nonself-adjoint, nonlocal and time-dependent.
In \cite{BW}, the authors dropped nonlocal part of $\mathcal{L}_{\nu}(t)$ and considered the following model equation
\beno
\pa_t\omega-\nu \Delta-a_0e^{-\nu t}\cos y\pa_xw=0.
\eeno
Using the hypocoercive method introduced by Villani \cite{Vill},
they proved the enhanced dissipation rate of the solution in some Banach space $X$(see (3.7) in \cite{BW}):
\beno
\|\om(t)\|_X\le Ce^{-M\sqrt{\nu}t}\|w_0\|_X
\eeno
for any $t\in [0,\tau/\nu]$. This decay rate is much faster than the diffusive decay of $e^{-\nu t}$. The mechanism leading to the enhanced dissipation
is due to mixing. See \cite{BZ, CKR, KX, BMV, BWV} and references therein for more relevant works.
Let us mention recent important progress \cite{BGM1,BGM2,BGM3} on the enhanced dissipation for Couette flow $(y,0)$ in $\mathbb{T}\times \R$ applied to the subcritical transition.

For the full linearized operator $\mathcal{L}_{\nu}(t)$,
Beck and Wayne numerically computed the eigenvalues of $-\mathcal{L}_{\nu}(t)$ for a fixed time. The real parts of the eigenvalues are negative and scale like
$\sqrt{\nu}$. Based on these analysis, Beck and Wayne conjectured that
the same result should hold for the full linearized equation. The following theorem gives an affirmative answer to Beck and Wayne's conjecture.

\begin{theorem}
\label{Thm:enhance}
Given $\d\in (0,1)$ and $\tau>0$, there exist constants $c_1>0,\ C>0,$ such that if $\omega$ satisfies \eqref{eq:NS-L} with $\omega_0\in L^2$ and $\int_{\bbT_{2\pi\d}}\omega_0(x,y)dx=0$, then it holds that for $0\leq t\leq \tau/\nu$,
\begin{align*}
&\|{\omega}(t)\|_{L^2}\leq Ce^{-c_1\sqrt{\nu}t}\|{\omega}_0\|_{L^2},\\
&\|V(t)\|_{\dot{H}_x^1L_y^2}\leq \frac{Ce^{-c_1\sqrt{\nu}t}}{\sqrt{1+\nu t^3}}\|{\omega}_0\|_{L^2}.
\end{align*}
When $\delta=1$, it holds that for $0\leq t\leq \tau/\nu$,
\begin{align*}
&\|P_{\geq 2}\om(t)\|_{L^2}\leq Ce^{-c_1\sqrt{\nu}t}\|P_{\geq 2}\om_0\|_{L^2},\\
&\|P_{\geq 2}V(t)\|_{\dot{H}_x^1L_y^2}\leq \frac{Ce^{-c_1\sqrt{\nu}t}}{\sqrt{1+\nu t^3}}\|P_{\geq 2}{\omega}_0\|_{L^2}.
\end{align*}
\end{theorem}

\begin{remark}
When $\delta=1$, the operator $1+\Delta^{-1}$ has two additional kernels $\{\sin x, \cos x\}$, which yield two
family of exact solutions $\{e^{-\nu t}\sin x, e^{-\nu t}\cos x\}$ of the linearized Navier-Stokes equations.
Thus, we have to make the projection $P_{\geq 2}$ for the solution in order to obtain the enhanced dissipation.
\end{remark}

In \cite{LX},  Lin and Xu  proved the enhanced dissipation in the sense: if $\int_{\bbT_{2\pi\d}}\omega_0(x,y)dx=0$, then for any $\tau>0$ and $\epsilon>0$, if $\nu$
is small enough,
\begin{align*}
&\|P_{\geq 2}\om(\tau/\nu)\|_{L^2}\leq \epsilon\|P_{\geq 2}\om_0\|_{L^2}\quad\text{for}\quad \d=1,\\
&\|\om(\tau/\nu)\|_{L^2}\leq \epsilon\|\om_0\|_{L^2}\quad\text{for}\quad \d<1.
\end{align*}
The proof is based on the Hamiltonian structure of the linearized Euler equation and RAGE theorem as in \cite{CKR}.

Very recently, Ibrahim, Maekawa and Masmoudi \cite{IMM} proved
the same decay rate of the vorticity for the following equation:
\beno
\pa_t\omega+\mathcal{L}_{\nu}w=0,\quad \mathcal{L}_{\nu}=-\nu \Delta-a\cos y\pa_x(1+\Delta^{-1}),
\eeno
which is the linearized equation of the 2-D Navier-Stokes equation with a force $(-a\nu\cos y,0)$ around the Kolmogorov flow $(-a\cos y,0)$.
Their proof is based on the pseudospectral bound of $\mathcal{L}_{\nu}$.
Pseudospectra is an important concept in understanding the hydrodynamic stability \cite{Tre, TRD} due to the non-normality of the linearized operators. Gallay et al. applied the pseudospectra method to study the stability of the Lamb-Oseen vortex in the regime of high circulation Reynolds number \cite{GGN, DW, Mae, GM, LWZ, Ga}. Recently, McQuighan and Wayne \cite{MW} applied a similar method to revisit the metastability problem of the viscous Burgers equation.  However, it seems difficult to apply the pseudospectra method to the time-dependent operators such as $\mathcal{L}_{\nu}(t)$.

As announced in a work by Li and the first two authors \cite{LWZ}, the wave operator method  could be used to solve Beck and Wayne's conjecture. In \cite{LWZ}, we solved Gallay's conjecture on pseudospectral and spectral bounds on the Oseen vortices operator. One of the key ideas is to transform a nonlocal operator into a local one by constructing a wave operator. In this work, the analysis on linear inviscid damping naturally gives rise to our desired wave operator. Moveover, we also use the completely different method to handle the local operator.

Let us point out that the extra decay factor $\frac 1{\sqrt{1+\nu t^3}}$ of the velocity in Theorem \ref{Thm:enhance}
comes from the inviscid damping mechanism of  Kolmogorov flow.
Even for the time-independent operator $\mathcal{L}_{\nu}$, it seems difficult to deduce this decay estimate from the resolvent estimate.
This decay estimate of velocity plays an important role in the following nonlinear enhanced dissipation result.

\begin{theorem}
\label{Thm:NS}
Given $\d\in (0,1), \tau>0, \gamma>\f23$ and $C_1>1$, there exist constants $c_2,\ c_3\in(0,1),$ such that if $0<\nu<c_2$,  $ \omega_0\in L^2$, $\|(I-P_2)\omega_0\|_{L^2}\leq \nu^{\gamma},$ and $C_1^{-1}\leq\|P_2\omega_0\|_{L^2}\leq C_1$,
then the solution to \eqref{eq:NS-vorticity} satisfies
\begin{align*}
\|P_{\neq0}{\omega}(t)\|_{L^2}\leq Ce^{-c_3\sqrt{\nu}t}\|P_{\neq0}{\omega}_0\|_{L^2}
\end{align*}
for all $0<t<\tau/\nu$. Here $P_2$ denotes the orthogonal projection of $L^2(\mathbb{T}_{\d}^2)$ to the subspace $W_2=\text{span}\{\sin y,\cos y\}$.
\end{theorem}

\begin{remark}
The stability threshold $\nu^\gamma$ for $\gamma>\f23$ may be not optimal. We can only achieve the stability threshold $\nu^\f34$
if we do not use the enhanced dissipation with an extra decay factor of the velocity.
\end{remark}

\begin{remark}
The case of $\delta=1$ is a challenging problem. In this case, we need to consider the linearized Navier-Stokes equations around the dipole states such as
\beno
e^{-\nu t}(-\sin y, \sin x),\quad  e^{-\nu t}(-\cos y, \cos x).
\eeno
\end{remark}

\section{Sketch of the proof}

In this section, we present a sketch of the proof of main theorems.

\subsection{Linear inviscid damping}
The basic ideas are similar to \cite{WZZ1, WZZ2}. We introduce the stream function $\psi$ so that $V=(\pa_y\psi,-\pa_x\psi)$ and $-\Delta \psi=\om$.
In terms of $\psi$, the linearized Euler equation \eqref{eq:Euler-L} takes
\beno
\partial_t\Delta{\psi}+u(y)\partial_x\Delta{\psi}-u''(y)\partial_x{\psi}=0,
\eeno
here $u(y)=-\cos y$. Taking the Fourier transform in $x$, we get
\beno
(\partial_y^2-\al^2)\partial_t\widehat{\psi}=i\al\big(u''(y)-u(y)(\partial_y^2-\al^2)\big)\widehat{\psi}.
\eeno
Inverting the operator $(\partial_y^2-\al^2)$, we obtain
\ben\label{eq:Euler-L-operator}
-\frac{1}{i\al}\partial_t\widehat{\psi}=\mathcal{R}_\al\widehat{\psi},
\een
where
\ben\label{def:ray ope}
\mathcal{R}_\al\widehat{\psi}=-(\partial_y^2-\al^2)^{-1}\big(u''(y)-u(\partial_y^2-\al^2)\big)\widehat{\psi}.
\een
Let $\Omega$ be a simple connected domain including the spectrum $\sigma(\mathcal{R}_\al)$ of $\mathcal{R}_\al$.
We have the following representation formula of the solution to (\ref{eq:Euler-L-operator}):
\ben\label{eq:stream formula}
\widehat{\psi}(t,\al,y)=\frac{1}{2\pi i}\int_{\partial\Omega}
e^{-i\al tc}(c-\mathcal{R}_\al)^{-1}\widehat{\psi}(0,\al,y)dc.
\een
In this way, the large time behaviour of the solution $\widehat{\psi}(t,\al,y)$ is reduced to the study of the resolvent $(c-\mathcal{R}_\al)^{-1}$.

Let $\Phi$ be a solution of the inhomogeneous Rayleigh equation:
\ben\label{eq:Ray-ihom}
\left\{
\begin{aligned}
&\pa_y^2\Phi-\al^2\Phi-\frac{u''}{u-c}\Phi=f,\\
&\pa_y\Phi(-\pi)=\pa_y\Phi(\pi),\quad \Phi(-\pi)=\Phi(\pi),
\end{aligned}
\right.
\een
with $f(\al,y,c)=\f{\widehat{\om}_0(\al,y)}{i\al (u(y)-c)}$.
Then we find that
\ben
(c-\mathcal{R}_\al)^{-1}\widehat{\psi}(0,\al,y)=i\al\Phi.
\een

To study the large time behaviour of $\widehat{\psi}$, one of key ingredients is to establish the limiting absorption principle:
\beno
\Phi(\al,y,c\pm i\varepsilon)\to \Phi_\pm(\al,y,c)\quad \text{for}\,\, c\in \text{Ran}\,u,
\eeno
as $\varepsilon\to 0$. In \cite{WZZ2}, we established the limiting absorption principle for shear flows in the class $\mathcal{K}$, and proved that $\Phi_\pm$
is bounded in $H_y^1$.
These information is enough to show that the velocity decays to 0 as the time tends to infinity in $L^2$ sense.
These results can be easily extended to the Kolmogorov flow.

To obtain the explicit decay estimates of the velocity, we need to know more precise behaviour of the limit function $\Phi_\pm$.
As in the case of symmetric flows, we decompose the solution of \eqref{eq:Ray-ihom} into the odd part and even part.
Each part can be handled as in the monotonic flow. Main difference is that $\f 1 {u(y)-c}$ is more singular for $c=u(0)$ and $c=u(\pi)$.
For the symmetric flow considered in \cite{WZZ2}, $\f 1 {u(y)-c}$ is  singular only at $c=u(0)$. Due to the singularity of $\f 1 {u(y)-c}$, we can only establish the uniform estimates
of the solution of the homogeneous Rayleigh equation in the weighted space:
\ben\label{eq:Rayleigh-H}
(u-c)(\phi''-\al^2\phi)-u''\phi=0.
\een
The details will be presented in section 3.

In section 4, we solve the inhomogeneous Rayleigh equation. In section 5, we prove the limiting absorption principle and give the solution formula  of the linearized Euler equations. Again, the decay estimates of the velocity will be proved by using the dual method. More precisely, for $f=(\pa_y^2-\al^2)g$ with $g\in H^2(0,\pi)\cap H_0^1(0,\pi)$,
\begin{align*}
&\int_0^{\pi}\widehat{\psi}_o(t,\al,y)f(y)dy
=-\int_{u(0)}^{u(1)}K_o(c,\al)e^{-i\al ct}dc,
\end{align*}
and for $f=(\pa_y^2-\al^2)g$ with $g\in H^2(0,\pi)$ and $g'(0)=g'(\pi)=0$,
\begin{align*}
&\int_0^{\pi}\widehat{\psi}_e(t,\al,y)f(y)dy=-\int_{u(0)}^{u(1)}K_e(c,\al)e^{-i\al ct}dc,
\end{align*}
where $\widehat{\psi}_o$ and $\widehat{\psi}_o$ are the odd part and even part of $\widehat{\psi}$ respectively.
Thus, the decay estimates are reduced to the regularity of the kernels $K_o$ and $K_e$.
In section 6, we prove the decay estimates of the velocity under the regularity assumptions on the kernels. Section 7 and section 8
are devoted to the regularity of the kernels.

In section \ref{critical points}, we prove the decay estimates of the vorticity and velocity at critical points.

\subsection{Enhanced dissipation}

Motivated by \cite{LWZ}, we will use the wave operator method to study the enhanced dissipation.
We will construct a wave operator $\bbD$, which satisfies the following basic properties:
\begin{itemize}

\item[1.] $\bbD\big(\cos y(1+(\pa_y^2-\al^2)^{-1})\om\big)=\cos y\bbD(\om);$

\item[2.] $(1-\al^{-2})\|\om\|_{L^2}^2\leq \|\bbD(\om)\|_{L^2}^2\leq \|\om\|_{L^2}^2$;

\item[3.] $\|\sin y(\bbD(\pa_y^2\omega)-\pa_y^2\bbD(\omega))\|_{L^2}\leq C(|\al|\|\omega\|_{L^2}+\|\pa_y\omega\|_{L^2})$.

\end{itemize}
See section 10 for more properties.

In section 11, we establish the linear enhanced dissipation.
First of all, we study the decay estimates for the model without nonlocal term in a short time scale $\nu^{-\f12}$. Our method is based on the change of unknown, which is very different from the hypocoercive method used by Beck and Wayne. Then we use the wave operator method to establish the decay estimates for the model with nonlocal term. Finally, the linear enhanced dissipation was obtained by using some kind of iteration argument.

In section 12, we establish the nonlinear enhanced dissipation. The proof
is based on the linear enhanced dissipation, especially the decay estimate
of the velocity.  Another key observation is that the basic energy dissipation ensures that if $C_1^{-1}\leq \|P_2\omega(0)\|_{L^2}\leq C_1$, then
\beno
{C}^{-1}\leq \|P_2\omega(t)\|_{L^2}\leq C\quad \text{for}\,\, 0<t<\tau/\nu.
\eeno
The proof used the iteration argument again.

\section{The homogeneous Rayleigh equation}\label{Rayleigh equations}
To solve \eqref{eq:Ray-ihom}, we first construct a smooth solution of the homogeneous Rayleigh equation on $[-\pi,\pi]$ with $u(y)=-\cos y$:
\beno
(u-c)(\phi''-\al^2\phi)-u''\phi=0,
\eeno
where the complex constant $c$ will be taken in four kinds of domains given by
\beno
&&D_0\triangleq\big\{c\in (-1,1)\big\},\\
&&D_{\epsilon_0}\triangleq\big\{c=c_r+i\epsilon,~c_r\in [-1,1], 0<|\epsilon|<\epsilon_0\big\},\\
&&B_{\epsilon_0}^l\triangleq\big\{c=-1+\epsilon e^{i\theta},~0<\epsilon<\epsilon_0,~\frac{\pi}{2}\leq \theta\leq\frac{3\pi}{2}\big\},\\
&&B_{\epsilon_0}^r\triangleq\big\{c=1-\epsilon e^{i\theta},~0<\epsilon<\epsilon_0,~\frac{\pi}{2}\leq \theta\leq\frac{3\pi}{2}\big\},
\eeno
for some $\epsilon_0\in (0,1)$. We denote
\ben\label{eq:Omega}
\Om_{\epsilon_0}\triangleq \overline{D}_0\cup D_{\epsilon_0}\cup B_{\epsilon_0}^l\cup B^r_{\epsilon_0}.
\een

\subsection{Rayleigh integral operator}

Given $|\al|\ge 1$, let $A$ be a constant larger than $C|\al|$ with $C\ge 1$ independent of $\al$.
\begin{definition}\label{Def: 1}
For a function $f(y,c)$ defined on $[0,\pi]\times \Om_{\epsilon_0}$, for $k=0,1,2$ and $j=0,1$, we define
\beno
&&\|f\|_{X^k_0}\eqdef\sup_{(y,c)\in [0,\pi]\times D_0}\bigg|\frac{u'(y_c)^kf(y,c)}{\cosh(A(y-y_c))}\bigg|,\\
&&\|f\|_{X_j}\eqdef\sup_{(y,c)\in [0,\pi]\times \Om_{\epsilon_0}}\bigg|\frac{u'(y_c)^jf(y,c)}{\cosh(A(y-y_c))}\bigg|,
\eeno
where $y_c\in [0,\pi]$ satisfies $u(y_c)=c_r$ with
\ben\label{def:cr}
c_r=\mathrm{Re}\, c \quad \textrm{for }c\in D_{\epsilon_0}\cup \overline{D}_0,\quad c_r=-1\quad \textrm{for }c\in B_{\epsilon_0}^l,
\quad c_r=1\quad \textrm{for }c\in B_{\epsilon_0}^r.
\een
\end{definition}

\begin{definition}
For a function $f(y,c)$ defined on $[0,\pi]\times \Om_{\epsilon_0}$, we define
\beno
&&\|f\|_{Y_C}\eqdef\sum_{k=0}^2\sum_{\b+\gamma=k}A^{-k}
\|\partial_y^{\b}\partial_c^{\gamma}f\|_{X^{\gamma}_0}+A^{-3}\|\pa_c^2\pa_yf\|_{X^3_1}.
\eeno
where
\beno
\|f\|_{X^3_1}\eqdef\sup_{(y,c)\in [0,\pi]\times D_0}\bigg|\frac{Au'(y_c)^2f(y,c)}{\cosh(A(y-y_c))(A+u'(y_c)/u_1(y,c)^2)}\bigg|
\eeno
with $u_1(y,c)=\f{\cos y_c-\cos y}{y-y_c}$.
\end{definition}

\begin{definition}
Let $y_c$ be as in Definition \ref{Def: 1} with $c_r$ defined by (\ref{def:cr}).
The Rayleigh integral operator $T$ is defined by
\beno
T(f)\eqdef T_0\circ T_{2,2}(f)=\int_{y_c}^y\f 1 {(u(y')-c)^2}\int_{y_c}^{y'}f(z,c)(u(z)-c)^2dzdy',
\eeno
where
\beno
&&T_0f(y,c)\eqdef\int_{y_c}^yf(z,c)dz,\\
&&T_{k,j}f(y,c)\eqdef\frac{1}{(u(y)-c)^j}\int_{y_c}^yf(z,c)(u(z)-c)^kdz,\quad j\le k+1.
\eeno
\end{definition}

Let us first list some basic properties of $\cos y$ and $\sin y$, which will be used frequently.
The proof is easy and we omit it.

\begin{lemma}\label{Rmk:u_1}
There exists a constant $C$ so that for any $y,y_c\in [0,\pi]$,
\begin{align*}
&C^{-1}\max_{|y'-y_c|+|y'-y|=|y-y_c|}\sin y'\leq \f{\cos y_c-\cos y}{y-y_c}\leq C\max_{|y'-y_c|+|y'-y|=|y-y_c|}\sin y',\\
&C^{-1}\sin \f{y+y_c}{2}\leq \f{\cos y_c-\cos y}{y-y_c}\leq C\sin \f{y+y_c}{2},\quad
\f{1}{\pi}y(\pi-y)\leq \sin y \leq \f{4}{\pi^2}y(\pi-y).
\end{align*}
In particular, we have
\beno
&&\Big|\f{(y-y_c)(\sin y+\sin y_c)}{\cos y_c-\cos y}\Big|\leq C,\\
&&\Big|\int_{y}^{y_c}\f{(y-y_c)^2}{(\cos y_c-\cos y)^2}dy'\Big|\leq \f{C}{\sin y_c}.
\eeno
\end{lemma}

Throughout this paper, we denote
\ben\label{eq:u_1}
u_1(y,c)=\f{\cos y_c-\cos y}{y-y_c}=\int_0^1\sin(y_c+t(y-y_c))dt.
\een

\begin{lemma}\label{lem:T-bound}
There exists a constant $C$ independent of $A$ such that
\beno
&&\|Tf\|_{Y_C}\leq \frac{C}{A^2}\|f\|_{Y_C},\quad\|Tf\|_{X_0}\leq \frac{C}{A^2}\|f\|_{X_0}.
\eeno
Moreover, if $f\in C\big([0,\pi]\times \Omega_{\epsilon_0}\big)$, then
\beno
T_0f,\,\,T_{2,2}f,\,\,Tf\in C\big([0,\pi]\times \Omega_{\epsilon_0}\big).
\eeno
\end{lemma}

\begin{proof}
Let us prove the first inequality.
A direct calculation shows that
\begin{align}
\|T_0f\|_{X^k_0}
=&\sup_{(y,c)\in[0,\pi]\times D_0}\bigg|
\frac{1}{\cosh A(y-y_c)}\int_{y_c}^y
\frac{u'(y_c)^kf(z,c)}{\cosh A(z-y_c)}\cosh A(z-y_c)dz\bigg|\nonumber\\
\leq &\sup_{(y,c)\in[0,\pi]\times D_0}\bigg|
\frac{1}{\cosh A(y-y_c)}\int_{y_c}^y
\cosh A(z-y_c)dz\bigg|\|f\|_{X^k_0}\nonumber\\
\leq &\frac{1}{A}\|f\|_{X^k_0}.\label{eq:T0-est}
\end{align}

Using the fact that for any $0<y<y'<y_c$ or $y_c<y'<y<\pi$,
\ben\label{eq:u-mono}
|u(y')-u(y_c)|\leq |u(y)-u(y_c)|,
\een
we deduce that
\begin{align}
\|T_{2,2}f\|_{X^k_0}
\leq &\sup_{(y,c)\in[0,\pi]\times D_0}
\bigg|\frac{y-y_c}{\cosh A(y-y_c)}
\int_0^1
\cosh tA(y-y_c)dt
\bigg|\|f\|_{X^k_0}\nonumber\\
\leq& \frac{C}{A}\|f\|_{X^k_0},\label{eq:T22-est}
\end{align}
which along with \eqref{eq:T0-est} shows that for $k\geq 0$,
\ben\label{eq:T-bound-1}
\|Tf\|_{X^k_0}\le \frac C {A^2}\|f\|_{X^k_0}.
\een

Using the fact that for any $0<y<y'<y_c$ or $y_c<y'<y<\pi$, $|u(y')-u(y_c)|\leq |u(y)-u(y_c)|$, thus $|u(y')-c|\leq |u(y)-c|$, we get by Lemma \ref{Rmk:u_1} that
\begin{align}
\nonumber&\|T_{k,k+1}f\|_{X^{j+1}_0}+\|u'(y)T_{k,k+1}f\|_{X^j_0}\\
\nonumber&\leq C\sup_{(y,c)\in[0,\pi]\times D_0}
\bigg|\frac{1}{\cosh A(y-y_c)}\int_0^1
\cosh tA(y-y_c)dt
\bigg|\|f\|_{X^j_0}\\
&\leq  C\|f\|_{X^j_0}.\label{eq:Tk-est}
\end{align}

It is easy to see that
\beno
&&\partial_yTf(y,c)=T_{2,2}f(y,c),\\
&&\partial_cTf(y,c)=2T_0\circ T_{2,3}f(y,c)-2T_0\circ T_{1,2}f(y,c)+T\partial_cf(y,c),\\
&&\partial_y^2Tf(y,c)=-2u'(y)T_{2,3}f(y,c)+f(y,c).
\eeno
Then it follows from \eqref{eq:T0-est}, \eqref{eq:T22-est} and  \eqref{eq:Tk-est} that
\ben\label{eq:T-bound-2}
&\|\partial_{y}Tf\|_{X_0^0}+\frac{1}{A}\|\partial_{y}^2Tf\|_{X_0^0}+\|\partial_{c}Tf\|_{X_0^1}\leq \frac{C}{A}\|f\|_{X_0^0}
+\frac C {A^2}\|\partial_{c}f\|_{X_0^1}.
\een

It is easy to see that
\begin{align*}
T_{k,k+1}f(y,c)
&=\int_0^1f(y_c+t(y-y_c),c)\frac{t^{k}u_1(y_c+t(y-y_c),c)^{k}}{u_1(y,c)^{k+1}}dt,
\end{align*}
from which and the fact that $|\pa_c u_1(y,c)|\leq \f{C}{u'(y_c)}$, we can deduce that
\begin{align}
\left|\frac{u'(y_c)^j\pa_cT_{k,k+1}f(y,c)}{\cosh A(y-y_c)}\right|
\leq C\left(\frac{\|f\|_{X_0^{j-1}}}{u_1(y,c)^2}+\frac{\|\pa_{y}f\|_{X_0^{j-1}}+\|\pa_{c}f\|_{X_0^{j}}}{u_1(y,c)}\right).\label{eq:Tk-est2}
\end{align}
Direct calculations show that
\begin{align*}
\partial_y\partial_cTf(y,c)
&=2T_{2,3}f(y,c)-2T_{1,2}f(y,c)+T_{2,2}\partial_cf(y,c),\\
\partial_y\partial_c^2Tf(y,c)
&=2\partial_cT_{2,3}f(y,c)-2\partial_cT_{1,2}f(y,c)\\
&\quad+2T_{2,3}\partial_cf(y,c)
-2T_{1,2}\partial_cf(y,c)+T_{2,2}\partial_c^2f(y,c),
\end{align*}
and
\begin{align*}
\partial_c^2Tf(y,c)=&2\pa_cT_0T_{2,3}f(y,c)-2\pa_cT_0T_{1,2}f(y,c)+\partial_cT\partial_cf(y,c)\\
=&2T_0\partial_cT_{2,3}f(y,c)-2T_0\partial_cT_{1,2}f(y,c)+\frac{f(y_c,c)}{3u'(y_c)^2}\\
&+2T_0T_{2,3}\partial_cf(y,c)-2T_0T_{1,2}\partial_cf(y,c)+T\partial_c^2f(y,c).
\end{align*}
It follows from \eqref{eq:Tk-est} and \eqref{eq:T22-est} that
\ben
\|\partial_y\partial_cTf\|_{X_0^1}\leq C\|f\|_{X_0^0}+\frac C {A}\|\partial_{c}f\|_{X_0^1}.\label{eq:T-est3}
\een

By \eqref{eq:Tk-est2}, we have
\begin{align*}
\left|\frac{u'(y_c)^2T_0\pa_cT_{k,k+1}f(y,c)}{\cosh A(y-y_c)}\right|
\le& \f {u'(y_c)} {\cosh A(y-y_c)}\int_{y_c}^y\cosh A(y'-y_c)\left|\frac{u'(y_c)\pa_cT_{k,k+1}f(y',c)}{\cosh A(y'-y_c)}\right|dy'\\
\le& Cu'(y_c)\left|\int_{y_c}^y\f 1 {u_1(y',c)^2}dy'\right|\|f\|_{X_0^0}+\f C A\big(\|\partial_{y}f\|_{X_0^0}+\|\partial_{c}f\|_{X_0^1}\big)\\
\le& C\|f\|_{X_0^0}+\f C A\big(\|\partial_{y}f\|_{X_0^0}+\|\partial_{c}f\|_{X_0^1}\big),
\end{align*}
from which, \eqref{eq:T0-est} and \eqref{eq:Tk-est}, we infer that
\ben\label{eq:T-bound-3}
\|\partial_c^2Tf\|_{X_0^{2}}\leq C\|f\|_{X_0^0}+\frac C {A}\|\partial_{y}f\|_{X_0^0}+\frac C {A}\|\partial_{c}f\|_{X_0^1}+\frac C {A^2}\|\partial_{c}^2f\|_{X_0^2},
\een

By \eqref{eq:Tk-est} and \eqref{eq:Tk-est2}, we get
\beno
\left|\frac{u'(y_c)\partial_y\partial_c^2Tf(y,c)}{\cosh A(y-y_c)}\right|\leq \frac {C \|f\|_{X^0_0}}{u_1(y,c)^2}+ \frac{C \|\partial_{y}f\|_{X^0_0}}{u'(y_c)}+ \frac{C \|\partial_{c}f\|_{X^1_0}}{u'(y_c)}+\frac {C \|\partial_{c}^2f\|_{X^2_0}}{Au'(y_c)},
\eeno
which implies that
\ben \label{eq:T-est5}
\|\pa_y\pa_c^2Tf\|_{X^{3}_1}\le CA\|f\|_{X^0_0}+ C\|\partial_{y}f\|_{X^0_0}+ C \|\partial_{c}f\|_{X^1_0}+\frac C {A}\|\partial_{c}^2f\|_{X^2_0}.
\een

Putting (\ref{eq:T-bound-1}), \eqref{eq:T-bound-2} and (\ref{eq:T-est3})-\eqref{eq:T-est5} together, we conclude the first inequality.
The estimate of $T$ in ${X}_0$ norm is similar to \eqref{eq:T-bound-1}.

Now we check the continuity. We have
\beno
T_{2,2}(f)=\int_0^1\f{(y-y_c)(u(y_c+t(y-y_c))-c)^2}{(u(y)-c)^2}f(y_c+t(y-y_c),c)dt.
\eeno
Using the fact that for any $0<y<y'<y_c$ or $y_c<y'<y<\pi$, $\left|\frac{(u(y')-c)^2}{(u(y)-c)^2}\right|\leq 1$ and the continuity of $\f{(y-y_c)(u(y_c+t(y-y_c))-c)^2}{(u(y)-c)^2}$, and Lebesgue's dominated convergence theorem, we conclude the continuity of $T_{2,2}f$.
The continuity of $Tf$ follows from $Tf=T_0\circ T_{2,2}f$.
\end{proof}

\subsection{Existence of the solution}

We solve the homogeneous Rayleigh equation on $[0,\pi]$:
\beq\label{eq:homRay}
\left\{
\begin{array}{l}
\phi''-\al^2\phi-\frac{u''}{u-c}\phi=0,\\
\phi(y_c,c)=u(y_c)-c,\quad \phi'(y,c)\big|_{y=y_c}=u'(y_c).
\end{array}\right.
\eeq

\begin{proposition}\label{prop:Rayleigh-Hom}
There exists  $\phi_1(y,c)\in C\big([0,\pi]\times \Omega_{\epsilon_0}\big)$ such that $\phi(y,c)=(u(y)-c)\phi_1(y,c) $ is a solution
of the Rayleigh equation \eqref{eq:homRay}. Moreover, $\pa_y\phi_1(y,c)\in C\big([0,\pi]\times \Omega_{\epsilon_0}\big)$, and  there exists $\epsilon_1>0, C>0$ such that for any $\epsilon_0\in[0,\epsilon_1)$ and $(y,c)\in [0,\pi]\times \Om_{\epsilon_0}$,
\beno
&&|\phi_{1}(y,c)|\ge \f12,\quad |\phi_{1}(y,c)-1|\leq C |y-y_c|^2,
\eeno where the constants $\epsilon_1, C$ may depend on $\al$.
 \end{proposition}

 \begin{proof}
 By Lemma \ref{lem:T-bound}, the operator $1-\al^2T$ is invertible in the spaces $Y_C$ and $X_0\cap C\big([0,\pi]\times \Omega_{\epsilon_0}\big).$ Therefore, we may take $$\phi_1(y,c)=(1-\al^2T)^{-1}1=\sum_{k=0}^{+\infty}\al^{2k}T^k1, $$
 which converges in both $Y_C$ and $X_0\cap C\big([0,\pi]\times \Omega_{\epsilon_0}\big)$. Then $\phi_1(y,c)\in C\big([0,\pi]\times \Omega_{\epsilon_0}\big)$. Moreover,
 for $c\in D_0$, we have
 \beno
 T^k1(y,c)\geq0,\quad \phi_1(y,c)\geq 1,
 \eeno
which ensures that there exists $\epsilon_1>0$ so that for any $\epsilon_0\in[0,\epsilon_1)$ and $(y,c)\in [0,\pi]\times \Om_{\epsilon_0}$,
\beno
|\phi_1(y,c)|\ge \f12,\ |\phi_1(y,c)|\le C.
\eeno

Thanks to $\phi_1=1+\al^2T\phi_1$, we have
\begin{align}
\phi_1(y,c)=&1+\int_{y_c}^y\frac{\al^2}{(u(y')-c)^2}\int_{y_c}^{y'}\phi_{1}(z,c)(u(z)-c)^2dzdy',
\label{eq:phi1}
\end{align}
from which, it follows that
\beno
\big((u-c)^2\phi_{1}'\big)'=\al^2\phi_{1}(u-c)^2,\quad \phi_1(y_c,c)=1.
\eeno
Let $\phi(y,c)=(u(y)-c)\phi_1(y,c)$. Then $\phi(y,c)$ satisfies the Rayleigh equation \eqref{eq:homRay}.

It follows from \eqref{eq:phi1} that
\begin{align*}
|\phi_{1}(y,c)-1|\leq &\al^2\int_{y_c}^y\int_{y_c}^{z}|\phi_{1}(y',c)|\left|\frac{u(y')-c}{u(z)-c}\right|^2dy'dz
\leq C|y-y_c|^2.
\end{align*}
As $\pa_y\phi_1(y,c)=\al^2T_{2,2}\phi_1(y,c),$ by Lemma \ref{lem:T-bound}, we have $\pa_y\phi_1(y,c)\in C\big([0,1]\times \Omega_{\epsilon_0}\big)$.
\end{proof}

We need the following extra properties of $\phi_1(y,c)$.

\begin{lemma}\label{lem:phi1}
For $c\in D_0$, it holds that
\begin{itemize}
\item[1.] $\phi_1(y,c)\geq 1$;
\item[2.] $\pa_y\phi_1(y,c)>0$ for $y\in (y_c,\pi]$ and $\pa_y\phi_1(y,c)<0$ for $y\in [0,y_c)$;
\item[3.] for any given $M_0> 0$, there exists a constant $C$ only depending on $M_0$ such that for
$\b+\g\leq 2$, $|y-y_c|\leq M_0/|\al|$
\beno
&&|u'(y_c)^\gamma\partial_y^{\b}\partial_c^{\gamma}\phi_1(y,c)|\leq C|\al|^{\beta+\gamma},\\
&&|u'(y_c)^2\pa_y\pa_c^2\phi_1(y,c)|\leq C\left(|\al|^3+\f{\al^2u'(y_c)}{u_1(y,c)^2}\right).
\eeno
\end{itemize}
\end{lemma}

\begin{proof}
The first two properties are a direct consequence of \eqref{eq:phi1}. By Lemma \ref{lem:T-bound}, we have $\|\phi_1\|_{Y_C}\leq 2 $ if we take $A=C|\al|$ for some constant $C$ independent of $\al$. The third property follows from the definition of $Y_C$ norm.
\end{proof}
\subsection{Uniform estimates of  the solution}

The goal of this subsection is to establish some uniform estimates in wave number $\al$ for the solution $\phi(y,c)$ for $c\in D_0$ of the Rayleigh equation given by Proposition \ref{prop:Rayleigh-Hom}.

Without loss of generality, we always assume $\al\ge 1$ in the sequel.
We introduce
\beno
&&\cF(y,c)=\f{\pa_y\phi_1(y,c)}{\phi_1(y,c)},\quad \cG(y,c)=\f{\pa_c\phi_1(y,c)}{\phi_1(y,c)},\\
&&\cG_1(y,c)=\f{\cF}{u'(y_c)}+\cG=\f{1}{\phi_1}\Big(\f{\pa_y}{u'(y_c)}+\pa_c\Big)\phi_1,
\eeno
where $\phi_1(y,c)=\frac {\phi(y,c)} {u(y)-c}$ and $y_c\in [0,\pi]$ satisfying $c=u(y_c)$ for $c\in D_0$.

\begin{proposition}\label{prop:phi1}
There exists a constant $C$ independent of $\al$ such that
\begin{align*}
&\phi_1(y,c)-1\leq C\min\{\al^2|y-y_c|^2,1\}\phi_1(y,c),\\
&C^{-1}\al\min\{\al|y-y_c|,1\}\leq |\cF(y,c)|\leq C\al\min\{\al|y-y_c|,1\},\\
&C^{-1}e^{C^{-1}\al |y-y_c|}\leq \phi_1(y,c)\leq e^{C\al|y-y_c|},\\
&|\partial_y^{\b}\partial_c^{\gamma}\phi_1(y,c)|\leq C \al^{\b+\g}\phi_1(y,c)/u'(y_c)^{\gamma},\quad \beta+\gamma\le 2,\\
&|\pa_c\phi_1(y,c)|\leq C\phi_1(y,c)\f{\al\min\{1,\al|y-y_c|\}}{u'(y_c)},
\end{align*}
and
\begin{align*}
&\Big|\Big(\frac{\partial_y}{u'(y_c)}+\partial_c\Big)\phi_1(y,c)\Big|\leq C\frac{\min\{1,\al^2|y-y_c|^2\}\phi_1}{u'(y_c)^2},\\
&\Big|\Big(\frac{\partial_y}{u'(y_c)}+\partial_c\Big)^2\phi_1(y,c)\Big|\leq C\frac{\min\{1,\al^2|y-y_c|^2\}\phi_1}{u'(y_c)^4}.
\end{align*}
\end{proposition}
\begin{proof}
{\bf Step 1.} Estimates of $\cF$ and $\pa_y^k\phi_1$ for $k=0,1,2$.
\smallskip

Recall that $\phi_1=1+\al^2T(\phi_1)$, which implies that
\ben\label{eq:phi-est1}
\phi_1(y,c)-1\leq C\min\{\al^2|y-y_c|^2,1\}\phi_1(y,c).
\een

It is easy to check that $\cF$ satisfies
\ben
\cF'+\cF^2+\f{2u'}{u-c}\cF=\al^2,\quad \cF(y_c,c)=0.\label{eq:cF}
\een
Thanks to $\cF(y,c)\geq 0$ for $y\geq y_c$ and $\cF(y,c)\leq 0$ for $y\leq y_c$, we have  $\f{2u'}{u-c}\cF\ge 0$.
Due to $\phi_1'(y_c,c)=0$, we get
\beno
\lim_{y\to y_c}\f{u'(y)\mathcal{F}(y,c)}{u-c}=\lim_{y\to y_c}\f{u'(y_c)\phi_1'(y,c)}{u-c}=\phi_1''(y_c,c)=\mathcal{F}'(y_c,c),
\eeno
which along with \eqref{eq:cF} implies that $\mathcal{F}'(y_c,c)=\f{\al^2}{3}>0$.
Take $y_0\in [y_c,\pi]$ such that $\mathcal{F}(y_0,c)=\sup\limits_{y\in [y_c,\pi]}\mathcal{F}(y,c)$. Then $y_0>y_c,\ \mathcal{F}'(y_0,c)\geq0$ and $\al^2\geq (\cF'+\cF^2)(y_0,c)\geq \cF^2(y_0,c),$ which implies that  $\mathcal{F}(y,c)\leq\cF(y_0,c)\leq \al$ for $y\in [y_c,\pi]$. For $y\in [0,y_c]$,  we consider the minimal point of $\mathcal{F}(y,c)$, and obtain $-\al\leq \mathcal{F}(y,c)\leq 0$.
This shows that $|\cF(y,c)|\leq \al$, which in particular gives
\ben
e^{-\al|y-y'|}\leq \f{\phi_1(y',c)}{\phi_1(y,c)}\leq e^{\al|y-y'|}.\label{eq:cF-2}
\een

Using the fact that
\begin{align*}
\pa_y\phi_1(y,c)=\f{\al^2}{(u(y)-c)^2}\int_{y_c}^y\phi_1(y',c)(u(y')-c)^2\,dy',
\end{align*}
and $(u(y')-c)^2\leq (u(y)-c)^2$ and $\phi_1(y',c)\leq \phi_1(y,c)$  for $y_c\leq y'\leq y$ or $y\leq y'\leq y_c$,  we infer that
\beno
|\cF(y,c)|\leq C\al^2|y-y_c|.
\eeno

On the other hand, we deduce from \eqref{eq:cF-2} and Lemma \ref{Rmk:u_1} that for $0\leq y_c\leq y\leq \min\{\f{1}{\al}+y_c, \pi\}$,
\begin{align*}
\cF(y,c)&=\f{\al^2}{(u(y)-c)^2}\int_{y_c}^{y}\f{\phi_1(y',c)}{\phi_1(y,c)}(u(y')-c)^2\,dy'\\
&\geq \f{\al^2}{(u(y)-c)^2}\int_{y_c}^{y}e^{-\al|y-y'|}(u(y')-c)^2\,dy'\\
&\geq C^{-1} \f{\al^2}{\sin^2\f{y+y_c}{2}(y-y_c)^2}\int_{\f{y_c+y}{2}}^y\sin^2\f{y'+y_c}{2}(y'-y_c)^2\,dy'\\
&\geq C^{-1}\al^2|y-y_c|.
\end{align*}
Similarly, for $\max\{0, y_c-\f{1}{\al}\}\leq y\leq y_c\leq \pi$,  we have
\begin{align*}
-\cF(y,c)&\geq \f{\al^2}{(u(y)-c)^2}\int_{y}^{y_c}e^{-\al|y-y'|}(u(y')-c)^2\,dy'\\
&\geq C^{-1} \f{\al^2}{\sin^2\f{y+y_c}{2}(y-y_c)^2}\int_{y}^{\f{y_c+y}{2}}\sin^2\f{y'+y_c}{2}(y'-y_c)^2\,dy'\\
&\geq C^{-1}\al^2|y-y_c|.
\end{align*}
For $0\leq \f{1}{\al}+y_c\leq y\leq \pi$, we have
\begin{align*}
\cF(y,c)&=\f{\al^2}{(u(y)-c)^2}\int_{y_c}^{y}\f{\phi_1(y',c)}{\phi_1(y,c)}(u(y')-c)^2\,dy'\\
&\geq \f{\al^2}{(u(y)-c)^2}\int_{y_c}^{y}e^{-\al|y-y'|}(u(y')-c)^2\,dy'\\
&\geq C^{-1}\f{\al^2}{\sin^2\f{y+y_c}{2}(y-y_c)^2}\int_{\max\big\{\f{y_c+y}{2},y-\f{1}{\al}\big\}}^ye^{-\al|y-y'|}\sin^2\f{y'+y_c}{2}(y'-y_c)^2\,dy'\\
&\geq C^{-1}\al.
\end{align*}
For $0\leq y\leq y_c-\f{1}{\al}\leq \pi$, we have
\begin{align*}
-\cF(y,c)&=\f{\al^2}{(u(y)-c)^2}\int_{y}^{y_c}\f{\phi_1(y',c)}{\phi_1(y,c)}(u(y')-c)^2\,dy'\\
&\geq C^{-1} \f{\al^2}{(u(y)-c)^2}\int_{y}^{y_c}e^{-\al|y-y'|}(u(y')-c)^2\,dy'\\
&\geq C^{-1}\f{\al^2}{\sin^2\f{y+y_c}{2}(y-y_c)^2}\int^{\min\big\{\f{y_c+y}{2},y+\f{1}{\al}\big\}}_ye^{-\al|y-y'|}\sin^2\f{y'+y_c}{2}(y'-y_c)^2\,dy'\\
&\geq C^{-1}\al.
\end{align*}
This shows that
\ben
C^{-1}\al\min\{\al|y-y_c|,1\}\leq |\cF(y,c)|\leq C\al\min\{\al|y-y_c|,1\},\label{eq:F-est1}
\een
which along with $\phi_1(y_c,c)=1$ implies that
\ben\label{eq: phi_1_up_low}
&&C^{-1}e^{C^{-1}\al |y-y_c|}\leq \phi_1(y,c)\leq e^{C\al|y-y_c|},\\
&&|\pa_y\phi_1|\leq C\al\min\{\al|y-y_c|,1\}\phi_1.
\een
Using $\phi_1''+\f{2u'}{u-c}\phi_1'=\al^2\phi_1$ and $\big(u'(y)+u'(y_c)\big)|y-y_c|\leq C|u-c|$(by Lemma \ref{Rmk:u_1}), we obtain
\beno
|\pa_y^2\phi_1|\leq C\al^2\phi_1.
\eeno

{\bf Step 2}. Estimates of $\pa_c \phi_1, \pa_y\pa_c\phi_1$ and $\Big(\frac{\partial_y}{u'(y_c)}+\partial_c\Big)\phi_1$.

It is easy to check that
\beno
&&\partial_c\cF=\partial_y\cG,\quad \Big(\frac{\partial_y}{u'(y_c)}+\partial_c\Big)\cF=\partial_y\cG_1,\\
&&(\phi^2\partial_c\cF)'+2\phi_1^2u'\cF=0,\quad (\phi^2\partial_y\cF)'+2\phi_1^2(u''(u-c)-u'^2)\cF=0,
\eeno
 and
\beno
&&\cF(y_c,c)=\cG(y_c,c)=\cG_1(y_c,c)=0,\\
&&\Big(\frac{\partial_y}{u'(y_c)}+\partial_c\Big)\cG_1(y_c,c)=0.
\eeno
Therefore, we obtain
\ben
&&\partial_c\cF(y,c)=\frac{-2}{\phi(y,c)^2}\int_{y_c}^y\phi_1(y',c)^2u'(y')\cF(y',c)dy'=\frac{-2}{\phi(y,c)^2}T_0(\phi_1^2u'\cF),\label{eq:pcF}\\
&&\partial_y\cF(y,c)=\frac{-2}{\phi(y,c)^2}T_0(\phi_1^2(u''(u-c)-u'^2)\cF),\label{eq:pyF}
\een
which also show that
\ben
&&\Big(\frac{\partial_y}{u'(y_c)}+\partial_c\Big)\cF=\partial_y\cG_1(y,c)=\frac{-2}{\phi(y,c)^2}T_0\left(\phi_1^2\frac{a_1(y,c)}{u'(y_c)}\cF\right),\label{eq:pycF}
\een
where
\beno
a_1(y,c)=u'(y)u'(y_c)+u''(y)(u(y)-c)-u'(y)^2.
\eeno
It is easy to see that
\beno
&&a_1(y_c,c)=0,\quad
\partial_ya_1(y,c)=u''(y)(u'(y_c)-u'(y))+u'''(y)(u(y)-c),\\
&&\partial_ca_1(y,c)=u'(y)\frac{u''(y_c)}{u'(y_c)}-u''(y),
\eeno
therefore,
\ben
|\partial_ya_1(y,c)|\leq C|y-y_c|,\quad |a_1(y,c)|\leq C|y-y_c|^2.\label{eq:a1}
\een

By Lemma \ref{lem:phi1},  we know that $\partial_y\phi_1(y,c)$ has the same sign as $y-y_c$. For $y_c\leq y'\leq y$ or $y\leq y'\leq y_c$, $u'(y')\leq C\sin \f{y+y_c}{2}$. So, we infer from \eqref{eq:pcF} and \eqref{eq:phi-est1} that
\begin{align}\label{eq: pa_cf}
&|\partial_y\cG(y,c)|=|\partial_c\cF(y,c)|\nonumber\\
&\leq \frac{C\sin \f{y+y_c}{2}}{\phi(y,c)^2}\int_{y_c}^y\phi_1(y',c)\partial_{y'}\phi_1(y',c)dy'\nonumber\\
&\leq C\frac{\phi_1(y,c)^2-1}{\phi(y,c)^2}\sin \f{y+y_c}{2}\nonumber\\
&\leq C\frac{\phi_1(y,c)-1}{\phi_1(y,c)}\f{1}{|y-y_c|^2\sin \f{y+y_c}{2}}
\leq \frac{C\min\{1,\al^2|y-y_c|^2\}}{|y-y_c|^2\sin y_c},
\end{align}
and by \eqref{eq:pycF} and \eqref{eq:a1},
\begin{align}\label{eq: good_1F}
&\Big|\Big(\frac{\partial_y}{u'(y_c)}+\partial_c\Big)\cF\Big|
=|\partial_y\cG_1(y,c)|\nonumber\\
&\leq \frac{C|y-y_c|^2}{u'(y_c)\phi(y,c)^2}\int_{y_c}^y\phi_1(y',c)\partial_{y'}\phi_1(y',c)dy'\nonumber\\
&\leq C\frac{\phi_1(y,c)^2-1}{u'(y_c)\phi(y,c)^2}|y-y_c|^2\nonumber\\
&\leq \frac{C\min\{1,\al^2|y-y_c|^2\}}{u'(y_c)}\frac{|y-y_c|^2}{(u(y)-c)^2},
\end{align}
from which and $\cG(y_c,c)=\cG_1(y_c,c)=0$,  it follows that
\beq\label{eq:g}
|\cG(y,c)|\leq C\left|\int_{y_c}^y\frac{\min\{1,\al^2|y'-y_c|^2\}}{|y'-y_c|^2u'(y_c)}dy'\right|\leq \frac{C\al{\min\{1,\al|y-y_c|\}}}{u'(y_c)},
\eeq
and by Lemma \ref{Rmk:u_1},
\beno
|\cG_1(y,c)|\leq C\frac{\min\{1,\al^2|y-y_c|^2\}}{u'(y_c)}\left|\int_{y_c}^y\frac{|y'-y_c|^2}{(u(y')-c)^2}dy'\right|\leq \frac{C\min\{1,\al^2|y-y_c|^2\}}{u'(y_c)^2}.
\eeno
Thus, we deduce that
\ben
|\partial_c\phi_1|\leq \f{C\al \phi_1}{u'(y_c)}\min(1,\al|y-y_c|),\quad |\partial_y\partial_c\phi_1|\leq \f{C\al^2 \phi_1}{u'(y_c)},\label{eq:phi-pc}
\een
and
\ben\label{eq: good_1phi_1}
\left|\Big(\frac{\partial_y}{u'(y_c)}+\partial_c\Big)\phi_1(y,c)\right|\leq C\frac{\min\{1,\al^2|y-y_c|^2\}\phi_1}{u'(y_c)^2}.
\een

{\bf Step 3.} Estimates of $\pa_c^2\phi_1$ and $\Big(\frac{\partial_y}{u'(y_c)}+\partial_c\Big)^2\phi_1$.\smallskip

By Lemma \ref{lem:phi1}, we get for $|y-y_c|\leq \f{M_0}{\al}$,
\begin{align*}
|\pa_c^2\cF(y,c)|&\leq \f{C|\pa_c^2\pa_y\phi_1|}{\phi_1}+\f{C|\pa_c\pa_y\phi_1\pa_c\phi_1|}{\phi_1^2}+\f{C|\pa_c^2\phi_1\pa_y\phi_1|}{\phi_1^2}+\f{C|\pa_c\phi_1|^2|\pa_y\phi_1|}{\phi_1^3}\\
&\leq \f{C\al^3}{u'(y_c)^2}+\f{C\al^2}{u'(y_c)\sin^2\f{y+y_c}{2}}.
\end{align*}
It follows from \eqref{eq:pcF} that
\beno
\partial_c^2\cF(y,c)=\partial_y\partial_c\cG(y,c)
=-2\int_{y_c}^y\partial_c\Big(\frac{\phi_1(z,c)^2}{\phi(y,c)^2}\cF(z,c)\Big)u'(z)dz.
\eeno
Then by \eqref{eq:phi-pc}, we have
\begin{align*}
|\partial_c^2\cF(y,c)|\le& 4\f {|\pa_c\phi(y,c)|} {|\phi(y,c)|^3}\sin(\f {y+y_c} 2)\int_{y_c}^y\phi_1(z,c)\pa_z\phi_1(z,c)dz\\
&+\f {C\al} {|\phi(y,c)|^2u'(y_c)}\sin(\f {y+y_c} 2)\int_{y_c}^y\phi_1(z,c)\pa_z\phi_1(z,c)dz\\
&+\f {C\al} {(u(y)-c)^2}\sin(\f {y+y_c} 2)\Big|\int_{y_c}^y|\pa_c\cF(z,c)|dz\Big|\end{align*}
which along with the fact $\int_{y_c}^y2\phi_1(z,c)\pa_z\phi_1(z,c)dz=\phi_1(y,c)^2-1$, \eqref{eq:phi-est1} and \eqref{eq: pa_cf}  implies that for $|y-y_c|\geq \f{M_0}{\al}$,
\begin{align*}
|\partial_c^2\cF(y,c)|\leq \frac{C\al}{|y-y_c|^2u'(y_c)^2}.
\end{align*}
Thanks to $\pa_c\cG(y_c,c)=\pa_c^2\phi_1(y_c,c)=\al^2\pa_c^2T\phi_1=\f{\al^2}{3u'(y_c)^2}$, we get
\beno
|\partial_c\cG(y,c)|\le |\pa_c\cG(y_c,c)|+\Big|\int_{y_c}^y\pa_g\pa_c\cG(z,c)dz\Big|\leq  \frac{C\al^2}{u'(y_c)^2},
\eeno
which implies that $|\partial_c^2\phi_1|\leq \frac{C\al^2\phi_1}{u'(y_c)^2}.$

Using \eqref{eq:pycF} and the formula
$$\Big(\frac{\partial_y}{u'(y_c)}+\partial_c\Big)(T_0a)=\frac{a(y,c)}{u'(y_c)}-\frac{a(y_c,c)}{u'(y_c)}+T_0(\partial_ca)
=T_0\Big(\Big(\frac{\partial_y}{u'(y_c)}+\partial_c\Big)a\Big),$$
we deduce that
\begin{align}
\Big(\frac{\partial_y}{u'(y_c)}+\partial_c\Big)\partial_y\cG_1(y,c)=&\left(\Big(\frac{\partial_y}{u'(y_c)}+\partial_c\Big)\frac{-2}{\phi(y,c)^2}\right)T_0\Big(\phi_1^2\frac{a_1(y,c)}{u'(y_c)}\cF\Big)\nonumber\\
&+\frac{-2}{\phi(y,c)^2}T_0\left(\Big(\frac{\partial_y}{u'(y_c)}+\partial_c\Big)\Big(\phi_1^2\frac{a_1(y,c)}{u'(y_c)}\cF\Big)\right).\label{eq:G-pcy}
\end{align}
Using \eqref{eq: good_1phi_1} and the fact that
\ben\label{eq:u-gd}
\left|\Big(\frac{\partial_y}{u'(y_c)}+\partial_c\Big)(u(y)-c)\right|=\Big|\frac{u'(y)}{u'(y_c)}-1\Big|\leq C\frac{|y-y_c|}{u'(y_c)}\leq C\frac{|u(y)-c|}{u'(y_c)^2},
\een
we infer that
\ben\label{eq:phi-pcy-1}
\left|\left(\frac{\partial_y}{u'(y_c)}+\partial_c\right)\frac{-2}{\phi(y,c)^2}\right|\leq \frac{C}{\phi(y,c)^2u'(y_c)^2}.
\een
On the other hand, we have
\begin{align*}
&\left(\frac{\partial_y}{u'(y_c)}+\partial_c\right)\left(\phi_1^2\frac{a_1(y,c)}{u'(y_c)}\cF\right)\\
&=2\phi_1\frac{a_1(y,c)}{u'(y_c)}\cF\left(\frac{\partial_y}{u'(y_c)}+\partial_c\right)\phi_1+\phi_1^2\frac{a_2(y,c)}{u'(y_c)^2}\cF\\
&\quad-\phi_1^2\frac{a_1(y,c)}{u'(y_c)^3}u''(y_c)\cF+\phi_1^2\frac{a_1(y,c)}{u'(y_c)}\left(\frac{\partial_y}{u'(y_c)}+\partial_c\right)\cF,
\end{align*}
where
\beno
a_2(y,c)=u'(y)(u''(y_c)-u''(y))+u'''(y)(u(y)-c)=0.
\eeno
Then we get by \eqref{eq:a1},  \eqref{eq: good_1phi_1} and \eqref{eq: good_1F}  that
\begin{align*}
&\left|\Big(\frac{\partial_y}{u'(y_c)}+\partial_c\Big)\left(\phi_1^2\frac{a_1(y,c)}{u'(y_c)}\cF\right)\right|\\
&\leq
C\phi_1^2\frac{|y-y_c|^2}{u'(y_c)^3}|\cF|+C\phi_1^2\frac{|y-y_c|^2}{u'(y_c)}
\frac{\phi_1(y,c)^2-1}{u'(y_c)\phi_1(y,c)^2}\frac{|y-y_c|^2}{(u(y)-c)^2}\\
&\leq C\phi_1^2\frac{|y-y_c|^2}{u'(y_c)^3}|\cF|+C|y-y_c|
\frac{\phi_1(y,c)^2-1}{u'(y_c)^3},
\end{align*}
which gives
\begin{align}
&\left|T_0\left(\Big(\frac{\partial_y}{u'(y_c)}+\partial_c\Big)\Big(\phi_1^2\frac{a_1(y,c)}{u'(y_c)}\cF\Big)\right)\right|\nonumber\\
&\leq
C \frac{|y-y_c|^2}{u'(y_c)^3}|T_0(\phi_1^2|\cF|)|+C\left|T_0\left(|y-y_c|
\frac{\phi_1(y,c)^2-1}{u'(y_c)^3}\right)\right|\nonumber\\
&\leq C|y-y_c|^2
\frac{\phi_1(y,c)^2-1}{u'(y_c)^3}.\label{eq:T-F-pc}
\end{align}
It follows from \eqref{eq:G-pcy}, \eqref{eq:phi-pcy-1} and \eqref{eq:T-F-pc} that
\begin{align*}
&\left|\Big(\frac{\partial_y}{u'(y_c)}+\partial_c\Big)\partial_y\cG_1(y,c)\right|\\
&\leq C \frac{|\partial_y\cG_1(y,c)|}{u'(y_c)^2}+C\frac{|y-y_c|^2}{\phi(y,c)^2}
\frac{\phi_1(y,c)^2-1}{u'(y_c)^3}\\
&\leq C\frac{\phi_1(y,c)^2-1}{u'(y_c)^3\phi_1(y,c)^2}\frac{|y-y_c|^2}{(u(y)-c)^2}.
\end{align*}
This together with $\Big(\frac{\partial_y}{u'(y_c)}+\partial_c\Big)\cG_1(y_c,c)=0$ gives
\beno
\left|\Big(\frac{\partial_y}{u'(y_c)}+\partial_c\Big)\cG_1(y,c)\right|\leq C\frac{\phi_1(y,c)^2-1}{u'(y_c)^3\phi_1(y,c)^2}\left|\int_{y_c}^y\frac{|y'-y_c|^2}{(u(y')-c)^2}dy'\right|\leq \frac{C(\phi_1(y,c)^2-1)}{u'(y_c)^4\phi_1(y,c)^2},
\eeno
from which, \eqref{eq: good_1phi_1} and \eqref{eq:phi-est1}, we infer that
\beno
\left|\Big(\frac{\partial_y}{u'(y_c)}+\partial_c\Big)^2\phi_1(y,c)\right|\leq C\frac{\min\{1,\al^2|y-y_c|^2\}\phi_1}{u'(y_c)^4}.
\eeno

The proposition follows by collecting the estimates in Step 1-Step 3.
\end{proof}

It is easy to see from the proof of Proposition \ref{prop:phi1} that

\begin{lemma}\label{Rmk: fg_periodic}
We have following estimates for $\cF$ and $\cG$ at $y=0,\pi$:
\begin{align*}
&C^{-1}\al\min\{\al y_c,1\}\leq |\cF(0,c)|\leq C\al\min\{\al y_c,1\},\\
&C^{-1}\al\min\{\al (\pi-y_c),1\}\leq |\cF(\pi,c)|\leq C\al\min\{\al (\pi-y_c),1\},\\
&|\pa_c\cF(0,c)|\leq \f{C|\cF(0,c)|^2}{\al^2y_c^2\sin y_c},\quad
|\pa_c^2\cF(0,c)|\leq \f{C\al\min\{\al y_c,1\}}{y_c^2\sin^2 y_c},\\
&|\pa_c\cF(\pi,c)|\leq \f{C|\cF(\pi,c)|^2}{\al^2(\pi-y_c)^2\sin y_c},\quad
|\pa_c^2\cF(\pi,c)|\leq \f{C\al\min\{\al (\pi-y_c),1\}}{(\pi-y_c)^2\sin^2y_c},\\
&|\cG(0,c)|\leq \f{C\al\min\{1,\al y_c\}}{\sin y_c},\quad
|\pa_c\cG(0,c)|\leq \f{C\al^2}{\sin^2y_c},\\
&|\cG(\pi,c)|\leq \f{C\al\min\{1,\al (\pi-y_c)\}}{\sin y_c},\quad
|\pa_c\cG(\pi,c)|\leq \f{C\al^2}{\sin^2y_c}.
\end{align*}
\end{lemma}

To obtain better estimates near critical points, we introduce the following functions:
for $0\leq y_c<\f{1}{\al}$ and $0\leq y\leq 1$, let
\beno
\widetilde{\phi_1}(y,c)=\phi_1(yy_c,c),
\eeno
and for $\pi-\f{1}{\al}< y_c\leq \pi$ and $0\leq y\leq 1$, let
\beno
\widetilde{\widetilde{\phi_1}}(y,c)=\phi_1\big((1-y)\pi+yy_c,c\big).
\eeno

\begin{proposition}\label{Prop:tphi_1}
It holds that for any $y\in [0,1]$ and $0\leq y_c<\f{1}{\al}$,
\begin{align*}
&1\leq \widetilde{\phi_1}\leq 2,\quad
|\pa_c\widetilde{\phi_1}|\leq C\al^2,\quad
|\pa_c^2\widetilde{\phi_1}|\leq C\al^4,\\
&\left|\partial_c^k\Big(\frac{\pa_y\widetilde{\phi_1}(y,c)}{y_c^2}\Big)\right|\leq C\al^{2+2k},\ \ \ k=0,1,2,
\end{align*}
and for any $y\in [0,1]$ and $\pi-\f{1}{\al}< y_c\leq \pi$, it holds that
\begin{align*}
&1\leq \widetilde{\widetilde{\phi_1}}\leq 2,\quad
|\pa_c\widetilde{\widetilde{\phi_1}}|\leq C\al^2,\quad
|\pa_c^2\widetilde{\widetilde{\phi_1}}|\leq C\al^4,\\
&\left|\partial_c^k\Big(\frac{\pa_y\widetilde{\widetilde{\phi_1}}(y,c)}{(\pi-y_c)^2}\Big)\right|\leq C\al^{2+2k},\ \ \ k=0,1,2.
\end{align*}
\end{proposition}
\begin{proof}
For $0\leq y_c<\frac{1}{\al}$ and $0\leq y\leq 1$, by \eqref{eq:phi1}, we have $\widetilde{\phi_1}=1+\al^2T_1\widetilde{\phi_1}$, where
\begin{align*}
(T_1f)(y,c)&=\int_{1}^y\f {y_c^2} {(u(y'y_c)-c)^2}\int_{1}^{y'}f(z,c)(u(zy_c)-c)^2dzdy'\\
&=\int_{1}^y\int_{1}^{y'}f(z,c)\f{F(z,y_c)^2}{F(y',y_c)^2}y_c^2dzdy'
\end{align*}
with
\begin{align*}
F(z,c)=&\f{u(zy_c)-c}{y_c^2}=\f{\cos y_c-\cos zy_c}{y_c^2}\\
=&\f{1-z^2}{2}m_1\left(\f{(1-z)y_c}{2}\right)m_1\left(\f{(1+z)y_c}{2}\right),
\end{align*}
and $m_1(y)=\sin y/y\in C^4([-1,1]).$ One can check that $1/C\leq m_1(-y)=m_1(y)\leq C,$ and $|\pa_c^km_1(ay_c)|\leq C$ for $y_c<\f1 \al, k=1,2$. Then we can deduce that for $y_c<\f{1}{\al}$,
\beno
\f{|z^2-1|}{C}\leq|F(z,c)|\leq C|z^2-1|,\ |\partial_cF(z,c)|\leq C|z^4-1|,\ |\partial_c^2F(z,c)|\leq C{|z^4-1|},
\eeno
and for $y_c<\f{1}{\al}$ and $0<y'<z<1$,
\beno
0\leq{\f{F(z,y_c)^2}{F(y',y_c)^2}}\leq C,\
\left|\partial_c\left({\f{F(z,y_c)^2}{F(y',y_c)^2}}\right)\right|\leq C,\ \left|\partial_c^2\left({\f{F(z,y_c)^2}{F(y',y_c)^2}}\right)\right|\leq C.
\eeno

Let $\Om_{\al}=\{(y,c):~y\in [0,1], \ c\in [u(0),u(1/\al)]\}$. Then we can infer that
\ben
&&\|T_1f\|_{L^{\infty}(\Omega_{\al})}\leq \frac{1}{2\al^2}\|f\|_{L^{\infty}(\Omega_{\al})},\label{eq:T1-1}\\
&&\|\partial_cT_1f-T_1\partial_cf\|_{L^{\infty}(\Omega_{\al})}\leq C\|f\|_{L^{\infty}(\Omega_{\al})},\label{eq:T1-2}\\
&&\|\partial_c^2T_1f-T_1\partial_c^2f\|_{L^{\infty}(\Omega_{\al})}\leq C(\|f\|_{L^{\infty}(\Omega_{\al})}+\|\partial_cf\|_{L^{\infty}(\Omega_{\al})}).\label{eq:T1-3}
\een

We infer from \eqref{eq:T1-1} that $1\leq \widetilde{\phi_1}\leq 2$ for $(y,c)\in \Om_{\al}$.
By Proposition \ref{prop:phi1}, we have for $(y,c)\in \Om_{\al}$,
\beno
|\pa_c\widetilde{\phi_1}|=\Big|(\pa_c\phi_1)(yy_c,c)+\f{y}{u'(y_c)}(\pa_y\phi_1)(yy_c,c)\Big|\leq \f{C\al\widetilde{\phi_1}}{u'(y_c)}\leq \f{C\al}{y_c}.
\eeno
Thus, we obtain
\beno
\|T_1(\pa_c\widetilde{\phi_1})\|_{L^{\infty}(\Om_{\al})}\leq C,
\eeno
which along with \eqref{eq:T1-2} shows that for $(y,c)\in \Om_{\al}$,
\ben
|\pa_c\widetilde{\phi_1}|\leq C\al^2.\label{eq:phi1-pa1}
\een

By Proposition \ref{prop:phi1}, we have for $(y,c)\in \Om_{\al}$,
\beno
|\pa_c^2\widetilde{\phi_1}|\leq \f{C\al^2\widetilde{\phi_1}}{u'(y_c)^2}\leq \f{C\al^2}{y_c^2}.
\eeno
Thus, we have
\beno
\|T_1(\pa_c^2\widetilde{\phi_1})\|_{L^{\infty}(\Om_{\al})}\leq C\al^2,
\eeno
which along with \eqref{eq:T1-2} shows that  for $(y,c)\in \Om_{\al}$,
\ben\label{eq:phi1-pa2}
|\pa_c^2\widetilde{\phi_1}|\leq C\al^4.
\een

Using the fact that $\frac{\pa_y\widetilde{\phi_1}(y,c)}{y_c^2}=\al^2\int_{1}^{y}\widetilde{\phi_1}(z,c)\f{F(z,y_c)^2}{F(y,y_c)^2}dz,$
we can deduce that for $(y,c)\in \Om_{\al}$,
\beno
\left|\partial_c^k\Big(\frac{\pa_y\widetilde{\phi_1}(y,c)}{y_c^2}\Big)\right|\leq C\al^{2+2k},\ \ \ k=0,1,2.
\eeno
Due to $\phi_1(0,c)=\widetilde{\phi_1}(0,c)$, we have by Proposition \ref{prop:phi1},
\ben\label{eq:tphi'/y_c^2_1}
-\frac{\pa_y\widetilde{\phi_1}(0,c)}{y_c^2}\geq C^{-1}\al^2.
\een

For $\pi-\f{1}{\al}<y_c<\pi$ and $0\leq y\leq 1$, we have $\widetilde{\widetilde{\phi_1}}=1+\al^2T_2\widetilde{\widetilde{\phi_1}}$, where
\beno
(T_2f)(y,c)=\int_{1}^y\f {(\pi-y_c)^2} {(u((1-y')\pi+y'y_c)-c)^2}\int_{1}^{y'}f(z,c)(u((1-z)\pi+zy_c)-c)^2dzdy',
\eeno
and we also have
\begin{align*}
\f{u((1-y')\pi+y'y_c)-c}{(\pi-y_c)^2}&=\cos (\pi-y_c)-\cos(y'(\pi-y_c))\\
&=\f{1-y'^2}{2}m_1\Big(\f{(1-y'){\pi-y_c}}{2}\Big)m_1\Big(\f{(1+y'){\pi-y_c}}{2}\Big).
\end{align*}
Using similar arguments as above, we can deduce that for $y\in [0,1]$ and $c\in [u(\pi-1/\al),u(\pi)]$,
\begin{align*}
&1\leq \widetilde{\widetilde{\phi_1}}\leq 2,\quad
|\pa_c\widetilde{\widetilde{\phi_1}}|\leq C\al^2,\quad
|\pa_c^2\widetilde{\widetilde{\phi_1}}|\leq C\al^4,\\
&\left|\partial_c^k\Big(\frac{\pa_y\widetilde{\widetilde{\phi_1}}(y,c)}{(\pi-y_c)^2}\Big)\right|\leq C\al^{2+2k},\ \ \ k=0,1,2,
\end{align*}
and
\ben\label{eq:tphi'/y_c^2_2}
\frac{\pa_y\widetilde{\widetilde{\phi_1}}(0,c)}{(\pi-y_c)^2}\geq C^{-1}\al^2.
\een
\end{proof}

\section{The inhomogeneous Rayleigh equation}\label{InhoRayequ}
In this section, we solve the inhomogeneous Rayleigh equation:
\beq\label{eq:Rayleigh-Ihom-periodic}
\left\{
\begin{aligned}
&\Phi''-\al^2\Phi-\frac{u''}{u-c}\Phi=f,\\
&\Phi'(-\pi)=\Phi'(\pi),\quad \Phi(-\pi)=\Phi(\pi),
\end{aligned}
\right.
\eeq
where $u(y)=-\cos y$.
We split the equation \eqref{eq:Rayleigh-Ihom-periodic} into two parts:
\beq\label{eq:Rayleigh-Ihom-odd}
\left\{
\begin{aligned}
&\Phi_{o}''-\al^2\Phi_{o}-\frac{u''}{u-c}\Phi_{o}=f_{o},\\
&\Phi_{o}(0)=\Phi_{o}(\pi)=0,
\end{aligned}
\right.
\eeq
and
\beq\label{eq:Rayleigh-Ihom-even-half1}
\left\{
\begin{aligned}
&\Phi_{e}''-\al^2\Phi_{e}-\frac{u''}{u-c}\Phi_{e}=f_e,\\
&\Phi_{e}'(0)=\Phi_{e}'(\pi)=0,
\end{aligned}
\right.
\eeq
where $f_o(y,c)=\frac{f(y,c)-f(-y,c)}{2}$ and $f_e(y,c)=\frac{f(y,c)+f(-y,c)}{2}$. Thus, $\Phi=\Phi_o+\Phi_e$ is a solution of \eqref{eq:Rayleigh-Ihom-periodic} with $\Phi_o$ being an odd function and $\Phi_e$ being an even function.

Throughout this section, we denote by $\phi(y,c)$ a solution of \eqref{eq:Rayleigh-H} obtained by Proposition \ref{prop:Rayleigh-Hom} and let $\phi_1(y,c)=\f{\phi(y,c)}{u(y)-c}$ for $c\in \Om_{\ep}$.
Here and in what follows, we denote $\phi'=\pa_y\phi$ and $\phi_1'=\pa_y\phi_1$.

\subsection{The Wronskian}
Before we present the solution formulation, let us first calculate the Wronskian of \eqref{eq:Rayleigh-Ihom-odd} and \eqref{eq:Rayleigh-Ihom-even-half1}.

\begin{lemma}\label{Rmk:Wronskian}
For any $c\in \Om_{\ep_0}\setminus \overline{D}_0$, the Wronskian of \eqref{eq:Rayleigh-Ihom-odd} has the form
\beno
W_o(c)=\phi(0,c)\phi(\pi,c)\int_0^{\pi}\f{1}{\phi(y,c)^2}dy,
\eeno
and the Wronskian of \eqref{eq:Rayleigh-Ihom-even-half1} has the form
\beno
W_e(c)=\f{\pa_y\phi(\pi,c)}{\phi(0,c)}-\f{\pa_y\phi(0,c)}{\phi(\pi,c)}-\phi'(\pi,c)\phi'(0,c)\int_0^{\pi}\f{1}{\phi(y',c)^2}dy'.
\eeno
\end{lemma}

\begin{proof}
For $c\notin \overline{D}_0$, the integration $\int_0^{y}\f{1}{\phi(y',c)^2}dy'$ is well defined.
It is easy to check that
\beno
\varphi_o(y,c)=\phi(0,c)\phi(y,c)\int_0^{y}\f{1}{\phi(y',c)^2}dy',
\eeno
is a solution to \eqref{eq:Rayleigh-H} with boundary conditions $\varphi(0,c)=0$ and $\varphi'(0,c)=1$.
So, $W_o(c)=\varphi_o(\pi,c)$.

One can also check that
\beno
\varphi_e(y,c)=\f{\phi(y,c)}{\phi(0,c)}-\phi(y,c)\pa_y\phi(0,c)\int_0^y\f{1}{\phi(y',c)^2}dy'
\eeno
is a solution of \eqref{eq:Rayleigh-H} with boundary conditions $\varphi(0,c)=1$ and $\pa_y\varphi(0,c)=0$. So, $W_e(c)=\pa_y\varphi_e(\pi,c).$
\end{proof}

By Proposition \ref{prop:Rayleigh-Hom}, $\phi(y,c)$ and $\pa_y\phi(y,c)$ are continuous. Thus, both $W_o(c)$ and $W_e(c)$ are continuous for $c\in \Om_{\ep_0}\setminus \overline{D}_0$. The following lemma will show that $\sin y_cW_o(c)$ is continuous up to the boundary. To this end, we introduce some notations.
For $c=u(y_c)\in D_0$ and $j=0,1$, $k=1,2,$ let
\begin{align}
&\mathrm{II}(c)=\int_0^{\pi}\frac{1}{(u(y)-c)^2}\Big(\frac{1}{\phi_{1}(y,c)^2}-1\Big)dy,\label{eq:Pi}\\
&\rmA(c)=\sin^3 y_c\mathrm{II}(c),\quad \rmB(c)=\pi\cos y_c,\label{eq:defAB}\\
&J_j^k(c)=\f{-u'(y_c)(u(j\pi)-c)^k}{\phi_1((1-j)\pi,c)^{k-1}\phi_1'((1-j)\pi,c)},\ \ J_j(c)=J_j^2(c),\label{eq:defJ}\\
&\rmA_1(c)= J_1(c)-J_0(c)+\sin^2 y_c\rmA(c),\quad \rmB_1(c)=\sin^2 y_c \rmB.\label{eq:defA_1B_1},\\
&\rho(c)=(u(\pi)-c)(c-u(0))=1-c^2.
\end{align}

\begin{lemma}\label{Lem:W(c)converge}
For $c_{\ep}=c+ i\ep=u(y_c)+ i\ep\in D_{\ep_0}$, $c\in(-1,1),$ we have
\beno
&&\lim_{\ep\to 0+}\rho(c_{\ep})^{1/2}W_o(c_{\ep})=-\phi_1(0,c)\phi_1(\pi,c)\big(\rmA(c)-i\rmB(c)\big),\\
&&\lim_{\ep\to 0-}\rho(c_{\ep})^{1/2}W_o(c_{\ep})=-\phi_1(0,c)\phi_1(\pi,c)\big(\rmA(c)+i\rmB(c)\big),\\
&&\lim_{\ep\to 0+}\phi(0,c_{\ep})\phi(\pi,c_{\ep})W_e(c_{\ep})=-\f{(\phi_1\phi_1')(\pi,c)(\phi_1\phi_1')(0,c)}{u'(y_c)}(\rmA_1(c)-i\rmB_1(c)),\\
&&\lim_{\ep\to 0-}\phi(0,c_{\ep})\phi(\pi,c_{\ep})W_e(c_{\ep})=-\f{(\phi_1\phi_1')(\pi,c)(\phi_1\phi_1')(0,c)}{u'(y_c)}(\rmA_1(c)+i\rmB_1(c)).
\eeno
Here $\rho(c_\ep)=1-c_\ep^2$ and $\rho(c_{\ep})^{1/2}$ is taken so that $\mathrm{Re}\, \rho(c_{\ep})^{1/2}>0$.
\end{lemma}
\begin{proof}
Let us prove the case of $\ep>0$. Let $c_{\ep}=u(y_c)+i\ep\in D_{\ep_0}$. Then
\begin{align*}
\rho(c_{\ep})^{3/2}\int_0^{\pi}\f{1}{(u(y)-c_{\ep})^2}dy
&=\rho(c_{\ep})^{3/2}\pa_{c}\left(\int_0^{\pi}\f{1}{u(y)-c_{\ep}}dy\right)\\
&=-\rho(c_{\ep})^{3/2}\pa_{c}\left(\int_0^{\pi}\f{1}{\cos y+c_{\ep}}dy\right)\\
&=-\rho(c_{\ep})^{3/2}\pa_{c}\left(\int_0^{\infty}\f{2}{1-t^2+c_{\ep}(1+t^2)}dt\right)\\
&=-\rho(c_{\ep})^{3/2}\pa_{c}\left(\f{1}{(1+c_{\ep})d_{\ep}}\ln\Big(\f{1+d_{\ep}t}{1-d_{\ep}t}\Big)\bigg|_{t=0}^{\infty}\right),
\end{align*}
here $d_{\ep}$ satisfies $d_{\ep}^2=\f{1-c_{\ep}}{1+c_{\ep}}$ with $\text{Re}\, d_{\ep}>0$. Then $\text{Im}\, d_{\ep}<0$, $\rho(c_{\ep})^{1/2}=(1+c_{\ep})d_{\ep}$ and
\begin{align*}
\rho(c_{\ep})^{3/2}\int_0^{\pi}\f{1}{(u(y)-c_{\ep})^2}dy
&=-\rho(c_{\ep})^{3/2}\pa_{c}\left(\f{-1}{(1+c_{\ep})d_{\ep}}\pi i\right)\\
&=-\pa_{c}\rho(c_{\ep})\pi i/2
\to -\cos y_c\pi i, \  \textrm{as}\ \ep\to 0+.
\end{align*}
By Proposition \ref{prop:Rayleigh-Hom} and Lemma \ref{Rmk:u_1}, we get
\begin{align*}
\left|\f{\rho(c_{\ep})}{(u(y)-c_{\ep})^2}\Big(\f{1}{\phi_1(y,c_{\ep})^2}-1\Big)\right|
\leq C\left|\f{(\rho(c)+\ep)|y-y_{c}|^2}{(u(y)-c)^2+\ep^2}\right|
\leq C,
\end{align*}
with $C$  a constant independent of $\ep$, which ensures that as $\ep\to 0+$,
\begin{align*}
\rho(c_{\ep})\int_{0}^{\pi}\f{1}{(u(y)-c_{\ep})^2}\Big(\f{1}{\phi_1(y,c_{\ep})^2}-1\Big)dy
\to
\sin^2 y_c\int_{0}^{\pi}\f{1}{(u(y)-c)^2}\Big(\f{1}{\phi_1(y,c)^2}-1\Big)dy.
\end{align*}
Thus, we deduce that as $\ep\to 0+$,
\beno
\rho(c_{\ep})^{3/2}\int_{0}^{\pi}\f{1}{\phi(y,c_{\ep})^2}dy
\to \sin^3y_c \mathrm{II}(c)-\cos y_c\pi i.
\eeno
This together with the continuity of $\phi_1(y,c)$ gives
\beno
\lim_{\ep\to 0+}\rho(c_{\ep})^{1/2}W_o(c_{\ep})=-\phi_1(0,c)\phi_1(\pi,c)\big(\rmA(c)-i\rmB(c)\big).
\eeno

Using the fact that $u'(0)=u'(\pi)=0$, we get $\phi'(j\pi,c)=(u(j\pi)-c)\phi_1'(j\pi,c)$ for $j=0,1$. Then we infer that
\begin{align*}
\phi(0,c_{\ep})\phi(\pi,c_{\ep})W_e(c_{\ep})
&=(\phi\phi')(\pi,c_{\ep})-(\phi\phi')(0,c_{\ep})-(\phi\phi')(\pi,c_{\ep})(\phi\phi')(0,c_{\ep})\int_0^{\pi}\f{1}{\phi(y',c_{\ep})^2}dy'\\
&\to (\phi\phi')(\pi,c)-(\phi\phi')(0,c)-\f{(\phi\phi')(\pi,c)(\phi\phi')(0,c)}{\sin^3 y_c}(\rmA(c)-i\rmB(c))\\
&=(1+\cos y_c)^2(\phi_1\phi_1')(\pi,c)-(1-\cos y_c)^2(\phi_1\phi_1')(0,c)\\
&\quad-\f{(\phi_1\phi_1')(\pi,c)(\phi_1\phi_1')(0,c)}{u'(y_c)}\sin^2 y_c(\rmA(c)-i\rmB(c))\\
&=-\f{(\phi_1\phi_1')(\pi,c)(\phi_1\phi_1')(0,c)}{u'(y_c)}(\rmA_1(c)-i\rmB_1(c)).
\end{align*}

The case of $\ep<0$ is similar, here we omit the details.
\end{proof}

\subsection{Uniform estimates of $\rmA(c),\rmB(c)$}

\begin{proposition}\label{Prop: A^2+B^2}
Let $\rmA(c), \rmB(c)$ be defined by \eqref{eq:defAB}. Then there exists a constant $C$ independent of $\al$, such that
\begin{align*}
&C^{-1}(1+\al\sin y_c)^2\leq \rmA(c)^2+\rmB(c)^2\leq C(1+\al\sin y_c)^2,\\
&\pa_c\Big(\f{1}{\rmA(c)^2+\rmB(c)^2}\Big)\leq \f{C}{(1+\al\sin y_c)^2\sin^2y_c},\\
&\pa_c^2\Big(\f{1}{\rmA(c)^2+\rmB(c)^2}\Big)\leq \f{C}{(1+\al\sin y_c)^2\sin^4y_c}.
\end{align*}
\end{proposition}

We need the following lemma.

\begin{lemma}\label{Lem: II}
Let $\mathrm{II}(c)$ be defined by \eqref{eq:Pi}. Then there exists a constant $C$ independent of $\al$, such that
\begin{align*}
&C^{-1}\min\left\{\f{\al^2}{\sin y_c},\f{|\al|}{\sin^2 y_c}\right\}\leq -\mathrm{II}(c)\leq C\min\left\{\f{\al^2}{\sin y_c},\f{|\al|}{\sin^2 y_c}\right\},\\
&\big|\rho(c)^k\pa_c^k\mathrm{II}(c)\big|\leq C\min\left\{\f{\al^2}{\sin y_c},\f{|\al|}{\sin^2 y_c}\right\},\quad k=1,2.
\end{align*}
\end{lemma}

\begin{proof}
Let us first prove the first inequality of the lemma. Due to $\phi_1(y,c)\geq 1$, we get by Proposition \ref{prop:phi1} and Lemma \ref{Rmk:u_1} that
\begin{align*}
-\mathrm{II}(c)=&\int_0^{\pi}\frac{1}{(u(y)-c)^2}\Big(1-\frac{1}{\phi_{1}(y,c)^2}\Big)dy\\
\geq&C^{-1}\int_0^{\pi}\frac{\al^2|y-y_c|^2}{(u(y)-c)^2}\chi_{\{|y-y_c|\leq\frac{1}{\al}\}}dy\\
\geq&
\left\{
\begin{aligned}
C^{-1}\int_{0}^{y_c}\dfrac{\al^2|y-y_c|^2}{\sin^{2}\f{y+y_c}{2}|y-y_c|^2}dy\geq \dfrac{\al^2}{Cu'(y_c)},&\  (0< y_c\leq\frac{1}{\al}),\\
C^{-1}\int_{y_c-\frac{1}{\al}}^{y_c}\dfrac{\al^2|y-y_c|^2}{u'(y_c)^2|y-y_c|^2}dy\geq \dfrac{\al}{Cu'(y_c)^2},&\  (\frac{1}{\al}\leq y_c<\pi-\f{1}{\al}),\\
C^{-1}\int_{y_c}^{\pi}\dfrac{\al^2|y-y_c|^2}{\sin^{2}\f{y+y_c}{2}|y-y_c|^2}dy\geq \dfrac{\al^2}{Cu'(y_c)},&\  (\pi-\f{1}{\al}< y_c\leq\pi).
\end{aligned}
\right.
\end{align*}
On the other hand, by Proposition \ref{prop:phi1} and Lemma \ref{Rmk:u_1} again, we have
\begin{align*}
-\mathrm{II}(c)=&\int_0^{\pi}\frac{1}{(u(y)-c)^2}\Big(1-\frac{1}{\phi_{1}(y,c)^2}\Big)dy\\
\leq&C\int_0^{\pi}\frac{\al^2|y-y_c|^2}{(u(y)-c)^2}dy
\leq C\frac{\al^2}{u'(y_c)},
\end{align*}
or
\begin{align*}
-\mathrm{II}(c)=&\int_0^{\pi}\frac{1}{(u(y)-c)^2}\Big(1-\frac{1}{\phi_{1}(y,c)^2}\Big)dy\\
\leq&C\int_0^{\pi}\frac{\min\{\al^2|y-y_c|^2,1\}}{u'(y_c)^2|y-y_c|^2}dy
\leq C\frac{\al}{u'(y_c)^2}.
\end{align*}

Next we prove the derivative estimates of $\mathrm{II}$. Direct calculations show that
\begin{align*}
\pa_c\mathrm{II}(c)=&\pa_c\int_0^{\pi}\frac{1}{(u(y)-c)^2}\Big(\frac{1}{\phi_{1}(y,c)^2}-1\Big)dy\\
=&\int_0^{\pi}\left(\frac{\partial_y}{u'(y_c)}+\partial_c\right)\Big(\frac{1}{(u(y)-c)^2}\Big(\frac{1}{\phi_{1}(y,c)^2}-1\Big)\Big)dy\\
&-\frac{1}{u'(y_c)(u(y)-c)^2}\Big(\frac{1}{\phi_{1}(y,c)^2}-1\Big)\Big|_{y=0}^{\pi}.
\end{align*}
and
\begin{align*}
\pa_c^2\mathrm{II}(c)
=&\int_0^{\pi}\left(\frac{\partial_y}{u'(y_c)}+\partial_c\right)^2\Big(\frac{1}{(u(y)-c)^2}\Big(\frac{1}{\phi_{1}(y,c)^2}-1\Big)\Big)dy\\
&-\frac{1}{u'(y_c)}\left(\frac{\partial_y}{u'(y_c)}+\partial_c\right)
\Big(\frac{1}{(u(y)-c)^2}\Big(\frac{1}{\phi_{1}(y,c)^2}-1\Big)\Big)\Big|_{y=0}^{\pi},\\
&-\partial_c\left(\frac{1}{u'(y_c)(u(y)-c)^2}\Big(\frac{1}{\phi_{1}(y,c)^2}-1\Big)\Big|_{y=0}^{\pi}\right).
\end{align*}
By \eqref{eq:u-gd} and Proposition \ref{prop:phi1}, we get
\beno
\left|\left(\frac{\partial_y}{u'(y_c)}+\partial_c\right)\Big(\frac{1}{(u(y)-c)^2}\Big(\frac{1}{\phi_{1}(y,c)^2}-1\Big)\Big)\right|\leq \frac{C\min\{\al^2|y-y_c|^2,1\}}{u'(y_c)^2(u(y)-c)^2},
\eeno
and
\beno
\left|\left(\frac{\partial_y}{u'(y_c)}+\partial_c\right)^2\Big(\frac{1}{(u(y)-c)^2}\Big(\frac{1}{\phi_{1}(y,c)^2}-1\Big)\Big)\right|\leq \frac{C\min\{\al^2|y-y_c|^2,1\}}{u'(y_c)^4(u(y)-c)^2}.
\eeno
Here we used the fact that
\begin{align*}
\left|\big(\frac{\partial_y}{u'(y_c)}+\partial_c\big)^2(u(y)-c)\right|=\f{|u''(y)u'(y_c)-u''(y_c)u'(y)|}{u'(y_c)^3}=\f{\sin|y-y_c|}{\sin^3y_c}\leq \f{C|u(y)-c|}{\sin^4y_c}.
\end{align*}

Using Proposition \ref{prop:Rayleigh-Hom} and the facts that $\rho(c)\le |u(y)-c|$ and $|u(y)-c|\sim |y-y_c|^2\ge C^{-1}u'(y_c)^2$
for $y=0,\pi$, we deduce that for $y=0,\pi$:
\begin{align*}
\left|\frac{1}{u'(y_c)(u(y)-c)^2}\Big(\frac{1}{\phi_{1}(y,c)^2}-1\Big)\right|\leq C\f{\min\{\al^2|y-y_c|^2,1\}}{u'(y_c)(u(y)-c)^2}\leq C\f{\min\{\al^2,\al/u'(y_c)\}}{u'(y_c)\rho(c)},
\end{align*}
and
\begin{align*}
&\left|\partial_c\left(\frac{1}{u'(y_c)(u(y)-c)^2}\Big(\frac{1}{\phi_{1}(y,c)^2}-1\Big)\right)\right|\\
&\leq
C\left(\frac{1}{u'(y_c)^3(u(y)-c)^2}+\frac{1}{u'(y_c)|u(y)-c|^3}\right)\Big|\frac{1}{\phi_{1}(y,c)^2}-1\Big|\\
&\quad+2\frac{1}{u'(y_c)(u(y)-c)^2}\f{|\cG(y,c)|}{\phi_{1}(y,c)^2}\\
&\leq
C\left(\frac{\min\{\al^2|y-y_c|^2,1\}}{u'(y_c)^3(u(y)-c)^2}+\frac{\min\{\al^2|y-y_c|^2,1\}}{u'(y_c)|u(y)-c|^3}
+\frac{\al\min\{\al|y-y_c|,1\}}{u'(y_c)^2(u(y)-c)^2\phi_{1}(y,c)^2}\right)\\
&\leq
C\left(\frac{\min\{\al^2|y-y_c|^2,1\}}{u'(y_c)|u(y)-c|\rho(c)^2}+
\frac{\al\min\{\al|y-y_c|,1\}}{u'(y_c)^2(u(y)-c)^2\phi_{1}(y,c)^2}\right)\\
&\leq \f{C}{\rho(c)^2}\min\left\{\frac{\al^2}{u'(y_c)},\frac{\al}{u'(y_c)^2}\right\},
\end{align*}
here $\cG(y,c)=\f{\pa_c\phi_1}{\phi_1}(y,c)$.

With the estimates above and Lemma \ref{Rmk:u_1}, we can deduce the high order derivative estimates of $\mathrm{II}(c)$.
\end{proof}

Let us begin with the proof of Proposition \ref{Prop: A^2+B^2}.

\begin{proof}
It follow from Lemma \ref{Lem: II} that
\begin{align*}
\rmA(c)^2+\rmB(c)^2\geq C^{-1}(2\min\{\al^2 \sin^2 y_c,\al^4 \sin^4 y_c\}+\cos^2 y_c)
\geq C^{-1}(1+\al\sin y_c)^2,
\end{align*}
and the upper bound of $\rmA(c)^2+\rmB(c)^2$.

By Lemma \ref{Lem: II} again, we have
\begin{align*}
&|\mathrm{A}(c)|=|\sin^3y_c\mathrm{II}(c)|\leq C\al \sin y_c,\\
&|\pa_{c}\mathrm{A}(c)|=|\pa_{c}(\sin^3y_c\mathrm{II}(c))|
\leq C\min\{\al^2, \al/\sin y_c\},\\
&|\pa_{c}^2\mathrm{A}(c)|=|\pa_{c}^2(u'(y_c)\rho\mathrm{II}(c))|
\leq C\min\{\al^2/\sin^2y_c,\al/\sin^3y_c\},
\end{align*}
which together with the fact $\pa_c\rmB(c)=-\pi$ and $\pa_c^2\rmB(c)=0$ show that
\begin{align*}
\left|\partial_c\left(\f{1}{\mathrm{A}(c)^2+\mathrm{B}(c)^2}\right)\right|
&\leq \f{C}{(1+\al\sin y_c)^2{\sin^2 y_c}},\\
\left|\partial_c^2\left(\f{1}{\mathrm{A}(c)^2+\mathrm{B}(c)^2}\right)\right|
&\leq \f{C}{(\mathrm{A}(c)^2+\mathrm{B}(c)^2){\sin^4y_c}}+
\f{C|\partial_c\mathrm{A}(c)|^2}{(\mathrm{A}(c)^2+\mathrm{B}(c)^2)^2}
+\f{C|\mathrm{A}(c)\partial_c^2\mathrm{A}(c)|}{(\mathrm{A}(c)^2+\mathrm{B}(c)^2)^2}\\
&\leq \f{C}{(1+\al\sin y_c)^2{\sin^4y_c}}
+\f{C(|\al/\sin y_c|^2)}{(1+\al\sin y_c)^4}
+\f{C\al^2/\sin^2 y_c}{(1+\al\sin y_c)^4}\\
&\leq \f{C}{(1+\al\sin y_c)^2\sin^4y_c}.
\end{align*}
\end{proof}

\subsection{Uniform estimates of $\rmA_1(c),\rmB_1(c)$}

\begin{proposition}\label{Prop: A_1^2+B_1^2}
Let $\rmA_1(c), \rmB_1(c)$ be defined by \eqref{eq:defA_1B_1}. Then there exists a constant $C$ independent of $\al$, such that
\begin{align*}
&\f{(1+\al\sin  y_c)^6}{C\al^4}\leq \rmA_1(c)^2+\rmB_1(c)^2\leq \f{C(1+\al\sin  y_c)^6}{\al^4},\\
&\left|\partial_c\Big(\f{1}{\mathrm{A}_1(c)^2+\mathrm{B}_1(c)^2}\Big)\right|\leq \f{C\al^6}{(1+\al\sin y_c)^8},\\
&\left|\partial_c^2\Big(\f{1}{\mathrm{A}_1(c)^2+\mathrm{B}_1(c)^2}\Big)\right|\leq \f{C\al^8}{(1+\al\sin y_c)^{10}}.
\end{align*}
\end{proposition}

To prove Proposition \ref{Prop: A_1^2+B_1^2}, we need a criterion on embedding eigenvalues.\begin{definition}\label{embed}
We say that $c\in \text{Ran}\ u$ is an embedding eigenvalue of $\mathcal{R}_\al$ if there exists $0\neq \psi\in H^1(\mathbb{T})$ such that  for all $\varphi\in H^1(\mathbb{T})$,
\begin{align}\label{u1}
\int_{\mathbb{T}}(\psi'\varphi'+\al^2\psi\varphi)dy+p.v.\int_{\mathbb{T}}\frac{u''\psi\varphi}{u-c}dy+i\pi\sum_{y\in u^{-1}\{c\},u'(y)\neq 0}\frac{(u''\psi\varphi)(y)}{|u'(y)|}=0.
\end{align}
\end{definition}

For $u(y)=-\cos y$, one can check that if \eqref{u1} holds, then $u''(y)=0$ for some point where $c=u(y)$(see Lemma 5.2 in \cite{WZZ2}),
which means that $c=0$. Thus, $\varphi\in H^2(\mathbb{T})$ is a classical solution to
\begin{align}\label{u2}
-\psi''+\al^2\psi+\frac{u''\psi}{u-c}=0\Leftrightarrow -\psi''+(\al^2-1)\psi=0,
\end{align} which implies that $|\al|=0$ or $1$. On the torus $\mathbb{T}_{\d}^2$ with $0<\d<1$, we have $\al\d\in\mathbb{Z}$. If $\al\neq0$, then $|\al|>1$, which implies that $\mathcal{R}_\al$  has no embedding eigenvalues.

\begin{proposition}\label{prop:spectral}
If $\rmA_1(c)^2+\rmB_1(c)^2=0 $, then $c\in D_0$ is an embedding eigenvalue of $\mathcal{R}_\al$. Thus, if $|\al|>1$, then $\rmA_1(c)^2+\rmB_1(c)^2>0. $
\end{proposition}

\begin{proof}
If $\rmA_1(c)^2+\rmB_1(c)^2=0$, then $\rmB_1(c)=0.$
As $\rmB_1(c)=\sin^2y_c\rmB(c)=(1-c^2)(-\pi c),$ we have $c=0$ or $\pm1.$
Due to $\sin y_c^2\rmA(c)=0$ at $c=\pm1,$ using $|J_j(c)|=\dfrac{-u'(y_c)(u(j\pi)-c)^2}{\phi_1((1-j)\pi,c)^2\mathcal{F}((1-j)\pi,c)}$ and Lemma \ref{Rmk: fg_periodic},
we find that
\beno
&&J_1(c)\neq 0,\,\,J_0(c)=0 \quad\text{for}\,\,c=-1,\\
&&J_1(c)= 0,\,\,J_0(c)\neq0\quad\text{for}\,\,c=1,
\eeno
here we define a function at $c=\pm1$ as its limit as $c\to\pm1.$
Therefore, $\rmA_1(c)\neq 0$ for $c=\pm1,$ and $c=0.$

Now we may assume that $c=0,\ \rmA_1(c)=0,$ then $y_c=\pi/2$. Let
\beno
{\phi}_e(y,c)={\phi(\pi,c)\phi_1'(\pi,c)}{\widetilde{\varphi}}(y,c)-\phi(y,c),
\eeno
where ${\widetilde{\varphi}}(y,c)$ is given by
\begin{align*}
{\widetilde{\varphi}}(y,c)=
&\frac{\phi_1(y,c)}{u'(y_c)}(u(y)-u(\pi))\\
&+\f{\phi_1(y,c)}{u'(y_c)}(u(y)-c)(u(\pi)-c)\int_{\pi}^y\frac{u'(y_c)-u'(z)}{(u(z)-c)^2}dz\\
&+\phi_1(y,c)(u(y)-c)(u(\pi)-c)\int_{\pi}^y\frac{1}{(u(z)-c)^2}\Big(\frac{1}{\phi_1(z,c)^2}-1\Big)dz.
\end{align*}
Then ${\phi}_e(y,c)$ satisfies the homogeneous Rayleigh equation for $y\in (0,\pi).$ Moreover,
\begin{align*}
\partial_y{{\phi}_e}(0,c)=&\frac{\phi(\pi,c)\phi_1'(\pi,c)\phi_1'(0,c)}{u'(y_c)}(u(0)-u(\pi))\\
&+\frac{\phi(\pi,c)\phi_1'(\pi,c)\phi_1'(0,c)}{u'(y_c)}(u(0)-c)(u(\pi)-c)\int_{\pi}^0\frac{u'(y_c)-u'(z)}{(u(z)-c)^2}dz\\
&+\phi(\pi,c)\phi_1'(\pi,c)\phi_1'(0,c)(u(0)-c)(u(\pi)-c)\int_{\pi}^0\frac{1}{(u(z)-c)^2}\Big(\frac{1}{\phi_1(z,c)^2}-1\Big)dz\\
&-\frac{\phi(\pi,c)\phi_1'(\pi,c)}{\phi_1(0,c)}-\phi'(0,c).
\end{align*}
Using the facts that $u(0)-u(\pi)=-2,\ (u(0)-c)(u(\pi)-c)=-1 $, and
\begin{align*}
&\int_{\pi}^0\frac{u'(y_c)-u'(z)}{(u(z)-c)^2}dz=\int_{\pi}^0\frac{1-\sin z}{(\cos z)^2}dz=\int_{\pi}^0\frac{dz}{1+\sin z}=\tan\Big(\frac{z}{2}-\frac{\pi}{4}\Big)\Big|_{\pi}^0=-2,
\end{align*}
we deduce that
\begin{align*}
\partial_y{{\phi}_e}(0,c)=
&\phi(\pi,c)\phi_1'(\pi,c)\phi_1'(0,c)\mathrm{II}(c)
-\frac{\phi(\pi,c)\phi_1'(\pi,c)}{\phi_1(0,c)}-\phi'(0,c).
\end{align*}
Using the facts that $u(0)-c=-1,\ u(\pi)-c=1,\ u'(0)=u'(\pi)=0,\ u'(y_c)=\sin y_c=1,$ we infer that $\phi_1(\pi,c)=\phi(\pi,c),\ -\phi'(0,c)=\phi_1'(0,c)$
and
\begin{align*}
\partial_y{{\phi}_e}(0,c)
=&\phi(\pi,c)\phi_1'(\pi,c)\phi_1'(0,c)(\mathrm{II}(c)+J_1(c)-J_0(c))\\
=&\phi(\pi,c)\phi_1'(\pi,c)\phi_1'(0,c)\rmA_1(c)=0.
\end{align*}
It is easy to verify that $\partial_y{\phi}_e(\pi,c)=0$.

Now we extend ${\phi}_e $ to be an even function with periodic $2\pi$. Then ${\phi}_e\in H^2(\mathbb{T})$ satisfies \eqref{u1}.
This means that $c=0$ is an embedding eigenvalue of $\mathcal{R}_\al$.
\end{proof}

\begin{lemma}\label{Lem: J_j high}
Let $j=0,1$, and $J_j^k(c)$ be defined by \eqref{eq:defJ}. Then there exists a constant $C$  independent of $\al$, such that for $k=1,2,\ m=0,1,2,$
\begin{align*}
&|\partial_c^m J_{1-j}^k(c)|
\leq C\min\Big\{\f{|(1-j)\pi-y_c|^{2k+1}}{\al^2|\sin y_c|^{2m}\phi_1(j\pi,c)^{k-1}},
\frac{\al^{2m-2}}{|(1-j)\pi-y_c|^{s}}
\Big\},
\end{align*}
where $s=1$ for $k=1,m=2,$ and $s=0$ otherwise.
\end{lemma}
\begin{proof}
Recall that $\cF(y,c)=\frac{\partial_y\phi_1}{\phi_1}(y,c)$ and $\cG(y,c)=\frac{\partial_c\phi_1}{\phi_1}(y,c)$.
 A direct calculation gives
\begin{align*}
&\f{1}{\phi_1^{k-1}\pa_y\phi_1}(0,c)=\f{1}{\phi_1(0,c)^k\cF(0,c)},\\
&\pa_c\Big(\f{1}{\phi_1^{k-1}\pa_y\phi_1}\Big)(0,c)=-\f{\pa_c\cF}{\phi_1^k\cF^2}(0,c)-\f{k\cG}{\phi_1^k\cF}(0,c),\\
&\pa_c^2\Big(\f{1}{\phi_1^{k-1}\pa_y\phi_1}\Big)(0,c)
=\f{-k\partial_c\cG+k^2\cG^2}{\phi_1^k\cF}(0,c)
+\f{2k\cG\partial_c\cF}{\phi_1^k\cF^2}(0,c)
-\f{\partial_c^2 \cF}{\phi_1^k\cF^2}(0,c)
+\f{2|\partial_c \cF|^2}{\phi_1^k\cF^3}(0,c),
\end{align*}
from which and Lemma \ref{Rmk: fg_periodic}, we infer that
\begin{align*}
&\left|\f{1}{\phi_1(0,c)^{k-1}\phi_1'(0,c)}\right|\sim \f{1+\al y_c}{\al^2\phi_1(0,c)^2 y_c},\\
&\left|\partial_c\Big(\f{1}{\phi_1^{k-1}\phi_1'}\Big)(0,c)\right|
\leq \f{C(1+\al y_c)^2}{\al^2\phi_1(0,c)^ky_c^2\sin y_c},\\
&\left|\partial_c^2\Big(\f{1}{\phi_1^{k-1}\phi_1'}\Big)(0,c)\right|
\leq \f{C(1+\al y_c)^3}{\al^2\phi_1(0,c)^ky_c^3\sin^2 y_c}.
\end{align*}
Recall that
\beno
J_1^k(c)=\f{-u'(y_c)(u(\pi)-c)^k}{\phi_1(0,c)^{k-1}\phi_1'(0,c)}.
\eeno
Thus, we have
\begin{align*}
\left|\frac{J_{1}^k}{(\pi-y_c)^{2k}}\right|
\leq \f{Cu'(y_c)}{\phi_1(0,c)^k|\cF(0,c)|}
\leq \f{C\sin y_c(1+\al y_c)}{\phi_1(0,c)^k\al^2y_c}
\leq \f{C|\pi-y_c|}{\al^2\phi_1(0,c)^{k-1}},
\end{align*}
here we used $\sin y_c\sim y_c(\pi-y_c)$ and $\phi_1(0,c)\ge C^{-1}e^{C^{-1}\al y_c}$. For $m=1,2$, we have
\begin{align*}
\left|\frac{\partial_c^m J_{1}^k}{\big(u(\pi)-c\big)^k}\right|
&\leq \sum_{n=0}^m{C}{u'(y_c)^{2n-2m+1}}\left|\partial_c^n\f{1}{\phi_1^{k-1}\phi_1'(0,c)}\right|\\
&\leq \f{C(1+\al y_c)^{m+1}}{\al^2\phi_1(0,c)^ky_cu'(y_c)^{2m-1}}\leq \f{C(\pi-y_c)}{\al^2(\sin y_c)^{2m}\phi_1(0,c)^{k-1}}.
\end{align*}

To complete the proof, it suffices to improve the estimate for $y_c< \f{1}{\al}$. In this case,
thanks to $y_c\pa_y\phi_1(0,c)=\pa_y\widetilde{\phi_1}(0,c),$ we get
\beno
J_1^k=\f{-u'(y_c)(u(\pi)-c)^2}{\phi_1(0,c)^{k-1}\phi_1'(0,c)}
=\f{-m_1(y_c)(u(\pi)-c)^ky_c^2}{\widetilde{\phi_1}(0,c)^{k-1}\widetilde{\phi_1}'(0,c)},
\eeno
where $m_1(y)=\f{\sin y}{y}$. Since $m_1\in C^{\infty}([-1,1])$ is even, we have $|\pa_c^mm_1(y_c)|\leq C$ for $y_c<1,\ m=0,1,2.$
Thus, for $0<y_c<\f{1}{\al}$, by Proposition \ref{Prop:tphi_1} and \eqref{eq:tphi'/y_c^2_1},
\beno
\left|J_1^k\right|\leq C\al^{-2},\quad
\left|\partial_cJ_1^k\right|\leq C,\quad
\left|\partial_c^2J_1^k\right|\leq C\al^2.
\eeno
This completes the proof of the case of $j=0$. The case of $j=1$ can be proved similarly.
\end{proof}

Now we are in a position to prove Proposition \ref{Prop: A_1^2+B_1^2}.

\begin{proof}
By Lemma \ref{Lem: J_j high} and Lemma \ref{Lem: II}, we get
\begin{align*}
|\rmA_1(c)|+|\rmB_1(c)|
&\leq |J_1|+|J_0|+C\al \sin^3y_c+\pi|\sin^2y_c\cos y_c|\\
&\leq C(\f{1}{\al}+\al \sin^3y_c+\sin^2y_c)
\leq C\frac{(1+\al \sin y_c)^3}{\al^2}.
\end{align*}

Due to $\phi_1'(y,c)\leq 0$ for $y\leq y_c$, we get $J_1(c)\geq 0$. Similarly, we have $J_0(c)\leq 0$.
Proposition \ref{prop:spectral} ensures that $|\rmA_1(c)|+|\rmB_1(c)|>C_{\al}>0$. Thus, we only need to consider the case of the large $\al$.
Let $M$ be a large number determined later.
Then for $y_c\leq \f{1}{M\al}$, we choose $M$ large enough so that $1\leq \phi_1(0,c)\leq 2$. Then
\beno
|J_1(c)|\geq \f{u'(y_c)}{C|\mathcal{F}(0,c)|}\geq \f{y_c}{C\al\min\{\al y_c,1\}}\geq\f{1}{C\al^2},
\eeno
here $C$ is a constant independent of $M$ and $\al$. Thus, for $y_c\leq \f{1}{M\al}$,
\beno
|\rmA_1(c)|\geq |J_1(c)|+|J_0(c)|-|\sin^2y_c\rmA(c)|\geq \f{1}{C\al^2}-C\al \sin^3y_c\geq \f{(1+\al\sin y_c)^3}{C\al^2}.
\eeno
For $\f{1}{M\al}\leq y_c\leq \f{M}{\al}<1$, we have $|\rmB(c)|\geq |\rmB(0)/2|\geq 1,$ then
$$
|\rmB_1(c)|=|\rho(c)\rmB(c)|
\geq {\rho(c)}\geq \frac{(1+\al \sin y_c)^3}{C\al^2}.
$$
This shows that for $y_c\leq \f{M}{\al}<1$, we have
\beno
|\rmA_1(c)|+|\rmB_1(c)|\geq \dfrac{(1+\al \sin y_c)^3}{C\al^2}.
\eeno
Similarly, for $\pi-y_c\leq \f{M}{\al}<1$, we have
\beno
|\rmA_1(c)|+|\rmB_1(c)|\geq \dfrac{(1+\al \sin y_c)^3}{C\al^2}.
\eeno
While, for $y_c\in [\f{M}{\al},\pi-\f{M}{\al}],\ \al>M,$ and $M$ large enough, we have
\begin{align*}
|\rmA_1(c)|\geq |\sin^2y_cA(c)|-|J_1(c)|-|J_0(c)|\geq\f{|\al\sin^3y_c|}{C}-\f{C}{\al^2}\geq \f{1+\al^3 \sin^3y_c}{C\al^2}.
\end{align*}
Therefore, we obtain
\beno
\f{(1+\al\sin  y_c)^6}{C\al^4}\leq \rmA_1(c)^2+\rmB_1(c)^2\leq \f{C(1+\al\sin  y_c)^6}{\al^4}.
\eeno

By Lemma \ref{eq:defJ} and Lemma \ref{Lem: II}, we have
\begin{align*}
&|\mathrm{A_1}(c)|\leq C(1+\al \sin y_c)^3/\al^2,\quad
|\pa_{c}\mathrm{A}_1(c)|\leq C(1+\al \sin y_c),\\
&|\pa_{c}^2\mathrm{A}_1(c)|\leq C\min\{\al^2,\f{\al}{\sin y_c}\},\\
&|\rmB_1(c)|\leq C\sin^2 y_c,\ |\pa_c\rmB_1(c)|\leq C,\ |\pa_c^2\rmB_1(c)|\leq C.
\end{align*}
Then we can deduce that
\begin{align*}
&\left|\partial_c\Big(\f{1}{\mathrm{A}_1(c)^2+\mathrm{B}_1(c)^2}\Big)\right|
\leq \f{C\al^6}{(1+\al\sin y_c)^8},\\
&\left|\partial_c^2\Big(\f{1}{\mathrm{A}_1(c)^2+\mathrm{B}_1(c)^2}\Big)\right|\\
&\leq \f{C}{(\mathrm{A}_1(c)^2+\mathrm{B}_1(c)^2)^2}
+\f{C|\partial_c\mathrm{A}_1(c)|^2}{(\mathrm{A}_1(c)^2+\mathrm{B}_1(c)^2)^2}
+\f{C|\mathrm{A}_1(c)\partial_c^2\mathrm{A}_1(c)|}{(\mathrm{A}_1(c)^2+\mathrm{B}_1(c)^2)^2}\\
&\leq \f{C\al^8}{(1+\al\sin y_c)^{12}}+\f{C\al^8}{(1+\al\sin y_c)^{10}}+\f{C\al^8}{(1+\al\sin y_c)^{10}}\leq \f{C\al^8}{(1+\al\sin y_c)^{10}}.
\end{align*}
\end{proof}
\begin{remark}
Since the functions $\mathrm{A},\mathrm{B},\mathrm{A}_1,\mathrm{B}_1 $ are continuous with respect to $\al,$ Proposition \ref{Prop: A^2+B^2} and \ref{Prop: A_1^2+B_1^2} are true for every $\al\geq c_0>1$ with the constant $C$ depending only on $c_0.$
\end{remark}

\subsection{Solution formula of the inhomogeneous Rayleigh equation}
The fact that  the Kolmogorov flow is stable for the case of $\al>1$,  implies that $W_o(c)\neq 0$ and $W_e(c)\neq 0$ for all $c$ with $\text{Im}\, c\neq 0$.
\smallskip

Now we  define for $y\in [0,\pi]$ and $c\notin [-1,1]$,
\begin{align}\label{eq: Phi_o}
\Phi_{o}(y,c)=&\phi(y,c)\int_{0}^y\frac{1}{\phi(y',c)^2}
\int_{y_c}^{y'}f_o\phi(y'',c)dy''dy'
+\mu^{o}(c)\phi(y,c)\int_{0}^y\frac{1}{\phi(y',c)^2}dy'\nonumber\\
=&\phi(y,c)\int_{\pi}^y\frac{1}{\phi(y',c)^2}
\int_{y_c}^{y'}f_o\phi(y'',c)dy''dy'
+\mu^{o}(c)\phi(y,c)\int_{\pi}^y\frac{1}{\phi(y',c)^2}dy',
\end{align}
where
\ben\label{eq:def mu^o}
\mu^o(c)=\frac{-\phi(0,c)\phi(\pi,c)\int_{0}^{\pi}\frac{1}{\phi(y',c)^2}
\int_{y_c}^{y'}f_o\phi(y'',c)dy''dy'}{W_o(c)}.
\een
For $y\in [-\pi,0]$, let $\Phi_o(y,c)=-\Phi_o(-y,c)$.\smallskip

We define for $y\in [0,\pi]$ and $c\notin [-1,1]$,
\begin{align}\label{eq: Phi_e1}
\Phi_{e}(y,c)
&=\phi(y,c)\int_{0}^y\frac{1}{\phi(y',c)^2}
\int_{y_c}^{y'}f_e\phi(y'',c)dy''dy'\nonumber\\
&\quad+\mu_*^e(c)\phi(y,c)\int_{0}^y\frac{1}{\phi(y',c)^2}dy'
+\nu_0^e(c)\phi(y,c)\nonumber\\
&=\phi(y,c)\int_{\pi}^y\frac{1}{\phi(y',c)^2}
\int_{y_c}^{y'}f_e\phi(y'',c)dy''dy'\nonumber\\
&\quad+\mu_*^e(c)\phi(y,c)\int_{\pi}^y\frac{1}{\phi(y',c)^2}dy'+\nu_1^e(c)\phi(y,c),
\end{align}
where $\mu_{*}^e$ and $\nu_0^e,\ \nu_1^e$ are determined by
\begin{align*}
&\int_{0}^{\pi}\frac{1}{\phi(y',c)^2}
\int_{y_c}^{y'}f_e\phi(y'',c)dy''dy'
+\mu_*^e(c)\int_{0}^{\pi}\frac{1}{\phi(y',c)^2}dy'
+\nu_0^e(c)-\nu_1^e(c)=0,\\
&\nu_j^e(c)\phi(j\pi,c)\phi'(j\pi,c)+\int_{y_c}^{j\pi}f_e\phi(y'',c)dy''
+\mu_*^e(c)=0,\ \ j=0,1.
\end{align*}
For $y\in [-\pi,0]$, let $\Phi_{e}(y,c)=\Phi_{e}(-y,c)$. We have
\begin{align}
&\nu_{1}^e(c)
=\f{\int_{0}^{\pi}f_e\phi(y',c)dy'+(\phi\phi')\big(0,c\big)\bigg(\int_0^{\pi}\f{\int_{y_c}^{y'}f_e\phi(y'',c)\,dy''}{\phi(y',c)^2}dy'
-\int_0^{\pi}\f{\int_{y_c}^{\pi}f_e\phi(y'',c)dy''}{\phi(y',c)^2}dy'\bigg)}{-\phi(0,c)\phi(\pi,c)W_e(c)},\label{eq:defnu_1^e}\\
&\nu_{0}^e(c)
=\f{\int_{0}^{\pi}f_e\phi(y',c)dy'+(\phi\phi')\big(\pi,c\big)\bigg(\int_0^{\pi}\f{\int_{y_c}^{y'}f_e\phi(y'',c)\,dy''}{\phi(y',c)^2}dy'-\int_0^{\pi}\f{\int_{y_c}^{0}f_e\phi(y'',c)dy''}{\phi(y',c)^2}dy'\bigg)}{-\phi(0,c)\phi(\pi,c)W_e(c)},\label{eq:defnu_0^e}\\
&\mu_*^e(c)
=\f{(\phi\phi')(\pi,c)(\phi\phi')(0,c)\int_{0}^{\pi}\frac{\int_{y_c}^{y'}f_e\phi(y'',c)dy'}{\phi(y',c)^2}dy'-(\phi\phi')(j\pi,c)\int_{y_c}^{(1-j)\pi}f_e\phi(y'',c)dy''\big|_{j=0}^1}{\phi(0,c)\phi(\pi,c)W_e(c)}.\label{eq:defmu_0^e}
\end{align}
\begin{proposition}\label{prop:Rayleigh-IHperiodic}
Let $c\in D_{\epsilon_0}$. Then the solution of \eqref{eq:Rayleigh-Ihom-periodic} takes as follows
\begin{align*}
\Phi(y,c)=\Phi_o(y,c)+\Phi_e(y,c),
\end{align*}
where $\Phi_o$ and $\Phi_e$ are defined by \eqref{eq: Phi_o} and \eqref{eq: Phi_e1} respectively.
\end{proposition}

\begin{proof}
It is easy to check that the inhomogeneous Rayleigh equations \eqref{eq:Rayleigh-Ihom-odd} and \eqref{eq:Rayleigh-Ihom-even-half1} are equivalent to
\ben\label{eq:Ray-ihom-o}
\left\{
\begin{aligned}
&\Big(\phi^2\big(\frac{\Phi_o}{\phi}\big)'\Big)'=f_o\phi,\\
&\Phi_o(0)=\Phi_o(\pi)=0,
\end{aligned}
\right.
\een
and
\ben\label{eq:Ray-ihom-e}
\left\{
\begin{aligned}
&\Big(\phi^2\big(\frac{\Phi_e}{\phi}\big)'\Big)'=f_e\phi,\\
&\Phi_e'(0)=\Phi_e'(\pi)=0.
\end{aligned}
\right.
\een
As $\phi(y,c)\neq 0$ for $c\in D_{\epsilon_0}$, we conclude the proposition by integration twice, and matching the boundary conditions
by using the fact that $\phi\int^y\f 1 {\phi^2}dy'$ and $\phi$ are two independent solutions of the homogeneous Rayleigh equation.
\end{proof}

For $c\in B_{\ep_0}^l\cup B_{\ep_0}^r$, Howard's semi-circle theorem tells us that $c\notin \s(\mathcal{R}_{\al})$. Thus, there exists a unique solution of \eqref{eq:Rayleigh-Ihom-periodic}.

\section{The linearized Euler equations }\label{formulation}

\subsection{The limiting absorption principle}
In this subsection, we establish the limiting absorption principle for the inhomogeneous Rayleigh equation
\beq\label{u3}
(u-c)(\Phi''-\al^2\Phi)-u''\Phi=\om,
\eeq
when $c\in \big\{z\in \mathbb{C}:~0<\inf\limits_{y\in \mathbb{T}}|u(y)-z|<\ep_0\big\}$ for $\ep_0>0$ small enough.

In \cite{WZZ2}, we established the limiting absorption principle
for a class of shear flows in $\mathcal{K}$ by using blow-up analysis and compactness argument.
This result can be easily extended to the following  periodic shear flows $(u(y),0)$:
 \begin{itemize}
\item[(a)] Periodic: $u(y+2\pi)=u(y)$;

\item[(b)] Regularity: $u\in H^3(\mathbb{T})$, where $\mathbb{T}=\mathbb{T}_{2\pi}$;

\item[(c)] Spectrum: $\mathcal{R}_\al$ has no embedding eigenvalues for $\al\neq 0$;

\item[(d)] Curvature: $u''(y)\neq 0$ at critical points(i.e., $u'(y)=0$).
\end{itemize}
We denote by $\mathcal{P}$ the set of periodic flows satisfying (a)-(d). When $\al>1$,  $u(y)=-\cos y\in \mathcal{P}$.

\begin{proposition}\label{prop:ub}
Assume that $u\in \mathcal{P}$. There exists $\epsilon_0$ such that for $c\in \big\{z\in \mathbb{C}:~0<\inf\limits_{y\in\mathbb{T}}|u(y)-z|<\ep_0\big\}$, the solution to \eqref{u3} has the following uniform bound
\beno
\|\Phi\|_{H^1(\mathbb{T})}\leq C\|\omega\|_{H^1(\mathbb{T})}.
\eeno
Here $C$ is a constant independent of $\epsilon_0$.  Moreover,
there exists $\Phi_{\pm}(\al,y,c)\in H_0^1(\mathbb{T}) $ for $c\in \mathrm{Ran}\, u,$ such that $\Phi(\al,\cdot,c\pm i\epsilon)\to \Phi_{\pm}(\al,\cdot,c) $ in $C(\mathbb{T})$ as $\epsilon\to0+ $ and
\beno
\|\Phi_{\pm}(\al,\cdot,c)\|_{H^1(\mathbb{T})}\leq C\|\om\|_{H^1(\mathbb{T})}.
\eeno
\end{proposition}
One can refer to section 6 in \cite{WZZ2} for more details.

\subsection{Explicit formulas of the limits}
In this subsection, we give the explicit formulas of the limits $\Phi_\pm$,
which are important to obtain the explicit decay estimates of the velocity.
In what follows, we denote by $\phi(y,c)$ a solution of \eqref{eq:Rayleigh-H}
obtained by Proposition \ref{prop:Rayleigh-Hom} and let $\phi_1(y,c)=\f{\phi(y,c)}{u(y)-c}$ for $c\in \Om_{\ep}$.

We will use the notations $\rmA, \rmB, \rmA_1,\rmB_1$ and $\rm{II}$ defined in section 4.1.
Let us also introduce some new notations. We define $\Int(\varphi) $ to be a $2\pi$ periodic function so that
\begin{align}
\Int(\varphi)(y)=\int_{0}^y\varphi(y')\,dy' \quad\textrm{for}\, y\in [0,\pi],\quad \Int(\varphi)(y)=\Int(\varphi)(-y)\quad\textrm{for}\, y\in [-\pi,0].\label{def:int}
\end{align}
We introduce
\begin{align}
&\textrm{II}_{1,1}(\varphi)(y_c)
=p.v.\int_0^{\pi}\f{\Int(\varphi)(y)-\Int(\varphi)(y_c)}{(u(y)-u(y_c))^2}\,dy,\\
\label{eq:defII_1}
&\mathrm{II}_1(\varphi)(y_c)=p.v.\int_{0}^{\pi}\f{\int_{y_c}^{y}\varphi(y')\phi_1(y',c)dy'}{\phi(y,c)^2}dy,\\
&\rmE_j(\varphi)(y_c)=\int_{y_c}^{j\pi}\varphi(y)\phi_1(y,c)dy\quad \textrm{for}\ j=0,1,\label{eq:defE_j},
\end{align}
and for $k=0,1,2,$
\beq\label{eq:defL_k}
\mathcal{L}_k(\varphi)(y_c)=\int_0^{\pi}\int_{y_c}^{z}{\varphi}(y)\left(\frac{\partial_z+\partial_y}{u'(y_c)}+\partial_c\right)^k
\left(\f{1}{(u(z)-c)^2}\left(\f{\phi_1(y,c)\,}{\phi_1(z,c)^2}-1\right)\right)\,dydz.
\eeq
It is easy to see that
\beno
\mathrm{II}_{1}(\varphi)=\mathrm{II}_{1,1}(\varphi)+\mathcal{L}_{0}
(\varphi).
\eeno

\no{\bf Note.}\, Since the map from $c\in [-1,1]$ to $y_c\in [0,\pi]$ is one-to-one, we do not make a difference between $c$ and $y_c$ in some places.

\begin{remark}\label{rem:Pi-11}
We have the following formulation:
\begin{align}
\mathrm{II}_{1,1}(\varphi)(y_c)
&=-\partial_c\Big(\frac{H(\Int(\varphi))(y_c)}{2\sin y_c}\Big),\label{eq:Pi11-new}
\end{align}
where $H$ is the Hilbert transform on $\mathbb{T}$ defined by
\beno
H(\varphi)(y_c)=p.v.\int_{-\pi}^{\pi}\cot\f{y_c-y}{2}\varphi(y)dy.
\eeno

Indeed, we have
\begin{align*}
\mathrm{II}_{1,1}(\varphi)
&=p.v.\int_0^{\pi}\f{\Int(\varphi)(y)-\Int(\varphi)(y_c)}{(u(y)-u(y_c))^2}\,dy\nonumber\\
&=\partial_c\Big(\int_0^{\pi}\frac{\Int(\varphi)(y)-\Int(\varphi)(y_c)}{u(y)-c}\,dy\Big)
+\frac{\varphi(y_c)}{u'(y_c)}p.v.\int_0^{\pi}\frac{1}{u(y)-c}\, dy\nonumber\\
&=\partial_c\Big(p.v.\int_0^{\pi}\frac{\Int(\varphi)(y)}{u(y)-c}dy\Big)
-\Int(\varphi)(y_c)\partial_c\Big(p.v.\int_0^{\pi}\frac{dy}{u(y)-c}\Big).
\end{align*}
For an even function $f$, we have
\begin{align*}
p.v.\int_0^{\pi}\frac{f(y)}{u(y)-{c}}\,dy
&=\f{1}{2}p.v.\int_{-\pi}^{\pi}\frac{f(y)}{u(y)-{c}}\,dy\\
&=\f{1}{2\sin y_c}p.v.\int_{-\pi}^{\pi}\frac{(\sin y_c+\sin y)f(y)}{\cos y_c-\cos y}\,dy
=-\f{(Hf)(y_c)}{2\sin y_c},
\end{align*}
which along with the facts that $H1=0,\ p.v.\int_0^{\pi}\frac{dy}{u(y)-c}=0$, shows \eqref{eq:Pi11-new}.
\end{remark}

We denote
\beno
&&\widehat{\om}_o(y)=\f{\widehat{\om}_{0}(\al,y)-\widehat{\om}_{0}(\al, -y)}{2},\quad \widehat{\om}_{e}(y)=\f{\widehat{\om}_{0}(\al, y)+\widehat{\om}_{0}(\al, -y)}{2},\\
&&\ f_o(y,c)=\f{\widehat{\om}_{o}(\al,y)}{i\al (u(y)-c)},\quad  f_e(y,c)=\f{\widehat{\om}_{e}(\al,y)}{i\al (u(y)-c)}.
\eeno
For $c\in (-1,1)$, we define
\begin{align*}
\Phi_\pm^{o}(y,c)
\eqdef\left\{
\begin{aligned}
&\phi(y,c)\int_0^y\frac{1}{\phi(z,c)^2}\int_{y_c}^z\phi f_o(y',c)dy'dz\\
&\quad+\mu_{\pm}^o(\widehat{\om}_{o})(c)\phi(y,c)\int_0^y\frac{1}{\phi(y',c)^2}dy'\quad 0\leq y< y_c,\\
&\phi(y,c)\int_{\pi}^y\frac{1}{\phi(z,c)^2}\int_{y_c}^z\phi f_o(y',c)dy'dz\\
&\quad+\mu_{\pm}^o(\widehat{\om}_{o})(c)\phi(y,c)\int_{\pi}^y\frac{1}{\phi(y',c)^2}dy'\quad y_c< y\leq \pi,
\end{aligned}
\right.
\end{align*}
where
\begin{align}
\mu_{+}^o(\widehat{\om}_{o})(c)&=\frac{1}{\al}\frac{i\sin^3y_c\mathrm{II}_1(\widehat{\om}_{o})(y_c)
-\sin y_c\widehat{\om}_{o}(\al,y_c)\pi}
{-i\pi\cos y_c+\sin^3 y_c\mathrm{II}(c)}
=\f{1}{\al}\f{-\rmC_o(\widehat{\om}_{o})(c)+i\rmD_{o}(\widehat{\om}_{o})(c)}{\rmA(c)-i\rmB(c)},
\label{eq:defmu_+^o}\\
\mu_{-}^o(\widehat{\om}_{o})(c)&=\frac{1}{\al}\frac{i\sin^3y_c\mathrm{II}_1(\widehat{\om}_{o})(y_c)
+\sin y_c\widehat{\om}_{o}(\al,y_c)\pi}
{i\pi\cos y_c+\sin^3 y_c\mathrm{II}(c)}
=\f{1}{\al}\f{\rmC_o(\widehat{\om}_{o})(c)+i\rmD_{o}(\widehat{\om}_{o})(c)}{\rmA(c)+i\rmB(c)},
\label{eq:defmu_-^o}
\end{align}
with
\begin{align}
\rmC_o(\varphi)(c)=\pi\sin y_c\varphi(y_c),\ \ \ \rmD_{o}(\varphi)(c)
&=u'(y_c)\rho(c)\mathrm{II}_1(\varphi)(y_c).\label{eq:defD}
\end{align}
For $y\in [-\pi,0]$, let  $\Phi_{\pm}^{o}(y,c)=-\Phi_{\pm}^{o}(-y,c)$.\smallskip

For $c\in (-1,1)$, we define
\beno
\Phi_{\pm}^{e}(y,c)
\eqdef
\left\{
\begin{aligned}
&\phi(y,c)\int_0^y\f{\int_{y_c}^{y'}\phi f_e(y'',c)dy''}{\phi(y',c)^2}dy'\\
&\quad+\mu_{*\pm}^e(\widehat{\om}_{e})(c)\phi(y,c)\int_0^y\f{1}{\phi(y',c)^2}dy'
+\nu_{0\pm}^e(\widehat{\om}_{e})(c)\phi(y,c),
\ 0\leq y< y_c,\\
&\phi(y,c)\int_{\pi}^y\f{\int_{y_c}^{y'}\phi f_e(y'',c)dy''}{\phi(y',c)^2}dy'\\
&\quad+\mu_{*\pm}^e(\widehat{\om}_{e})(c)\phi(y,c)\int_{\pi}^y\f{1}{\phi(y',c)^2}dy'+\nu_{1\pm}^e(\widehat{\om}_{e})(c)\phi(y,c), \  y_c< y\leq \pi,
\end{aligned}
\right.
\eeno
where
\begin{align}
\label{eq:defmu_*+^e}
\mu_{*+}^e(\varphi)(c)&=\f{1}{\al}\f{I(c)\big(i\rmD_{o}(\varphi)(c)-\rmC_{o}(\varphi)(c)\big)-i(\phi\phi')(j\pi,c)\rmE_{1-j}(\varphi)(c)\big|_{j=0}^1 }{I(c)(\rmA_1(c)-i\rmB_1(c))/\rho(c)},\\
\mu_{*-}^e(\varphi)(c)&=\f{1}{\al}\f{I(c)\big(i\rmD_{o}(\varphi)(c)+\rmC_{o}(\varphi)(c)\big)-i(\phi\phi')(j\pi,c)\rmE_{1-j}(\varphi)(c)\big|_{j=0}^1}
{I(c)(\rmA_1(c)+i\rmB_1(c))/\rho(c)},\\
\nu_{k+}^e(\varphi)(c)&=\f{1}{\al}\f{i\rmE(\varphi)+(\phi\phi')\big((1-k)\pi,c\big)\left(\rmD_{+}(\varphi)-i\rmE_k(\varphi)(\rmA-i\rmB)\right)/\sin^3y_c}
{I(c)(\rmA_1(c)-i\rmB_1(c))/\rho(c)},\\
\label{eq:defnu_k+^e}
\nu_{k-}^e(\varphi)(c)&=\f{1}{\al}\f{i\rmE(\varphi)+(\phi\phi')\big((1-k)\pi,c\big)\left(\rmD_{-}(\varphi)-i\rmE_k(\varphi)(\rmA+i\rmB)\right)/\sin^3y_c}
{I(c)(\rmA_1(c)+i\rmB_1(c))/\rho(c)},
\end{align}
with
\beno
\rmE(\varphi)=\rmE_1(\varphi)(y_c)-\rmE_0(\varphi)(y_c),\quad \rmD_\pm(\varphi)=i\rmD_{o}(\varphi)(c)\mp\rmC_{o}(\varphi)(c),
\eeno
and
\begin{align}
I(c)&=\f{\phi_1(\pi,c)\phi'(\pi,c)\phi_1(0,c)\phi'(0,c)}{u'(y_c)}\nonumber\\
&=-\sin y_c(\phi_1\phi_1')(\pi,c)(\phi_1\phi_1')(0,c)
=-\f{(\phi\phi')(\pi,c)(\phi\phi')(0,c)}{u'(y_c)\rho(c)}.\nonumber
\end{align}

\begin{proposition}\label{prop:psi-conv-periodic}
Let $\Phi(y,c)$ be a solution of \eqref{eq:Rayleigh-Ihom-periodic} given by Proposition \ref{prop:Rayleigh-IHperiodic}
with $f=\f {\widehat{\om_0}(\al,y)} {i\al(u-c)}$ for $\widehat{\om}_0\in H^1(\mathbb{T})$.
Then it holds that for any $y\in[-\pi,\pi]$, $y_c\in(-\pi,\pi)\setminus\{0\}$ and $y\neq \pm y_c$,
\beno
\lim_{\epsilon\to 0+}\Phi(y,c_\epsilon)=\Phi_+(y,u(y_c)),\quad
\lim_{\epsilon\to 0-}\Phi(y,c_\epsilon)=\Phi_-(y,u(y_c)),
\eeno
where $\Phi_\pm=\Phi^o_\pm+\Phi^e_\pm$, and $c_\epsilon=c+i\epsilon\in D_{\ep_0}$ and $c=u(y_c)$.
\end{proposition}

The proposition is an immediate consequence of Lemma \ref{Lem:convergemu} and Lemma \ref{lem:5.5}.

\begin{lemma}\label{Lem: II_1 convergence}
Let $c_{\epsilon}=c+i\ep\in D_{\epsilon_0},\ c\in(-1,1)$. Then for any $\varphi\in H^1(0,\pi)$,
\begin{align*}
&\lim_{\ep\to 0+}
\rho(c_{\epsilon})\int_{0}^{\pi}\f{\int_{y_{c}}^y\varphi(y')dy'}{(u(y)-c_{\epsilon})^2}dy
=\rho(c)\mathrm{II}_{1,1}(\varphi)(c)
+i\pi {\varphi(y_c)},\\
&\lim_{\ep\to 0-}
\rho(c_{\epsilon})\int_{0}^{\pi}\f{\int_{y_{c}}^y\varphi(y')dy'}{(u(y)-c_{\epsilon})^2}dy
=\rho(c)\mathrm{II}_{1,1}(\varphi)(c)
-i\pi {\varphi(y_c)}.
\end{align*}
\end{lemma}

\begin{proof}
Set $g(y)=\int_{y_{c}}^y\varphi(y')dy'-\frac{\varphi(y_c)}{u'(y_c)}(u(y)-c)$. Then we have
\begin{align*}
&\rho(c_{\epsilon})\int_{0}^{\pi}\f{\int_{y_{c}}^y\varphi(y')dy'}{(u(y)-c_{\epsilon})^2}dy\\
&=\rho(c_\epsilon)\int_0^{\pi}\frac{g(y)}{(u(y)-c_{\epsilon})^2}dy
+\rho(c_\epsilon)\frac{\varphi(y_c)}{u'(y_c)}\int_0^{\pi}\frac{(u(y)-c)}{(u(y)-c_{\epsilon})^2}dy.
\end{align*}
As $\varphi\in H^1(0,\pi)$, we get
\beno
|g'(y)|=|\varphi(y)-u'(y)\varphi(y_c)/u'(y_c)|\leq C|y-y_c|^{\f12}/u'(y_c),
\eeno
which along with $g(y_c)=0$ ensures that
\beno
\int_0^{\pi}\frac{g(y)}{(u(y)-c_{\epsilon})^2}dy\lto p.v. \int_0^{\pi}\frac{g(y)}{(u(y)-c)^2}dy.
\eeno
From the proof of Lemma \ref{Lem:W(c)converge} and $p.v.\int_{0}^{\pi}\frac{dy}{u(y)-c}=0$, we deduce that as $\epsilon\to 0+$,
\begin{align*}
\rho(c_\epsilon)\int_0^{\pi}\frac{(u(y)-c)}{(u(y)-c_{\epsilon})^2}dy
&= \rho(c_\epsilon)\int_0^{\pi}\frac{dy}{(u(y)-c_{\epsilon})}dy+i\ep\rho(c_\epsilon)\int_0^{\pi}\frac{dy}{(u(y)-c_{\epsilon})^2}\\
&=\rho(c_\epsilon)\frac{\pi i}{\rho(c_\epsilon)^{\f12}}-i\ep\frac{\partial_c\rho(c_\epsilon)\pi i}{2\rho(c_\epsilon)^{\f12}}\\
&\lto \pi i\rho(c)^{\f12}=p.v.\int_{0}^{\pi}\frac{(u(y)-c)}{(u(y)-c)^2}dy+i\pi u'(y_c).
\end{align*}
This shows that as $\epsilon\to 0+$,
\begin{align*}
\rho(c_{\epsilon})\int_{0}^{\pi}\f{\int_{y_{c}}^y\varphi(y')dy'}{(u(y)-c_{\epsilon})^2}dy
\lto& \rho(c)p.v. \int_0^{\pi}\frac{g(y)+\frac{\varphi(y_c)}{u'(y_c)}(u(y)-c)}{(u(y)-c)^2}dy+i\pi u'(y_c)\frac{\varphi(y_c)}{u'(y_c)}\\&=\rho(c)\mathrm{II}_{1,1}(\varphi)(c)
+i\pi {\varphi(y_c)}.
\end{align*}
The case of $\epsilon\to 0-$ can be proved similarly.
\end{proof}

\begin{lemma}\label{Lem:convergemu}
Let $k\in \{0,1\}$ and $\widehat{\om}_0(\al,y)\in H^1(0,\pi)$. Let $\mu^o$, $\nu_k^e$ and $\mu^e_*$ be defined by \eqref{eq:def mu^o}, \eqref{eq:defnu_1^e}, \eqref{eq:defnu_0^e} and \eqref{eq:defmu_0^e}. Let $\mu_{\pm}^o$, $\mu_{*\pm}^e$ and $\nu_{k\pm}^e$ be defined by \eqref{eq:defmu_+^o}, \eqref{eq:defmu_-^o} and \eqref{eq:defmu_*+^e}-\eqref{eq:defnu_k+^e}. Then it holds that for any $y_c\in (0,\pi)$,
\begin{align*}
&\lim_{\epsilon\to 0+}\mu^o(c_{\ep})=\mu^{o}_+(c),\quad
\lim_{\epsilon\to 0-}\mu^o(c_{\ep})=\mu^o_-(c),\\
&\lim_{\epsilon\to 0+}\mu^e_*(c_{\ep})=\mu^e_{*+}(c),\quad
\lim_{\epsilon\to 0-}\mu^e_*(c_{\ep})=\mu^e_{*-}(c),\\
&\lim_{\epsilon\to 0+}\nu^e_k(c_{\ep})=\nu^e_{k+}(c),\quad
\lim_{\epsilon\to 0-}\nu^e_k(c_{\ep})=\nu^e_{k-}(c).
\end{align*}
Here $c_{\ep}=c+i\ep=u(y_c)+i\epsilon\in D_{\ep_0}$.
\end{lemma}

\begin{proof}
By Proposition \ref{prop:Rayleigh-Hom},  $\phi_1(y,c)$ is continuous for
$(y,c)\in [0,\pi]\times\Om_{\epsilon_0}$ and for $\epsilon_0$ small enough,
$$|\phi_1(y,c)|>\f12,\quad |\phi_1(y,c)-1|\le C|y-y_c|^2.$$
Thus, for $(y,c)\in [0,\pi]\times\Om_{\epsilon_0}$ with  $(z-y_c)(z-y)\leq 0$, there exists a constant $C$ so that
\beno
\left|\f{\rho(c)}{(u(y)-c)^2}\Big(\f{\phi_1(z,c)}{\phi_1(y,c)^2}-1\Big)\right|\leq C.
\eeno
This implies that  as $c_{\epsilon}\to u(y_c)=c$,
\begin{align*}
&\rho(c_{\epsilon})\int_0^{\pi}\int_{y_c}^y\f{\varphi(z)}{(u(y)-c_{\epsilon})^2}\Big(\f{\phi_1(z,c_{\ep})}{\phi_1(y,c_{\epsilon})^2}-1\Big)dzdy\lto \rho(c)\mathcal{L}_{0}(\varphi)(c),
\end{align*}
from which and Lemma \ref{Lem: II_1 convergence}, it follows that as $\ep\to 0\pm$,
\beno
&&i\al u'(y_c)\rho(c_\ep)\int_{0}^{\pi}\frac{1}{\phi(y',c_\ep)^2}
\int_{y_c}^{y'}f_o\phi(y'',c_\ep)dy''dy'\\
&&\lto \rho(c)u'(y_c)\mathcal{L}_{0}(\widehat\om)(c)+\rho(c)u'(y_c)\mathrm{II}_{1,1}(\widehat\om)(c)
\pm i\pi {\widehat\om(y_c)} {u'(y_c)}.
\eeno
Then by Lemma \ref{Lem:W(c)converge}, we infer that
\beno
\lim_{\ep\to 0\pm}\mu^o(c_\ep)=\mu_\pm^o(c).
\eeno

Notice that for $c_{\ep}=u(y_c)+ i\ep$ and $k\in\{0,1\}$,
\beno
\int_{y_{c}}^{k\pi}f_e\phi(y,c_{\ep})dy\lto \f{1}{i\al}\rmE_k(\widehat{\om}_e),\quad \textrm{as}\ \ep\to 0.
\eeno
The other limits  can be proved similarly.
\end{proof}

\begin{lemma}\label{lem:5.5}
Let $\varphi(y)\in H^1(0,\pi)$. Then it holds that for $0\leq y<y_c\leq \pi$,
\begin{align*}
&\lim_{\ep\to 0}\phi(y,c_{\ep})\int_0^y\f{\int_{y_c}^{y'}\varphi(y'')\phi_1(y'',c_{\ep})dy''}{\phi(y',c_{\ep})^2}dy'= \phi(y,u(y_c))\int_0^y\f{\int_{y_c}^{y'}\varphi(y'')\phi_1(y'',u(y_c))dy''}{\phi(y',u(y_c))^2}dy',\\
&\lim_{\ep\to 0}\phi(y,c_{\ep})\int_0^y\f{1}{\phi(y',c_{\ep})^2}dy'= \phi(y,u(y_c))\int_0^y\f{1}{\phi(y',u(y_c))^2}dy',
\end{align*}
and for $0\leq y_c<y\leq \pi$,
\begin{align*}
&\lim_{\ep\to 0}\phi(y,c_{\ep})\int_{\pi}^y\f{\int_{y_c}^{y'}\varphi(y'')\phi_1(y'',c_{\ep})dy''}{\phi(y',c_{\ep})^2}dy'= \phi(y,u(y_c))\int_{\pi}^y\f{\int_{y_c}^{y'}\varphi(y'')\phi_1(y'',u(y_c))dy''}{\phi(y',u(y_c))^2}dy',\\
&\lim_{\ep\to 0}\phi(y,c_{\ep})\int_{\pi}^y\f{1}{\phi(y',c_{\ep})^2}dy'= \phi(y,u(y_c))\int_{\pi}^y\f{1}{\phi(y',u(y_c))^2}dy'.
\end{align*}
Here $c_{\ep}=c+i\ep=u(y_c)+ i\ep\in D_{\ep_0}$.
\end{lemma}

\begin{proof}
We consider the case of $0\leq y<y_c\leq \pi$. Another case is similar. Notice that for $0\leq y'\leq y,\ y'\leq y''\leq y_c,$
\begin{align}\label{eq:term 1}
&\left|\f{\phi(y,c_{\ep})}{\phi(y',c_{\ep})^2}\right|+\left|\f{\phi(y,c_{\ep})\varphi(y'')\phi_1(y'',c_{\ep})}{\phi(y',c_{\ep})^2}\right|\\
&\leq C\f{|u(y)-c_{\ep}|}{|u(y')-c_{\ep}|^2}\leq \f{C}{|u(y')-c_{\ep}|}\leq \f{C}{|u(y)-u(y_c)|}.\nonumber
\end{align}
Then the result follows from Lebesgue's dominated convergence theorem and the continuity of $\phi_1$, $\phi$.
\end{proof}

\subsection{Solution formula  of the linearized Euler equations}
Let $\widehat{\psi}(t,\al,y)$ be a solution of \eqref{eq:Euler-L-operator} with initial data $\widehat{\psi}(0,\al,y)=-(\pa_y^2-\al^2)^{-1}\widehat{\om}_0(\al,y)$. Then we have
\begin{align*}
\widehat{\psi}(t,\al,y)=\f{1}{2\pi i}\oint_{\pa \Om} e^{-i\al tc}(c-\mathcal{R}_{\al})^{-1}\widehat{\psi}(0,\al,y) dc,
\end{align*}
where $\Om$ is a simply connected domain including the spectrum $\s(\mathcal{R}_{\al})=[-1,1]$ of $\mathcal{R}_{\al}$.

Let $\Phi(\al,y,c)$ be a solution of \eqref{eq:Rayleigh-Ihom-periodic} with $f(\al,y,c)=\f{\widehat{\om}_0(\al,y)}{i\al (u(y)-c)}$. We have
\beno
(c-\mathcal{R}_{\al})^{-1}\widehat{\psi}(0,\al,y)=i\al \Phi(\al,y,c).
\eeno
Therefore, we have
\ben\label{eq:psi(t)1}
\widehat{\psi}(t,\al,y)=\f{1}{2\pi}\oint_{\pa \Om_{\ep_0}} e^{-i\al tc}\al \Phi(\al,y,c) dc.
\een
In particular,
\ben\label{eq:psi(0)1}
\widehat{\psi}(0,\al,y)=\f{1}{2\pi}\oint_{\pa \Om_{\ep_0}}\al \Phi(\al,y,c) dc.
\een
\begin{lemma}\label{Lem:psi(t)}
Let $\widehat{\psi}(t,\al,y)$ be as in \eqref{eq:psi(t)1}. Then we have
\begin{align*}
\widehat{\psi}(t,\al,y)=&\f{1}{2\pi}\int_{-1}^1e^{-i\al tc}\al \widetilde{\Phi}_o(\al,y,c)dc+\f{1}{2\pi}\int_{-1}^1e^{-i\al tc}\al \widetilde{\Phi}_e(\al,y,c)dc\\
\triangleq&\widehat{\psi}_o(t,\al,y)+\widehat{\psi}_e(t,\al,y).
\end{align*}
In particular,
\beno
\widehat{\psi}(0,\al,y)=\f{1}{2\pi}\int_{-1}^1\al \widetilde{\Phi}_o(\al,y,c)dc+\f{1}{2\pi}\int_{-1}^1\al \widetilde{\Phi}_e(\al,y,c)dc,
\eeno
where
\beno
\widetilde{\Phi}_{o}(y,c)
=\left\{
\begin{aligned}
&(\mu_{-}^o(c)-\mu_{+}^o(c))\phi(y,c)\int_0^y\frac{1}{\phi(z,c)^2}dz,\ 0\leq y\leq y_c,\\
&(\mu_{-}^o(c)-\mu_{+}^o(c))\phi(y,c)\int_{\pi}^y\frac{1}{\phi(z,c)^2}dz,\ y_c\leq y\leq \pi,
\end{aligned}
\right.
\eeno
and
\beno
\widetilde{\Phi}_{e}(y,c)
=\left\{
\begin{aligned}
&(\mu_{*-}^e(c)-\mu_{*+}^e(c))\int_0^y\frac{\phi(y,c)}{\phi(z,c)^2}dz
+(\nu_{0-}^e(c)-\nu_{0+}^e(c))\phi(y,c),\ 0\leq y<y_c,\\
&(\mu_{*-}^e(c)-\mu_{*+}^e(c))\int_{\pi}^y\frac{\phi(y,c)}{\phi(z,c)^2}dz+(\nu_{1-}^e(c)-\nu_{1+}^e(c))\phi(y,c),\  y_c< y\leq \pi,
\end{aligned}
\right.
\eeno
and $\widetilde{\Phi}_{o}(y,c)=-\widetilde{\Phi}_{o}(-y,c)$, $\widetilde{\Phi}_{e}(y,c)=\widetilde{\Phi}_{e}(-y,c)$ for $y\in [-\pi,0]$.
\end{lemma}

\begin{proof}
For any $0<\ep\leq \ep_0$, we have
\begin{align*}
\widehat{\psi}(t,\al,y)&=\f{1}{2\pi}\oint_{\pa \Om_{\ep}} e^{-i\al tc}\al \Phi(\al,y,c) dc=I_1+I_2+I_3+I_4,
\end{align*}
where
\begin{align*}
I_1&=\f{1}{2\pi}\int_{-1}^1 e^{-i\al t(c-i\ep)}\al \Phi(\al,y,c-i\ep) dc,\\
I_2&=\f{1}{2\pi}\int_{1}^{-1} e^{-i\al t(c+i\ep)}\al \Phi(\al,y,c+i\ep) dc,\\
I_3&=\f{1}{2\pi}\int_{-\f{\pi}{2}}^{\f{\pi}{2}} e^{-i\al t(1+\ep e^{i\th})}\al \Phi(\al,y,1+\ep e^{i\th}) i\ep e^{i\th}d\th,\\
I_4&=\f{1}{2\pi}\int_{\f{\pi}{2}}^{\f{3\pi}{2}} e^{-i\al t(-1+\ep e^{i\th})}\al \Phi(\al,y,-1+\ep e^{i\th}) i\ep e^{i\th}d\th.
\end{align*}
Proposition \ref{prop:psi-conv-periodic}, Proposition \ref{prop:ub} and Lebesgue's dominated convergence theorem ensure that
\begin{align*}
&\lim_{\ep\to 0+}I_{1}
=\f{1}{2\pi}\int_{-1}^1 \al e^{-i\al tc} (\Phi_-^o(\al,y,c)+\Phi_-^e(\al,y,c)) dc, \\
&\lim_{\ep\to 0+}I_{2}
=-\f{1}{2\pi}\int_{-1}^1 \al e^{-i\al tc} (\Phi_+^o(\al,y,c)+\Phi_+^e(\al,y,c))dc,
\end{align*}
and
\beno
\lim_{\ep\to 0+}I_{3}=0,\quad \lim_{\ep\to 0+}I_{4}=0.
\eeno
Now, the lemma follows from the facts that $\widetilde{\Phi}_o=\Phi_-^o-\Phi_+^o,\ \widetilde{\Phi}_e=\Phi_-^e-\Phi_+^e$.
\end{proof}

\subsection{Dual formulation}

As in \cite{WZZ2}, we will use the dual method to derive the decay estimates of the velocity. So, we introduce the dual formulation of the stream function.

Let
\begin{align*}
\mu_{-}^o(c)-\mu_{+}^o(c)&=\f{2}{\al}\f{\mathrm{A}\mathrm{C}_{o}(\widehat{\om}_{o})+\mathrm{B}\mathrm{D}_{o}(\widehat{\om}_{o})}{\mathrm{A}^2+\mathrm{B}^2}
\eqdef\f{2}{\al}\rho(c)\mu_1(c),\\
\mu_{-}^e(c)-\mu_{+}^e(c)
&=\f{2}{\al}\f{I(c)^2(\rmA\rmC_{e}+\rmB\rmD_{e})+I(c)(\phi\phi')(j\pi,c)(\rmC_{e}+E_{1-j}\rmB)|_{j=0}^1
}{I(c)^2(\rmA_1^2+\rmB_1^2)/\rho(c)^2}\eqdef \f{2}{\al}\mu_2(c),\\
\nu_{j-}^e(c)-\nu_{j+}^e(c)
&=-\f{2}{\al}\f{\mu_2}
{\phi(j\pi,c)\phi'(j\pi,c)}\eqdef \f{2}{\al}\nu_{j}(c).
\end{align*}
Here $\rmC_{e}=\mathrm{C}_{o}(\widehat{\om}_{e})(c),\ \rmD_{e}=\mathrm{D}_{o}(\widehat{\om}_{e})(c)$ and $E_{1-j}=E_{1-j}(\widehat{\om}_{e})(y_c)$.

We denote
\ben
&&\Lambda_{1}(\varphi)(y_c)=\Lambda_{1,1}(\varphi)(y_c)+\Lambda_{1,2}(\varphi)(y_c),\label{def:Lam1}\\
&&\Lambda_{2}(\varphi)(y_c)=\Lambda_{2,1}(\varphi)(y_c)+\Lambda_{2,2}(\varphi)(y_c),\label{def:Lam2}
\een
where
\begin{align}
&\Lambda_{1,1}(\varphi)(y_c)=\rho(c)u''(y_c)\mathrm{II}_{1,1}(\varphi)(c),\label{eq:defLam_1,1}\\
&\Lambda_{1,2}(\varphi)(y_c)=\rho(c)u''(y_c)\mathcal{L}_0(\varphi)(c)+u'(y_c)\rho(c)\mathrm{II}(c)\varphi(y_c),\label{eq:defLam_1,2}\\
&\Lambda_{2,1}(\varphi)(y_c)=\rho(c)\mathrm{II}_{1,1}(u''\varphi)(c),\label{eq:defLam_2,1}\\
&\Lambda_{2,2}(\varphi)(y_c)=\rho(c)\mathcal{L}_0(u''\varphi)(c)+u'(y_c)\rho(c)\mathrm{II}(c)\varphi(y_c),\label{eq:defLam_2,2}
\end{align}

We denote
\begin{align}
&\Lambda_3(\widehat{\omega}_{e})(y_c)=\rho(c)\Lambda_{1}(\widehat{\om}_e)(y_c)+\Lambda_{3,1}(\widehat{\om}_e)(y_c),\label{eq:def Lam_3}\\
&\Lambda_4(g)(y_c)=\rho(c)\Lambda_2(g)(y_c)+\Lambda_{4,1}(g)(y_c),\label{eq:def Lam_4}
\end{align}
where
\begin{align}
&\Lambda_{3,1}(\widehat{\om}_e)(y_c)=J_j\left(\frac{u''(y_c)}{u'(y_c)}\rmE_{1-j}(\widehat{\omega}_{e})
+{\widehat{\omega}_{e}(y_c)}\right)\bigg|_{j=0}^{1},\label{eq:defLambda_3,1}\\
&\Lambda_{4,1}(g)(y_c)=J_j\left(\frac{\rmE_{1-j}(gu'')}{u'(y_c)}+g(y_c)\right)\bigg|_{j=0}^{1},\label{eq:defLambda_4,1}
\end{align}
with
\begin{align}
J_j(c)=&\f{(\phi\phi')(j\pi,c)\rho(c)}{I(c)}=\f{-u'(y_c)\rho(c)^2}{(\phi\phi')((1-j)\pi,c)}\nonumber\\
=&\f{-u'(y_c)(u(j\pi)-c)^2}{(\phi_1\phi_1')((1-j)\pi,c)}\quad \textrm{for}\ j=0,1.\label{def:J_j}
\end{align}

\begin{lemma}\label{Lem: K_o identity}
Let $f(\al,y)=(\pa_y^2-\al^2)g(\al,y)$ with $g\in H^2(0,\pi)\cap H_0^1(0,\pi)$. Then we have
\beno
\int_0^{\pi}\widehat{\psi}_o(t,\al,y)f(\al,y)dy=
-\int_{-1}^{1}K_o(c,\al)e^{-i\al ct}dc,
\eeno
where
\begin{align}\label{eq:defK_o}
K_o(c,\al)=\f{\Lambda_1(\widehat{\om}_o)(y_c)\Lambda_2(g)(y_c)}{(\mathrm{A}(c)^2+\mathrm{B}(c)^2){u'(y_c)}}.
\end{align}

\end{lemma}
\begin{proof}
By Lemma \ref{Lem:psi(t)}, we get
\begin{align*}
&\int_0^{\pi}\widehat{\psi}_o(t,\al,y)f(\al,y)dy\\
&=\dfrac{1}{\pi}\int_0^{\pi}f(\al,y)\int_{-1}^{u(y)}\rho(c)\mu_1(c)\phi(y,c)
\int_{\pi}^y\dfrac{1}{\phi(z,c)^2}dze^{-i\al ct}dc dy\\
&\quad+\dfrac{1}{\pi}\int_0^{\pi}f(\al,y)\int_{u(y)}^{1}\rho(c)\mu_1(c)\phi(y,c)
\int_0^y\dfrac{1}{\phi(z,c)^2}dze^{-i\al ct}dc dy\\
&=-\f{1}{\pi}\int_{-1}^{1}\rho(c)\mu_1(c)e^{-i\al ct}
\int_0^{\pi}\dfrac{\int_{y_c}^zf(\al,y)\phi(y,c)dy}{\phi(z,c)^2}dz dc.
\end{align*}
First of all, we have
$$
\rho(c)\mu_1(c)=\pi\f{\rmA(c)\rho(c)\frac{\widehat{\omega}_{o}(y_c)}{u'(y_c)}
+\rho(c)^2\frac{u''(y_c)}{u'(y_c)}\textrm{II}_1(\widehat{\om}_{o})(c)}{\rmA(c)^2+\rmB(c)^2}=\pi\f{\sin y_c{\Lambda}_1(\widehat{\om}_{o})}{\rmA(c)^2+\rmB(c)^2}.
$$
We write
\begin{align*}
&\int_0^{\pi}\dfrac{\int_{y_c}^zf(\al,y)\phi(y,c)dy}{\phi(z,c)^2}dz\\
&=\int_0^{\pi}\dfrac{\int_{y_c}^z(g''(y)-\al^2g(y))\phi(y,c)dy}{\phi(z,c)^2}dz\\
&=\int_0^{\pi}\dfrac{\int_{y_c}^zg(y)(\phi''-\al^2\phi)(y,c)dy+g'(z)\phi(z,c)-g(z)\phi'(z,c)+g(y_c)\phi'(y_c,c)}{\phi(z,c)^2}dz\\
&=\int_0^{\pi}\left[\dfrac{\int_{y_c}^zg(y)u''(y)\phi_1(y,c)dy}{\phi(z,c)^2}+\left(\dfrac{g(z)}{\phi(z,c)}\right)'
+\dfrac{g(y_c)u'(y_c)}{\phi(z,c)^2}\right]dz\\
&=\mathrm{II}_1(gu'')+p.v.\int_0^{\pi}\left[\left(\dfrac{g(z)}{\phi(z,c)}\right)'
+\dfrac{g(y_c)u'(y_c)}{\phi(z,c)^2}\right]dz.
\end{align*}
Notice that for $c\in(-1,1)$, we have
\beno
\f{g(z)}{\phi(z,c)}-\f{g(y_c)}{u(z)-c}=\f{g(z)(1-\phi_1(z,c))}{(u(z)-c)\phi_1(z,c)}+\f{g(z)-g(y_c)}{u(z)-c}\in C([0,\pi]),
\eeno
then we get
\begin{align*}
&p.v.\int_0^{\pi}\left[\left(\f{g(z)}{\phi(z,c)}\right)'
+\f{g(y_c)u'(y_c)}{\phi(z,c)^2}\right]dz\\
&=\left(\f{g(z)}{\phi(z,c)}-\f{g(y_c)}{u(z)-c}\right)\bigg|_0^{\pi}
+p.v.\int_0^{\pi}\left[\left(\f{g(y_c)}{u(z)-c}\right)'
+\f{g(y_c)u'(y_c)}{\phi(z,c)^2}\right]dz\\
&=\f{g(y_c)}{u(0)-c}-\f{g(y_c)}{u(\pi)-c}+
p.v.\int_0^{\pi}\left[-\f{g(y_c)(u'(z)-u'(y_c))}{(u(z)-c)^2}
+\f{g(y_c)u'(y_c)}{(u(z)-c)^2}\left(\f{1}{\phi(z,c)^2}-1\right)\right]dz\\
&=\f{g(y_c)(u(0)-u(\pi))}{\rho(c)}-g(y_c)
p.v.\int_0^{\pi}\f{(u'(z)-u'(y_c))}{(u(z)-c)^2}dz+g(y_c)u'(y_c)\mathrm{II}(c).
\end{align*}
From Remark \ref{rem:Pi-11} and $p.v.\int_{0}^{\pi}\f{1}{u(z)-c}dz=0$,
we know that
\begin{align*}
&p.v.\int_{0}^{\pi}\f{u'(z)-u'(y_c)}{(u(z)-c)^2}dz=\mathrm{II}_{1,1}(u'')(c)\\
&=\partial_c\Big(p.v.\int_0^{\pi}\frac{u'(y)-u'(0)}{u(y)-c}dy\Big)=\partial_c\Big(\ln\frac{u(\pi)-c}{c-u(0)}\Big)=\frac{u(0)-u(\pi)}{\rho(c)}.
\end{align*}
This gives
\beno
p.v.\int_0^{\pi}\left[\left(\f{g(z)}{\phi(z,c)}\right)'
+\f{g(y_c)u'(y_c)}{\phi(z,c)^2}\right]dz=u'(y_c)g(y_c)\mathrm{II}(c),
\eeno
 which implies
\begin{align*}
&\int_0^{\pi}\dfrac{\int_{y_c}^zf(\al,y)\phi(y,c)dy}{\phi(z,c)^2}dz
=\mathrm{II}_1(gu'')+u'(y_c)g(y_c)\mathrm{II}(c)=\Lambda_2(g)(y_c)/\rho(c).
\end{align*}
Therefore,
\begin{align*}
\int_0^{\pi}\widehat{\psi}_o(t,\al,y)f(y)dy
=-\int_{-1}^{1}\f{\Lambda_1(\widehat{\om}_{o})(y_c)\Lambda_2(g)(y_c)}{(\mathrm{A}(c)^2+\mathrm{B}(c)^2)u'(y_c)}e^{-i\al ct}dc,
\end{align*}
which proves the lemma.
\end{proof}

\begin{remark}\label{Rmk:5.9}
From the above proof, we know that if $g\in H^2(0,\pi)$, then  we have
$$p.v.\int_0^{\pi}\left[\left(\f{g(z)}{\phi(z,c)}\right)'
+\f{g(y_c)u'(y_c)}{\phi(z,c)^2}\right]dz=\f{g(z)}{\phi(z,c)}\Big|_{z=0}^{\pi}+u'(y_c)g(y_c)\mathrm{II}(c).$$\end{remark}
\begin{lemma}\label{lem:kernel identityeven}
Let $f(\al,y)=(\pa_y^2-\al^2)g(\al,y)$ with $g\in H^2(0,\pi)$ and $g'(0)=g'(\pi)=0$. Then we have
\beno
\int_0^{\pi}\widehat{\psi}_e(t,\al,y)f(\al,y)dy=
-\int_{-1}^{1}K_e(c,\al)e^{-i\al ct}dc,
\eeno
where
\begin{align}\label{eq:defK_e}
K_e(c,\al)=\f{\Lambda_3(\widehat{\omega}_{e})(y_c)\Lambda_4(g)(y_c)}{u'(y_c)(\rmA_1(c)^2+\rmB_1(c)^2)}.
\end{align}
\end{lemma}
\begin{proof}
By Lemma \ref{Lem:psi(t)}, we have
\begin{align*}
&\int_0^{\pi}\widehat{\psi}_e(t,\al,y)f(\al,y)dy\\
&=\dfrac{1}{\pi}\int_0^{\pi}f(\al,y)\int_{-1}^{u(y)}\mu_2(c)\phi(y,c)\int_{\pi}^y\dfrac{1}{\phi(z,c)^2}dze^{-i\al ct}dc dy\\
&\quad+\dfrac{1}{\pi}\int_0^{\pi}f(\al,y)\int_{u(y)}^{1}\mu_2(c)\phi(y,c)\int_0^y\dfrac{1}{\phi(z,c)^2}dze^{-i\al ct}dc dy\\
&\quad+\dfrac{1}{\pi}\int_0^{\pi}f(\al,y)\int_{u(y)}^{1}\nu_0(c)\phi(y,c)e^{-i\al ct}dc dy
+\dfrac{1}{\pi}\int_0^{\pi}f(\al,y)\int_{-1}^{u(y)}\nu_1(c)\phi(y,c)e^{-i\al ct}dc dy\\
&=-\f{1}{\pi}\int_{-1}^{1}\mu_2(c)e^{-i\al ct}
\int_0^{\pi}\dfrac{\int_{y_c}^zf(y)\phi(y,c)dy}{\phi(z,c)^2}dz dc\\
&\quad+\f{1}{\pi}\int_{-1}^{1}\nu_0(c)e^{-i\al ct}
\int_0^{y_c}f(y)\phi(y,c)dy dc
+\f{1}{\pi}\int_{-1}^{1}\nu_1(c)e^{-i\al ct}
\int_{y_c}^{\pi}f(y)\phi(y,c)dy dc.
\end{align*}

We have
\begin{align*}
\int_0^{\pi}\dfrac{\int_{y_c}^zf(y)\phi(y,c)dy}{\phi(z,c)^2}dz
=\mathrm{II}_1(gu'')+p.v.\int_0^{\pi}\left[\left(\dfrac{g(z)}{\phi(z,c)}\right)'
+\dfrac{g(y_c)u'(y_c)}{\phi(z,c)^2}\right]dz,
\end{align*}
and by Remark \ref{Rmk:5.9},
\begin{align*}
p.v.\int_0^{\pi}\left[\left(\f{g(z)}{\phi(z,c)}\right)'
+\f{g(y_c)u'(y_c)}{\phi(z,c)^2}\right]dz
=\f{g(y_c)}{\rho(c)}\rmA(c)-\f{g(0)}{\phi(0,c)}+\f{g(\pi)}{\phi(\pi,c)}.
\end{align*}
On the other hand, we have
\begin{align*}
&\int_0^{y_c}f(y)\phi(y,c)dy=-\rmE_0(gu'')-g(y_c)u'(y_c)+g(0)\phi'(0,c),\\
&\int_{y_c}^{\pi}f(y)\phi(y,c)dy=\rmE_1(gu'')+g(y_c)u'(y_c)-g(\pi)\phi'(\pi,c).
\end{align*}
Then we conclude that
\begin{align*}
&\int_0^{\pi}\widehat{\psi}_e(t,\al,y)f(y)dy\\
&=-\f{1}{\pi}\int_{-1}^{1}
\Big[\mu_2\Big(\f{g(y_c)\rmA(c)}{\rho(c)}
+\mathrm{II}_1(gu'')\Big)
+\nu_0\big(\rmE_0(gu'')+g(y_c)u'(y_c)\big)\\
&\quad\quad\quad\quad\quad\quad
-\nu_1\big(\rmE_1(gu'')+g(y_c)u'(y_c)\big)\Big]e^{-i\al ct}dc,
\end{align*}
Here we used
\beno
\nu_{0}(c)=-\f{\mu_2(c)}
{\phi(0,c)\phi'(0,c)},\quad\nu_{1}(c)=-\f{\mu_2(c)}
{\phi(\pi,c)\phi'(\pi,c)}.
\eeno

Recall that
\begin{align*}
\mu_2(c)&=\f{I(c)^2(\rmA\rmC_{e}+\rmB\rmD_{e})+I(c)\phi(j\pi,c)\phi'(j\pi,c)(\rmC_{e}+E_{1-j}B)|_{j=0}^1}
{I(c)^2(\rmA_1(c)^2+\rmB_1(c)^2)/\rho(c)^2}\\
&=\f{\rho(c)^2(\rmA\rmC_{e}+\rmB\rmD_{e})+\rho(c)J_j(c)(\rmC_{e}+E_{1-j}B)|_{j=0}^1
}{\rmA_1(c)^2+\rmB_1(c)^2}\\
&=\pi\rho^2\f{\big(\rmA\rho(c)\frac{\widehat{\omega}_{e}(y_c)}{u'(y_c)}+\rho(c)\frac{u''(y_c)}{u'(y_c)^2}u'(y_c)
\rho(c)\textrm{II}_1(\widehat{\om}_{e})\big)+J_j\big(\frac{u''(y_c)}{u'(y_c)^2}\rmE_{1-j}
+\frac{\widehat{\omega}_{e}(y_c)}{u'(y_c)}\big)|_{j=0}^1}
{\rmA_1(c)^2+\rmB_1(c)^2}\\
&=\f{\pi\rho(c)^2}{u'(y_c)}\f{\rho(c)\Lambda_1(\widehat{\omega}_{e})(y_c)+\Lambda_{3,1}(\widehat{\omega}_{e})(y_c)}
{\rmA_1(c)^2+\rmB_1(c)^2}=\f{\pi\rho(c)^2}{u'(y_c)}\f{\Lambda_3(\widehat{\omega}_{e})(y_c)}{\rmA_1(c)^2+\rmB_1(c)^2}.
\end{align*}
Then we have
\begin{align*}
&\mu_2\Big(\f{g(y_c)\rmA(c)}{\rho(c)}
+\mathrm{II}_1(gu'')\Big)
+\nu_0\big(\rmE_0(gu'')+g(y_c)u'(y_c)\big)
-\nu_1\big(\rmE_1(gu'')+g(y_c)u'(y_c)\big)\\&=\mu_2(c)\f{\Lambda_2(g)(y_c)}{\rho(c)}-\f{\mu_2(c)\big(\rmE_0(gu'')+g(y_c)u'(y_c)\big)}
{\phi(0,c)\phi'(0,c)}+\f{\mu_2(c)\big(\rmE_1(gu'')+g(y_c)u'(y_c)\big)}
{\phi(\pi,c)\phi'(\pi,c)}\\
&=\mu_2\f{\Lambda_2(g)(y_c)}{\rho(c)}+\f{\mu_2J_1\big(\rmE_0(gu'')+g(y_c)u'(y_c)\big)}
{u'(y_c)\rho(c)^2}-\f{\mu_2J_0\big(\rmE_1(gu'')+g(y_c)u'(y_c)\big)}
{u'(y_c)\rho(c)^2}\\
&=\mu_2(c)\f{\Lambda_2(g)(y_c)}{\rho(c)}+\f{\mu_2(c)\Lambda_{4,1}(g)(y_c)}
{\rho(c)^2}=\f{\mu_2(c)\Lambda_{4}(g)(y_c)}
{\rho(c)^2}\\
&=\f{\pi\rho(c)^2}{u'(y_c)}\f{\Lambda_3(\widehat{\omega}_{e})(y_c)}{\rmA_1(c)^2+\rmB_1(c)^2}\f{\Lambda_{4}(g)(y_c)}
{\rho(c)^2}=\pi\f{\Lambda_3(\widehat{\omega}_{e})(y_c)\Lambda_{4}(g)(y_c)}{u'(y_c)(\rmA_1(c)^2+\rmB_1(c)^2)}=\pi K_e(c,\al),
\end{align*}
which gives the lemma.
\end{proof}

\section{Linear inviscid damping}

The decay estimates of the velocity rely on the following uniform $W^{1,2}$ estimates of the kernels $K_o$ and $K_e$.

\begin{proposition}\label{Prop: estimate K_o}
Let $K_o$ be defined by \eqref{eq:defK_o} with $g\in H^2(0,\pi)$, $g(0)=g(\pi)=0$ and $\widehat{\om}_o=\f{1}{2}(\widehat{\om}_0(\al,y)-\widehat{\om}_0(\al,-y))\in H^2$. Then $K_o(-1, \al)=K_o(1,\al)=0$, and there exists a constant $C$ independent of $\al$ such that,
\begin{align*}
&\|K_o(\cdot,\al)\|_{L_c^1(-1,1)}
\leq C\|\widehat{\om}_o\|_{L^2}\|g\|_{L^2},\\
&\|(\pa_cK_o)(\cdot,\al)\|_{L_c^1(-1,1)}
\leq C\|\widehat{\om}_o\|_{H^1}\|g\|_{H^1},\\
&\|(\pa_c^2K_o)(\cdot, \al)\|_{L_c^1(-1,1)}
\leq C\al^{\f12}\|\widehat{\om}_o\|_{H^2}\|f\|_{L^2}.
\end{align*}
\end{proposition}

\begin{proposition}\label{Prop: estimate K_e}
Let $K_e$ be defined by \eqref{eq:defK_e} with $g\in H^2(0,\pi)$, $g'(0)=g'(\pi)=0$ and
$\widehat{\om}_e=\f{1}{2}(\widehat{\om}_0(\al,y)+\widehat{\om}_0(\al,-y))\in H^2$. Then $K_e(-1,\al)=K_e(1,\al)=0$,
and there exists a constant $C$ independent of $\al$ such that,
\begin{align*}
&\|K_e(\cdot,\al)\|_{L_c^1(-1,1)}
\leq C\|\widehat{\om}_e\|_{L^2}\|g\|_{L^2},\\
&\|(\pa_cK_e)(\cdot, \al)\|_{L_c^1(-1,1)}
\leq C|\al|^{\f12}\|\widehat{\om}_e\|_{H^1}(\|g'\|_{L^2}+\al \|g\|_{L^2}),\\
&\|(\pa_c^2K_e)(\cdot,\al)\|_{L_c^1(-1,1)}
\leq C|\al|^{\f32}\|\widehat{\om}_e\|_{H^2}\|f\|_{L^2}.
\end{align*}
\end{proposition}

For the moment,  let us admit these two propositions and complete
the proof of Theorem \ref{Thm:LinInviDam}.

We have  $\widehat{\psi}(t,\al,y)=\widehat{\psi}_o(t,\al,y)+\widehat{\psi}_e(t,\al,y)$.  Using integration by parts,  Proposition \ref{Prop: estimate K_o} and Proposition \ref{Prop: estimate K_e}  ensure that
\begin{align*}
\|\widehat{\psi}_o(t,\al,\cdot)\|_{L^2}
&=2\sup_{\|f\|_{L^2}\leq 1}\left|\int_0^{\pi}\widehat{\psi}_o(t,\al,y)f(y)dy\right|\\
&\leq 2\sup_{\|f\|_{L^2}\leq1}\left|\int_{-1}^{1}K_o(c,\al)e^{-i\al ct}dc\right|
\leq C\f{1}{|\al|^{\f32}t^2}\|\widehat{\om}_o\|_{H^2},
\end{align*}
and
\begin{align*}
\|\widehat{\psi}_e(t,\al,\cdot)\|_{L^2}
&=2\sup_{\|f\|_{L^2}\leq1}\left|\int_0^{\pi}\widehat{\psi}_e(t,\al,y)f(y)dy\right|\\
&\leq 2\sup_{\|f\|_{L^2}\leq 1}\left|\int_{-1}^{1}K_e(c,\al)e^{-i\al ct}dc\right|
\leq C\f{1}{|\al|^{\f12}t^2}\|\widehat{\om}_e\|_{H^2}.
\end{align*}

Proposition \ref{Prop: estimate K_o} and Proposition \ref{Prop: estimate K_e} also imply that
\begin{align*}
&\al^2\|\widehat{\psi_o}(t,\al,\cdot)\|_{L^2}^2+\|\partial_y\widehat{\psi_o}(t,\al,\cdot)\|_{L^2}^2\\
&=-2\int_0^{\pi}\widehat{\psi_o}(t,\al,y)(\overline{\widehat{\psi_o}}''-\al^2\overline{\widehat{\psi_o}})(t,\al,y)dy
\leq \left\{
\begin{aligned}
&C\|\widehat{\om}_o\|_{L^2}\|\widehat{\psi_o}(t,\al,\cdot)\|_{L^2},\\
&C\f{1}{|\al t|}\|\widehat{\om}_o\|_{H^1}\|\widehat{\psi_o}(t,\al,\cdot)\|_{H^1},
\end{aligned}
\right.
\end{align*}
and
\begin{align*}
&\al^2\|\widehat{\psi_e}(t,\al,\cdot)\|_{L^2}^2+\|\partial_y\widehat{\psi_e}(t,\al,\cdot)\|_{L^2}^2\\
&=-2\int_0^{\pi}\widehat{\psi_e}(t,\al,y)(\overline{\widehat{\psi_e}}''-\al^2\overline{\widehat{\psi_e}})(t,\al,y)dy\\
&\leq\left\{
\begin{aligned}
&C\|\widehat{\om}_e(\al,\cdot)\|_{L^2_y}\|\widehat{\psi}_e(t,\al,\cdot)\|_{L^2_y},\\
&C\f{1}{|\al|^{\f12}|t|}\|\widehat{\om}_e(\al,\cdot)\|_{H^1_y}
(|\al|\|\widehat{\psi}_e(t,\al,\cdot)\|_{L^2_y}+\|\partial_y\widehat{\psi}_e(t,\al,\cdot)\|_{L^2_y}).
\end{aligned}
\right.
\end{align*}
Thus, we conclude that for $|t|\leq 1$,
\begin{align*}
&\|\widehat{V_o^1}\|_{L^2}+\|\widehat{V_o^2}\|_{L^2}
\leq |\al|\|\widehat{\psi_o}(t,\al,\cdot)\|_{L^2}+\|\partial_y\widehat{\psi_o}(t,\al,\cdot)\|_{L^2}
\leq C|\al|^{-1}\|\widehat{\om}_o\|_{L^2},\\
&\|\widehat{V_e^1}\|_{L^2}+\|\widehat{V_e^2}\|_{L^2}
\leq |\al|\|\widehat{\psi_e}(t,\al,\cdot)\|_{L^2}+\|\partial_y\widehat{\psi_e}(t,\al,\cdot)\|_{L^2}
\leq C|\al|^{-1}\|\widehat{\om}_e\|_{L^2},
\end{align*}
and for $|t|\geq 1$,
\begin{align*}
&\|\widehat{V_o^1}\|_{L^2}+\|\widehat{V_o^2}\|_{L^2}
\leq |\al|\|\widehat{\psi_o}(t,\al,\cdot)\|_{L^2}+\|\partial_y\widehat{\psi_o}(t,\al,\cdot)\|_{L^2}
\leq C\f{1}{|\al t|}\|\widehat{\om}_o\|_{H^1},\\
&\|\widehat{V_e^1}\|_{L^2}+\|\widehat{V_e^2}\|_{L^2}
\leq |\al|\|\widehat{\psi_e}(t,\al,\cdot)\|_{L^2}+\|\partial_y\widehat{\psi_e}(t,\al,\cdot)\|_{L^2}
\leq C\f{1}{|\al|^{\f12}|t|}\|\widehat{\om}_e\|_{H^1},
\end{align*}
Due to $V^1=V_o^1+V_e^1$ and $V^2=V_o^2+V_e^2$, we obtain
\beno
\|V^1\|_{L^2_{x,y}}+\|V^2\|_{L^2_{x,y}}\leq \f{C}{\langle t\rangle}\|\om_0\|_{H^{-\f12}_xH^1_y}.
\eeno
As $\|\widehat{V^2}(t,\al,\cdot)\|_{L^2}\leq C|\al|\|\widehat{\psi}(t,\al,\cdot)\|_{L^2}
\leq C\dfrac{|\al|^{\f12}}{t^2}\|\widehat{\om}_0(t,\al,\cdot)\|_{H^2}$, we infer that
\beno
\|V^2\|_{L^2_{x,y}}\leq \f{C}{\langle t^2\rangle}\|\om_0\|_{H^{\f12}_xH^2_y}.
\eeno
This completes the proof of Theorem \ref{Thm:LinInviDam}.\smallskip

Let us remark that if $\om_0(x,y)$ is odd in $y$, then we have
\begin{align*}
&\|V\|_{L^2_{x,y}}\leq \f{C}{\langle t\rangle}\|\om_0\|_{H^{-1}_xH^1_y},\quad \|V^2\|_{L^2_{x,y}}\leq \f{C}{\langle t^2\rangle}\|\om_0\|_{H^{-\f12}_xH^2_y}.
\end{align*}

\section{Estimates of some integral operators}\label{key term}

In this section, we present the estimates of some integral operators, which will be used to establish the $W^{2,1}$ estimates of $K_o$ and $K_e$.

\subsection{Basic properties of Hilbert transform}

Recall that the Hilbert transform $H$  on $\mathbb{T}$  is defined by
\beno
H(f)(x)=p.v.\int_{-\pi}^{\pi}\cot\f{x-y}{2}f(y)dy.
\eeno

\begin{lemma}\label{Lem: identity-Hilbert-periodic}
It holds that
\begin{align*}
&\f{d}{dx}\big(H(\varphi)(x)\big)=H(\varphi')(x),\\
&\f{d}{dx}\big(\sin x (H\varphi)'( x)-\cos x H\varphi(x)\big)
=\sin x (H\varphi)''(x)+\sin x H(\varphi)(x).
\end{align*}
If $\varphi$ is even, then
\begin{align*}
&\cot  x H(\varphi)(x)-H(\cot y\varphi)(x)
=-\int_{-\pi}^{\pi}\varphi(z)\,dz,\\
&\sin x(H\varphi)(x)-H(\sin y \varphi)(x)
=\int_{-\pi}^{\pi}(\cos x+\cos y) \varphi(y)\,dy,\\
&H(\cos y \varphi)-\cos x H(\varphi)
=\sin x\int_{-\pi}^{\pi} \varphi(y)dy,\\
&H(\varphi)(x)=-2p.v.\int_{0}^{\pi}\f{1}{\sin y}\ln\left|\frac{\sin\frac{x+y}{2}}{\sin\frac{x-y}{2}}\right|\sin y\varphi'(y)\,dy,
\end{align*}
and if $\varphi$ is odd, then
\begin{align*}
\sin x(H\varphi)(x)-H(\sin y\varphi)(x)=0.
\end{align*}
\end{lemma}

\begin{proof}
The first two equalities are obvious. Notice that
\begin{align*}
\sin x(H\varphi)(x)-H(\sin y\varphi)(x)
&=p.v.\int_{-\pi}^{\pi}(\sin x-\sin y) \cot\frac{x-y}{2}\varphi(y)dy\\
&=\int_{-\pi}^{\pi}(\cos x+\cos y) \varphi(y)dy,
\end{align*}
in particular, if $\varphi$ is odd, then
\begin{align*}
\sin x(H\varphi)(x)-H(\sin x \varphi)(x)=0.
\end{align*}

If  $\varphi$ is  even, then we have
\begin{align*}
H(\varphi)(x)&=p.v.\int_{-\pi}^{\pi}\cot\frac{x-y}{2}\varphi(y)dy\\
&=p.v.\int_{0}^{\pi}\left(\cot\frac{x-y}{2}+\cot\frac{x+y}{2}\right)\varphi(y)dy\\
&=-2p.v.\int_{0}^{\pi}\f{1}{\sin y}\ln\left|\frac{\sin\frac{x+y}{2}}{\sin\frac{x-y}{2}}\right|\sin y\varphi'(y)dy,
\end{align*}
and
\begin{align*}
&\cot x(H\varphi)(x)-H(\cot y\varphi)(x)\\
&=p.v.\int_{-\pi}^{\pi}(\cot x-\cot y) \cot\frac{x-y}{2}\varphi(y)dy\\
&=-p.v.\int_{-\pi}^{\pi}\f{1+\cos(x-y)}{\sin  x\sin y} \varphi(y)dy
=-\int_{-\pi}^{\pi} \varphi(y)dy,
\end{align*}
and
\begin{align*}
H(\cos y\varphi)-\cos x H(\varphi)
&=p.v.\int_{-\pi}^{\pi}(\cos y-\cos x) \cot\frac{x-y}{2}\varphi(y)dy\\
&=\int_{-\pi}^{\pi}(\sin x+\sin y) \varphi(y)dy
=\sin x\int_{-\pi}^{\pi} \varphi(y)dy.
\end{align*}
This proves the lemma.
\end{proof}

The following lemma can be easily proved by Hardy's inequality.

\begin{lemma}\label{Lem: Sobolev_Hardy}
If $\varphi\in H^2(\mathbb{T})$ is even, then $\varphi'(0)=\varphi'(\pi)=0$ and
\beno
\Big\|\f{\pa_{ y}\varphi}{\sin y}\Big\|_{L^2(0,\pi)}\leq C\|\varphi\|_{H^2(0,\pi)}.
\eeno
If $\varphi$ is even and  $\varphi(0)=\varphi(\pi)=0,$ then
\beno
&&\Big\|\f{\varphi}{\sin^2 y}\Big\|_{L^2(0,\pi)}
+\Big\|\f{\varphi}{\sin y}\Big\|_{L^{\infty}(0,\pi)}
\leq C\|\varphi\|_{H^2(0,\pi)},\\
&&\Big\|\f{\varphi}{\sin y}\Big\|_{L^2(0,\pi)}\leq C\|\varphi\|_{H^1(0,\pi)}.
\eeno
\end{lemma}

\begin{lemma}\label{lem:Z}
Let $Z(f)(y)=f'(y)-\cot y f(y)$. If $f\in H^{m}(0,\pi)$ is an even function for $m=1,2,3$, then $Z(f)\in H^{m-1}(0,\pi)$ can be extended to a periodic odd function. Moreover, we have
\beno
\|Z(f)\|_{H^{m-1}(\mathbb{T})}\leq C\|f\|_{H^{m}(0,\pi)}.
\eeno
If we $f\in H^2(0,\pi)\cap H_0^1(0,\pi)$ is an even function, then
\beno
\|{Z(f)}/{\sin y}\|_{L^2(\mathbb{T})}\leq \|f\|_{H^2(0,\pi)}.
\eeno
\end{lemma}

\begin{proof}
For $0<y\leq \f{\pi}{4}$,  we may write
\begin{align*}
Z(f)(y)=m(y)y\big(\f{f}{ym(y)}\big)'=f'-\frac {\cos y f} {ym(y)},
\end{align*}
where $m(y)=\f{\sin y}{y}\in C^{\infty}(\R)$ and $\f{1}{m}\in C^{\infty}([0,\f{3\pi}{4}])$.  On the other hand, for $\pi-y\leq \f{3\pi}{4}$, we write
\begin{align*}
Z(f)(y)=m(\pi-y)(\pi-y)\big(\f{f}{(\pi-y)m(\pi-y)}\big)'=f'-\f {\cos y f} {(\pi-y)m(\pi-y)}.
\end{align*}
Thus, we can conclude that
\beno
\|Z(f)\|_{H^{m-1}(0,\pi)}\leq C\|f\|_{H^m(0,\pi)}.
\eeno

To prove that $Z(f)$ can be extended to an even $\pi$-periodic function in $H^{m-1}(\mathbb{T})$, we only need to show that $Z(f)(0)=Z(f)(\pi)=0$ and $\pa_yZ(f)$ is an even function. This is because $\lim\limits_{y\to0+}Z(f)(y)=\lim\limits_{y\to0+}f'(y)-\lim\limits_{y\to0+}\cot y f(y)=f'(0+)-f'(0+)=0,\ \lim\limits_{y\to\pi-}Z(f)(y)=f'(\pi-)-f'(\pi-)=0$, and $Z(f)$ is an odd function.
So,  $\|Z(f)\|_{H^{m-1}(\mathbb{T})}\leq C\|f\|_{H^m(0,\pi)}$.
\end{proof}

\subsection{Estimates of $\mathrm{II}_{1,1}$} Recall that $\mathrm{II}_{1,1}$ is defined by
\begin{align*}
\textrm{II}_{1,1}(\varphi)(y_c)
=p.v.\int_0^{\pi}\f{\Int(\varphi)(y)-\Int(\varphi)(y_c)}{(u(y)-u(y_c))^2}\,dy,
\end{align*}
where $\Int(\varphi)$ is defined by
\begin{align*}
\Int(\varphi)(y)=\int_{0}^y\varphi(y')\,dy' \quad\textrm{for}\, y\in [0,\pi],\quad \Int(\varphi)(y)=\Int(\varphi)(-y)\quad\textrm{for}\, y\in [-\pi,0].
\end{align*}
In what follows, we denote
\beno
\|f\|_{H^k_{y_c}}=\displaystyle\sum_{n=0}^k\|\pa_{y_c}^nf\|_{L^2_{y_c}(\mathbb{T})},\quad \|\cdot\|_{H^k}=\|\cdot\|_{H^k(0,\pi)}.
\eeno

\begin{lemma}\label{Lem: II_1,1L^2}
It holds that for $k=0,1,2$,
\beno
&&\|\rho\mathrm{II}_{1,1}(\varphi)\|_{H_{ y_c}^k}
\leq C\|\varphi\|_{H^k},\\
&&\left\|{\pa_{y_c}(\rho\mathrm{II}_{1,1}(\varphi))}/{\sin y_c}\right\|_{L^2_{y_c}}
\leq C\|\varphi\|_{H^2}.
\eeno
\end{lemma}
\begin{proof}

By Lemma \ref{Lem: identity-Hilbert-periodic}, we have for $\varphi$ even,
\begin{align*}
\f{\sin y_c}{2} \pa_{ y_c}\Big(\frac{H(\varphi)(y_c)}{\sin y_c}\Big)
&=\f {1}{2}( (H\varphi)'( y_c)-\cot y_c H\varphi( y_c))\nonumber\\
&=\f {1}{2}H(\varphi'-\cot y_c \varphi)( y_c)+\f {1}{2}\int_{-\pi}^{\pi} \varphi(y)dy.
\end{align*}
Let $\In(\varphi)(y)=\Int(\varphi)(y)-\f{\Int(\varphi)(\pi)}{2}(1-\cos y)=\int_0^y\varphi(y')\,dy'-\sin^2\f{y}{2}\int_0^\pi\varphi(y')\,dy'$ for $y\in [0,\pi]$.  We extend $\In(\varphi)$ to be an even $\pi$-periodic function.  Then we get
by Remark \ref{rem:Pi-11} that
\begin{align}
\mathrm{II}_{1,1}(\varphi)(y_c)&=-\f{1}{2\sin y_c}\pa_{y_c}\Big(\frac{H(\In(\varphi))(y_c)}{\sin y_c}\Big)\nonumber\\
&=-\f{1}{2\sin^2 y_c}\mathcal{H}(\varphi)(y_c)-\f{1}{2\sin^2 y_c}\int_{-\pi}^{\pi}\In(\varphi)(y')dy'.\label{eq: OpH}
\end{align}
Here $\mathcal{H}=H\circ Z\circ \In$. Then by Lemma \ref{lem:Z}, we get
\begin{align*}
\|\rho\mathrm{II}_{1,1}(\varphi)\|_{H^k_{y_c}}
&\leq C\big
(\|\mathcal{H}(\varphi)\|_{H^k(\mathbb{T})}+\|\varphi\|_{L^2}\big)\\
&\leq C\big(\|Z\circ \In(\varphi)\|_{H^k(\mathbb{T})}
+\|\varphi\|_{L^2}\big)\\
&\leq C\big(\|\In(\varphi)\|_{H^{k+1}(0,\pi)}+\|\varphi\|_{L^2}\big)
\leq C\|\varphi\|_{H^k}.
\end{align*}

As $\pa_{y_c}(\rho\mathrm{II}_{1,1}(\varphi))=H\Big(\pa_{y}(Z\circ\In(\varphi))\Big)$, $H\Big(\pa_{y}(Z\circ\In(\varphi))\Big)\in H^1(\mathbb{T})$ is odd, and by Lemma \ref{lem:Z}, we have
\beno
\left\|\pa_{y_c}(\rho\mathrm{II}_{1,1})/\sin y_c\right\|_{L^2_{y_c}}
\leq C\big\|H\big({\pa_{y}(Z\circ\In(\varphi))}\big)\big\|_{H^1}
\leq C\|\In(\varphi)\|_{H^3}\leq C\|\varphi\|_{H^2}.
\eeno
This proves the lemma.
\end{proof}

The following lemma shows that for $\varphi$ vanishing at $0$ and $\pi$, we have better estimates, which will be used in the odd case.

\begin{lemma}\label{Lem: II_1,1odd}
If $\varphi\in H_0^1(0,\pi)$, then
\beno
\|\sin y_c\mathrm{II}_{1,1}(\varphi)\|_{L_{ y_c}^2}\leq C\|\varphi\|_{H^1}.
\eeno
If $\varphi\in H^2(0,\pi)\cap H_0^1(0,\pi)$, then
\beno
\|\sin y_c\mathrm{II}_{1,1}(\varphi)\|_{L^{\infty}}
+\|\mathrm{II}_{1,1}(\varphi)\|_{L_{ y_c}^2}
\leq C\|\varphi\|_{H^2}.
\eeno
If $\varphi\in H^2(0,\pi)\cap H_0^1(0,\pi)$ and $\varphi''/\sin y_c\in L^2(0,\pi)$, then
\begin{align*}
\|\pa_{y_c}\mathrm{II}_{1,1}(\varphi)\|_{L^2_{y_c}}
&+\|\mathrm{II}_{1,1}(\varphi)\|_{L^{\infty}}+\|\pa_{y_c}^2(\sin y_c\mathrm{II}_{1,1}(\varphi))\|_{L^2_{y_c}}\\
&+\|\sin y_c\pa_{y_c}\mathrm{II}_{1,1}(\varphi)\|_{L^{\infty}}
\leq C\big(\|\varphi\|_{H^2}+\|\varphi''/\sin y_c\|_{L^2_{y_c}}\big).
\end{align*}
\end{lemma}

\begin{proof}
For $\varphi\in H^1_0(0,\pi)$, we get by Hardy's inequality that
\beno
\Big\|\f{\In(\varphi)}{\sin^2 y}\Big\|_{L^2}
+\Big\|\f{\pa_y\In(\varphi)}{\sin y}\Big\|_{L^2}
+\Big\|\pa_y\big(\frac{\In(\varphi)}{\sin y}\big)\Big\|_{L^2}
\leq C\|\varphi\|_{H^1}.
\eeno
Recall that
\begin{align}
\mathrm{II}_{1,1}(\varphi)(y_c)&=-\f{1}{2\sin y_c}\pa_{y_c}\Big(\frac{H(\In(\varphi))}{\sin y_c}\Big),
\end{align}
and $\f{\In(\varphi)}{\sin y}$ is odd, by Lemma \ref{Lem: identity-Hilbert-periodic}, we have
\begin{align*}
\sin y_c\mathrm{II}_{1,1}(\varphi)(y_c)&=-\f{1}{2}\pa_{y_c}H\Big(\frac{\In(\varphi)}{\sin y}\Big)=-\f12H\Big(\pa_y\big(\frac{\In(\varphi)}{\sin y}\big)\Big),
\end{align*}
which implies that for $\varphi\in H^1_0(0,\pi)$,
\beno
\|\sin y_c\mathrm{II}_{1,1}(\varphi)\|_{L_{ y_c}^2}\leq C\|\varphi\|_{H^1}.
\eeno

For $\varphi\in H^2(0,\pi)\cap H_0^1(0,\pi)$, $\pa_y^2\big(\frac{\In(\varphi)}{\sin y}\big)\in L^2(\mathbb{T})$ and
\beno
\pa_{y_c}\big(\sin y_c\mathrm{II}_{1,1}(\varphi)(y_c)\big)=
-\f12H\Big(\pa_y^2\big(\frac{\In(\varphi)}{\sin y}\big)\Big),
\eeno
which shows that
\beno
\big\|\pa_{y_c}\big(\sin y_c\mathrm{II}_{1,1}(\varphi)(y_c)\big)\big\|_{L^2_{y_c}}\leq C\Big\|\pa_y^2\big(\frac{\In(\varphi)}{\sin y}\big)\Big\|_{L^2}\leq C\|\varphi\|_{H^2},
\eeno
which implies  that
\beno
\|\sin y_c\mathrm{II}_{1,1}(\varphi)\|_{L^{\infty}}+\|\mathrm{II}_{1,1}(\varphi)\|_{L^2_{y_c}}\leq C\|\varphi\|_{H^2}.
\eeno

For $f(0)=f'(0)=f''(0)=0$ and $f'''/y\in L^2$, we have
\beno
\Big\|\f{f}{y^4}\Big\|_{L^2}+\Big\|\f{f'}{y^3}\Big\|_{L^2}+\Big\|\f{f''}{y^3}\Big\|_{L^2}\leq C\|f'''/y\|_{L^2}.
\eeno
For $\varphi\in H^2(0,\pi)\cap H_0^1(0,\pi)$ and $\varphi''/u'\in L^2(0,\pi)$, let
\beno
\In_2(\varphi)=\In(\varphi)-(\sin y)^2\big(\varphi'(0)(1+\cos y)+\varphi'(\pi)(1-\cos y)+\Int(\varphi)(\pi)\cos y\big)/4,
\eeno
then $\In_2(\varphi)=\In_2(\varphi)'=\In_2(\varphi)''=0 $ for $y=0,\pi.$ Thus,
\beno
\Big\|\f{\In_2(\varphi)}{\sin^4 y}\Big\|_{L^2}
+\Big\|\f{\In_2(\varphi)'}{\sin^3 y}\Big\|_{L^2}
+\Big\|\f{\In_2(\varphi)''}{\sin^2 y}\Big\|_{L^2}
\leq C\big(\|\In_2(\varphi)'''/\sin y\|_{L^2}+\|\In_2(\varphi)\|_{H^2}\big),
\eeno
which implies that
\beno
\left\|\pa_y^3\big(\frac{\In(\varphi)}{\sin y}\big)\right\|_{L^2}\leq\left\|\pa_y^3\big(\frac{\In_2(\varphi)}{\sin y}\big)\right\|_{L^2}+C\|\varphi\|_{H^2}\leq
C\big(\|\varphi''/\sin y\|_{L^2}+\|\varphi\|_{H^2}\big).
\eeno
Therefore, by Lemma \ref{lem:Z},
\begin{align*}
\|\pa_{y_c}\mathrm{II}_{1,1}(\varphi)\|_{L^2_{y_c}}
&\leq C\Big\|\f{Z\circ H\big(\pa_y(\In(\varphi)/\sin y)\big)}{\sin y}\Big\|_{L^2}\\
&\leq C\|H\big(\pa_y(\In(\varphi)/\sin y)\big)\|_{H^2}\\
&\leq C\|\In(\varphi)/\sin y\|_{H^3}
\leq C\big(\|\varphi''/\sin y\|_{L^2}+\|\varphi\|_{H^2}\big),
\end{align*}
as well as
\begin{align*}
\|\pa_{y_c}^2(\sin y_c\mathrm{II}_{1,1}(\varphi))\|_{L^2_{y_c}}
\leq \|H(\pa_y^3(\In(\varphi/\sin y)))\|_{L^2}
\leq C(\|\varphi''/\sin y\|_{L^2}+\|\varphi\|_{H^2}),
\end{align*}
which implies that
$\|\sin y_c\pa_{y_c}\mathrm{II}_{1,1}(\varphi)\|_{L^{\infty}}\leq C(\|\varphi\|_{H^2}+\|\varphi''/u'\|_{L^2}).$
\end{proof}

\subsection{Estimates of $\mathcal{L}_k$}

In this subsection, we denote by $\phi(y,c)$ a solution of \eqref{eq:Rayleigh-H} obtained by Proposition \ref{prop:Rayleigh-Hom} and let $\phi_1(y,c)=\f{\phi(y,c)}{u(y)-c}$ for $c\in D_0$. Here $u(y)=-\cos y$. We denote $\|\cdot\|_{L^p}=\|\cdot\|_{L^p(0,\pi)}$. Recall that
\beno
\mathcal{L}_k(\varphi)(y_c)=\int_0^{\pi}\int_{y_c}^{z}{\varphi}(y)\left(\frac{\partial_z+\partial_y}{u'(y_c)}+\partial_c\right)^k
\left(\f{1}{(u(z)-c)^2}\left(\f{\phi_1(y,c)\,}{\phi_1(z,c)^2}-1\right)\right)\,dydz.
\eeno

\begin{lemma}\label{Lem:L^2L_k}
For any $p\in [1,\infty)$, there exists a constant $C$ independent of $\al$ such that for $k=0,1,2,$
\begin{align*}
&\|u'(y_c)^{2k+2}\mathcal{L}_k(\varphi)\|_{L^{\infty}}\leq C|\al|^{\f{1}{p}}\left\|\varphi\right\|_{L^{p}},\\
&\|u'(y_c)^{2k+1}\mathcal{L}_k(\varphi)\|_{L^{\infty}}\leq C|\al|^{1+\f{1}{p}}\left\|\varphi\right\|_{L^{p}},\\
&\big\|u'(y_c)^{2k+1}\mathcal{L}_k(\varphi)\big\|_{L^\infty}\leq C|\al|\|\varphi\|_{L^{\infty}}.
\end{align*}
\end{lemma}

\begin{proof}
By Proposition \ref{prop:phi1}, we know that for $k=0,1,2,$ and $0\leq z\leq y\leq y_c\leq \pi$ or $0\leq y_c\leq y\leq z\leq \pi$,
\begin{align}\label{eq:directive_k}
\left|\left(\frac{\partial_z+\partial_y}{u'(y_c)}+\partial_c\right)^k
\left(\f{1}{(u(z)-c)^2}\left(\f{\phi_1(y,c)\,}{\phi_1(z,c)^2}-1\right)\right)\right|\leq C\f{\min\{\al^2|z-y_c|^2,1\}}{(u(z)-c)^2u'(y_c)^{2k}}.
\end{align}
Recall that $u_1(y,c)=\f{\cos y_c-\cos y}{y-y_c}$ and
\beno
C^{-1}\sin \f{y+y_c}{2}\leq u_1(y,c)\leq C\sin \f{y+y_c}{2}.
\eeno
Thus,
\begin{align*}
&\int_0^{\pi}\f{|z-y_c|^{1-\f{1}{p}}\min\{\al^2|z-y_c|^2,1\}}{|u(z)-c|^2}\,dz\\
&\leq C\int_0^{\pi}\f{\min\{\al^2|z-y_c|^2,1\}}{\sin^2 \f{z+y_c}{2}|z-y_c|^{1+\f1p}}\,dz\\
&\leq C\int_{|z-y_c|\leq \f{1}{\al},0\leq z\leq \pi}\f{\al^2|z-y_c|^{1-\f1p}}{\sin^2 \f{z+y_c}{2}}\,dz
+C\int_{|z-y_c|> \f{1}{\al},0\leq z\leq \pi}\f{|z-y_c|^{-1-\f1p}}{\sin^2 \f{z+y_c}{2}}\,dz.
\end{align*}
Using the fact that
\begin{align*}
\int_{|z-y_c|\leq \f{1}{\al},0\leq z\leq \pi}\f{\al^2|z-y_c|^{1-\f1p}}{\sin^2 \f{z+y_c}{2}}\,dz\leq C\min\Big\{\f{\al^{1+\f1p}}{\sin y_c},\f{\al^{\f1p}}{\sin^2 y_c}\Big\},\\
\int_{|z-y_c|> \f{1}{\al},0\leq z\leq \pi}\f{|z-y_c|^{-1-\f1p}}{\sin^2 \f{z+y_c}{2}}\,dz\leq C\min\Big\{\f{\al^{1+\f1p}}{\sin y_c},\f{\al^{\f1p}}{\sin^2 y_c}\Big\},
\end{align*}
and for $p=\infty$,
\begin{align*}
&\int_0^{\pi}\f{|z-y_c|\min\{\al^2|z-y_c|^2,1\}}{|u(z)-c|^2}\,dz\\
&\leq C\int_{|z-y_c|\leq \f{1}{\al},0\leq z\leq \pi}\f{\al^2|z-y_c|}{\sin^2 \f{z+y_c}{2}}\,dz
+C\int_{|z-y_c|> \f{1}{\al},0\leq z\leq \pi}\f{|z-y_c|^{-1}}{\sin^2 \f{z+y_c}{2}}\,dz\\
&\leq C\f{|\al|}{\sin y_c},
\end{align*}
we deduce that for $k=0,1,2,$ and $p\in [1,\infty)$,
\begin{align*}
\big|u'(y_c)^{2k}\mathcal{L}_k(\varphi)(c)\big|
&\leq C\int_0^{\pi}\f{|z-y_c|^{1-\f{1}{p}}\min\{\al^2|z-y_c|^2,1\}}{|u(z)-c|^2}dz\left\|\varphi\right\|_{L^{p}}\\
&\leq C\min\Big\{\frac{\al^{1+\f{1}{p}}}{u'(y_c)},\frac{\al^{\f{1}{p}}}{u'(y_c)^2}\Big\}\left\|\varphi\right\|_{L^{p}},
\end{align*}
and
\beq\label{eq:estL_kL^infty}
|u'(y_c)^{2k}\mathcal{L}_k(\varphi)|\leq C\f{|\al|}{\sin y_c}\|\varphi\|_{L^{\infty}}.
\eeq
This proves the lemma.
\end{proof}

To estimate the derivatives of $\mathcal{L}_0$, we introduce for $k,j=0,1$,
\begin{align}\label{eq:defI_k,j}
I_{k,j}(\varphi)(c)=\int_{y_c}^{z}{\varphi}(y)\left(\frac{\partial_z+\partial_y}{u'(y_c)}+\partial_c\right)^k
\left(\f{1}{(u(z)-c)^2}\left(\f{\phi_1(y,c)\,}{\phi_1(z,c)^2}-1\right)\right)\,dy\bigg|_{z=j\pi}.
\end{align}
Thus, for $k=0,1,$
\ben\label{eq:pa_cL_0}
\partial_c\mathcal{L}_k(\varphi)=\frac{1}{u'(y_c)}\left(\mathcal{L}_k(\varphi')-I_{k,1}(\varphi)+I_{k,0}(\varphi)\right)+\mathcal{L}_{k+1}(\varphi).
\een
\begin{lemma}\label{Lem:H^1II_1,2_periodic}
It holds that
\beno
&&\big\|u'(y_c)^3\pa_c\mathcal{L}_{0}(\varphi)\big\|_{L^{\infty}}\leq C\al\|\varphi\|_{L^{\infty}}+C\al^{\f12}\|\varphi'\|_{L^2},\\
&&\|u'(y_c)^3\pa_c\mathcal{L}_{0}(\varphi)\|_{L^{\infty}}\leq C\al\|\varphi\|_{L^{\infty}}+C\al u'(y_c)\|\varphi'\|_{L^{\infty}}.
\eeno
If ${\varphi}/{u'}\in L^p(0,\pi)$ for $p\in (1,\infty)$, then
\begin{align*}
&\|\sin y_c\mathcal{L}_{0}(\varphi)\|_{L^{\infty}}
\leq C\al^{\f1p}\|\varphi\|_{W^{1,p}},\\
&\|\mathcal{L}_{0}(\varphi)\|_{L^{\infty}}
\leq C\al\|\varphi\|_{W^{1,\infty}},\\
&\|\sin^3 y_c\pa_c\mathcal{L}_{0}(\varphi)\|_{L^{\infty}}\leq C\al^{\f12}\|\varphi\|_{H^1}.
\end{align*}
\end{lemma}

\begin{proof}
By \eqref{eq:directive_k}, we have for $k=0,1$,
\begin{align}
\nonumber|I_{k,j}(\varphi)(c)|&\leq C\left|\int_{y_c}^{j\pi}\varphi(y)dy\right|\f{\min\{\al^2|j\pi-y_c|^2,1\}}{(\cos j\pi-\cos y_c)^2u'(y_c)^{2k}}\\
\nonumber&\leq C\f{|j\pi-y_c|\min\{\al^2|j\pi-y_c|^2,1\}}{(\cos j\pi-\cos y_c)^2u'(y_c)^{2k}}\|\varphi\|_{L^{\infty}}\\
\label{eq: I_k,jinfty}&\leq C\f{\min\{\al^2|j\pi-y_c|^2,1\}}{|j\pi-y_c|^3u'(y_c)^{2k}}\|\varphi\|_{L^{\infty}},
\end{align}
thus by the fact that $\sin y_c\leq C|j\pi-y_c|$, we deduce that for $j=0,1,$
\begin{align}
\label{6.10}|u(y_c)^{2k+3}I_{k,j}(\varphi)(c)|&\leq C\|\varphi\|_{L^{\infty}},\\
\label{6.11}|u(y_c)^{2k+2}I_{k,j}(\varphi)(c)|&\leq C\al\|\varphi\|_{L^{\infty}}.
\end{align}
Notice that
\begin{align}
\pa_c\mathcal{L}_{0}(\varphi)
=\frac{1}{u'(y_c)}\left(\mathcal{L}_0(\varphi')-I_{0,1}(\varphi)+I_{0,0}(\varphi)\right)+\mathcal{L}_{1}(\varphi).\label{eq:L0-pa}
\end{align}
We get by Lemma \ref{Lem:L^2L_k} that
\beno
&&\|u'(y_c)^2\mathcal{L}_0(\varphi')\|_{L^{\infty}}
\leq C\al^{\f12}\|\varphi'\|_{L^2},\\
&&\|u'(y_c)^3\mathcal{L}_1(\varphi)\|_{L^{\infty}}
\leq C\al\|\varphi\|_{L^{\infty}},
\eeno
which along with the estimates of $I_{0,0}$ and $I_{0,1}$ imply that
\beno
\|u'(y_c)^3\pa_c\mathcal{L}_{0}(\varphi)\|_{L^{\infty}}\leq C\al\|\varphi\|_{L^{\infty}}+C\al^{\f12}\|\varphi'\|_{L^2}.\eeno
By \eqref{eq:estL_kL^infty}, we get
\beno
\|u'(y_c)\mathcal{L}_0(\varphi')\|_{L^{\infty}}
\leq C|\al|\|\varphi'\|_{L^{\infty}}.
\eeno
This along with the estimates of $I_{0,0}$ and $I_{0,1}$ gives
\beno
\|u'(y_c)^3\pa_c\mathcal{L}_{0}(\varphi)\|_{L^{\infty}}\leq C\al\|\varphi\|_{L^{\infty}}+C\al u'(y_c)\|\varphi'\|_{L^{\infty}}.
\eeno

If  $\varphi/u'\in L^p$, then $\f{\varphi(y)}{u'(y)}=\f{1}{u'(y)}\int_{j\pi}^y\varphi'(z)dz$, thus, $\big\|\f{\varphi(y)}{u'(y)}\big\|_{L^p(0,\pi)}\leq C\|\varphi'\|_{L^p(0,\pi)}$ for any $p\in (1,\infty]$.
Due to $u(y)-c=u_1(y,c)(y-y_c)$, we have
\begin{align*}
\left|\int_{y_c}^zu'(y)^{p'}dy\right|^{\f{1}{p'}}
\leq C\max_{|2y-y_c-z|\leq |z-y_c|}u'(y)|z-y_c|^{\f{1}{p'}}\leq Cu_1(z,c)|z-y_c|^{\f{1}{p'}}.
\end{align*}
Then we have
\begin{align*}
|\mathcal{L}_k(\varphi)(c)|
&\leq C\int_0^{\pi}\f{\min\{\al^2|z-y_c|^2,1\}{\left|\int_{y_c}^zu'(y)^{p'}dy\right|^{\f{1}{p'}}}}{{u_1(z,c)^2}(z-y_c)^2u'(y_c)^{2k}}dz\left\|\f{\varphi(y)}{u'(y)}\right\|_{L^{p}}\\
&\leq C\int_0^{\pi}\f{\min\{\al^2|z-y_c|^2,1\}}{u'(y_c)^{2k+1}|z-y_c|^{1+\f1p}}\|\varphi'\|_{L^p}
\leq C\f{\al^{\f1p}}{u'(y_c)^{2k+1}}\|\varphi'\|_{L^p},
\end{align*}
and similarly,
\ben\label{eq:Lk-5}
&&\|u'(y_c)^{2k+1}\mathcal{L}_k(\varphi)\|_{L^{\infty}}\leq C\al\|{\varphi}\|_{L^{\infty}}.
\een
In particular, we get
\beno
|\sin y_c\mathcal{L}_{0}(\varphi)\|_{L^{\infty}}
\leq C\al^{\f1p}\|\varphi\|_{W^{1,p}}.
\eeno

As $|u(z)-c|=\Big|\sin \f{y_c+z}{2}\sin \f{y_c-z}{2}\Big|\geq \Big|\sin \f{y_c-z}{2}\Big|^2\geq C^{-1}|z-y_c|^2$, we have
\begin{align}\nonumber
|u'(y_c)^{2k}\mathcal{L}_k(\varphi)(c)|
&\leq C\int_0^{\pi}\f{\min\{\al^2|z-y_c|^2,1\}{\int_{y_c}^zu'(y)dy}}{(u(z)-c)^2}dz\left\|\f{\varphi(y)}{u'(y)}\right\|_{L^{\infty}}\\ \label{6.12}
&\leq C\int_0^{\pi}\f{\min\{\al^2|z-y_c|^2,1\}}{|u(z)-c|}dz\|\varphi'\|_{L^{\infty}}\\ \nonumber
&\leq C\int_0^{\pi}\f{\min\{\al^2|z-y_c|^2,1\}}{|z-y_c|^2}dz\|\varphi'\|_{L^{\infty}}
\leq C|\al|\|\varphi'\|_{L^{\infty}},
\end{align}
In particular, we get
\beno
\|\mathcal{L}_{0}(\varphi)\|_{L^{\infty}}
\leq C\al\|\varphi\|_{W^{1,\infty}}.
\eeno
On the other hand, we have
\begin{align*}
|u'(y_c)^{2k+2}I_{k,0}(\varphi)(c)|
\leq C\f{u'(y_c)^{2k+2}y_c^{2-\f1p}\min\{\al^2y_c^2,1\}}{u'(y_c)^{2k+2}y_c^2}\left\|\f{\varphi}{u'}\right\|_{L^{p}}\leq C\al^{\f1p}\|\varphi\|_{W^{1,p}(0,\pi)}.
\end{align*}
The bound of $I_{k,1}$ is similar. Then we can conclude the last inequality of the lemma by using \eqref{eq:L0-pa}.
\end{proof}

\begin{lemma}\label{Lem:H^2II_1,2_periodic}
It holds that
\begin{align*}
\big|u'(y_c)^5\pa_c^2\mathcal{L}_{0}(\varphi)(c)\big|
\leq C\big(\al^{\f12}u'(y_c)\|\varphi''\|_{L^2}+\al\|\varphi\|_{L^{\infty}}+\al u'(y_c)\|\varphi'\|_{L^{\infty}}+\al^{\f12}\|\varphi'\|_{L^2}\big).
\end{align*}
If $\varphi/u'\in L^{\infty}$, then for $k=0,1,2$,
\beno
\|u'(y_c)^{2k}\partial_c^k\mathcal{L}_0(\varphi)\|_{L^{\infty}}\leq C\al\left\|{\varphi}'\right\|_{L^{\infty}}+C\al^{\f{1}{2}}\left\|{\varphi}''\right\|_{L^{2}}.
\eeno
\end{lemma}
\begin{proof}
Since we have
\begin{align*}
\partial_cI_{0,j}(\varphi)(c)
&=-\f{\varphi(y_c)}{u'(y_c)}\f{1}{(u(z)-c)^2}\left(\f{1\,}{\phi_1(z,c)^2}-1\right)\bigg|_{z=j\pi}\\
&\quad+\int_{y_c}^{j\pi}\partial_c
\left(\f{{\varphi}(y)}{(u(z)-c)^2}\left(\f{\phi_1(y,c)\,}{\phi_1(z,c)^2}-1\right)\right)\,dy\bigg|_{z=j\pi},
\end{align*}
and by Proposition \ref{prop:phi1}, we get for $|2y-y_c-z|\leq |z-y_c|$,
\begin{align*}
&\left|\partial_c\left(\f{1}{(u(z)-c)^2}\left(\f{\phi_1(y,c)\,}{\phi_1(z,c)^2}-1\right)\right)\right|\\
&\leq C\left|\left(\f{1}{|u(z)-c|^3}\left(\f{\phi_1(y,c)\,}{\phi_1(z,c)^2}-1\right)\right)\right|
+C\left|\f{1}{|u(z)-c|^2}\pa_c\left(\f{\phi_1(y,c)\,}{\phi_1(z,c)^2}\right)\right|\\
&\leq C\f{\min\{\al^2|z-y_c|^2,1\}}{|u(z)-c|^3}
+C\f{\min\{\al^2|z-y_c|,\al\}}{u'(y_c)|u(z)-c|^2}.
\end{align*}
Thus, we obtain
\begin{align*}
|\pa_cI_{0,1}(\varphi)(c)|
&\leq C\|\varphi\|_{L^{\infty}}\f{\min\{\al^2(\pi-y_c)^2,1\}}{(\pi-y_c)^2u'(y_c)^3}\\
&\quad+C\left|\partial_c\left(\f{1}{(u(\pi)-c)^2}\left(\f{\phi_1(y,c)\,}{\phi_1(\pi,c)^2}-1\right)\right)\right||\pi-y_c|\|\varphi\|_{L^{\infty}}\\
&\leq C\|\varphi\|_{L^{\infty}}\f{\min\{\al^2(\pi-y_c)^2,1\}}{(\pi-y_c)^2u'(y_c)^3}
+C\|\varphi\|_{L^{\infty}}\f{\min\{\al^2(\pi-y_c),\al\}}{|\pi-y_c|u'(y_c)^3},
\end{align*}
and
\begin{align*}
|\pa_cI_{0,0}(\varphi)(c)|\leq C\|\varphi\|_{L^{\infty}}\Big(\f{\min\{\al^2 y_c^2,1\}}{y_c^2u'(y_c)^3}+\f{\min\{\al^2 y_c,\al\}}{y_cu'(y_c)^3}\Big).
\end{align*}
This gives
\begin{align*}
|\pa_cI_{0,1}(\varphi)(c)|+|\pa_cI_{0,0}(\varphi)(c)|
\leq C\|\varphi\|_{L^{\infty}}\Big(\f{\min\{\al^2 \sin^2y_c,1\}}{\sin^5 y_c}+\f{\min\{\al^2 \sin^2y_c,\al\sin y_c\}}{\sin^5 y_c}\Big).
\end{align*}
Therefore, for any $\g\in [0,1]$,
\begin{align}\label{eq:pa_cI_0,j}
\big|u'(y_c)^{3+\g}\pa_cI_{0,0}(\varphi)(c)\big|
+\big|u'(y_c)^{3+\g}\pa_cI_{0,1}(\varphi)(c)\big|
\leq C\al^{2-\g}\|\varphi\|_{L^{\infty}}.
\end{align}
Notice that
\begin{align*}
\pa_c^2\mathcal{L}_{0}(\varphi)
&=\pa_c\left(\frac{1}{u'(y_c)}\left(\mathcal{L}_0(\varphi')-I_{0,1}(\varphi)+I_{0,0}(\varphi)\right)+\mathcal{L}_{1}(\varphi)\right)\\
&=\frac{1}{u'(y_c)^2}\left(\mathcal{L}_0(\varphi'')-I_{0,1}(\varphi')+I_{0,0}(\varphi')\right)+\f{1}{u'(y_c)}\mathcal{L}_{1}(\varphi')\\
&\quad-\f{1}{u'(y_c)}\pa_c(I_{0,1}(\varphi)(c)-I_{0,0}(\varphi)(c))\\
&\quad-\f{u''(y_c)}{u'(y_c)^3}\left(\mathcal{L}_0(\varphi')-I_{0,1}(\varphi)+I_{0,0}(\varphi)\right)\\
&\quad+\frac{1}{u'(y_c)}\left(\mathcal{L}_1(\varphi')-I_{1,1}(\varphi)+I_{1,0}(\varphi)\right)+\mathcal{L}_{2}(\varphi).
\end{align*}
Thus, by Lemma \ref{Lem:L^2L_k}, \eqref{6.10}, \eqref{6.11},  \eqref{eq:Lk-5} and \eqref{eq:pa_cI_0,j}, we obtain
\begin{align*}
|\sin^5y_c\pa_c^2\mathcal{L}_{0}(\varphi)(c)|
\leq C\al^{\f12}u'(y_c)\|\varphi''\|_{L^2}+C\al\|\varphi\|_{L^{\infty}}
+C\al u'(y_c)\|\varphi'\|_{L^{\infty}}+C\al^{\f12}\|\varphi'\|_{L^2}.
\end{align*}

If  $\varphi(0)=\varphi(\pi)=0$, then
\beno
|\pa_cI_{0,0}(\varphi)(c)|\leq C\Big\|\f{\varphi}{u'}\Big\|_{L^{\infty}}\Big(\f{\min\{\al^2 y_c^2,1\}}{y_c^2u'(y_c)^2}+\f{\min\{\al^2 y_c,\al\}}{y_cu'(y_c)^2}\Big).
\eeno
Therefore,
\begin{align*}
|u'(y_c)^{3}\pa_cI_{0,0}(\varphi)(c)|\leq C\al\Big\|\f{\varphi}{u'}\Big\|_{L^{\infty}}.
\end{align*}
Similarly, we have
\begin{align*}
|u'(y_c)^{3}\pa_cI_{0,1}(\varphi)(c)|\leq C\al\Big\|\f{\varphi}{u'}\Big\|_{L^{\infty}}.
\end{align*}
We also have
\beno
&&\|u'(y_c)^{2k+1}I_{k,j}(\varphi)\|_{L^{\infty}}\leq C\al\left\|\f{\varphi}{u'}\right\|_{L^{\infty}}.
\eeno
With the estimates above, we deduce the second inequality of the lemma by using \eqref{6.12}, \eqref{6.11} and Lemma \ref{Lem:L^2L_k}.
\end{proof}

\section{$W^{2,1}$ estimates of $K_o$ and $K_e$}
\label{K_oK_e}

This section is devoted to the proof of Proposition \ref{Prop: estimate K_o} and Proposition \ref{Prop: estimate K_e}.
For the sake of simplicity, we introduce some notations:
\begin{itemize}
\item We use $\cL^p$ to denote a function $f$ which satisfies $\|f\|_{L_{y_c}^p}\leq C$.

\item We use $\cL^{p}\cap \cL^q$  to denote a function $f$ which satisfies $\|f\|_{L_{y_c}^p}+\|f\|_{L_{y_c}^q}\leq C$.

\item We use $\rho \cL^{p}$ denote a function $f$ which satisfies $\left\|\f{f}{\rho}\right\|_{L^p_{y_c}}\leq C$.

\item We use $\cL^{p}+\cL^q$  to denote a function $f$ which can be divided into two parts $f=f_1+f_2$ with $f_1,f_2$ satisfing $\|f_1\|_{L_{y_c}^p}\leq C,\|f_2\|_{L_{y_c}^q}\leq C$.
\end{itemize}

\subsection{$W^{2,1}$ estimate of $K_o$}

In this subsection, we prove Proposition \ref{Prop: estimate K_o}.Recall that
\begin{align*}
K_o(c,\al)=\f{\Lambda_1(\widehat{\om}_o)(y_c)\Lambda_2(g)(y_c)}{(\mathrm{A}(c)^2+\mathrm{B}(c)^2){u'(y_c)}},
\end{align*}
where $\rmA, \rmB$ are defined by \eqref{eq:defAB} and $\Lambda_1, \Lambda_2$ are defined by \eqref{def:Lam1} and \eqref{def:Lam2}.
\smallskip

{\bf Step 1. $L^1$ estimate.} \smallskip

We normalize $\widehat{\om}_0, g$ so that $\|\widehat{\om}_{o}\|_{L^2}\leq 1,\ \|g\|_{L^2}\leq 1$.
By \eqref{eq:defLam_1,1}, \eqref{eq:defLam_2,1} and Lemma \ref{Lem: II_1,1L^2}, we get $\Lambda_{1,1}(\widehat{\omega}_{o})(y_c)=\cL^2$ and $\Lambda_{2,1}(g)(y_c)=\cL^2$.
By \eqref{eq:defLam_1,2}, \eqref{eq:defLam_2,2} Lemma \ref{Lem: II} and Lemma \ref{Lem:L^2L_k}, we get $\Lambda_{1,2}(\widehat{\omega}_{o})(y_c)=|\al|^{\f12}\cL^{\infty}+|\al| \sin y_c\cL^2$ and $\Lambda_{2,2}(g)(y_c)=|\al|^{\f12}\cL^{\infty}+|\al| \sin y_c\cL^2$. Thus, we have
\beno
\Lambda_{1}(\widehat{\omega}_{o})(y_c)=(1+|\al|\sin y_c)\cL^2+|\al|^{\f{1}{2}}\cL^{\infty},\
\Lambda_2(g)(y_c)=(1+|\al|\sin y_c)\cL^2+|\al|^{\f{1}{2}}\cL^{\infty}.
\eeno
Then by Proposition \ref{Prop: A^2+B^2}, we infer that
\begin{align*}
u'(y_c)K_o(c,\al)=\cL^1+\f{|\al|^{\f{1}{2}}\cL^2}{1+|\al|\sin y_c}+\f{\al \cL^{\infty}}{(1+|\al|\sin y_c)^2},
\end{align*}
which implies that
\begin{align*}
&\|K_o(\cdot,\al)\|_{L_c^1(-1,1)}\leq C\|\widehat{\om}_o\|_{L^2}\|g\|_{L^2}.
\end{align*}

{\bf Step 2. $W^{1,1}$ estimate.} \smallskip

We normalize $\widehat{\om}_0, g$ so that  $\|\widehat{\om}_{o}\|_{H^1}\leq 1,\  \|g\|_{H^1}\leq 1$. By the definition of $\widehat{\om}_{o}$, we know that $\widehat{\om}_{o}(\al,0)=\widehat{\om}_{o}(\al,\pi)=g(0)=g(\pi)=0$, then $\big\|\f{\widehat{\om}_{o}}{u'}\big\|_{L^2}+\big\|\f{g}{u'}\big\|_{L^2}\leq C$.
By \eqref{eq:defLam_1,1}, \eqref{eq:defLam_2,1} and Lemma \ref{Lem: II_1,1odd}, we get
\beno
\Lambda_{1,1}(\widehat{\omega}_{o})(y_c)=\sin y_c\cL^2,\quad \Lambda_{2,1}(g)(y_c)=\sin y_c\cL^2,
\eeno
 and by Lemma \ref{Lem: II_1,1L^2}, we get
 \beno
 \pa_{y_c}\Lambda_{1,1}(\widehat{\omega}_{o})(y_c)=\cL^2,\quad \pa_{y_c}\Lambda_{2,1}(g)(y_c)=\cL^2.
 \eeno
 By \eqref{eq:defLam_1,2}, \eqref{eq:defLam_2,2}, Lemma \ref{Lem: II} and Lemma \ref{Lem:H^1II_1,2_periodic}, we get
 \beno
 &&\Lambda_{1,2}(\widehat{\omega}_{o})(y_c)=|\al|^{\f12}\sin y_c\cL^{\infty}+|\al| \sin^2 y_c\cL^2,\\
 &&\Lambda_{2,2}(g)(y_c)=|\al|^{\f12}\sin y_c\cL^{\infty}+|\al| \sin^2 y_c\cL^2,\\
&&\pa_c\Lambda_{1,2}(\widehat{\omega}_{o})(y_c)=\f{\al^{\f12}\cL^{\infty}}{\sin y_c}+|\al|\cL^2,\\
&&\pa_c\Lambda_{2,2}(g)(y_c)=\f{\al^{\f12}\cL^{\infty}}{\sin y_c}+|\al|\cL^2.
\eeno
Thus, we obtain
\beno
&&\Lambda_{1}(\widehat{\omega}_{o})(y_c)=\sin y_c((1+|\al|\sin y_c) \cL^2+|\al|^{\f{1}{2}}\cL^{\infty}),\\
&&\Lambda_2(g)(y_c)=\sin y_c((1+|\al|\sin y_c)\cL^2+|\al|^{\f{1}{2}}\cL^{\infty}),\\
&&\pa_c\Lambda_{1}(\widehat{\omega}_{o})(y_c)=\f{1}{\sin y_c}\big((1+|\al|\sin y_c)\cL^2+\al^{\f12}\cL^{\infty}\big),\\
&&\pa_c\Lambda_2(g)(y_c)=\f{1}{\sin y_c}\big((1+|\al|\sin y_c)\cL^2+\al^{\f12}\cL^{\infty}\big).
\eeno
Therefore,
\beno
&&\Lambda_{1}(\widehat{\omega}_{o})\Lambda_2(g)=\sin^2 y_c\big((1+|\al|\sin y_c)^2\cL^1+\al^{\f12}(1+|\al|\sin y_c)\cL^{2}+\al \cL^{\infty}\big),\\
&&\pa_c(\Lambda_{1}(\widehat{\omega}_{o})\Lambda_2(g))
=(1+|\al|\sin y_c)^2\cL^1+\al^{\f12}(1+|\al|\sin y_c)\cL^{2}+\al \cL^{\infty},
\eeno
which along with Proposition \ref{Prop: A^2+B^2} give
\begin{align*}
&\|\pa_cK_o(\cdot,\al)\|_{L_c^1(-1,1)}
\leq C\|\widehat{\om}_o\|_{H^1}\|g\|_{H^1}.
\end{align*}

{\bf Step 3. $W^{2,1}$ estimate.}\smallskip

 We normalize $\widehat{\om}_0, g$ so that $\|\widehat{\om}_{o}\|_{H^2}\leq 1,\ \|f\|_{L^2}\leq 1$.
 Then we have
 \beno
 &&\|g''\|_{L^2}^2+2\al^2\|g'\|_{L^2}^2+\al^4\|g\|_{L^2}^2\leq 1,\\
 &&\big\|\f{u''g}{u'}\big\|_{L^\infty}\leq C\|(u''g)'\|_{L^\infty}\leq C/\al^{\f{1}{2}},\quad \|(u''g)''\|_{L^2}\leq C.
\eeno
By \eqref{eq:defLam_1,1}, \eqref{eq:defLam_2,1} and Lemma \ref{Lem: II_1,1odd}, we get
\beno
\Lambda_{1,1}(\widehat{\omega}_{o})(y_c)=\sin^2 y_c\cL^2\cap \sin y_c\cL^{\infty},\quad \Lambda_{2,1}(g)(y_c)=\sin^2 y_c\cL^2\cap \sin y_c\cL^{\infty}.
\eeno
By Lemma \ref{Lem: II_1,1L^2}, we get
\beno
&&\pa_c\Lambda_{1,1}(\widehat{\omega}_{o})(y_c)=\cL^2,\quad \pa_c\Lambda_{2,1}(g)(y_c)=\cL^2,\\
&&\pa_c^2\Lambda_{1,1}(\widehat{\omega}_{o})(y_c)=\f{\cL^2}{\sin^2y_c},\quad \pa_c^2\Lambda_{2,1}(g)(y_c)=\f{\cL^2}{\sin^2y_c}.
\eeno
By \eqref{eq:defLam_1,2}, \eqref{eq:defLam_2,2}, Lemma \ref{Lem: II} and Lemma \ref{Lem:H^2II_1,2_periodic}, we get
\beno
&&\Lambda_{1,2}(\widehat{\omega}_{o})(y_c)=|\al|\sin^2 y_c\cL^{\infty},\quad \Lambda_{1,2}(g)(y_c)=|\al|^{\f12}\sin^2 y_c\cL^{\infty},\\
&&\pa_c\Lambda_{1,2}(\widehat{\omega}_{o})(y_c)=|\al| \cL^{\infty},\quad \pa_c\Lambda_{1,2}(g)(y_c)=|\al|^{\f12}{\cL^\infty},\\
&&\pa_c^2\Lambda_{1,2}(\widehat{\omega}_{o})(y_c)=\f{|\al|\cL^{\infty}}{\sin^2y_c}+\f{|\al|\cL^2}{\sin y_c},\quad
\pa_c^2\Lambda_{1,2}(g)(y_c)=\f{|\al|^{\f12}\cL^\infty} {\sin^2y_c}+\f{|\al|\cL^2}{\sin y_c}.
\eeno
Therefore, we obtain
\begin{align*}
&\Lambda_1(\widehat{\om}_{o})=\sin^2y_c \cL^2\cap \sin y_c\cL^{\infty}+|\al|\sin^2y_c\cL^{\infty},\\
&\partial_c(\Lambda_1(\widehat{\om}_{o}))=\cL^2+\al \cL^{\infty},\\
&\partial_c^2(\Lambda_1(\widehat{\om}_{o}))=\f{(1+|\al|\sin y_c)\cL^2+|\al| \cL^{\infty}}{\sin^2y_c},\\
&\Lambda_2(g)=\sin^2y_c(\cL^2+|\al|^{\f{1}{2}}\cL^{\infty}),\\
&\partial_c(\Lambda_2(g))= \cL^2+|\al|^{\f{1}{2}}\cL^{\infty},\\
&\partial_c^2(\Lambda_2(g))=\f{(1+|\al|\sin y_c)\cL^2+|\al|^{\f{1}{2}} \cL^{\infty}}{\sin^2y_c}.
\end{align*}
Thus,
\begin{align*}
&\Lambda_1(\widehat{\om}_{o})\Lambda_2(g)=\sin^4y_c\big(\cL^1+\al \cL^2+\al^{\f32}\cL^{\infty}\big),\\
&\pa_c\big(\Lambda_1(\widehat{\om}_{o})\Lambda_2(g)\big)
=\sin^2y_c\big(\cL^1+\al \cL^2+\al^{\f32}\cL^{\infty}\big),\\
&\pa_c^2\big(\Lambda_1(\widehat{\om}_{o})\Lambda_2(g)\big)
=(1+\al\sin y_c)\big(\cL^1+\al\cL^2+\al^2\cL^{\infty}\big),
\end{align*}
from which and  Proposition \ref{Prop: A^2+B^2}, we infer that
\begin{align*}
&\|\pa_c^2K_o(\cdot, \al)\|_{L_c^1(-1,1)}
\leq C\al^{\f12}\|\widehat{\om}_o\|_{H^2}\|f\|_{L^2},
\end{align*}
and
\beno
K_o(c, \al)=C(\al)\sin^2 y_c\cL^2,
\eeno
which implies that $K_o(\pm 1,\al)=0$.

\subsection{$W^{2,1}$ estimate of $K_e$}

In this subsection, we prove Proposition \ref{Prop: estimate K_e}. Recall that
\begin{align*}
K_e(c,\al)=\f{\Lambda_3(\widehat{\omega}_{e})(y_c)\Lambda_4(g)(y_c)}{u'(y_c)(\rmA_1(c)^2+\rmB_1(c)^2)},
\end{align*}
where $\rmA_1, \rmB_2$ are defined by \eqref{eq:defA_1B_1},
and $\Lambda_3, \Lambda_4$ are defined by \eqref{eq:def Lam_3} and \eqref{eq:def Lam_4}.\smallskip

We need the following lemma.

\begin{lemma}\label{Lem:Lambda_3,1 4,1}
Let  $\phi(y,c)$ be as in Proposition \ref{prop:Rayleigh-Hom} with $(y,c)\in [0,\pi]\times [-1,1]$, and $\Lambda_{3,1}(\widehat{\om}_e)$ and $\Lambda_{4,1}(g)$ be defined by \eqref{eq:defLambda_3,1} and \eqref{eq:defLambda_4,1} with $\rmE_j$ defined by \eqref{eq:defE_j} for $j=0,1$.  It holds that

(1) if $\|\widehat{\om}_e\|_{L^2}\leq 1$ and $\|g\|_{L^2}\leq 1$, then\begin{align*}
\Lambda_{3,1}(\widehat{\om}_e)=\f{\cL^2}{\al^2},\quad
\Lambda_{4,1}(g)=\f{\cL^2}{\al^2};
\end{align*}

(2) if $\|\widehat{\om}_e\|_{H^1}\leq 1$ and $|\al|\|g\|_{L^2}+\|g'\|_{L^2}\leq 1$, then
\begin{align*}
&\Lambda_{3,1}(\widehat{\om}_e)=\frac{\sin y_c(\cL^2+\al \cL^{\infty})}{\al^2},\quad
\Lambda_{4,1}(g)=\frac{\sin y_c(\cL^2+\al^{\frac{1}{2}}\cL^{\infty})}{\al^2},\\
&\pa_c\big(\Lambda_{3,1}(\widehat{\om}_e)\big)
=\frac{\cL^2+\al \cL^{\infty} }{\al^2\sin y_c},\quad
\pa_c\big(\Lambda_{4,1}(g)\big)
=\frac{\cL^2+\al^{\frac{1}{2}}\cL^{\infty}}{\al^2\sin y_c};
\end{align*}

(3) if  $\|\widehat{\om}_e\|_{H^2}\leq 1$ and $\|g''-\al^2g\|_{L^2}\leq 1$ and     $\widehat{\om}_e'(0)=\widehat{\om}_e'(\pi)=g'(0)=g'(\pi)=0$, then \begin{align*}
&\Lambda_{3,1}(\widehat{\om}_e)
=\al^{-2}(\rho \cL^2\cap u' \cL^{\infty})+\rho \cL^{\infty},\quad
\Lambda_{4,1}(g)
=\al^{-2}(\rho \cL^2\cap u' \cL^{\infty})+ \al^{-\f32}\rho \cL^{\infty},\\
&\pa_c\big(\Lambda_{3,1}(\widehat{\om}_e)\big)
=\al^{-2}\cL^2+\cL^{\infty},\quad
\pa_c\big(\Lambda_{4,1}(g)\big)
=\al^{-2}(\cL^2+\al^{\frac{1}{2}}\cL^{\infty}),\\
&\pa_c^2\big(\Lambda_{3,1}(\widehat{\om}_e)\big)
=\f{\al^{-2}\cL^2+\cL^{\infty}}{\sin^2y_c},\quad
\pa_c^2\big(\Lambda_{4,1}(g)\big)
=\f{\cL^2+\al^{\frac{1}{2}}\cL^{\infty}}{\al^2\sin^2 y_c};
\end{align*}

(4) if  $\|\widehat{\om}_e\|_{H^2}\leq 1$ and $\|\pa_y\widehat{\om}_e/u'\|_{H^1}\leq 1$, then
\begin{align*}
\pa_c\big(\Lambda_{3,1}(\widehat{\om}_e)\big)=\cL^{\infty},\quad
\pa_c^2\big(\Lambda_{3,1}(\widehat{\om}_e)\big)=\f{\al^2\cL^{\infty}}{(1+\al \sin y_c)^2}+\f{\cL^2}{\al\sin y_c}.
\end{align*}
\end{lemma}

\begin{remark}\label{Rmk:difference1/2}
Following the proof of Case 2, we can show that if  $|\al|\|\om\|_{L^2}+\|\pa_y\om\|_{L^2}\leq 1$, then
\begin{align*}
&\Lambda_{3,1}(\om)=\frac{\sin y_c(\cL^2+\al^{\f12} \cL^{\infty})}{\al^2},\quad
\pa_c\big(\Lambda_{3,1}(\om)\big)
=\frac{\cL^2+\al^{\f12}\cL^{\infty}}{\al^2\sin y_c}.
\end{align*}
This will be used in the proof of Proposition \ref{Prop: commutator}.
\end{remark}

\begin{proof}
{\bf Case 1.} $\|\widehat{\om}_e\|_{L^2}\leq 1$ and $\|g\|_{L^2}\leq 1$.\smallskip

Using the fact that $\phi_1(y,c)\leq \phi_1(0,c)$  for $0<y<y_c$ and
$\phi_1(y,c)\leq \phi_1(\pi,c)$ for $y_c<y<\pi$, we get
\begin{align}\label{7.1}
\Big\|\f{\rmE_j(\varphi)}{\phi_1(j\pi,c)(j\pi-y_c)}\Big\|_{L^p}\leq C\|\varphi\|_{L^p},\ \ \ 1<p\leq+\infty,
\end{align}
which implies that
\begin{align*}
&{u''(y_c)}\rmE_0(\widehat{\om}_e)+u'(y_c){\widehat{\om}_e(y_c)}
=\phi_1(0,c)y_c\cL^{2},\\
&{u''(y_c)}\rmE_1(\widehat{\om}_e)+u'(y_c){\widehat{\om}_e(y_c)}
=\phi_1(\pi,c)(\pi-y_c)\cL^{2},
\end{align*}
and
\begin{align*}
&{\rmE_0(gu'')}+{u'(y_c)}g(y_c)=\phi_1(0,c)y_c\cL^{2},\\
&{\rmE_1(gu'')}+{u'(y_c)}g(y_c)=\phi_1(\pi,c)(\pi-y_c)\cL^{2},
\end{align*}
from which and Lemma \ref{Lem: J_j high}, we deduce (1). \smallskip

{\bf Case 2.} $\|\widehat{\om}_e\|_{H^1}\leq 1$ and $|\al|\|g\|_{L^2}+\|g'\|_{L^2}\leq 1$(thus, $\|g\|_{L^{\infty}}\leq C|\al|^{-\f 12}.$)
\smallskip

We write
\begin{align}
\nonumber&u''(y_c)\rmE_{j}(\widehat{\om}_e)+u'(y_c)\widehat{\om}_e(y_c)\\
\nonumber&=\int_{y_c}^{j\pi}\big(u''(y_c)\phi_1(y,c)\widehat{\om}_e(y)-u''(y)\widehat{\om}_e(y_c)\big)dy\\
\nonumber&=\int_{y_c}^{j\pi}\big(u''(y_c)-u''(y)\big)\phi_1(y,c)\widehat{\om}_e(y)dy\\
\nonumber&\quad+\int_{y_c}^{j\pi}\big(\phi_1(y,c)-1\big)u''(y)\widehat{\om}_e(y)dy+\int_{y_c}^{j\pi}u''(y)\big(\widehat{\om}_e(y)-\widehat{\om}_e(y_c)\big)dy\\
&=I_1+I_2+I_3.\label{eq:I_1+I_2+I_3}
\end{align}
Using the fact that $\phi_1(y,c)-1\leq C\min\{\al^2|y-y_c|^2,1\}\phi_1(y,c)\leq C\min\{\al^2|y-y_c|^2,|\al||y-y_c|\}\phi_1(y,c)$ and $|u''(y)-u''(y_c)|\leq \sin \f{y+y_c}{2}|y-y_c|$, we obtain
\begin{align}
&|I_1|\leq C|j\pi-y_c|^3\phi_1(j\pi,c)\|\widehat{\om}_e\|_{L^{\infty}},\label{eq:estI_1}\\
&|I_2|\leq C|\al| |j\pi-y_c|^2\min\big\{1,|\al||j\pi-y_c|\big\}\phi_1(j\pi,c)\|\widehat{\om}_e\|_{L^{\infty}},\label{eq:estI_2}\\
&\Big\|\f{I_3}{|j\pi-y_c|^2}\Big\|_{L^2}\leq C\|\widehat{\om}_e'\|_{L^2}.\label{eq:estI_3}
\end{align}
Thus, we infer that
\begin{align*}
&{u''(y_c)}\rmE_0(\widehat{\om}_e)+u'(y_c){\widehat{\om}_e(y_c)}
=\phi_1(0,c)y_c^2\al \cL^{\infty}+y_c^2\cL^2,\\
&{u''(y_c)}\rmE_1(\widehat{\om}_e)+u'(y_c){\widehat{\om}_e(y_c)}
=\phi_1(\pi,c)(\pi-y_c)^2\al\cL^{\infty}+(\pi-y_c)^2\cL^2,
\end{align*}
which along with Lemma \ref{Lem: J_j high} show that
\beno
\Lambda_{3,1}(\widehat{\om}_e)=\frac{\sin y_c(\cL^2+\al \cL^{\infty})}{\al^2}.
\eeno
We write
\begin{align}
\nonumber&\rmE_{j}(u''g)+u'(y_c)g(y_c)\\
\nonumber&=\int_{y_c}^{j\pi}\big(u''(y)\phi_1(y,c)g(y)-u''(y)g(y_c)\big)dy\\
\nonumber&=\int_{y_c}^{j\pi}\big(\phi_1(y,c)-1\big)u''(y)g(y)dy+\int_{y_c}^{j\pi}u''(y)\big(g(y)-g(y_c)\big)dy\\
&=I_2'+I_3'.\label{eq:I_2'+I_3'}
\end{align}
We have
\begin{align}
\label{eq:estI_2'}
&|I_2'|\leq C|\al| |j\pi-y_c|^2\min\big\{1,|\al| |j\pi-y_c|\big\}\phi_1(j\pi,c)\|g\|_{L^{\infty}},\\
&\Big\|\f{I_3'}{|j\pi-y_c|^2}\Big\|_{L^2}\leq C\|g'\|_{L^2}.\nonumber
\end{align}
Thus, we obtain
\begin{align*}
&{\rmE_0(gu'')}+{u'(y_c)}g(y_c)=\phi_1(0,c)y_c^2\al^{\frac{1}{2}}\cL^{\infty}+y_c^2\cL^2,\\
&{\rmE_1(gu'')}+{u'(y_c)}g(y_c)=\phi_1(\pi,c)(\pi-y_c)^2\al^{\frac{1}{2}} \cL^{\infty}+(\pi-y_c)^2\cL^2,
\end{align*}
which along with Lemma \ref{Lem: J_j high} show that
\beno
\Lambda_{4,1}(g)=\frac{\sin y_c(\cL^2+\al^{\frac{1}{2}}\cL^{\infty})}{\al^2}.
\eeno

A direct calculation gives
\begin{align}\label{7.7}
&\partial_c\left({u''(y_c)}\rmE_j(\widehat{\om}_e)+u'(y_c){\widehat{\om}_e(y_c)}\right)\\ \nonumber
&=\f{u'''(y_c)}{u'(y_c)}\rmE_j(\widehat{\om}_e)
-\f{u''(y_c)\widehat{\om}_e(y_c)}{u'(y_c)}
+{u''(y_c)}\int_{y_c}^{j\pi}\widehat{\om}_e\partial_c\phi_1(y,c)dy
+\f{({u'\widehat{\om}_e})'(y_c)}{u'(y_c)}\\ \nonumber
&=\f{u'''(y_c)}{u'(y_c)}\rmE_j(\widehat{\om}_e)
+\widehat{\om}_e'(y_c)+{u''(y_c)}\int_{y_c}^{j\pi}\widehat{\om}_e\partial_c\phi_1(y,c)dy,
\end{align}
and
\begin{align*}
\partial_c\left({\rmE_j(gu'')}+{u'(y_c)}g(y_c)\right)
&=-\f{gu''(y_c)}{u'(y_c)}+\int_{y_c}^{j\pi}gu''\partial_c\phi_1(y,c)dy
+\f{(u'g)'(y_c)}{u'(y_c)}\\
&=g'(y_c)+\int_{y_c}^{j\pi}gu''\partial_c\phi_1(y,c)dy.
\end{align*}
By Proposition \ref{prop:phi1}, we get for $0<y<y_c,\, j=0$ or $y_c<y<\pi,\, j=1$,
\beno
\left|\f{\partial_c\phi_1(y,c)}{\phi_1(j\pi,c)}\right|
\leq \f{C|\al|}{u'(y_c)}\min\{|\al||y-y_c|,1\}
\leq \f{C|\al|}{u'(y_c)}\min\{|\al||j\pi-y_c|,1\}.
\eeno
Then we have for $j=0$
\begin{align}\label{7.9}
\left\|\f{(\pi-y_c)}{\min\{\al y_c,1\}}\int_{y_c}^0\varphi(y)\f{\partial_c\phi_1(y,c)}{\phi_1(0,c)}dy\right\|_{L^p}\leq C\al \|\varphi\|_{L^{p}},\ \ 1<p\leq \infty.
\end{align}
Similarly, we get for $j=1$
\begin{align}\label{7.10}
\left\|\f{y_c}{\min\{\al (\pi-y_c),1\}}\int_{y_c}^{\pi}\varphi(y)\f{\partial_c\phi_1(y,c)}{\phi_1(\pi,c)}dy\right\|_{L^p}\leq C\al \|\varphi\|_{L^{p}},\ \ 1<p\leq \infty.
\end{align}
Recall that $\|\widehat{\om}_e\|_{H^1}\leq 1$ and $\al\|g\|_{L^2}+\|g'\|_{L^2}\leq 1$, then $\|g\|_{L^{\infty}}\leq C\al^{-\frac{1}{2}}$.
Thus, we obtain
\begin{align*}
&\partial_c\left({u''(y_c)}\rmE_0(\widehat{\om}_e)+u'(y_c){\widehat{\om}_e(y_c)}\right)
=\phi_1(0,c)\f{\al \cL^{\infty}}{\pi-y_c}+\cL^2,\\
&\partial_c\left({u''(y_c)}\rmE_1(\widehat{\om}_e)+u'(y_c){\widehat{\om}_e(y_c)}\right)
=\phi_1(\pi,c)\f{\al \cL^{\infty}}{y_c}+\cL^2,\\
&\partial_c\left({\rmE_0(gu'')}+{u'(y_c)}g(y_c)\right)=\f{\phi_1(0,c)\al^{\frac{1}{2}} \cL^{\infty}}{\pi-y_c}+\cL^2,\\
&\partial_c\left({\rmE_1(gu'')}+{u'(y_c)}g(y_c)\right)=\f{\phi_1(\pi,c)\al^{\frac{1}{2}}\cL^{\infty}}{y_c}+\cL^2,
\end{align*}
from which and Lemma \ref{Lem: J_j high}, we infer that
\beno
\pa_c\big(\Lambda_{3,1}(\widehat{\om}_e)\big)
=\frac{\cL^2+\al \cL^{\infty} }{\al^2\sin y_c},\quad
\pa_c\big(\Lambda_{4,1}(g)\big)
=\frac{\cL^2+\al^{\frac{1}{2}}\cL^{\infty}}{\al^2\sin y_c}.
\eeno

{\bf Case 3.} $\|\widehat{\om}_e\|_{H^2}\leq 1$ and $\|g''-\al^2g\|_{L^{2}}\leq 1$ and $\widehat{\om}_e'(0)=\widehat{\om}_e'(\pi)=g'(0)=g'(\pi)=0$. Thus, we have
\beno
&&\|g''\|_{L^2}^2+2\al^2\|g'\|_{L^2}^2+\al^4\|g\|_{L^2}^2\leq 1,\quad \big\|\f{\widehat{\om}_e'}{u'}\big\|_{L^{2}}\leq C\|\widehat{\om}_e''\|_{L^{2}},\\
&&\|g\|_{L^{\infty}}\leq C\al^{-\frac{3}{2}},\quad
\big\|\f{g'}{u'}\big\|_{L^{2}}\leq C\|g''\|_{L^{2}}\leq C,\quad
\|g'\|_{L^{\infty}}\leq C\al^{-\frac{1}{2}}.
\eeno

First of all, we have for $I_3$ in \eqref{eq:I_1+I_2+I_3} and $I_3'$ in \eqref{eq:I_2'+I_3'},
\begin{align*}
|I_3|\leq C\Big|\int_{y_c}^{j\pi}\int_{y}^{y_c}\widehat{\om}_e'(y')dy'dy\Big|\leq C\Big|\int_{y_c}^{j\pi}|j\pi-y'|\widehat{\om}_e'(y')dy'\Big|\leq C|j\pi-y_c|^2\Big|\int_{y_c}^{j\pi}\f{\widehat{\om}_e'(y')}{j\pi-y'}dy'\Big|,
\end{align*}
which implies that
\beno
\Big\|\f{I_3}{|j\pi-y_c|^3}\Big\|_{L^2}\leq C\Big\|\f{\widehat{\om}_e'}{u'}\Big\|_{L^2}\leq C\|\widehat{\om}_e''\|_{L^2},\quad \Big\|\f{I_3}{|j\pi-y_c|^2}\Big\|_{L^{\infty}}\leq C\|\widehat{\om}_e'\|_{L^{\infty}}.
\eeno
Similarly, for $I'_3$, we have
\beno
\Big\|\f{I_3'}{|j\pi-y_c|^3}\Big\|_{L^2}\leq C\Big\|\f{g'}{u'}\Big\|_{L^2}\leq C\|g''\|_{L^2},\quad \Big\|\f{I_3'}{|j\pi-y_c|^2}\Big\|_{L^{\infty}}\leq C\|g'\|_{L^{\infty}}.
\eeno
Thus, by \eqref{eq:estI_1}, \eqref{eq:estI_2} and \eqref{eq:estI_2'}, we obtain
\begin{align*}
&{u''(y_c)}\rmE_0(\widehat{\om}_e)+u'(y_c){\widehat{\om}_e(y_c)}
=\phi_1(0,c)y_c^3\al^2 \cL^{\infty}+(y_c^3\cL^2\cap y_c^2\cL^{\infty}),\\
&{u''(y_c)}\rmE_1(\widehat{\om}_e)+u'(y_c){\widehat{\om}_e(y_c)}
=\phi_1(\pi,c)(\pi-y_c)^3\al^2 \cL^{\infty}+(\pi-y_c)^2((\pi-y_c)\cL^2\cap \cL^{\infty}),\\
&\rmE_0(u''g)+u'(y_c){g(y_c)}
=\phi_1(0,c)y_c^3\al^{\frac{1}{2}}\cL^{\infty}+(y_c^3\cL^2\cap y_c^2\cL^{\infty}),\\
&\rmE_1(u''g)+u'(y_c){g(y_c)}
=\phi_1(\pi,c)(\pi-y_c)^3\al^{\frac{1}{2}}\cL^{\infty}+(\pi-y_c)^2((\pi-y_c)\cL^2\cap \cL^{\infty}),
\end{align*}
which along with Lemma \ref{Lem: J_j high} show that
\beno
\Lambda_{3,1}(\widehat{\om}_e)
=\al^{-2}(\rho\cL^2\cap u' \cL^{\infty})+ \rho \cL^{\infty},\quad
\Lambda_{4,1}(g)
=\al^{-2}(\rho \cL^2\cap u' \cL^{\infty})+ \al^{-\f32}\rho \cL^{\infty}.
\eeno

We have
\begin{align*}
&\pa_c\left(u''(y_c)E_{j}(\widehat{\om}_e)+u'(y_c)\widehat{\om}_e(y_c)\right)\\
&=\int_{y_c}^{j\pi}u''(y_c)\pa_c\phi_1(y,c)\widehat{\om}_e(y)dy+\widehat{\om}_e'(y_c)-\int_{y_c}^{j\pi}\phi_1(y,c)\widehat{\om}_e(y)dy\\
&=I_4+I_5+I_6.
\end{align*}
By Proposition \ref{prop:phi1}, we have
\begin{align*}
&|I_4|\leq \dfrac{C|\al||j\pi-y_c|\phi_1(j\pi,c)}{u'(y_c)}\min\{|\al||j\pi-y_c|,1\}\|\widehat{\om}_e\|_{L^{\infty}},\\
&\big\|\f{I_5}{u'}\big\|_{L^2}\leq C\big\|\f{\widehat{\om}_e'}{u'}\big\|_{L^2},\ |I_6|\leq C|\phi_1(j\pi,c)|j\pi-y_c|\|\widehat{\om}_e\|_{L^{\infty}},
\end{align*}
which implies that
\beno
\pa_c\left(u''(y_c)E_{j}(\widehat{\om}_e)+u'(y_c)\widehat{\om}_e(y_c)\right)=\f{\phi_1(j\pi,c)\al^2|j\pi-y_c|\cL^{\infty}}{((1-j)\pi-y_c)}+u'(y_c)\cL^2.
\eeno
Similarly, due to
\begin{align*}
&\pa_c\left(E_{j}(u''g)+u'(y_c)g(y_c)\right)
=\int_{y_c}^{j\pi}u''(y)\pa_c\phi_1(y,c)g(y)dy+g'(y_c),
\end{align*}
we get
\beno
\partial_c\left(\rmE_j(u''g)+u'(y_c){g(y_c)}\right)
=\f{\phi_1(j\pi,c)\al^{\frac{1}{2}}|j\pi-y_c| \cL^{\infty}}{((1-j)\pi-y_c)}+u'(y_c)\cL^2,
\eeno
from which and Lemma \ref{Lem: J_j high}, we infer that
\beno
\pa_c\big(\Lambda_{3,1}(\widehat{\om}_e)\big)
=\al^{-2}\cL^2+\cL^{\infty},\quad
\pa_c\big(\Lambda_{4,1}(g)\big)
=\al^{-2}(\cL^2+\al^{\frac{1}{2}}\cL^{\infty}).
\eeno

We have
\begin{align*}
\partial_c^2\left({\rmE_j(gu'')}+{u'(y_c)}g(y_c)\right)
=\f{g''(y_c)}{u'(y_c)}+\int_{y_c}^{j\pi}gu''\partial_c^2\phi_1(y,c)dy,
\end{align*}
and
\begin{align}\label{7.8}
&\partial_c^2\left({u''(y_c)}\rmE_j(\widehat{\om}_e)+u'(y_c){\widehat{\om}_e(y_c)}\right)\\ \nonumber
&=\left(\f{u''''(y_c)}{u'(y_c)^2}-\f{(u'''u'')(y_c)}{u'(y_c)^3}\right)\rmE_j(\widehat{\om}_e)
+\f{\widehat{\om}_e''(y_c)}{u'(y_c)}+2\f{u'''(y_c)}{u'(y_c)}\int_{y_c}^{j\pi}\widehat{\om}_e\partial_c\phi_1(y,c)dy\\ \nonumber
&\quad+{u''(y_c)}\int_{y_c}^{j\pi}\widehat{\om}_e\partial_c^2\phi_1(y,c)dy-\f{u'''(y_c)}{u'(y_c)^2}\widehat{\om}_e(y_c)\\ \nonumber
&=\f{\widehat{\om}_e''(y_c)}{u'(y_c)}-2\int_{y_c}^{j\pi}\widehat{\om}_e\partial_c\phi_1(y,c)dy
+{u''(y_c)}\int_{y_c}^{j\pi}\widehat{\om}_e\partial_c^2\phi_1(y,c)dy-\f{u'''(y_c)}{u'(y_c)^2}\widehat{\om}_e(y_c)
\end{align}
By Proposition \ref{prop:phi1}, we deduce that for $0<y<y_c,\ j=0$ or $y_c<y<\pi ,\ j=1$, $\Big|\f{\partial_c^2\phi_1(y,c)}{\phi_1(j\pi,c)}\Big|\leq \dfrac{C\al^2}{u'(y_c)^2}$, and then
\begin{align}\label{7.11}
&\left\|y_c(\pi-y_c)^2\int_{y_c}^0\varphi(y)\f{\partial_c^2\phi_1(y,c)}{\phi_1(0,c)}dy\right\|_{L^p}\leq C\al^2\|\varphi\|_{L^{p}},\ \ 1<p\leq \infty,\\ \label{7.12}
&\left\|y_c^2(\pi-y_c)\int_{y_c}^{\pi}\varphi(y)\f{\partial_c^2\phi_1(y,c)}{\phi_1(\pi,c)}dy\right\|_{L^p}\leq C\al^2\|\varphi\|_{L^{p}},\ \ 1<p\leq \infty.
\end{align}
Then we can deduce that
\begin{align*}
&y_c\partial_c^2\left({u''(y_c)}\rmE_0(\widehat{\om}_e)+u'(y_c){\widehat{\om}_e(y_c)}\right)
=\f{\phi_1(0,c)\al^2 \cL^{\infty}}{(\pi-y_c)^2}+\cL^2,\\
&(\pi-y_c)\partial_c^2\left({u''(y_c)}\rmE_1(\widehat{\om}_e)+u'(y_c){\widehat{\om}_e(y_c)}\right)
=\f{\phi_1(\pi,c)\al^2 \cL^{\infty}}{y_c^2}+\cL^2,
\end{align*}
from which and Lemma \ref{Lem: J_j high}, we infer that
\begin{align}
\label{eq:partLambda_3,1}&\pa_c^2\left(\f{J_{1-j}}{u'(y_c)}\big({u''(y_c)}\rmE_j(\widehat{\om}_e)+u'(y_c){\widehat{\om}_e(y_c)}\big)\right)
=\f{\cL^{\infty}}{(j\pi-y_c)^2}+\f{\cL^2}{\al^2(j\pi-y_c)^2},
\end{align}
which implies that
\beno
\pa_c^2\big(\Lambda_{3,1}(\widehat{\om}_e)\big)
=\f{\al^{-2}\cL^2+\cL^{\infty}}{\sin^2y_c}.
\eeno

Similarly, we have
\begin{align*}
&y_c\partial_c^2\left(\rmE_0(u''g)+u'(y_c){g(y_c)}\right)
=\f{\phi_1(0,c)\al^{\frac{1}{2}}\cL^{\infty}}{(\pi-y_c)^2}+\cL^2,\\
&(\pi-y_c)\partial_c^2\left(\rmE_1(u''g)+u'(y_c){g(y_c)}\right)
=\f{\phi_1(\pi,c)\al^{\frac{1}{2}} \cL^{\infty}}{y_c^2}+\cL^2,
\end{align*}
from which and Lemma \ref{Lem: J_j high}, we infer that
\beno
\pa_c^2\big(\Lambda_{4,1}(g)\big)
=\f{\cL^2+\al^{\frac{1}{2}}\cL^{\infty}}{\al^2\sin^2 y_c}.
\eeno

{\bf Case 4. } $\|\widehat{\om}_e\|_{H^2}\leq 1$ and $\|\pa_y\widehat{\om}_e/u'\|_{H^1}\leq 1$.\smallskip

Recall that
\begin{align*}
&\pa_c\left(u''(y_c)E_{j}(\widehat{\om}_e)+u'(y_c)\widehat{\om}_e(y_c)\right)\\
&=\int_{y_c}^{j\pi}u''(y_c)\pa_c\phi_1(y,c)\widehat{\om}_e(y)dy+\widehat{\om}_e'(y_c)-\int_{y_c}^{j\pi}\phi_1(y,c)\widehat{\om}_e(y)dy.
\end{align*}
Now since $\pa_y\widehat{\om}_e/u'\in H^1\subset L^{\infty}$, we get
\beno
\pa_c\left(u''(y_c)E_{j}(\widehat{\om}_e)+u'(y_c)\widehat{\om}_e(y_c)\right)=\f{\phi_1(j\pi,c)\al^2|j\pi-y_c|\cL^{\infty}}{((1-j)\pi-y_c)}+u'(y_c)\cL^{\infty},
\eeno
from which and Lemma \ref{Lem: J_j high}, we infer that $\pa_c\Lambda_{3,1}=\cL^{\infty}$.

For the estimates of higher derivative, we only need to improve the estimate of $\f{J_{1-j}}{u'(y_c)}\big({u''(y_c)}\rmE_j(\widehat{\om}_e)+u'(y_c){\widehat{\om}_e(y_c)}\big)$ for $|y_c-j\pi|\leq \f{1}{\al}$. We show the proof of the case $j=0$, the other case can be proved similarly.

Let $\om_1=\f{\pa_y\widehat{\om}_e}{y}\in H^1(0,1)$, and then
\beno
E_0(\widehat{\omega}_{e})
=\int_{y_c}^0\widehat{\om}_{e}(y)\phi_1(y,c)dy
=y_c\int_{1}^0\widehat{\om}_{e}(yy_c)\widetilde{\phi_1}(y,c)dy,
\eeno
here we recall the definition of $\widetilde{\phi_1}(y,c)=\phi_1(yy_c,c)$ and  its estimate in the proof of Lemma \ref{Lem: J_j high}.

Using the fact that
\beno
\pa_c\widehat{\om}_{e}(yy_c)=\f{y\widehat{\omega}_{e}'(yy_c)}{u'(y_c)}=\f{y^2{\omega}_{1}(yy_c)}{m_1(y_c)},
\eeno
with $m_1(y)=\f{\sin y}{y}$ and
\beno
\pa_c^2\widehat{\om}_{e}(yy_c)=\f{y^3{\omega}_{1}'(yy_c)}{m_1(y_c)u'(y_c)}+\f{y^2{\omega}_{1}(yy_c)m_1'(y_c)}{m_1(y_c)^2u'(y_c)},
\eeno
and that $1\leq |\widetilde{\phi_1}(y,c)|\leq 2$ and \eqref{eq:phi1-pa1}, \eqref{eq:phi1-pa2} hold for $y\in [0,1]$ and $c\in [u(0),u(1/\al)]$, we deduce that for $0<y_c<\frac{1}{\al},$
\begin{align*}
&|\partial_c(E_0(\widehat{\om}_{e})/y_c)|
\leq C(\|{\om}_{1}\|_{L^{\infty}}+\al^2\|\widehat{\om}_{e}\|_{L^{\infty}}),\\
&|\partial_c^2(E_0(\widehat{\om}_{e})/y_c)|\leq C(\al^2\|{\om}_{1}\|_{L^{\infty}}+\al^4\|\widehat{\om}_{e}\|_{L^{\infty}})
+\frac{C}{y_c}\int_0^1y^3|{\om}_{1}'(yy_c)|dy.
\end{align*}
Now using  the fact that
\beno
J_1\left(\f{u''(y_c)}{u'(y_c)}\rmE_{0}(\widehat{\om}_{e})+{\widehat{\om}_{e}(y_c)}\right)
=\f{-(u(\pi)-c)^2y_c^2u''(y_c)}{\widetilde{\phi_1}(0,c)\widetilde{\phi_1}'(0,c)}\f{\rmE_{0}(\widehat{\om}_{e})}{y_c}
+J_1{\widehat{\om}_{e}(y_c)},
\eeno
and by Proposition \ref{Prop:tphi_1} and Lemma \ref{Lem: J_j high}, we deduce that for $0<y_c<\frac{1}{\al},$
\beno
\pa_c^2\left(J_1\Big(\f{u''(y_c)}{u'(y_c)}\rmE_{0}(\widehat{\om}_{e})+{\widehat{\om}_{e}(y_c)}\Big)\right)\chi_{[u(0),u(1/\al)]}=\al^2\cL^{\infty}+\f{\cL^2}{\al^2y_c},
\eeno
which along with \eqref{eq:partLambda_3,1}  gives
\beno
\pa_c^2\left(J_1\Big(\f{u''(y_c)}{u'(y_c)}\rmE_{0}(\widehat{\om}_{e})+{\widehat{\om}_{e}(y_c)}\Big)\right)=\f{\al^2\cL^{\infty}}{(1+\al\sin y_c)^2}+\f{\cL^2}{\al y_c}.
\eeno
Similarly, for $j=1$, we have
\beno
E_1(\widehat{\om}_e)=\int_{y_c}^{\pi}\widehat{\om}_e(y)\phi_1(y,c)dy=(y_c-\pi)\int_{1}^0\widehat{\om}_e\big((1-z)\pi+zy_c\big)\widetilde{\widetilde{\phi_1}}(z,c)dz.
\eeno
Then by the same argument, we obtain
\beno
\pa_c^2\left(J_0\Big(\f{u''(y_c)}{u'(y_c)}\rmE_{1}(\widehat{\om}_{e})+{\widehat{\om}_{e}(y_c)}\Big)\right)=\f{\al^2\cL^{\infty}}{(1+\al\sin y_c)^2}+\f{\cL^2}{\al(\pi-y_c)},
\eeno
from which, we deduce (4) of the lemma.
\end{proof}

Now we are in a position to prove Proposition \ref{Prop: estimate K_e}.

\begin{proof}
{\bf Step1. $L^1$ estimate.}\smallskip

We normalize $\widehat{\om}_e, g$ so that $\|\widehat{\om}_e\|_{L^2}\leq 1$ and $\|g\|_{L^2}\leq 1$. By Lemma \ref{Lem: II_1,1L^2}, we get
$\Lambda_{1,1}(\widehat{\om}_e)=\cL^2$ and $\Lambda_{2,1}(g)=\cL^2$. By Lemma \ref{Lem:L^2L_k} and Lemma \ref{Lem: II}, we get $\Lambda_{1,2}(\widehat{\om}_e)=|\al|^{\f12}\cL^{\infty}+|\al|\sin y_c\cL^2$ and $\Lambda_{2,2}(g)=|\al|^{\f12}\cL^{\infty}+|\al|\sin y_c\cL^2$. Then by Lemma \ref{Lem:Lambda_3,1 4,1}, we get
\beno
&&\Lambda_3(\widehat{\om}_e)=\rho(1+|\al|\sin y_c)\cL^2+|\al|^{\f12}\sin^2 y_c\cL^{\infty}+\f{\cL^2}{\al^2},\\
&&\Lambda_4(g)=\rho(1+|\al|\sin y_c)\cL^2+|\al|^{\f12}\sin^2 y_c\cL^{\infty}+\f{\cL^2}{\al^2},
\eeno
from which and Proposition \ref{Prop: A_1^2+B_1^2}, we deduce that
\beno
\|K_e(\cdot, \al)\|_{L_c^1(-1,1)}
\leq C\|\widehat{\om}_e\|_{L^2}\|g\|_{L^2}.
\eeno

{\bf Step 2. $W^{1,1}$ estimate.}\smallskip

We normalize $\widehat{\om}_e, g$ so that $\|\widehat{\om}_e\|_{H^1}\leq 1$ and $|\al|\|g\|_{L^2}+\|g'\|_{L^2}\leq 1$, then $\|g\|_{L^{\infty}}\leq C|\al|^{-\frac{1}{2}}$.
By Lemma \ref{Lem: II_1,1L^2}, we get $\Lambda_{1,1}(\widehat{\om}_e)=\cL^{\infty}$ and $\Lambda_{2,1}(g)=|\al|^{-\f12}\cL^{\infty}$. By Lemma \ref{Lem:L^2L_k} and Lemma \ref{Lem: II}, we get $\Lambda_{1,2}(\widehat{\om}_e)=|\al|{\sin y_c}\cL^{\infty}$ and $\Lambda_{2,2}(g)=|\al|^{\f12}{\sin y_c}\cL^{\infty}$. Then by Lemma \ref{Lem:Lambda_3,1 4,1}, we get
\beno
&&\Lambda_3(\widehat{\om}_e)=\sin^2y_c(1+|\al|\sin y_c)\cL^{\infty}+\frac{C\sin y_c(\cL^2+\al \cL^{\infty})}{\al^2}\\
&&\quad=\frac{\sin y_c}{\al}(1+|\al|\sin y_c)^2\cL^{\infty}+\frac{\sin y_c}{\al^2}\cL^2={|\al|^{-\f32}\sin y_c}(1+|\al|\sin y_c)^3\cL^2,\\
&&\Lambda_4(g)=|\al|^{-\f12}\sin^2y_c(1+|\al|\sin y_c)\cL^{\infty}+\frac{\sin y_c(\cL^2+|\al|^{\f12}\cL^{\infty})}{\al^2}\\
&&\quad=\frac{\sin y_c}{|\al|^{-\f32}}(1+|\al|\sin y_c)^2\cL^{\infty}+\frac{\sin y_c}{\al^2}\cL^2=\f{\sin y_c}{\al^2}(1+|\al|\sin y_c)^3\cL^2.
\eeno

By Lemma \ref{Lem: II_1,1L^2}, we get $\pa_c\Lambda_{1,1}(\widehat{\om}_e)=\f{\cL^{2}}{\sin y_c}$ and $\pa_c\Lambda_{2,1}(g)=\f{\cL^{2}}{\sin y_c}$. By Lemma \ref{Lem:H^1II_1,2_periodic} and Lemma \ref{Lem: II}, we get
$\pa_c\Lambda_{1,2}(\widehat{\om}_e)=\f{\al\cL^{\infty}}{\sin y_c}+\al \cL^2$ and $\pa_c\Lambda_{2,2}(g)=\f{\al^{\f12}\cL^{\infty}}{\sin y_c}+\al \cL^2$. Then by Lemma \ref{Lem:Lambda_3,1 4,1}, we obtain
\begin{align*}
\partial_c\big(\Lambda_3(\widehat{\om}_e)\big)
&=(1+\al \sin y_c)(\sin y_c\cL^{2}+ \cL^{\infty})+\f{(\cL^2+\al \cL^{\infty})}{\al^2\sin y_c}\\
&=\frac{(1+|\al|\sin y_c)^2}{\al\sin y_c}\cL^{\infty}+\frac{(1+|\al|\sin y_c)^3}{\al^2\sin y_c}\cL^2=\frac{(1+|\al|\sin y_c)^3}{|\al|^{\f32}\sin y_c}\cL^2,\\
\partial_c\big(\Lambda_4(g)\big)&=\al^{\f12}\sin y_c\cL^{\infty}+\sin y_c(1+\al\sin y_c)\cL^2+\f{(\cL^2+\al^{\frac{1}{2}}\cL^{\infty})}{\al^2\sin y_c}\\
&=\frac{(1+|\al|\sin y_c)^2}{|\al|^{\f32}\sin y_c}\cL^{\infty}+\frac{(1+|\al|\sin y_c)^3}{\al^2\sin y_c}\cL^2=\frac{(1+|\al|\sin y_c)^3}{|\al|^{2}\sin y_c}\cL^2,
\end{align*}
from which and  Proposition \ref{Prop: A_1^2+B_1^2},  we infer that
\beno
\|\pa_cK_e(\cdot, \al)\|_{L_c^1(-1,1)}
\leq C\al^{\f12}\|\widehat{\om}_e\|_{H^1}\|g\|_{H^1}.
\eeno

{\bf Step 3. $W^{2,1}$ estimate.}\smallskip

We normalize $\widehat{\om}_e, g$ so that $\|\widehat{\om}_e\|_{H^2}\le 1$ with $\widehat{\om}_e'(0)=\widehat{\om}_e'(1)=0$ and $\|g''-\al^2g\|_{L^{2}}\leq 1$ with $g'(0)=g'(1)=0$. Then we have
\beno
&&\|g''\|_{L^2}^2+\al^2\|g'\|_{L^2}^2+\al^4\|g\|_{L^2}\leq 1,\quad \big\|\f{\widehat{\om}_e'}{u'}\big\|_{L^{2}}\leq C\|\widehat{\om}_e''\|_{L^{2}}\leq C,\\
 &&\|g\|_{L^{\infty}}\leq C\al^{-\frac{3}{2}},\quad
\big\|\f{g'}{u'}\big\|_{L^{2}}\leq C\|g''\|_{L^{2}}\leq C,\quad
 \|g'\|_{L^{\infty}}\leq C\al^{-\frac{1}{2}},\ \|u''g\|_{H^k}\leq C\|g\|_{H^k}
 \eeno
 for $k=0,1,2.$

By Lemma \ref{Lem: II_1,1L^2}, we get $\Lambda_{1,1}(\widehat{\om}_e)=\cL^{\infty}$ and $\Lambda_{2,1}(g)=|\al|^{-\f32}\cL^{\infty}$. By Lemma \ref{Lem:L^2L_k} and Lemma \ref{Lem: II}, we get $\Lambda_{1,2}(\widehat{\om}_e)=|\al|{\sin y_c}\cL^{\infty}$ and $\Lambda_{2,2}(g)=|\al|^{-\f12}{\sin y_c}\cL^{\infty}$. Then by Lemma \ref{Lem:Lambda_3,1 4,1}, we get
\begin{align*}
\Lambda_3(\widehat{\om}_e)
&=\al^{-2}(\sin^2 y_c \cL^2\cap \sin y_c \cL^{\infty})+ \sin^2 y_c(1+\al \sin y_c)\cL^{\infty}=\al^{-\f12}\sin^2 y_c(1+\al \sin y_c)^2\cL^2,\\
\Lambda_4(g)
&=\al^{-2}(\rho \cL^2\cap u' \cL^{\infty})+ \al^{-\f32}\rho (1+\al\sin y_c) \cL^{\infty}=\al^{-2}\sin^2 y_c(1+\al \sin y_c)^2\cL^2.
\end{align*}

By Lemma \ref{Lem: II_1,1L^2}, we get $\pa_c\Lambda_{1,1}(\widehat{\om}_e)=\cL^{2}$ and $\pa_c\Lambda_{2,1}(g)=\cL^{2}$.
By Lemma \ref{Lem:H^1II_1,2_periodic} and Lemma \ref{Lem: II}, we get $\pa_c\Lambda_{1,2}(\widehat{\om}_e)=\f{\al \cL^{\infty}}{\sin y_c}$ and $\pa_c\Lambda_{2,2}(g)=\f{\al^{-\f12}\cL^{\infty}+\al^{\f12}\sin y_c\cL^{\infty}}{\sin y_c}$. Then by Lemma \ref{Lem:Lambda_3,1 4,1},
we obtain
\beno
&&\partial_c\big(\Lambda_3(\widehat{\om}_e)\big)
=\sin^2 y_c\cL^{2}+\al^{-2}\cL^2+ (1+\al \sin y_c)\cL^{\infty}=\al^{-\f12}(1+\al \sin y_c)^2\cL^2,\\
&&\partial_c\big(\Lambda_4(g)\big)
=\al^{-2}(\cL^{2}+\al^{\frac{1}{2}}(1+\al\sin y_c)\cL^{\infty})+\sin^2 y_c(\cL^2+\al^{\frac{1}{2}}\cL^{\infty})=\al^{-2}(1+\al \sin y_c)^3\cL^2.
\eeno

By Lemma \ref{Lem: II_1,1L^2}, we get $\pa_c^2\Lambda_{1,1}(\widehat{\om}_e)=\f{\cL^{2}}{\sin^2y_c}$ and $\pa_c^2\Lambda_{2,1}(g)=\f{\cL^{2}}{\sin^2y_c}$.
By Lemma \ref{Lem:H^2II_1,2_periodic} and Lemma \ref{Lem: II}, we get $\pa_c^2\Lambda_{1,2}(\widehat{\om}_e)=\f{\al\cL^{\infty}}{\sin^3 y_c}+\f{\al\cL^2}{\sin y_c}$ and $\pa_c^2\Lambda_{2,2}(g)=\f{(\al^{\f12}\sin y_c+\al^{-\f12})}{\sin^3 y_c}\cL^{\infty}+\f{\al\cL^2}{\sin y_c}$.
Then by Lemma \ref{Lem:Lambda_3,1 4,1}, we obtain
\begin{align*}
\partial_c^2\big(\Lambda_3(\widehat{\om}_e)\big)
&=(1+\al \sin y_c)(\cL^{2}+ \rho^{-1}\cL^{\infty})+\rho^{-1}(\al^{-2}\cL^2+ \cL^{\infty})=\al^{-\f12}\rho^{-1}(1+\al \sin y_c)^3\cL^2,\\
\partial_c^2\big(\Lambda_4(g)\big)
&=(\al^{-2}\rho^{-1}+1)(\cL^2+\al^{\frac{1}{2}}\cL^{\infty})+(1+\al\sin y_c)\cL^2=\al^{-2}\rho^{-1}(1+\al \sin y_c)^3\cL^2,
\end{align*}
from which and Proposition \ref{Prop: A_1^2+B_1^2}, we infer that
\beno
\|\pa_c^2K_e(\cdot, \al)\|_{L_c^1(-1,1)}
\leq C\al^{\f32}\|\widehat{\om}_e\|_{H^2}\|f\|_{L^2},
\eeno
and
\beno
K_e(c,\al)=C(\al)\sin^2 y_c\cL^2,
\eeno
which implies that $K_e(\pm 1, \al)=0$.
\end{proof}

\section{Decay estimates at critical points}\label{critical points}

\subsection{Vorticity depletion phenomena}

In this subsection, we give a formula of the vorticity at a critical point and
confirm the vorticity depletion phenomena found by Bouchet and Morita \cite{BM}.

\begin{lemma}\label{Lem:repre of vorticity}
Let $\om(t,x,y)$ be a solution to \eqref{eq:Euler-L}. Then we have
\begin{align*}
\widehat{w}(t,\al,0)=\int_{-1}^1\f{(1+\cos y_c)^2\Lambda_3(\widehat{\omega}_{e})e^{-i\al c t}}{u'(y_c)\phi_1'(0,c)(\rmA_1^2+\rmB_1^2)}dc.
\end{align*}
and $\f{(1+\cos y_c)^2\Lambda_3(\widehat{\omega}_{e})}{u'(y_c)\phi_1'(0,c)(\rmA_1^2+\rmB_1^2)}\in L^1_c(-1,1).$ In particular, for any  $\al\neq 0$,
\beno
\lim_{t\to\infty}\widehat{w}(t,\al,0)=0.
\eeno
\end{lemma}

\begin{proof}
Let $\eta\in C^{\infty}(\R)$ so that $\eta(y)=1$ for $|y|<\f12$ and $\eta(y)=0$ for $|y|>1$ and $\eta'(y)\leq 0$ for $y>0$.  We denote $\eta_n(y)=\eta(ny)$, $g_n(y)=\eta_n'(y)=n\eta'(ny)$ and $f_n(y)=g_n''(y)-\al^2g_n(y)$.

Assume that $\om_0(x,y)\in H^2$, then for any $t\geq 0$, $\om(t,x,y)\in H^2$. Then $\widehat{\om}_o(t,\al,0)=0$, and by Lemma \ref{lem:kernel identityeven},
\begin{align*}
\widehat{w}(t,\al,0)
&=-\lim_{n\to+\infty}\int_0^{\pi}\widehat{\om}_e(t,\al,y)g_n(y)dy
\\
&=\lim_{n\to+\infty}\int_0^{\pi}\widehat{\psi}_e(t,\al,y)f_n(y)dy\\
&=-\lim_{n\to+\infty}
\int_{u(0)}^{u(\pi)}
\f{\Lambda_3(\widehat{\omega}_{e})\Lambda_4(g_n)}{u'(y_c)(\rmA_1^2+\rmB_1^2)}e^{-i\al c t}dc.
\end{align*}

{\bf Pointwise limit of $\Lambda_4(g_n)$.} \smallskip

Since $\lim\limits_{n\to+\infty}\eta_n(y_c)=0$ and $ \lim\limits_{n\to+\infty}g_n(y_c)=0$ for $y_c\in (0,\pi],$  we get
\beno
&&\lim\limits_{n\to+\infty}\rmE_0(g_nu'')(c)=\lim\limits_{n\to+\infty}-\int_0^{y_c}u''(y)\phi_1(y,c)d\eta_n(y)=u''(0)\phi_1(0,c),\\
&&\lim\limits_{n\to+\infty}\rmE_1(g_nu'')(c)=\lim\limits_{n\to+\infty}\int_{y_c}^{\pi}g_n(y)u''(y)\phi_1(y,c)dy=0.
\eeno
By the proof of Lemma \ref{Lem:convergemu}, we get
$\mathrm{II}_{1,1}(\varphi)=p.v.\int_{0}^{\pi}\f{\int_{y_c}^y\varphi(y')dy'}{(u(y)-c)^2}dy.$ Then for $1/n<y_c$, as $n\to \infty$,
\begin{align*}
|\textrm{II}_{1,1}(g_nu'')(c)|
&\leq\int_{0}^{1/n}\frac{\int_y^{y_c}|g_nu''|(z)dz}{(u(y)-c)^2}dy\\
&\leq \|g_nu''\|_{L^1}\int_{0}^{1/n}\f{dy}{(u(y)-c)^2}\to 0.
\end{align*}
By \eqref{eq:directive_k}, we also have for $1/n<y_c$, as $n\to \infty$,
\begin{align}\label{eq:L_0limit}
|\mathcal{L}_0(g_nu'')|
&\leq \bigg|\int_0^{\f1n}\int_{y_c}^{z}{g_nu''}(y)\left(\f{1}{(u(z)-c)^2}\left(\f{\phi_1(y,c)\,}{\phi_1(z,c)^2}-1\right)\right)\,dydz\bigg|\\
\nonumber&\leq C\|g_nu''\|_{L^1}\int_0^{\f1n}\sup_{y\in [z,1/n]}\left|\f{1}{(u(z)-c)^2}\left(\f{\phi_1(y,c)\,}{\phi_1(z,c)^2}-1\right)\right|\,dz\\
\nonumber&\leq \f{C|\al|}{ny_c^3}\|g_nu''\|_{L^1}\to 0
\end{align}
Thus, for $y_c\in (0,\pi]$, we obtain
\begin{align*}
\lim\limits_{n\to+\infty}\Lambda_4(g_n)(c)
=&\dfrac{J_1(c)u''(0)\phi_1(0,c)}{u'(y_c)}\\
=&-\dfrac{u'(y_c)\rho(c)^2u''(0)\phi_1(0,c)}{\phi(0,c)\phi'(0,c)u'(y_c)}
=-\dfrac{(u(\pi)-c)^2}{\phi_1'(0,c)}.
\end{align*}

{\bf Uniform bound of $\Lambda_4(g_n)$.}\smallskip

For $c\in (-1,1],$ by Lemma \ref{Lem: J_j high} we have for $j=0,1,$
\beno
\|\rmE_j(g_nu'')(c)|\leq \phi_1(j\pi,c)\|g_nu''\|_{L^1}\leq C\phi_1(j\pi,c),\quad |J_j|\leq \f{C}{\al^2\phi_1((1-j)\pi,c)}.
\eeno
By Lemma \ref{Lem: II}, we get $|\rho\mathrm{II}(c)|\leq C\al \sin y_c.$
By Lemma \ref{Lem:L^2L_k}, we have
\begin{align*}
|\sin y_c^2\mathcal{L}_0(g_nu'')|\leq C|\al|\|g_nu''\|_{L^1}.
\end{align*}
Thus, $|\Lambda_{2,2}(g_nu'')|\leq C|\al|$.

By \eqref{eq: OpH}, we get
\begin{align}
\sin^3y_c \mathrm{II}_{1,1}(\varphi)(y_c)=-\f{\sin y_c}{2}\mathcal{H}(\varphi)-\f{\sin y_c}{2}\int_{-\pi}^{\pi}\In(\varphi)(y')dy'.
\end{align}
As $Z\circ \In(u''g_n)$ is odd, we get by Lemma \ref{Lem: identity-Hilbert-periodic} that
\beno
\sin y_cH(Z\circ \In(u''g_n))=H(\sin yZ\circ \In(u''g_n)).
\eeno
Using the fact that for $y\in[0,\pi]$,
\beno
\sin yZ\circ \In(u''g_n)(y)=\sin yu''g_n(y)-\sin y\int_0^{\pi}u''g_n(y)dy-\cos y\In(u''g_n)(y),
\eeno
we deduce that
\beno
\sin y_c\mathcal{H}(u''g_n)=H(\sin yu''g_n(y))-\int_0^{\pi}u''g_n(y)dyH(\sin y)-H(\cos y\In(u''g_n)(y)).
\eeno
Thanks to $\|\sin yu''g_n\|_{H^1}\leq C\sqrt{n}$ and $\|\sin yu''g_n\|_{L^2}\leq \f{C}{\sqrt{n}}$, we get
\beno
\|H(\sin yu''g_n(y))\|_{L^{\infty}}\leq C\|H(\sin yu''g_n)\|_{H^1}^{\f12}\|H(\sin yu''g_n)\|_{L^2}^{\f12}\leq C.
\eeno
Thanks to  $\|\cos y\In(u''g_n)(y)\|_{H^1}\leq C\sqrt{n}$ and $\|\cos y\In(u''g_n)(y)\|_{L^2}\leq \f{C}{\sqrt{n}}$, we get
\beno
\|H(\cos y\In(u''g_n))\|_{L^{\infty}}\leq C\|H(\cos y\In(u''g_n))\|_{H^1}^{\f12}\|H(\cos y\In(u''g_n))\|_{L^2}^{\f12}\leq C.
\eeno

With the estimates above, we conclude $\|u'(y_c)\Lambda_4(g_n)\|_{L^{\infty}}\leq C|\al|$.\smallskip

By the proof of Proposition \ref{Prop: estimate K_e}, we get for $\widehat{\om}_e\in H^2$,
\beno
\Lambda_3(\widehat{\om}_e)
=C\al^{-2}(\sin^2 y_c \cL^2\cap \sin y_c \cL^{\infty})+ C\sin^2 y_c(1+\al \sin y_c)\cL^{\infty}.
\eeno
Thus, by Lebesgue's dominated convergence theorem, we obtain
\begin{align*}
\widehat{\om}(t,\al,0)
&=\lim_{n\to+\infty}
-\int_{u(0)}^{u(\pi)}
\f{\Lambda_3(\widehat{\omega}_{e})\Lambda_4(g_n)}{u'(y_c)(\rmA_1^2+\rmB_1^2)}e^{-i\al c t}dc\\
&=\int_{u(0)}^{u(\pi)}
\f{(u(\pi)-c)^2\Lambda_3(\widehat{\omega}_{e})e^{-i\al c t}}{u'(y_c)\phi_1'(0,c)(\rmA_1^2+\rmB_1^2)}dc.
\end{align*}
By Proposition \ref{Prop: A_1^2+B_1^2}, we get $\f{(u(\pi)-c)^2\Lambda_3(\widehat{\omega}_{e})e^{-i\al c t}}{u'(y_c)\phi_1'(0,c)(\rmA_1^2+\rmB_1^2)}\in L^2_{y_c}\subset L^1_{y_c}\subset L^1_c$.
\end{proof}

\subsection{Decay estimates at critical points}

First of all, we need the following formula of the stream function at a critical point.

\begin{lemma}\label{lem:psi-c}
Let $\om(t,x,y)$ be the solution to \eqref{eq:Euler-L}, and $-\Delta \psi=\om$. Then it holds that
\begin{align*}
&\widehat{\psi}(t,\al,0)
=\int_{-1}^{1}
\f{(1+\cos y_c)\sin y_c\Lambda_3(\widehat{\omega}_{e})e^{-i\al c t}}{\phi_1'(0,c)(\rmA_1^2+\rmB_1^2)}dc,\\
&\partial_y\widehat{\psi}(t,\al,0)
=\int_{-1}^{1}\frac{\sin y_c\Lambda_1(\widehat{\omega}_{o})}{(\rmA^2+\rmB^2)\phi(0,c)}e^{-i\al ct}dc.
\end{align*}
\end{lemma}

\begin{proof}
By \eqref{eq:Euler-L-operator}, we get
\beno
u''(y)i\al \widehat{\psi}=-\pa_t\widehat{\om}-i\al u(y)\widehat{\om},
\eeno
which gives
\beno
\widehat{\psi}(t,\al,0)=\f{-1}{i\al u''(0)}\big(\pa_t\widehat{\om}_e(t,\al,0)+i\al u(0)\widehat{\om}_e(t,\al,0)\big).
\eeno
By Lemma \ref{Lem:repre of vorticity}, we get
\begin{align*}
\pa_t\widehat{\om}(t,\al,0)
&=\int_{u(0)}^{u(\pi)}
-i\al c\f{(u(\pi)-c)^2\Lambda_3(\widehat{\omega}_{e})e^{-i\al c t}}{u'(y_c)\phi_1'(0,c)(\rmA_1^2+\rmB_1^2)}dc,
\end{align*}
which implies the first identity.

Thanks to $\pa_y\psi(t,x,0)=\pa_y\psi_o(t,x,0)$, we get by Lemma \ref{Lem:psi(t)}
that
\begin{align*}
\pa_y\psi(t,x,0)&=\f{1}{2\pi}\int_{-1}^{1}e^{-i\al tc}\al\pa_y\widetilde{\Phi}_o(\al,0,c)dc\\
&=\f{1}{2\pi}\int_{-1}^{1}e^{-i\al tc}\f{\al(\mu_-^o(c)-\mu_+^o(c))}{\phi(0,c)}dc,\\
&=\int_{-1}^1\f{\sin y_c\Lambda_1(\widehat{\om}_o)}{(\rmA^2+\rmB^2)\phi(0,c)}e^{-i\al tc}dc.
\end{align*}
This finishes the proof.
\end{proof}

We denote
\beno
&&K_o^0(c,\al)=\f{\sin y_c\Lambda_1(\widehat{\om}_o)(y_c)}{(\rmA^2+\rmB^2)\phi(0,c)},\\
&&K_e^0(c,\al)=\f{(1+\cos y_c)\sin y_c\Lambda_3(\widehat{\omega}_{e})}{\phi_1'(0,c)(\rmA_1^2+\rmB_1^2)}=-\f{J_1^1(c)\Lambda_3(\widehat{\omega}_{e})(y_c)}{(\rmA_1^2+\rmB_1^2)}.
\eeno

\begin{lemma}\label{prop:K_o^0}
If $\|\widehat{\omega}_{o}\|_{H^2}+\|\widehat{\omega}_{o}''/u'\|_{L^2}\leq 1$, then we have
\beno
\partial_cK_o^0={\al} y_c^{-2}\sin y_c\cL^{\infty},\quad \partial_c^2K_o^0={\al y_c^{-2}}( (\sin y_c)^{-1}\cL^{\infty}+\cL^2).
\eeno
\end{lemma}

\begin{proof}
By the identity $\f{\rho(c)}{\phi(0,c)}=-\f{1-c}{\phi_1(0,c)}$ and Remark \ref{Rmk: fg_periodic}, we get
\ben\label{eq:rho/phi1}
\left|\f{\rho(c)}{\phi(0,c)}\right|\leq C(1-c),\
\left|\partial_c\f{\rho(c)}{\phi(0,c)}\right|\leq \left|\dfrac{1}{\phi_1(0,c)}\right|+\left|\f{(c-1)\cG}{\phi_1(0,c)^2}\right|\leq\f{C\al(\pi-y_c)}{y_c}+C,
\een
and
\ben
\label{eq:rho/phi2}\left|\partial_c^2\dfrac{\rho(c)}{\phi(0,c)}\right|\leq \left|\dfrac{2\cG+(c-1)\partial_c\cG}{\phi_1(0,c)^2}\right|+2\left|\dfrac{(c-1)\cG^2}{\phi_1(0,c)^3}\right|\leq
\frac{C\al^2}{y_c^2}+\frac{C\al}{\sin y_c}.
\een
Then by \eqref{eq:defLam_1,1} and Lemma \ref{Lem: II_1,1odd}, we have
\beno
&&\pa_c\Lambda_{1,1}(\widehat{\omega}_{o})(y_c)=\cL^{\infty},\quad\Lambda_{1,1}(\widehat{\omega}_{o})(y_c)=\sin^2 y_c\cL^{\infty},\quad\pa_c^2\Lambda_{1,1}(\widehat{\omega}_{o})(y_c)=\f{\cL^2}{\sin y_c}.
\eeno
By \eqref{eq:defLam_1,2}, Lemma \ref{Lem: II} and Lemma \ref{Lem:H^2II_1,2_periodic}, we obtain
\beno
\Lambda_{1,2}(\widehat{\omega}_{o})(y_c)=|\al|\sin^2 y_c\cL^{\infty},\quad
\pa_c\Lambda_{1,2}(\widehat{\omega}_{o})(y_c)=|\al| \cL^{\infty},\quad
\pa_c^2\Lambda_{1,2}(\widehat{\omega}_{o})(y_c)=\f{|\al| \cL^{\infty}}{\sin^2y_c}+\f{|\al|L^2}{\sin y_c}.
\eeno
Therefore, we obtain
\ben\label{eq:Lambda_1new}
\Lambda_1(\widehat{\omega}_{o})=\al\rho \cL^{\infty},\quad
\partial_c\Lambda_1(\widehat{\omega}_{o})=\al \cL^{\infty},\quad
\partial_c^2\Lambda_1(\widehat{\omega}_{o})=\rho^{-1}\al(\sin y_c \cL^2+ \cL^{\infty}).
\een
Notice that $K_o^0(c,\al)=\f{\rho(c)}{\phi(0,c)}\f{\Lambda_1(\widehat{\om}_o)}{\sin y_c(\rmA^2+\rmB^2)}$. By \eqref{eq:rho/phi1}, \eqref{eq:rho/phi2}, \eqref{eq:Lambda_1new} and Proposition \ref{Prop: A^2+B^2}, we deduce the result.
\end{proof}

\begin{lemma}\label{prop:K_e^0}
Assume that $\widehat{\omega}_{e}\in H^2$ and $\f{\pa_y\widehat{\omega}_{e}}{u'}\in H^1$. Then $K_e^0\big|_{c=\pm1}=0$ and
\beno
\|\pa_c^2K_e^0\|_{L_c^1}\leq C{|\al|^2}(\|{\widehat{\om}_{e}}\|_{H^2}+\|{\widehat{\om}_{e}'}/{u'}\|_{H^1}).
\eeno
\end{lemma}

\begin{proof}
We may assume $\|{\widehat{\om}_{e}}\|_{H^2}+\|{\widehat{\om}_{e}'}/{u'}\|_{H^1}\leq 1$ by normalization.

By Lemma \ref{Lem: J_j high}, we obtain
\beno
\left|J_1^1(c)\right|\leq \f{C}{\al^2},\ \left|\pa_cJ_1^1(c)\right|\leq \f{C}{(1+\al\sin y_c)^{2}},\ \left|\pa_c^2J_1^1(c)\right|\leq \f{C\al}{(1+\al\sin y_c)^{3}\sin y_c}.
\eeno
By the proof of Proposition \ref{Prop: estimate K_e}, we have
\beno
\Lambda_3(\widehat{\om}_e)
=\al^{-2}(\sin^2 y_c \cL^2\cap \sin y_c \cL^{\infty})+ \sin^2 y_c(1+\al \sin y_c)\cL^{\infty},
\eeno
and by Lemma \ref{Lem: II_1,1L^2}, we get $\pa_c\Lambda_{1,1}(\widehat{\om}_e)=\cL^{2}$;
by Lemma \ref{Lem:H^1II_1,2_periodic} and Lemma \ref{Lem: II}, we get $\pa_c\Lambda_{1,2}(\widehat{\om}_e)=\f{\al \cL^{\infty}}{\sin y_c}$.
We also have $\Lambda_{1}(\widehat{\om}_e)=(1+\al \sin y_c)\cL^{\infty}.$ Then by Lemma \ref{Lem:Lambda_3,1 4,1}, we get
\beno
\pa_c\big(\Lambda_3(\widehat{\om}_e)\big)
=\sin^2 y_c\cL^{2}+ (1+\al \sin y_c)\cL^{\infty}.
\eeno
By Lemma \ref{Lem: II_1,1L^2}, we get $\pa_c^2\Lambda_{1,1}(\widehat{\om}_e)=\f{\cL^{2}}{\sin^2y_c}$;
and by Lemma \ref{Lem:H^2II_1,2_periodic} and Lemma \ref{Lem: II}, we get $\pa_c^2\Lambda_{1,2}(\widehat{\om}_e)=\f{\al \cL^{\infty}}{\sin^3 y_c}+\f{\al \cL^{2}}{\sin y_c}$.
Then by Lemma \ref{Lem:Lambda_3,1 4,1}, we get
\beno
\pa_c^2\big(\Lambda_3(\widehat{\om}_e)\big)
=\Big(1+\al\sin y_c+\f{1}{\al\sin y_c}\Big)\cL^{2}+ \f{\al^2}{(1+\al\sin y_c)^2}\cL^{\infty}+\f{\al}{\sin y_c}\cL^{\infty}.
\eeno
Then by Proposition \ref{Prop: A_1^2+B_1^2}, we obtain
\beno
u'(y_c)\pa_c^2K_e^0=\f{\al^3\cL^{\infty}}{(1+\al\sin y_c)^6}+\f{\al \cL^{2}}{(1+\al\sin y_c)^4}=\al^2\cL^1,
\eeno
and
\beno
K_{e}^{0}=C(\al)\sin y_c\cL^{\infty},
\eeno
which implies that $K_{e}^{0}\big|_{c=\pm 1}=0$ and $\pa_cK_e^0=\al^2\cL^{\infty}$.
\end{proof}

Now we are in a position to prove Theorem \ref{Thm: critical}.

\begin{proof}
{\bf Step 1.}
Let $c'\in (u(0),u(1))$, which will be determined later. Without loss of generality, let us consider the case of $|\al t|>1.$ Then by Lemma \ref{lem:psi-c}, we have
\begin{align*}
\partial_y\widehat{\psi}(t,\al,0)
&=\int_{u(0)}^{u(\pi)}K_o^0(c,\al)e^{-i\al ct}dc\\
&=\frac{1}{i\al t}\int_{u(0)}^{u(\pi)}\partial_cK_o^0(c,\al)e^{-i\al ct}dc\\
&=\frac{1}{i\al t}\int_{u(0)}^{c'}\partial_cK_o^0(c,\al)e^{-i\al ct}dc\\
&\quad-\frac{1}{(\al t)^2}\left(\partial_cK_o^0(c',\al)e^{-i\al c't}+\int_{c'}^{u(\pi)}\partial_c^2K_o^0(c,\al)e^{-i\al ct}dc\right).
\end{align*}
By Lemma \ref{prop:K_o^0}, we have
\begin{align*}
\|\partial_c^2K_o^0\|_{L_c^1(c',u(\pi))}&=\|u'(y_c)\partial_c^2K_o^0\|_{L_{y_c}^1(y_{c'},1)}\\
&\leq C\al(\|y_c^{-2}\|_{L_{y_c}^1(y_{c'},1)}+\|y_c^{-1}\|_{L_{y_c}^2(y_{c'},1)})\leq C\al y_{c'}^{-1}.
\end{align*}
Therefore, we deduce that for any $y_{c'}\in (0,1)$,
\beno
|\partial_y\widehat{\psi}(t,\al,0)|\leq \frac{C\al}{|\al t|}\int_{u(0)}^{c'}\frac{dc}{y_c}+C\frac{1}{(\al t)^2}\frac{\al}{y_{c'}}\leq C\Big(\frac{y_{c'}}{|t|}+\frac{1}{\al y_{c'}t^2}\Big).
\eeno
Taking $y_{c'}=|\al t|^{-\frac{1}{2}}, $ we deduce that
\beno
 |\partial_y\widehat{\psi}(t,\al,0)|\leq C\al^{-\frac{1}{2}}|t|^{-\frac{3}{2}}.
\eeno

 {\bf Step 2.} Using the fact that $\widehat{\psi}(t,\al,0)=\int_{-1}^1K_{e}^0(c,\al)e^{-i\al ct}dc$, and Lemma \ref{prop:K_e^0}, we obtain
\begin{align*}
|\widehat{\psi}(t,\al,0)|&=\left|\f{1}{i\al t}\int_{-1}^1\pa_cK_{e}^0(c,\al)e^{-i\al ct}dc\right|\\
&\leq \left|\f{1}{\al^2 t^2}\int_{-1}^1\pa_c^2K_{e}^0(c,\al)e^{-i\al ct}dc\right|+\f{1}{\al^2 t^2}\|\pa_cK_{e}^0(c,\al)\|_{L^{\infty}}\\
&\leq Ct^{-2}(\|\widehat{\om}_e\|_{H^2}+\|\widehat{\om}_e'/u'\|_{H^1}).
\end{align*}
Let $W(t,x,y)=\om(t,x+u(y)t,y)$. Then $\widehat{W}(t,\al,y)=e^{i\al tu(y)}\widehat{\om}(t,\al,y)$ and
\beno
\pa_t\widehat{W}=-i\al e^{i\al u(y)t}u''\widehat{\psi}.
\eeno
By Lemma \ref{Lem:repre of vorticity}, $\widehat{\om}(t,\al,0)\to 0$ as $t\to \infty$, so does $\widehat{W}(t,\al,0)$. Thus,
\beno
\widehat{W}(t,\al,0)=\int_t^{\infty}i\al e^{i\al u(0)s}u''(0)\widehat{\psi}(s,\al,0)ds.
\eeno
Thus, if $t>1$ then
\begin{align*}
|\widehat{\om}(t,\al,0)|
&=|\widehat{W}(t,\al,0)|
\leq C|\al|\int_t^{+\infty}|\widehat{\psi}(s,\al,0)|ds\\
&\leq C|\al|\int_t^{+\infty}s^{-2}(\|\widehat{\om}_e\|_{H^2}+\|\widehat{\om}_e'/u'\|_{H^1})ds\\
&\leq C|\al||t|^{-1}(\|\widehat{\om}_e\|_{H^2}+\|\widehat{\om}_e'/u'\|_{H^1}),
\end{align*}
and if $t<-1$ then
\begin{align*}
|\widehat{\om}(t,\al,0)|
&=|\widehat{W}(t,\al,0)|
\leq C|\al|\int_{-\infty}^t|\widehat{\psi}(s,\al,0)|ds\\
&\leq C|\al||t|^{-1}(\|\widehat{\om}_e\|_{H^2}+\|\widehat{\om}_e'/u'\|_{H^1}).
\end{align*}

{\bf Step 3.} By Step 1 and Step 2, we deduce that
\begin{align*}
&|{\om}(t,x,0)|=\Big|\sum_{\al\neq 0}\widehat{\om}(t,\al,0)e^{i\al x}\Big|\leq \sum_{\al\neq 0}|\widehat{\om}(t,\al,0)|
\leq \sum_{\al\neq 0}C|\al||t|^{-1}\|\widehat{\om}_e(\al,\cdot)\|_{H^3}\\&\quad\leq C|t|^{-1}\Big(\sum_{\al\neq 0}|\al|^{-2s}\Big)^{\f12}\Big(\sum_{\al\neq 0}|\al|^{2+2s}\|\widehat{\om}_e(\al,\cdot)\|_{H^3}^2\Big)^{\f12}
\leq C|t|^{-1}\|{\om}_e\|_{H_x^{s+1}H_y^3},\\
&|V^2(t,x,0)|=\Big|\sum_{\al\neq 0}i\al\widehat{\psi}(t,\al,0)e^{i\al x}\Big|\leq \sum_{\al\neq 0}|\al||\widehat{\psi}(t,\al,0)|
\leq \sum_{\al\neq 0}C|\al||t|^{-2}\|\widehat{\om}_e(\al,\cdot)\|_{H^3}\\&\quad\leq C|t|^{-2}\Big(\sum_{\al\neq 0}|\al|^{-2s}\Big)^{\f12}\Big(\sum_{\al\neq 0}|\al|^{2+2s}\|\widehat{\om}_e(\al,\cdot)\|_{H^3}^2\Big)^{\f12}
\leq C|t|^{-2}\|{\om}_e\|_{H_x^{s+1}H_y^3},\\
&|V^1(t,x,0)|=\Big|\sum_{\al\neq 0}\pa_y\widehat{\psi}(t,\al,0)e^{i\al x}\Big|\leq \sum_{\al\neq 0}|\pa_y\widehat{\psi}(t,\al,0)|
\leq \sum_{\al\neq 0}C|\al|^{-\f12}|t|^{-\f32}\|\widehat{\om}_o(\al,\cdot)\|_{H^3}\\&\quad\leq C|t|^{-\f32}\Big(\sum_{\al\neq 0}|\al|^{-2s}\Big)^{\f12}\Big(\sum_{\al\neq 0}|\al|^{2s-1}\|\widehat{\om}_o(\al,\cdot)\|_{H^3}^2\Big)^{\f12}
\leq C|t|^{-\f32}\|{\om}_o\|_{H_x^{s-\f12}H_y^3}.
\end{align*}

The above results also hold with $0$ replaced by $\pi,$ thus all $k\pi(k\in\mathbb{Z}).$
\end{proof}

\section{Wave operator}\label{Operator D}

To prove the enhanced dissipation, we use the wave operator method introduced in \cite{LWZ}.

\subsection{Basic properties of wave operator}

\begin{definition}\label{Def:oper_D}
Let $|\al|>1$.
Let $\phi(y,c)$ be defined in Proposition \ref{prop:Rayleigh-Hom} with $\phi_1(y,c)=\f{\phi(y,c)}{u(y)-c}$ for $(y,c)\in [0,\pi]\times [-1,1]$. For any function $\varphi(y)$ defined on $[-\pi,\pi]$, we define
\begin{align*}
-\bbD({\varphi}_{o})(-y_c)=\bbD({\varphi}_{o})(y_c)=\f{\Lambda_{1}(\varphi_{o})(y_c)}{(\mathrm{A}^2+\mathrm{B}^2)^{\f12}},\\
\bbD({\varphi}_{e})(-y_c)=\bbD({\varphi}_{e})(y_c)=\f{\Lambda_3({\varphi}_{e})(y_c)}{(\mathrm{A}_1^2+\mathrm{B}_1^2)^{\f12}},
\end{align*}
where $\varphi_{o}=\f{\varphi(y)-\varphi(-y)}{2}$ and $\varphi_{e}=\f{\varphi(y)+\varphi(-y)}{2}$.
\end{definition}

The operator $\bbD$ is called the wave operator, which
has the following basic properties.

\begin{proposition}\label{Prop: oper_D}
It holds that
\begin{align*}
&\bbD\big(\cos y(1+(\pa_y^2-\al^2)^{-1})\om\big)=\cos y\bbD(\om),\\
&\|\bbD(\om)\|_{L^2}^2=\langle \om, \om+(\pa_y^2-\al^2)^{-1}\om\rangle.
\end{align*}
Moreover, there exists a constant $C$ independent of $\al$ so that
\begin{align*}
&(1-\al^{-2})\|\om\|_{L^2}^2\leq \|\bbD(\om)\|_{L^2}^2\leq \|\om\|_{L^2}^2,\\
&\|\sin y\bbD (\om)\|_{L^2}^2\geq \|\pa_y\psi\|_{L^2}^2+(\al^2-1)\|\psi\|_{L^2}^2,\\
&\|\pa_{y}\bbD(\om)\|_{L^2}\leq C|\al|^{\f12}\|\om\|_{H^1},\quad
\|\pa_{y}^2\bbD(\om)\|_{L^2}\leq C|\al|^{\f32}\|\om\|_{H^2}.
\end{align*}
where $-(\pa_y^2-\al^2)\psi=\om$.
\end{proposition}
\begin{proof}
{\bf Step 1.} If $\om$ is odd, then $\psi$ is odd and $\psi(0)=\psi(\pi)=0$. Then we have
\begin{align*}
&\mathrm{II}_1(u''\psi)\\
&={p.v.}\int_0^{\pi}\f{\int_{y_c}^y\psi(y')u''(y')\phi_1(y',c)dy'}{\phi(y,c)^2}dy\\
&={p.v.}\int_0^{\pi}\f{\int_{y_c}^y\psi(y')(\pa_y^2-\al^2)\phi(y',c)dy'}{\phi(y,c)^2}dy\\
&=-\int_0^{\pi}\f{\int_{y_c}^y\om(y')\phi(y',c)dy'}{\phi(y,c)^2}dy
+{p.v.}\int_0^{\pi}\f{\psi(y)\pa_y\phi(y,c)-u'(y_c)\psi(y_c)}{\phi(y,c)^2}dy\\
&\quad-{p.v.}\int_0^{\pi}\f{\pa_y\psi(y)\phi(y,c)}{\phi(y,c)^2}dy\\
&=-\int_0^{\pi}\f{\int_{y_c}^y\om(y')\phi(y',c)dy'}{\phi(y,c)^2}dy
-\f{\psi(y)}{\phi(y,c)}\bigg|_{y=0}^{\pi}
-\f{\psi(y_c)}{\rho(c)}\rmA\\
&=u(y_c)\mathrm{II}_1(\om)-\f{\psi(y_c)}{\rho(c)}\rmA
-p.v.\int_0^{\pi}\f{\int_{y_c}^y\om(y')u(y')\phi_1(y',c)dy'}{\phi(y,c)^2}dy.
\end{align*}
Here we used Remark  \ref{Rmk:5.9}.
Thus, we get
\beno
\rho\mathrm{II}_1\big(u''\psi+u\om\big)
=\rho u\mathrm{II}_1(\om)-\psi(y_c)\rmA.
\eeno
Recall that $\Lambda_1(\varphi)=\rmA\varphi+u''\rho\mathrm{II}_1(\varphi)$. This shows that for $\om$ odd,
\beno
\bbD(u''\psi+u\om)=u \bbD(\om).
\eeno

If $\om$ is even, then $\psi$ is also even and  $\pa_y\psi(0)=\pa_y\psi(\pi)=0$. Then we have
\begin{align*}
&\mathrm{II}_1(u''(y)\psi(\al,y))\\
&={p.v.}\int_0^{\pi}\f{\int_{y_c}^y\psi(y')u''(y')\phi_1(y',c)dy'}{\phi(y,c)^2}dy\\
&={p.v.}\int_0^{\pi}\f{\int_{y_c}^y\psi(y')(\pa_y^2-\al^2)\phi(y',c)dy'}{\phi(y,c)^2}dy\\
&=-\int_0^{\pi}\f{\int_{y_c}^y\om(y')\phi(y',c)dy'}{\phi(y,c)^2}dy
+{p.v.}\int_0^{\pi}\f{\psi(y)\pa_y\phi(y,c)-u'(y_c)\psi(y_c)}{\phi(y,c)^2}dy\\
&\quad-{p.v.}\int_0^{\pi}\f{\pa_y\psi(y)\phi(y,c)}{\phi(y,c)^2}dy\\
&=-\int_0^{\pi}\f{\int_{y_c}^y\om(y')\phi(y',c)dy'}{\phi(y,c)^2}dy
-\f{\psi(y)}{\phi(y,c)}\bigg|_{y=0}^{\pi}
-\f{\psi(y_c)}{\rho(c)}\rmA\\
&=u(y_c)\mathrm{II}_1(\om)-\f{\psi(y_c)}{\rho(c)}\rmA
-\mathrm{II}_1(u\om)-\f{\psi(\pi)}{\phi(\pi,c)}+\f{\psi(0)}{\phi(0,c)},
\end{align*}
and for $j=0,1,$
\begin{align*}
\mathrm{E}_j(u''\psi)&=\int_{y_c}^{j\pi}u''\psi\phi_1(y,c)\,dy
=\int_{y_c}^{j\pi}\psi(\pa_y^2-\al^2)\phi(y',c)\,dy\\
&=-\int_{y_c}^{j\pi}\om\phi(y',c)\,dy
+\psi\pa_y\phi(y',c)\big|_{y'=y_c}^{j\pi}
-\pa_y\psi\phi(y',c)\big|_{y'=y_c}^{j\pi}\\
&=u(y_c)\rmE_j(\om)-\rmE_j(u\om)+\psi(j\pi)(u(j\pi)-c)\pa_y\phi_1(j\pi,c)-\psi(y_c)u'(y_c).
\end{align*}
Thus, we obtain
\begin{align*}
&\Lambda_3(u''\psi+u\om)\\
&=\rho(c)\rmA(u''\psi+u\om)+u''\rho^2\mathrm{II}_1(u''\psi+u\om)\\
&\quad-\f{u'(y_c)(u(\pi)-c)^2}{\phi_1(0,c)\phi_1'(0,c)}\Big(\f{u''(y_c)}{u'(y_c)}\rmE_0(u''\psi+u\om)+u''\psi+u\om\Big)\\
&\quad+\f{u'(y_c)(u(0)-c)^2}{\phi_1(\pi,c)\phi_1'(\pi,c)}\Big(\f{u''(y_c)}{u'(y_c)}\rmE_1(u''\psi+u\om)+u''\psi+u\om\Big)\\
&=\rho(c)\rmA(u''\psi+u\om)
+u''\rho^2\Big(u(y_c)\mathrm{II}_1(\om)-\f{\rmA\psi}{\rho(c)}-\f{\psi(\pi)}{\phi(\pi,c)}+\f{\psi(0)}{\phi(0,c)}\Big)\\
&\quad-\f{u'(y_c)(u(\pi)-c)^2}{\phi_1(0,c)\phi_1'(0,c)}\Big(\f{u''(y_c)}{u'(y_c)}\big(u\rmE_0(\om)+\psi(0)(u(0)-c)\phi_1'(0,c)-u'\psi\big)+u''\psi+u\om\Big)\\
&\quad+\f{u'(y_c)(u(0)-c)^2}{\phi_1(\pi,c)\phi_1'(\pi,c)}\Big(\f{u''(y_c)}{u'(y_c)}\big(u\rmE_1(\om)+\psi(\pi)(u(\pi)-c)\phi_1'(\pi,c)-u'\psi\big)+u''\psi+u\om\Big)\\
&=u(y_c)\Lambda_3(\om).
\end{align*}
This shows that for $\om$ even,
\beno
\bbD(u''\psi+u\om)=u\bbD(\om).
\eeno

{\bf Step 2.}
Let $\om=\om_o+\om_e$ with $\om_o=\f{\om(y)-\om(-y)}{2}$ and $\om_e=\f{\om(y)+\om(-y)}{2}$. Let $(\pa_y^2-\al^2)g_o=\overline{\om_o}+(\pa_y^2-\al^2)\overline{\om_o}$, where we denote by $\overline{f}$ the complex conjugate of $f$. Then by Lemma \ref{Lem: K_o identity}, we get
\begin{align*}
\langle \psi_o,\om_o+(\pa_y^2-\al^2)\om_o\rangle
&=-2\int_{u(0)}^{u(\pi)}\f{\Lambda_1(\om_o){\Lambda_2(g_o)}}{(\rmA^2+\rmB^2)u'(y_c)}dc.
\end{align*}
As $g_o=\overline{\om_o}-\overline{\psi_o}$, we have
\begin{align*}
\Lambda_2(g_o)&=\rho\mathrm{II}_1(u''g_o)+\rmA g_o=-\rho\mathrm{II}_1(u''\overline{\psi_o}+u\overline{\omega_o})+\rmA g_o\\
&=-\rho u\mathrm{II}_1(\overline{{\om_o}})+\rmA\overline{\psi_o}+\rmA g_o=\rho u''\mathrm{II}_1(\overline{{\om_o}})+\rmA\overline{\omega_o}=\Lambda_1(\overline{\omega_o}).
\end{align*}
This shows that
\begin{align*}
\langle \om_o,\om_o-\psi_o\rangle&=-\langle \psi_o,\om_o+(\pa_y^2-\al^2)\om_o\rangle\\
&=2\int_{u(0)}^{u(\pi)}\f{\Lambda_1(\om_o)\overline{\Lambda_1(\omega_o)}}{(\rmA^2+\rmB^2)u'(y_c)}dc=\|\bbD(\om_o)\|_{L^2}^2.
\end{align*}

Let $(\pa_y^2-\al^2)g_e=\overline{\om_e}+(\pa_y^2-\al^2)\overline{\om_e}$. Then we have
\begin{align*}
\langle \psi_e,\om_e+(\pa_y^2-\al^2)\om_e\rangle
&=-2\int_{u(0)}^{u(\pi)}\f{\Lambda_3(\om_e)\Lambda_4(g_e)}{(\rmA_1^2+\rmB^2_1)u'(y_c)}dc.
\end{align*}
Due to $g_e=\overline{\om_e}-\overline{\psi_e}$, we have
\begin{align*}
u''(y_c)\Lambda_4(g_e)
&=\Lambda_3(u''g_e)=-\Lambda_3(u''\overline{\psi_e}+u\overline{\omega_e})\\
&=-u(y_c)\Lambda_3(\overline{\omega_e})=u''(y_c)\Lambda_3(\overline{\omega_e}),
\end{align*}
thus, $\Lambda_4(g_e)=\Lambda_3(\overline{\omega_e})\ a.e. $ and
\begin{align*}
\langle \om_e,\om_e-\psi_e\rangle&=-\langle \psi_e,\om_e+(\pa_y^2-\al^2)\om_e\rangle\\
&=2\int_{u(0)}^{u(\pi)}\f{\Lambda_3(\om_e)\overline{\Lambda_3(\omega_e)}}{(\rmA_1^2+\rmB_1^2)u'(y_c)}dc
=\|\bbD(\om_e)\|_{L^2}^2.
\end{align*}
Thus we get $\|\bbD(\om)\|_{L^2}^2=\langle \om, \om+(\pa_y^2-\al^2)^{-1}\om\rangle$. \\

{\bf Step 3.}
Let $-(\pa_y^2-\al^2)\psi=\om$. Then we have
\beno
\|\om\|_{L^2}^2=\|\pa_y^2\psi\|_{L^2}^2+2\al^2\|\pa_y\psi\|_{L^2}^2+\al^4\|\psi\|_{L^2}^2\geq \al^2(\|\pa_y\psi\|_{L^2}^2+\al^2\|\psi\|_{L^2}^2),
\eeno
which gives
\begin{align*}
&\|\bbD(\omega)\|_{L^2}^2=\langle \om, \om-\psi\rangle
=\|\om\|_{L^2}^2-(\|\pa_y\psi\|_{L^2}^2+\al^2\|\psi\|_{L^2}^2)
\leq \|\om\|_{L^2}^2,\\
&\|\bbD(\omega)\|_{L^2}^2\geq (1-\al^{-2})\|\om\|_{L^2}^2.
\end{align*}

\begin{align*}
\|\sin y\bbD(\omega)\|_{L^2}^2
&=\|\bbD(\omega)\|_{L^2}^2
-\|\cos y\bbD(\omega)\|_{L^2}^2\\
&=\|\bbD(\omega)\|_{L^2}^2
-\|\bbD(\cos y(\omega-\psi))\|_{L^2}^2\\
&\geq\|\bbD(\omega)\|_{L^2}^2
-\|\cos y(\omega-\psi)\|_{L^2}^2\geq \|\bbD(\omega)\|_{L^2}^2
-\|\omega-\psi\|_{L^2}^2\\
&=\langle\omega,\omega-\psi\rangle-\|\omega-\psi\|_{L^2}^2
=\langle\psi,\omega-\psi\rangle\\
&=\|\partial_y\psi\|_{L^2}^2+(\al^2-1)\|\psi\|_{L^2}^2\geq 0.
\end{align*}

{\bf Step 4.} By Proposition \ref{Prop: A^2+B^2}, we obtain
\begin{align*}
\left|\pa_{y_c}\Big(\f{1}{(\rmA^2+\rmB^2)^{1/2}}\Big)\right|
&\leq C\left|(\rmA^2+\rmB^2)^{1/2}\sin y_c\pa_{c}\Big(\f{1}{\rmA^2+\rmB^2}\Big)\right|\\
&\leq \f{C}{(1+\al\sin y_c)\sin y_c},
\end{align*}
and by the proof of Proposition \ref{Prop: estimate K_o}, we get for $\om\in H^1(\mathbb{T})$ odd, $ \|\om\|_{H^1}=1,$
\beno
&&\Lambda_{1}(\om)(y_c)=\sin y_c((1+|\al|\sin y_c)\cL^2+|\al|^{\f{1}{2}}\cL^{\infty}),\\
&&\pa_{y_c}\Lambda_{1}(\om)(y_c)=\big((1+|\al|\sin y_c)\cL^2+\al^{\f12}\cL^{\infty}\big).
\eeno
Thus, for $\om\in H^1(\mathbb{T})$ odd, $ \|\om\|_{H^1}=1,$
\begin{align*}
\pa_{y_c}\bbD(\om)&=\Lambda_{1}(\om)(y_c)\pa_{y_c}\Big(\f{1}{(\rmA^2+\rmB^2)^{1/2}}\Big)+\f{\pa_{y_c}\Lambda_{1}(\om)(y_c)}{(\rmA^2+\rmB^2)^{1/2}}\\
&=\cL^2+\f{|\al|^{\f12}\cL^{\infty}}{1+\al\sin y_c}=\cL^2.
\end{align*}
Thus, for $\om\in H^1(\mathbb{T})$ odd,
\beno
\|\bbD(\om)\|_{H^1}\leq C\|\om\|_{H^1}.
\eeno

By Proposition \ref{Prop: A^2+B^2}, we get
\begin{align*}
\left|\pa_{y_c}^2\Big(\f{1}{(\rmA^2+\rmB^2)^{1/2}}\Big)\right|
\leq \f{C}{(1+\al\sin y_c)\sin^2 y_c},
\end{align*}
and by the proof of Proposition \ref{Prop: estimate K_o}, for $\om\in H^2(\mathbb{T})$ odd, $ \|\om\|_{H^2}=1,$ we get
\beno
&&\Lambda_1(\om)=(\sin^2y_c\cL^2\cap \sin y_c\cL^{\infty})+|\al|\sin^2y_c \cL^{\infty},\\
&&\partial_{y_c}(\Lambda_1(\om))=\sin y_c(\cL^2+\al \cL^{\infty}),\\
&&\partial_{y_c}^2(\Lambda_1(\om))=((1+|\al|\sin y_c)\cL^2+|\al|\cL^{\infty})
\eeno
Thus for $\om\in H^2(\mathbb{T})$ odd, $ \|\om\|_{H^2}=1,$ we get
\begin{align*}
\pa_{y_c}^2\bbD(\om)&=\Lambda_{1}(\om)\pa_{y_c}^2\Big(\f{1}{(\rmA^2+\rmB^2)^{1/2}}\Big)+\f{\pa_{y_c}^2\Lambda_{1}(\om)}{(\rmA^2+\rmB^2)^{1/2}}+2\pa_{y_c}\Lambda_{1}(\om)\pa_{y_c}\Big(\f{1}{(\rmA^2+\rmB^2)^{1/2}}\Big)\\
&=\cL^2+\f{|\al|^{\f12}\cL^{\infty}}{1+\al\sin y_c}+\f{|\al|\cL^{\infty}}{1+\al\sin y_c}=\al^{\f12}\cL^2.
\end{align*}
Thus, for $\om\in H^2(\mathbb{T})$ odd,
\beno
\|\bbD(\om)\|_{H^2}\leq C\al^{\f12}\|\om\|_{H^2}.
\eeno

By Proposition \ref{Prop: A_1^2+B_1^2}, we get
\beno
&&\f{1}{(\rmA_1^2+\rmB_1^2)^{\f12}}\leq \f{C\al^2}{(1+\al\sin y_c)^3},\\
&&\pa_{y_c}\Big(\f{1}{(\rmA_1^2+\rmB_1^2)^{\f12}}\Big)\leq \f{C\al^2}{(1+\al\sin y_c)^3\sin y_c},\\
&&\pa_{y_c}^2\Big(\f{1}{(\rmA_1^2+\rmB_1^2)^{\f12}}\Big)\leq \f{C\al^2}{(1+\al\sin y_c)^3\sin^2 y_c}.
\eeno
For $\om\in H^1$ even, $ \|\om\|_{H^1}=1,$ by the proof of Proposition \ref{Prop: estimate K_e}, we get
\beno
&&\Lambda_3(\om)={|\al|^{-\f32}\sin y_c}(1+|\al|\sin y_c)^3\cL^2,\\
&&\partial_{y_c}\big(\Lambda_3(\om)\big)
={|\al|^{-\f32}}(1+|\al|\sin y_c)^3\cL^2.
\eeno
Thus, for $\om\in H^1$ even, $ \|\om\|_{H^1}=1,$
\begin{align*}
\pa_{y_c}\bbD(\om)&=\Lambda_{3}(\om)(y_c)\pa_{y_c}\Big(\f{1}{(\rmA_1^2+\rmB_1^2)^{1/2}}\Big)
+\f{\pa_{y_c}\Lambda_{3}(\om)(y_c)}{(\rmA_1^2+\rmB_1^2)^{1/2}}=\al^{\f12}\cL^2.
\end{align*}
This shows that for $\om\in H^1$ even,
\beno
\|\bbD(\om)\|_{H^1}\leq C|\al|^{\f12}\|\om\|_{H^1}.
\eeno

For $\om\in H^2$ even, $\|\om\|_{H^2}=1$, by the proof of Proposition \ref{Prop: estimate K_e}, we can get
\begin{align*}
\partial_{y_c}^m\Lambda_3(\om)
&=\al^{-1/2}(\sin y_c)^{2-m}(1+\al \sin y_c)^3\cL^2,\ \ \ m=0,1,2.
\end{align*}
Thus, for $\om\in H^2$ even, $\|\om\|_{H^2}=1$,
\begin{align*}
\pa_{y_c}^2\bbD(\om)&=\Lambda_{3}(\om)\pa_{y_c}^2\Big(\f{1}{(\rmA_1^2+\rmB_1^2)^{1/2}}\Big)+\f{\pa_{y_c}^2\Lambda_{3}(\om)}{(\rmA_1^2+\rmB_1^2)^{1/2}}+2\pa_{y_c}\Lambda_{3}(\om)\pa_{y_c}\Big(\f{1}{(\rmA_1^2+\rmB_1^2)^{1/2}}\Big)\\
&=\al^{\f32}\cL^2.
\end{align*}
This shows that for $\om\in H^2$ even,
\beno
\|\bbD(\om)\|_{H^2}\leq C|\al|^{\f32}\|\om\|_{H^2}.
\eeno

This completes the proof of the proposition.
\end{proof}

\subsection{Commutator estimate}

\begin{proposition}\label{Prop: commutator}
It holds that
\beno
\|\sin y(\bbD(\pa_y^2\omega)-\pa_y^2\bbD(\omega))\|_{L^2}\leq C(|\al|\|\omega\|_{L^2}+\|\pa_y\omega\|_{L^2}).
\eeno
\end{proposition}

We need following lemmas.

\begin{lemma}\label{lem: commu_II_1,1}
It holds that
\beno
\|\sin y_c[\pa_y^2,\rho\mathrm{II}_{1,1}]\om\|_{L^2}\leq C\|\om\|_{H^1}.
\eeno
\end{lemma}

\begin{proof}
By \eqref{eq: OpH}, we get
\beno
\rho\mathrm{II}_{1,1}(\om)=-\f12H(Z\circ \In(\om))-\f12\int_{-\pi}^{\pi}\In(\om)(y')dy'.
\eeno
As $Z\circ \In(\om)$ is odd, so does $\pa_{y_c}^2Z\circ \In(\om)$. Thus,
\begin{align*}
\sin y_c\pa_{y_c}^2\rho\mathrm{II}_{1,1}(\om)
=-\f12H(\sin y_c\pa_{y_c}^2Z\circ \In(\om))
\end{align*}
Then we  get
\begin{align*}
\sin y_c[\pa_y^2,\rho\mathrm{II}_{1,1}]\om=-\f12H(\sin y_c[\pa_{y_c}^2,Z\circ \In](\om))+\f{\sin y_c}{2}\int_{-\pi}^{\pi}\In(\pa_y^2\om)(y')dy'.
\end{align*}
On one hand, we have
\begin{align*}
\f12\int_{-\pi}^{\pi}\In(\pa_y^2\om)(y)dy=\int_{0}^{\pi}\In(\pa_y^2\om)(y)dy
&=\int_{0}^{\pi}\om'(y)-\om'(0)-\sin^2\f{y}{2}(\om'(\pi)-\om'(0))dy\\
&=\om(\pi)-\om(0)-\pi(\om'(0)+\om'(\pi))/2.
\end{align*}
On the other hand,
\begin{align*}
&\sin y_c\big(Z\circ\In(\pa_y^2\om)-\pa_y^2Z\circ\In(\om)\big)\\
&=\f{\om'(0)+\om'(\pi)}{2}\cos y+\f{\om'(0)-\om'(\pi)}{2}-\f{2\om}{\sin y}+\f{2\cos y}{\sin^2 y}\In(\om)(y)-\f{\Int(\om)(\pi)}{2}.
\end{align*}
and $H(\cos y)=2\pi \sin y$. Then we deduce that
\begin{align*}
\sin y_c[\pa_y^2,\rho\mathrm{II}_{1,1}]\om
=\f12H\Big(-\f{2\om}{\sin y}+\f{2\cos y}{\sin^2 y}\In(\om)(y)\Big)+\sin y_c(\om(\pi)-\om(0)),
\end{align*}
from which, it follows that
\begin{align*}
\big\|\sin y_c[\pa_y^2,\rho\mathrm{II}_{1,1}]\om\big\|_{L^2}
&\leq C\left\|H\Big(-\f{2\om}{\sin y}+\f{2\cos y}{\sin^2 y}\In(\om)(y)\Big)\right\|_{L^2}+C\|\om\|_{L^\infty}\\
&\leq C\left\|-\f{2\om}{\sin y}+\f{2\cos y}{\sin^2 y}\In(\om)(y)\right\|_{L^2}+C\|\om\|_{L^\infty}\\
&\leq C\left\|-\f{\om}{\sin y}+\f{\cos y}{\sin^2 y}\int_0^y\om(y')dy'\right\|_{L^2(0,\f{\pi}{2})}\\
&\quad+C\left\|-\f{\om}{\sin y}+\f{\cos y}{\sin^2 y}\int_{\pi}^y\om(y')dy'\right\|_{L^2[\f{\pi}{2},\pi]}+C\|\om\|_{L^\infty}.
\end{align*}
Using the fact that
\beno
-\f{\om}{\sin y}+\f{\cos y}{\sin^2 y}\int_0^y\om(y')dy'=\f{-1}{\sin^2 y}\big(\int_0^y\sin y'\om'(y')dy'+\int_0^y\om(y')(\cos y'-\cos y)dy'\big),\\
-\f{\om}{\sin y}+\f{\cos y}{\sin^2 y}\int_{\pi}^y\om(y')dy'=\f{-1}{\sin^2 y}\big(\int_{\pi}^y\sin y'\om'(y')dy'+\int_{\pi}^y\om(y')(\cos y'-\cos y)dy'\big),
\eeno
and Hardy's inequality, we infer that
\beno
\left\|-\f{\om}{\sin y}+\f{\cos y}{\sin^2 y}\int_0^y\om(y')dy'\right\|_{L^2(0,\f{\pi}{2})}+\left\|-\f{\om}{\sin y}+\f{\cos y}{\sin^2 y}\int_{\pi}^y\om(y')dy'\right\|_{L^2(\f{\pi}{2},\pi)}\leq C\|\om\|_{H^1}.
\eeno

Summing up, we conclude the lemma.
\end{proof}

\begin{lemma}\label{lem: commL_0}
It holds that
\beno
|\sin^3 y_c[\pa_{y_c}^2,\mathcal{L}_0]\varphi|\leq C\al^{\f12}(\|\varphi\|_{H^1}+\al\|\varphi\|_{L^2}).
\eeno
\end{lemma}

\begin{proof}
By \eqref{eq:pa_cL_0}, we get
\begin{align*}
\partial_{y_c}^2\mathcal{L}_0(\varphi)&=u'(y_c)\pa_c\left(\mathcal{L}_0(\varphi')-I_{0,1}(\varphi)+I_{0,0}(\varphi)\right)+u'(y_c)\pa_c(u'(y_c)\mathcal{L}_{1}(\varphi))\\
&=\mathcal{L}_0(\varphi'')-I_{0,1}(\varphi')+I_{0,0}(\varphi')+u'(y_c)\mathcal{L}_{1}(\varphi')\\
&\quad-u'(y_c)\pa_cI_{0,1}(\varphi)+u'(y_c)\pa_cI_{0,0}(\varphi)+u''(y_c)\mathcal{L}_1(\varphi)\\
&\quad+u'(y_c)\left(\mathcal{L}_1(\varphi')-I_{1,1}(\varphi)+I_{1,0}(\varphi)\right)+u'(y_c)^2\mathcal{L}_{2}(\varphi),
\end{align*}
which gives
\begin{align*}
[\pa_{y_c}^2,\mathcal{L}_0](\varphi)
&=-I_{0,1}(\varphi')+I_{0,0}(\varphi')+u'(y_c)\big(-I_{1,1}(\varphi)+I_{1,0}(\varphi)\big)\\
&\quad-u'(y_c)\pa_cI_{0,1}(\varphi)+u'(y_c)\pa_cI_{0,0}(\varphi)\\
&\quad+u''(y_c)\mathcal{L}_1(\varphi)+2u'(y_c)\mathcal{L}_1(\varphi')+u'(y_c)^2\mathcal{L}_{2}(\varphi).
\end{align*}

By \eqref{eq:directive_k}, we have for $j=0,1$
\begin{align*}
|I_{0,j}(\varphi)(c)|&\leq C\left|\int_{y_c}^{j\pi}\varphi(y)dy\right|\f{\min\{\al^2|j\pi-y_c|^2,1\}}{(\cos j\pi-\cos y_c)^2u'(y_c)}\\
&\leq C\f{|j\pi-y_c|^{\f12}\min\{\al^2|j\pi-y_c|^2,1\}}{(\cos j\pi-\cos y_c)^2u'(y_c)}\|\varphi\|_{L^{2}}\\
&\leq C\f{\min\{\al^2|j\pi-y_c|^2,1\}}{|j\pi-y_c|^{7/2}u'(y_c)}\|\varphi\|_{L^{2}},
\end{align*}
which gives
\beno
|\sin^3 y_cI_{0,j}(\varphi')|\leq C\al^{\f12}\|\varphi'\|_{L^2}.
\eeno

We get by \eqref{eq: I_k,jinfty} that
\beno
|\sin^4 y_cI_{1,j}(\varphi)|\leq C\al\|\varphi\|_{L^{\infty}}\leq C\al^{\f12}(\|\varphi'\|_{L^2}+\al\|\varphi\|_{L^2}).
\eeno
Using \eqref{eq:pa_cI_0,j} with $\gamma=1$, we get
\beno
|\sin^4 y_c\pa_cI_{0,j}(\varphi)|\leq C\al\|\varphi\|_{L^{\infty}}\leq C\al^{\f12}(\|\varphi'\|_{L^2}+\al\|\varphi\|_{L^2}).
\eeno
By \eqref{eq:estL_kL^infty}, we get
\beno
|\sin^3y_c\mathcal{L}_1(\varphi)|+|\sin^5y_c\mathcal{L}_2(\varphi) |\leq C\al\|\varphi\|_{L^{\infty}}\leq C\al^{\f12}\big(\|\varphi'\|_{L^2}+\al\|\varphi\|_{L^2}\big),
\eeno
and by Lemma \ref{Lem:L^2L_k}, we get
\beno
|u'(y_c)^4\mathcal{L}_1(\varphi')|\leq C\al^{\f12}\|\varphi'\|_{L^2}.
\eeno

Summing up, we conclude the lemma.
 \end{proof}

\begin{lemma}\label{Lem:commLambda_3,1}
For $\om\in H^2[0,\pi],\ \om'(0)=\om'(\pi)=0,$ we have
\beno
\sin y_c[\pa_{y_c}^2,\Lambda_{3,1}](\om)=C(\al^{-2}+\rho)(\|\om\|_{H^1}+\al\|\om\|_{L^2})(\al^{\f12}\cL^{\infty}+\cL^2).
\eeno
\end{lemma}

\begin{proof}
We have
\begin{align*}
-\rmE_j(\pa_y^2\om)
&=\om'(y_c)+\int_{y_c}^{j\pi}\pa_y\phi_1(y,c)\pa_y\om dy.
\end{align*}
Then by \eqref{7.7} and \eqref{7.8},  we get
\begin{align*}
&\pa_{y_c}^2\Lambda_{3,1}(\om)-\Lambda_{3,1}(\pa_y^2\om)\\
&=\pa_{y_c}^2\Big(\f{J_{j}}{u'}\Big)(u''\rmE_{1-j}(\om)+u'\om)
+2\pa_{y_c}\Big(\f{J_{j}}{u'}\Big)\pa_{y_c}(u''\rmE_{1-j}(\om)+u'\om)\\
&\quad+\f{J_{j}}{u'}\Big(\pa_{y_c}^2(u''\rmE_{1-j}(\om)+u'\om)-(u''\rmE_{1-j}(\pa_y^2\om)+u'\pa_y^2\om)\Big)\bigg|_{j=0}^1\\
&=\pa_{y_c}^2\Big(\f{J_{j}}{u'}\Big)(u''\rmE_{1-j}(\om)+u'\om)
+2\pa_{y_c}\Big(\f{J_{j}}{u'}\Big)\Big(u'''\rmE_{1-j}(\om)+u'\om'+u''u'\int_{y_c}^{(1-j)\pi}\pa_c\phi_1(y,c)\om(y) dy\Big)\\
&\quad+\f{J_{j}}{u'}\Big(u\rmE_{1-j}(\om)+2u''\pa_y\om+u''\int_{y_c}^{(1-j)\pi}\pa_y\phi_1(y,c)\pa_y\om dy-u'''{\om}\\
&\quad+(u''^2-2u'^2)\int_{y_c}^{(1-j)\pi}\pa_c\phi_1(y,c)\om(y) dy+u'^2u''\int_{y_c}^{(1-j)\pi}\pa_c^2\phi_1(y,c)\om(y) dy\Big)\bigg|_{j=0}^1\\
&=T_1+T_2+T_3.
\end{align*}
By \eqref{eq:I_1+I_2+I_3}-\eqref{eq:estI_3} and Lemma \ref{Lem: J_j high}, we get
\beno
\sin y_cT_1=\al^{-2}(\|\om\|_{H^1}+\al\|\om\|_{L^2})(\al^{\f12}\cL^{\infty}+\cL^2).
\eeno
Using $u'''(y)=-\sin y$, \eqref{7.1}, \eqref{7.9}, \eqref{7.10}  and Lemma \ref{Lem: J_j high}, we deduce that
\beno
&&|\sin y_cT_2|\leq C\al^{-2}(\|\om\|_{H^1}+\al\|\om\|_{L^2})(\al^{\f12}\cL^{\infty}+\cL^2),
\eeno
By \eqref{7.1} and Lemma \ref{Lem: J_j high}, we get
\beno
&&\big|J_{j}\rmE_{1-j}(\om)\big|+|J_{j}\om|\leq C\al^{-2}\|\om\|_{L^{\infty}}\leq C\al^{-2}(\|\om\|_{H^1}+\al\|\om\|_{L^2}),\\
&&\big\|J_{j}\pa_y\om\big\|_{L^2}\leq C\al^{-2}\|\om\|_{H^1},
\eeno
and
\begin{align*}
\Big|J_{j}\int_{y_c}^{(1-j)\pi}\pa_y\phi_1(y,c)\pa_y\om dy\Big|
&\leq C|J_{j}|\int_{y_c}^{(1-j)\pi}\al \al^{\f12}|y-y_c|^{\f12}|\pa_y\om| dy\phi_1((1-j)\pi,c)\\
&\leq C\f{\sin y_c}{\al^{\f12}}\|\pa_y\om\|_{L^2}.
\end{align*}By \eqref{7.9},\eqref{7.10} and Lemma \ref{Lem: J_j high}, we get\begin{align*}
\Big|J_{j}\int_{y_c}^{(1-j)\pi}\pa_c\phi_1(y,c)\om(y) dy \Big|
&\leq C|J_{j}|\al\|\om\|_{L^{\infty}}\frac{\phi_1((1-j)\pi,c)}{|j\pi-y_c|}
\leq C\al^{-1}\|\om\|_{L^{\infty}}.
\end{align*}
By \eqref{7.11}, \eqref{7.12} and Lemma \ref{Lem: J_j high}, we get\begin{align*}
\Big|J_{j}u'^2\int_{y_c}^{(1-j)\pi}\pa_c^2\phi_1(y,c)\om(y) dy\Big|
&\leq C|J_{j}|\al^2\|\om\|_{L^{\infty}}{\phi_1((1-j)\pi,c)}{|(1-j)\pi-y_c|}\\
&\leq C\sin y_c\|\om\|_{L^{\infty}}.
\end{align*}
Notice that\begin{align*}
\|\om\|_{L^{\infty}}\leq C\al^{-\f12}(\|\om\|_{H^1}+\al\|\om\|_{L^2}),\ \ \ \ \ \sin y_c\leq \al^{-1}+\al\rho.
\end{align*}
Then we conclude that
\beno
\sin y_cT_3=(\al^{-2}+\rho)(\|\om\|_{H^1}+\al\|\om\|_{L^2})(\al^{\f12}\cL^{\infty}+\cL^2).
\eeno
This proves the lemma.
\end{proof}

Now we are in a position to prove the proposition.

\begin{proof}
{\bf Step 1. the odd case.} Let $\om\in H^2$ be odd and $\om(0)=\om(\pi)=0$. Direct calculation gives
\begin{align*}
&\pa_{y_c}^2\bbD(\om)-\bbD(\pa_y^2\om)\\
&=\rho\mathrm{II}_{1,1}(\om)\pa_{y_c}^2\Big(\f{u''}{(\rmA^2+\rmB^2)^{\f12}}\Big)+2\pa_{y_c}\big(\rho\mathrm{II}_{1,1}(\om)\big)\pa_{y_c}\Big(\f{u''}{(\rmA^2+\rmB^2)^{\f12}}\Big)+\f{ u''[\pa_{y_c}^2,\rho\mathrm{II}_{1,1}]\om}{(\rmA^2+\rmB^2)^{\f12}}\\
&\quad+\mathcal{L}_{0}(\om)\pa_{y_c}^2\Big(\f{\rho u''}{(\rmA^2+\rmB^2)^{\f12}}\Big)+2\pa_{y_c}\mathcal{L}_{0}(\om)\pa_{y_c}\Big(\f{\rho u''}{(\rmA^2+\rmB^2)^{\f12}}\Big)+\f{\rho u''[\pa_{y_c}^2,\mathcal{L}_{0}]\om}{(\rmA^2+\rmB^2)^{\f12}}\\
&\quad+\om\pa_{y_c}^2\Big(\f{\rho u'\mathrm{II}}{(\rmA^2+\rmB^2)^{\f12}}\Big)+2\pa_{y_c}\om\pa_{y_c}\Big(\f{\rho u'\mathrm{II}}{(\rmA^2+\rmB^2)^{\f12}}\Big)\\
&=I_1+I_2+I_3+I_4+I_5+I_6+I_7+I_8.
\end{align*}

By Lemma \ref{Lem: II_1,1odd} and Proposition \ref{Prop: A^2+B^2}, we get
\begin{align*}
\|\sin y_cI_1\|_{L^2}\leq C\|\sin y_c\mathrm{II}_{1,1}\|_{L^2}\Big\|\rho\pa_{y_c}^2\Big(\f{u''}{(\rmA^2+\rmB^2)^{\f12}}\Big)\Big\|_{L^{\infty}}\leq C\|\om\|_{H^1},
\end{align*}
and by Lemma \ref{Lem: II_1,1L^2} and Proposition \ref{Prop: A^2+B^2},
\begin{align*}
\|\sin y_cI_2\|_{L^2}\leq C\|\rho\mathrm{II}_{1,1}\|_{H^1}\Big\|\sin y_c\pa_{y_c}\Big(\f{u''}{(\rmA^2+\rmB^2)^{\f12}}\Big)\Big\|_{L^{\infty}}\leq C\|\om\|_{H^1},
\end{align*}
and by Lemma \ref{lem: commu_II_1,1} and Proposition \ref{Prop: A^2+B^2},
\beno
\|\sin y_cI_3\|_{L^2}\leq C\|\sin y_c[\pa_y^2,\rho\mathrm{II}_{1,1}]\om\|_{L^2}\leq C\|\om\|_{H^1},
\eeno
and by Lemma \ref{Lem:H^1II_1,2_periodic} and Proposition \ref{Prop: A^2+B^2},
\begin{align*}
\|\sin y_cI_4\|_{L^2}&\leq C\al^{-\f12}\|\sin y_c\mathcal{L}_0(\om)\|_{L^{\infty}}\Big\|\al^{\f12}\pa_{y_c}^2\Big(\f{\rho u''}{(\rmA^2+\rmB^2)^{\f12}}\Big)\Big\|_{L^2}\\
&\leq C\|\om\|_{H^1}\Big\|\f{\al^{\f12}}{(1+\al\sin y_c)}\Big\|_{L^2}\leq C\|\om\|_{H^1},\\
\|\sin y_cI_5\|_{L^2}&\leq C\al^{-\f12}\|\sin^2 y_c\pa_{y_c}\mathcal{L}_0(\om)\|_{L^{\infty}}\Big\|\f{\al^{\f12}}{\sin y_c}\pa_{y_c}\Big(\f{\rho u''}{(\rmA^2+\rmB^2)^{\f12}}\Big)\Big\|_{L^2}\\
&\leq C\|\om\|_{H^1}\Big\|\f{\al^{\f12}}{(1+\al\sin y_c)}\Big\|_{L^2}\leq C\|\om\|_{H^1}.
\end{align*}
By Lemma \ref{lem: commL_0}, we get
\beno
\|\sin y_cI_6\|_{L^2}\leq C\al^{-\f12}\Big\|\f{\al^{\f12}}{(1+\al\sin y_c)}\Big\|_{L^2}\|\sin^3 y_c[\pa_{y_c}^2,\mathcal{L}_0]\om\|_{L^{\infty}}\leq C(\|\om\|_{H^1}+\al\|\om\|_{L^2})
\eeno
By Lemma \ref{Lem: II} and Proposition \ref{Prop: A^2+B^2}, we get
\begin{align*}
\|\sin y_cI_7\|_{L^2}&\leq C\|\om/u'\|_{L^2}\Big\|\rho\pa_{y_c}^2\Big(\f{\rho u'\mathrm{II}}{(\rmA^2+\rmB^2)^{\f12}}\Big)\Big\|_{L^{\infty}}\leq C\|\om\|_{H^1},\\
\|\sin y_cI_8\|_{L^2}&\leq C\|\pa_y\om\|_{L^2}\Big\|\sin y_c\pa_{y_c}\Big(\f{\rho u'\mathrm{II}}{(\rmA^2+\rmB^2)^{\f12}}\Big)\Big\|_{L^{\infty}}\leq C\|\pa_y\om\|_{L^2}.
\end{align*}
Summing up, we conclude the odd case.\smallskip

{\bf Step 2. the even case.} Let $\om\in H^2$ be even, $\om'(0)=\om'(\pi)=0 $. Direct calculation gives
\begin{align*}
&\pa_{y_c}^2\bbD(\om)-\bbD(\pa_y^2\om)\\
&=\f{\rho u''}{(\rmA_1^2+\rmB_1^2)^{\f12}}[\pa_{y_c}^2,\rho\mathrm{II}_{1,1}]\om
+2\pa_{y_c}\Big(\f{\rho u''}{(\rmA_1^2+\rmB_1^2)^{\f12}}\Big)\pa_{y_c}\big(\rho\mathrm{II}_{1,1}(\om)\big)
+\pa_{y_c}^2\Big(\f{\rho u''}{(\rmA_1^2+\rmB_1^2)^{\f12}}\Big)\rho\mathrm{II}_{1,1}(\om)\\
&\quad+\f{\rho^2 u''}{(\rmA_1^2+\rmB_1^2)^{\f12}}[\pa_{y_c}^2,\mathcal{L}_{0}]\om
+2\pa_{y_c}\Big(\f{\rho^2 u''}{(\rmA_1^2+\rmB_1^2)^{\f12}}\Big)\pa_{y_c}\big(\mathcal{L}_{0}(\om)\big)
+\pa_{y_c}^2\Big(\f{\rho^2 u''}{(\rmA_1^2+\rmB_1^2)^{\f12}}\Big)\mathcal{L}_{0}(\om)\\
&\quad+\om\pa_{y_c}^2\Big(\f{\rho^2 u'\mathrm{II}}{(\rmA_1^2+\rmB_1^2)^{\f12}}\Big)+2\pa_{y_c}\Big(\f{\rho^2 u'\mathrm{II}}{(\rmA_1^2+\rmB_1^2)^{\f12}}\Big)\pa_{y_c}\om\\
&\quad+\pa_{y_c}^2\Big(\f{1}{(\rmA_1^2+\rmB_1^2)^{\f12}}\Big)\Lambda_{3,1}(\om)+\pa_{y_c}\Big(\f{1}{(\rmA_1^2+\rmB_1^2)^{\f12}}\Big)\pa_{y_c}\Lambda_{3,1}(\om)+\f{[\pa_{y_c}^2,\Lambda_{3,1}]\om}{(\rmA_1^2+\rmB_1^2)^{\f12}}\\
&=I_1'+I_2'+I_3'+I_4'+I_5'+I_6'+I_7'+I_8'+I_9'+I_{10}'+I_{11}'.
\end{align*}

By Proposition \ref{Prop: A_1^2+B_1^2} and Lemma \ref{lem: commu_II_1,1}, we get
\beno
\|\sin y_cI_1'\|_{L^2}\leq C\left\|\f{\rho u''}{(\rmA_1^2+\rmB_1^2)^{\f12}}\right\|_{L^{\infty}}\big\|\sin y_c[\pa_{y_c}^2,\rho\mathrm{II}_{1,1}]\om\big\|_{L^2}\leq C(\al\|\om\|_{L^2}+\|\pa_y\om\|_{L^2}),
\eeno
and by Proposition \ref{Prop: A_1^2+B_1^2} and Lemma \ref{Lem: II_1,1L^2},
\beno
&&\|\sin y_cI_2'\|_{L^2}\leq C\left\|\sin y_c\pa_{y_c}\Big(\f{\rho u''}{(\rmA_1^2+\rmB_1^2)^{\f12}}\Big)\right\|_{L^{\infty}}\big\|\pa_{y_c}\big(\rho\mathrm{II}_{1,1}(\om)\big)\big\|_{L^2}\leq C\|\om\|_{H^1},\\
&&\|\sin y_cI_3'\|_{L^2}\leq C\left\|\sin y_c\pa_{y_c}^2\Big(\f{\rho u''}{(\rmA_1^2+\rmB_1^2)^{\f12}}\Big)\right\|_{L^{\infty}}\big\|\rho\mathrm{II}_{1,1}(\om)\big\|_{L^2}\leq C\al\|\om\|_{L^2}.
\eeno
By Proposition \ref{Prop: A_1^2+B_1^2} and Lemma \ref{lem: commL_0}, we get
\beno
\|\sin y_cI_4'\|_{L^2}\leq C\left\|\f{\rho u''}{(\rmA_1^2+\rmB_1^2)^{\f12}}\right\|_{L^2}\|\sin^3 y_c[\pa_{y_c}^2,\mathcal{L}_{0}]\om\|_{L^{\infty}}\leq C(\al\|\om\|_{L^2}+\|\pa_y\om\|_{L^2}),
\eeno
and by Proposition \ref{Prop: A_1^2+B_1^2} and Lemma \ref{Lem:H^1II_1,2_periodic},
\begin{align*}
\|\sin y_cI_5'\|_{L^2}&\leq C\left\|\f{1}{\sin y_c}\pa_{y_c}\Big(\f{\rho^2 u''}{(\rmA_1^2+\rmB_1^2)^{\f12}}\Big)\right\|_{L^2}\big\|\sin^2 y_c\pa_{y_c}\big(\mathcal{L}_{0}(\om)\big)\|_{L^{\infty}}\\
&\leq C\big(\al^\f12\|\om\|_{L^\infty}+\|\om'\|_{L^2}\big)\le C\big(\al\|\om\|_{L^2}+\|\om'\|_{L^2}\big),
\end{align*}
and by Proposition \ref{Prop: A_1^2+B_1^2} and Lemma \ref{Lem:L^2L_k},
\begin{align*}
\|\sin y_cI_6'\|_{L^2}&\leq C\left\|\f{1}{\sin y_c}\pa_{y_c}^2\Big(\f{\rho^2 u''}{(\rmA_1^2+\rmB_1^2)^{\f12}}\Big)\right\|_{L^2}\big\|\rho\mathcal{L}_{0}(\om)\big\|_{L^{\infty}}\\
&\leq C\al\|\om\|_{L^2}.
\end{align*}
By Proposition \ref{Prop: A_1^2+B_1^2} and Lemma \ref{Lem: II}, we get
\beno
&&\|\sin y_cI_7'\|_{L^2}\leq C\|\om\|_{L^2}\left\|\sin y_c\pa_{y_c}^2\Big(\f{\rho^2 u'\mathrm{II}}{(\rmA_1^2+\rmB_1^2)^{\f12}}\Big)\right\|_{L^{\infty}}\leq C\al\|\om\|_{L^2},\\
&&\|\sin y_cI_8'\|_{L^2}\leq C\|\pa_y\om\|_{L^2}\left\|\sin y_c\pa_{y_c}\Big(\f{\rho^2 u'\mathrm{II}}{(\rmA_1^2+\rmB_1^2)^{\f12}}\Big)\right\|_{L^{\infty}}\leq C\|\pa_y\om\|_{L^2}.
\eeno
By Proposition \ref{Prop: A_1^2+B_1^2} and Remark \ref{Rmk:difference1/2}, we get
\beno
&&\sin y_cI_9'=\frac{(\cL^2+\al^{\f12}\cL^{\infty})}{(1+\al\sin y_c)^3}(\al\|\om\|_{L^2}+\|\pa_y\om\|_{L^2})=\cL^2,\\
&&\sin y_cI_{10}'=\frac{(\cL^2+\al^{\f12}\cL^{\infty})}{(1+\al\sin y_c)^3}(\al\|\om\|_{L^2}+\|\pa_y\om\|_{L^2})=\cL^2.
\eeno
By Proposition \ref{Prop: A_1^2+B_1^2} and Lemma \ref{Lem:commLambda_3,1}, we get
\beno
\sin y_cI_{11}'=\frac{(1+\al^2\rho)(\cL^2+\al^{\f12}\cL^{\infty})}{(1+\al\sin y_c)^3}(\al\|\om\|_{L^2}+\|\pa_y\om\|_{L^2})=\cL^2.
\eeno
Summing up, we conclude the even case.
\end{proof}

\section{Linear enhanced dissipation}\label{Enhanced}

\subsection{Decay estimates on the model without nonlocal term}

In this subsection, we consider a toy model without nonlocal term:
\ben\label{eq:Toymodel1}
\pa_t\omega-\nu \pa_y^2\omega-ia e^{-\nu t}(\cos y)\omega=0.
\een

\begin{lemma}\label{Lem:toymodel1}
Let $ 0<\nu<|a|,\ a\in\R$, and $\omega(t,y)$ solve \eqref{eq:Toymodel1} for $ t\geq0,\ y\in\bbT$. Then $\|\omega(t)\|_{L^2}\leq \|\omega(0)\|_{L^2}$ and
\beno
\|\sin y\omega(t)\|_{L^2}\leq C(\nu t^3a^2)^{-\frac{1}{2}}\|\omega(0)\|_{L^2}
\eeno
for $0<t\leq(\nu|a|)^{-\frac{1}{2}}$.
\end{lemma}
\begin{proof}
First of all, we have
\begin{align*}
\partial_t\|\omega\|_{L^2}^2=2{\rm Re}\langle \partial_t\omega,\omega\rangle=2{\rm Re}\langle \nu \pa_y^2\omega+ia e^{-\nu t}\cos y\omega,\omega\rangle=-2\nu\|\pa_y\omega\|_{L^2}^2\leq 0,
\end{align*}
which gives $\|\omega(t)\|_{L^2}\leq \|\omega(0)\|_{L^2}$.

Let $u(y)=-\cos y,\ \gamma(t,s)=a\f{e^{-\nu s}-e^{-\nu t}}{\nu}$. Then $\gamma(s,s)=0,$ $\partial_t\gamma=ae^{-\nu t}$, and
\beno
\partial_t(e^{i\gamma u}\omega)
=e^{i\gamma u}(\om_t+iu\gamma_t\om)
=\nu e^{i\gamma u}\pa_y^2\omega,
\eeno
from which, we infer that
\begin{align*}
\partial_t\|\pa_y(e^{i\gamma u}\omega)\|_{L^2}^2
&=2{\rm Re}\langle \partial_t\pa_y(e^{i\gamma u}\omega),\pa_y(e^{i\gamma u}\omega)\rangle\\
&=-2{\rm Re}\langle \partial_t(e^{i\gamma u}\omega),\pa_y^2(e^{i\gamma u}\omega)\rangle\\
&=-2{\rm Re}\langle \nu e^{i\gamma u}\pa_y^2\omega,\pa_y^2(e^{i\gamma u}\omega)\rangle\\
&=-2\nu {\rm Re}\langle  \pa_y^2\omega,e^{-i\gamma u}\pa_y^2(e^{i\gamma u}\omega)\rangle\\
&=-2\nu {\rm Re}\langle  \pa_y^2\omega,\pa_y^2\omega+2i\gamma u'\pa_y\omega+(i\gamma u''-\gamma^2 u'^2)\omega\rangle\\
&=-2\nu \|\pa_y^2\omega+i\gamma u'\pa_y\omega+{i}\gamma u''\omega/2\|_{L^2}^2
+2\nu\gamma^2 \| u'\pa_y\omega+ u''\omega/2\|_{L^2}^2\\
&\quad +2\nu\gamma^2{\rm Re}\langle  \pa_y^2\omega, u'^2\omega\rangle.
\end{align*}
For the last term, we get by integration by parts that
\begin{align*}
{\rm Re}\langle  \pa_y^2\omega, u'^2\omega\rangle
&=-{\rm Re}\langle  \pa_y\omega, \pa_y(u'^2\omega)\rangle\\
&=-{\rm Re}\langle  \pa_y\omega, u'^2\pa_y\omega+2u'u''\omega\rangle\\
&=-{\rm Re}\langle  u'\pa_y\omega, u'\pa_y\omega+2u''\omega\rangle\\
&=-\| u'\pa_y\omega+ u''\omega/2\|_{L^2}^2+\| u''\omega/2\|_{L^2}^2-{\rm Re}\langle  u'\pa_y\omega, u''\omega\rangle,
\end{align*}
where
\begin{align*}
-2{\rm Re}\langle  u'\pa_y\omega, u''\omega\rangle
&=-\int_{-\pi}^{\pi}u'u''(\pa_y\omega\overline{\omega}+\pa_y\overline{\omega}{\omega})dy\\
&=-\int_{-\pi}^{\pi}u'u''\pa_y|\omega|^2dy=\int_{-\pi}^{\pi}(u'u'')'|\omega|^2dy.
\end{align*}
Putting these identities above together and using the fact that
\beno
(u'')^2/2+(u'u'')'=3(u'')^2/2+u'u'''=3 (\cos y)^2/2-(\sin y)^2\leq 3/2,
\eeno
we deduce that
\begin{align*}
\partial_t\|\pa_y(e^{i\gamma u}\omega)\|_{L^2}^2
&\leq 2\nu\gamma^2\big(\| u''\omega/2\|_{L^2}^2-{\rm Re}\langle  u'\pa_y\omega, u''\omega\rangle\big)\\
&=\nu\gamma^2\int_{-\pi}^{\pi}\big((u'')^2/2+(u'u'')'\big)|\omega|^2dy\leq \frac{3}{2}\nu\gamma^2\|\omega\|_{L^2}^2.
\end{align*}
 Therefore, for $0<s<t $,
\begin{align}\label{eq: pa_yom1}
\|\pa_y(e^{i\gamma u}\omega)(t,s)\|_{L^2}^2
&\leq \|\pa_y(e^{i\gamma u}\omega)(s,s)\|_{L^2}^2
+\frac{3}{2}\nu\int_s^t\gamma(\tau,s)^2\|\omega(\tau)\|_{L^2}^2d\tau\\ \nonumber
&\leq\|\pa_y\omega(s)\|_{L^2}^2
+\frac{3}{2}\nu\|\omega(0)\|_{L^2}^2\int_s^t\gamma(\tau,s)^2d\tau.
\end{align}
Thanks to $\pa_y(e^{i\gamma u}\omega)=e^{i\gamma u}(i\gamma u'\omega+\pa_y\omega), $ $\|\pa_y(e^{i\gamma u}\omega)\|_{L^2}^2=\|i\gamma u'\omega+\pa_y\omega\|_{L^2}^2 $. Then by \eqref{eq: pa_yom1}, we get
\begin{align*}
\int_0^t\|i\gamma(t,s) u'\omega(t)+\pa_y\omega(t)\|_{L^2}^2ds
&=\int_0^t\|\pa_y(e^{i\gamma u}\omega)(t,s)\|_{L^2}^2ds\\
&\leq \int_0^t\|\pa_y\omega(s)\|_{L^2}^2ds+ \frac{3}{2}\nu\|\omega(0)\|_{L^2}^2\int_0^t\int_s^t\gamma(\tau,s)^2d\tau ds.
\end{align*}

On the other hand, let
\beno
\gamma_1(t)=\int_0^t\gamma(t,s)ds,\quad \gamma_2(t)=\int_0^t\gamma(t,s)^2ds,\quad \gamma_0(t)=\gamma_2(t)-\gamma_1(t)^2/t.
\eeno
We find that
\begin{align*}
&\int_0^t\|i\gamma(t,s) u'\omega(t)+\pa_y\omega(t)\|_{L^2}^2ds\\
&=\gamma_2(t)\|u'\omega(t)\|_{L^2}^2+2\gamma_1(t){\rm Re}\langle iu'\omega(t),\pa_y\omega(t)\rangle+t\|\pa_y\omega(t)\|_{L^2}^2\\
&=\gamma_0(t)\|u'\omega(t)\|_{L^2}^2+t\|i\gamma_1(t) u'\omega(t)/t+\pa_y\omega(t)\|_{L^2}^2.
\end{align*}
Moreover,
\begin{align*}
t\gamma_0(t)&
=t\gamma_2(t)-\gamma_1(t)^2\\
&=\int_0^t\int_0^t\gamma(t,s)^2dsd\tau-\int_0^t\int_0^t\gamma(t,s)\gamma(t,\tau)dsd\tau\\
&=\int_0^t\int_0^t(\gamma(t,s)^2/2+\gamma(t,\tau)^2/2-\gamma(t,s)\gamma(t,\tau))dsd\tau\\
&=\frac{1}{2}\int_0^t\int_0^t(\gamma(t,s)-\gamma(t,\tau))^2dsd\tau\\
&=\frac{1}{2}\int_0^t\int_0^t\gamma(\tau,s)^2dsd\tau
=\int_0^t\int_s^t\gamma(\tau,s)^2d\tau ds,
\end{align*}
and
\beno
 \int_0^t\|\pa_y\omega(s)\|_{L^2}^2ds=\frac{1}{2\nu}\big(\|\omega(0)\|_{L^2}^2-\|\omega(t)\|_{L^2}^2\big).
\eeno
Therefore, we obtain
\beno
\gamma_0(t)\|u'\omega(t)\|_{L^2}^2\leq \frac{1}{2\nu}(\|\omega(0)\|_{L^2}^2-\|\omega(t)\|_{L^2}^2)+\frac{3}{2}\nu t\gamma_0(t)\|\omega(0)\|_{L^2}^2,
\eeno
which gives
\beno
\|u'\omega(t)\|_{L^2}^2\leq \left({1}/({2\nu\gamma_0(t)})+{3\nu t}/{2}\right)\|\omega(0)\|_{L^2}^2.
\eeno

Now, we give a lower bound of $\gamma_0(t). $ As $\gamma(\tau,s)=\int_s^{\tau}ae^{-\nu t'} dt'\geq a(\tau-s)e^{-\nu t}$ for $0<s<\tau<t,$ we have
\beno
t\gamma_0(t)
=\int_0^t\int_s^t\gamma(\tau,s)^2d\tau ds\geq \int_0^t\int_s^ta^2(\tau-s)^2e^{-2\nu t}d\tau ds=a^2e^{-2\nu t}\frac{t^4}{12},
\eeno
which gives
\beno
\frac{1}{2\nu\gamma_0(t)}\leq \frac{1}{2\nu a^2}e^{2\nu t}\frac{12}{t^3}= \frac{6e^{2\nu t}}{\nu a^2t^3}.
\eeno
For $0<t\leq (\nu|a|)^{-\frac{1}{2}},\ 0<\nu<|a|,$ we have $\nu t\leq 1,\ \nu t\leq 1/(\nu a^2t^3). $ Then we conclude that
\beno
\|u'\omega(t)\|_{L^2}^2\leq (6e^2+2)\|\omega(0)\|_{L^2}^2/(\nu a^2t^3).
\eeno
This completes the proof.
\end{proof}

Now, we deal with the inhomogeneous equation.

\begin{lemma}\label{Lem:Inhomtoymodel}
Let $0<\nu<|a|,\ a\in\R$, and $\omega(t,y)$ solve
\ben\label{eq:toymodel_2}
\pa_t\omega-\nu \pa_y^2\omega-ia e^{-\nu t}(\cos y)\omega=f.
\een
Then we have
\beno
\|\sin y\omega(t)\|_{L^2}\leq \frac{C}{\sqrt{\nu t^3a^2}}\sup_{0\leq s\leq t}\|\omega(s)\|_{L^2}+\int_0^t\|\sin yf(s)\|_{L^2}ds
\eeno
for $0<t\leq(\nu|a|)^{-\frac{1}{2}}$.
\end{lemma}

\begin{proof}
We  decompose $ \omega=\omega_1+\omega_2$, where
\begin{align*}
&\pa_t\omega_1-\nu \pa_y^2\omega_1-ia e^{-\nu t}(\cos y)\omega_1=0,\ \ \omega_1(0)=\omega(0),\\
&\pa_t\omega_2-\nu \pa_y^2\omega_2-ia e^{-\nu t}(\cos y)\omega_2=f,\ \ \omega_2(0)=0.
\end{align*}

By Lemma \ref{Lem:toymodel1}, we have $\|\omega_1(t)\|_{L^2}\leq \|\omega_1(0)\|_{L^2}=\|\omega(0)\|_{L^2}$, and
\beno
\|\sin y\omega_1(t)\|_{L^2}\leq C(\nu t^3a^2)^{-\frac{1}{2}}\|\omega_1(0)\|_{L^2}=C(\nu t^3a^2)^{-\frac{1}{2}}\|\omega(0)\|_{L^2}
\eeno
 for $0<t\leq(\nu|a|)^{-\frac{1}{2}}$. Let $A(t)=\sup\limits_{0\leq s\leq t}\|\omega(s)\|_{L^2}$.  Then
 \beno
 \|\omega_2(t)\|_{L^2}\leq \|\omega_1(t)\|_{L^2}+\|\omega(t)\|_{L^2}\leq 2A(t).
\eeno
For any $b>0$, we have
\begin{align*}
\f{d}{dt}\|\sin y\omega_2\|_{L^2}^2
&=2{\rm Re}\langle \sin y\partial_t\omega_2,\sin y\omega_2\rangle\\
&=2{\rm Re}\langle \sin y(\nu \pa_y^2\omega_2+ia e^{-\nu t}\cos y\omega_2+f),\sin y\ \omega_2\rangle\\
&=-2\nu\|\sin y\ \pa_y\omega_2+\cos y\ \omega_2\|_{L^2}^2+2\nu\|\cos y\ \omega_2\|_{L^2}^2+2{\rm Re}\langle \sin yf,\sin y\ \omega_2\rangle\\
&\leq 2\nu\|\omega_2\|_{L^2}^2+2\|\sin yf\|_{L^2}\|\sin y\omega_2\|_{L^2}\\
&\leq 2\big(\|\sin yf\|_{L^2}+4\nu A(t)^2/b\big)\big(\|\sin y\omega_2\|_{L^2}^2+b^2\big)^{\frac{1}{2}},
\end{align*}
therefore,
\beno
\f{d}{dt}\big(\|\sin y\omega_2\|_{L^2}^2+b^2\big)^{\frac{1}{2}}\leq \|\sin yf\|_{L^2}+4\nu A(t)^2/b,
\eeno
which implies that
\begin{align*}
\|\sin y\omega_2(t)\|_{L^2}
&\leq (\|\sin y\omega_2(t)\|_{L^2}^2+b^2)^{\frac{1}{2}}\\
&\leq b+\int_0^t(\|\sin yf(s)\|_{L^2}+4\nu A(s)^2/b)ds\\
&\leq b+4\nu t A(t)^2/b+\int_0^t\|\sin yf(s)\|_{L^2}ds.
\end{align*}
As it is true for any $b>0$, we minimize the right hand side to obtain
\beno
\|\sin y\omega_2(t)\|_{L^2}\leq 4\sqrt{\nu t} A(t)+\int_0^t\|\sin yf(s)\|_{L^2}ds.
\eeno
For $0<t\leq(\nu|a|)^{-\frac{1}{2}}$, we have $\sqrt{\nu t}\leq (\nu t^3a^2)^{-\frac{1}{2}}$. Then
\begin{align*}
\|\sin y\omega(t)\|_{L^2}
&\leq\|\sin y\omega_1(t)\|_{L^2}+\|\sin y\omega_2(t)\|_{L^2}\\ &\leq C(\nu t^3a^2)^{-\frac{1}{2}}\|\omega(0)\|_{L^2}
+4\sqrt{\nu t} A(t)+\int_0^t\|\sin yf(s)\|_{L^2}ds \\
&\leq C(\nu t^3a^2)^{-\frac{1}{2}}A(t)+\int_0^t\|\sin yf(s)\|_{L^2}ds.
\end{align*}
This gives our result.
\end{proof}

\subsection{Decay estimates on the model with nonlocal term}

\begin{lemma}
\label{lem: corss_term}
Let $ 0<\nu<|a|,\ |\al|>1,\ a,\al\in\R,$ and $\omega(t,y)$ solve
\ben\label{eq:toymodel-3}
\pa_t\omega-\nu \pa_y^2\omega-ia e^{-\nu t}\cos y(\omega-\psi)=0,\quad
-(\pa_y^2-\al^2)\psi=\omega.
\een
Then for $0<t\leq(\nu|a|)^{-\frac{1}{2}},\ \nu\al^2t<1,$
$$-\langle\psi(t),\omega(t)-\psi(t)\rangle\leq \f{C}{\nu t^3a^2}
\langle\omega(0),\omega(0)-\psi(0)\rangle.$$
\end{lemma}
\begin{proof}
By \eqref{eq:toymodel-3}, we have
\begin{align*}
\partial_t\langle\omega,\omega-\psi\rangle
&=\langle\omega_t,\omega-\psi\rangle
+\langle\omega,\omega_t\rangle
+\langle(\pa_y^2-\al^2)\psi,\psi_t\rangle\\
&=\langle\omega_t,\omega-\psi\rangle
+\langle\omega,\omega_t\rangle
-\langle\psi,\om_t\rangle\\
&=-2\nu\langle\pa_y\omega,\pa_y\omega-\pa_y\psi\rangle.
\end{align*}
By Proposition \ref{Prop: oper_D}, we have $\| \bbD(\om)\|_{L^2}^2=\langle\om,\om-\psi\rangle.$ Using the facts that
\beno
&&\|\pa_y\om\|_{L^2}^2=\|\pa_y^3\psi\|_{L^2}^2+\al^4\|\pa_y\psi\|_{L^2}^2+2\al^2\|\pa_y^2\psi\|_{L^2}^2\geq \al^4\|\pa_y\psi\|_{L^2}^2+\al^2\|\pa_y^2\psi\|_{L^2}^2,\\
&&-\langle \pa_y\om,\pa_y\psi\rangle=-\|\pa_y^2\psi\|_{L^2}^2-\al^2\|\pa_y\psi\|_{L^2}^2,
\eeno
we deduce that
\begin{align}\label{eq: part_of_K_2}
\langle\pa_y\omega,\pa_y\omega-\pa_y\psi\rangle
=\|\pa_y\omega\|_{L^2}^2-\langle \pa_y\om,\pa_y\psi\rangle
\geq (1-\al^{-2})\|\pa_y\omega\|_{L^2}^2.
\end{align}
Then we obtain $\pa_t\|\bbD(\om)(t)\|_{L^2}^2\leq 0$, thus  $ \|\bbD(\om)(t)\|_{L^2}\leq \| \bbD(\om)(0)\|_{L^2}$ and
\beno
2\nu\int_0^t \|\pa_y\omega(s)\|_{L^2}^2\leq \frac{\al^2}{\al^2-1}\|\bbD(\om)(0)\|_{L^2}^2.
\eeno

By Proposition \ref{Prop: oper_D}, $\bbD(\om)$ satisfies \eqref{eq:toymodel_2} with $f(s,y)=\nu[\bbD,\pa_y^2]\om$. Thus by Lemma \ref{Lem:Inhomtoymodel}, we know that for $0<t\leq (\nu |a|)^{-\f12}$,
\begin{align*}
\|\sin y\bbD(\om)(t)\|_{L^2}
\leq C(\nu t^3a^2)^{-\frac{1}{2}}\|\bbD(\om)(0)\|_{L^2}+\int_0^t\nu\|\sin y[\bbD,\pa_y^2]\omega(s)\|_{L^2}ds.\end{align*}
Again by Proposition \ref{Prop: oper_D}, we have \begin{align*}
\int_0^t\nu\|\sin y\ [\bbD,\pa_y^2]\omega(s)\|_{L^2}ds
&\leq \int_0^t\nu C(|\al|\|\om(s)\|_{L^2}+\|\pa_y\omega(s)\|_{L^2})ds\\&\leq \int_0^t\nu C|\al|\sqrt{\f{\al^2}{\al^2-1}}\|\bbD(\om)(s)\|_{L^2}ds+C\nu t^{\f12}\|\pa_y\omega\|_{L^2_t(L^2)}\\&\leq C\nu t|\al|\sqrt{\f{\al^2}{\al^2-1}}\|\bbD(\om)(0)\|_{L^2}
+C\nu t^{\f12}\sqrt{\f{\al^2/(2\nu)}{\al^2-1}}\|\bbD(\om)(0)\|_{L^2}.
\end{align*}
Therefore, for $0<t\leq (\nu |a|)^{-\f12},$
\beno
\|\sin y\bbD(\om)(t)\|_{L^2}\leq C\big((\nu t^3a^2)^{-\frac{1}{2}}+\nu t|\al|+(\nu t)^{\f12}\big)\|\bbD(\om)(0)\|_{L^2}.
\eeno
Thus, for $0<t\leq (\nu |a|)^{-\f12},\ \nu \al^2t< 1$, we have $ \nu t|\al|\leq(\nu t)^{\f12}\leq(\nu t^3a^2)^{-\frac{1}{2}}$. As $-\langle\psi,\omega\rangle=-\|\pa_y\psi\|_{L^2}^2-\al^2\|\psi\|_{L^2}^2$, then by Proposition \ref{Prop: oper_D},
\begin{align}\label{eq: K_3}
\langle\psi,\omega-\psi\rangle(t)
&=(\al^2-1)\|\psi\|_{L^2}^2+\|\pa_y\psi\|_{L^2}^2\\\nonumber
&\leq \|\sin y\ \bbD(\om)(t)\|_{L^2}^2\\\nonumber
&\leq C\|\bbD(\om)(0)\|_{L^2}^2/(\nu t^3a^2)
=C\langle\omega,\omega-\psi\rangle(0)/(\nu t^3a^2).\nonumber
\end{align}
The proof is completed.
\end{proof}

Let $-(\pa_y^2-\al^2)\widehat{\psi}=\widehat{\om}$.  We introduce
\beno
&&K_1(t,\al)=\langle\widehat{\om},\widehat{\om}-\widehat{\psi}\rangle(t,\al),\\
&&K_2(t,\al)=-\langle(\pa_y^2-\al^2)\widehat{\omega},\widehat{\omega}-\widehat{\psi}\rangle(t,\al),\\
&&K_3(t,\al)=\langle\widehat{\psi},\widehat{\omega}-\widehat{\psi}\rangle(t,\al).
\eeno

\begin{lemma}\label{lem:K-relation}
It holds that
\begin{align*}
&K_1=\|\bbD(\widehat{\om})\|_{L^2}^2\geq \al^2K_3,\\
&K_2=(\al^2-1)\|\widehat{\om}\|_{L^2}^2+\|\pa_y\widehat{\om}\|_{L^2}^2\geq \al^2K_1,\\
&K_3=(\al^2-1)\|\widehat{\psi}\|_{L^2}^2+\|\pa_y\widehat{\psi}\|_{L^2}^2\geq 0,\\
&K_2K_3\geq K_1^2.
\end{align*}
\end{lemma}

\begin{proof}
By \eqref{eq: K_3} and Proposition \ref{Prop: oper_D}, we get \begin{align*}
\|\bbD(\widehat{\om})\|_{L^2}^2
&=K_1\geq (1-\al^{-2})\|\widehat{\om}\|_{L^2}^2\\
&=(1-\al^{-2})(\|\pa_y^2\widehat{\psi}\|_{L^2}^2+\al^2\|\pa_y\widehat{\psi}\|_{L^2}^2+\al^4\|\widehat{\psi}\|_{L^2}^2)\\
&\geq \al^2\big((\al^2-1)\|\widehat{\psi}\|_{L^2}^2+\|\pa_y\widehat{\psi}\|_{L^2}^2\big)
= \al^2K_3\geq 0 .
\end{align*}
It is easy to check that
\beno
K_2=-\langle(\pa_y^2-\al^2)\widehat{\om},\widehat{\om}-\widehat{\psi}\rangle
=\langle\pa_y\widehat{\om},\pa_y\widehat{\om}-\pa_y\widehat{\psi}\rangle+\al^2\langle\widehat{\om},\widehat{\om}-\widehat{\psi}\rangle,
\eeno
which along with \eqref{eq: part_of_K_2} gives $K_2\geq \al^2K_1\geq 0$.

Using the facts that
\beno
&&\|\widehat{\om}\|_{L^2}^2=\|\pa_y^2\widehat{\psi}\|_{L^2}^2+
\al^2\|\pa_y\widehat{\psi}\|_{L^2}^2+\al^4\|\widehat{\psi}\|_{L^2}^2,\\
&&\|\pa_y\widehat{\om}\|_{L^2}^2=\|\pa_y^3\widehat{\psi}\|_{L^2}^2+\al^2\|\pa_y^2\widehat{\psi}\|_{L^2}^2+\al^4\|\pa_y\widehat{\psi}\|_{L^2}^2,
\eeno
we can deduce that
\begin{align*}
K_2K_3&=\big((\al^2-1)\|\widehat{\om}\|_{L^2}^2+\|\pa_y\widehat{\om}\|_{L^2}^2\big)
\big((\al^2-1)\|\widehat{\psi}\|_{L^2}^2+\|\pa_y\widehat{\psi}\|_{L^2}^2\big)\\
&\ge\big((\al^2-1)\langle\widehat{\omega},\widehat{\psi}\rangle+\langle\pa_y\widehat{\omega},\pa_y\widehat{\psi}\rangle\big)^2
=\big(\langle\widehat{\omega},(\al^2-1)\widehat{\psi}-\pa_y^2\widehat{\psi}\rangle\big)^2\\
&= \langle\widehat{\omega},\widehat{\omega}-\widehat{\psi}\rangle^2=K_1^2.
\end{align*}
\end{proof}

\subsection{Proof of Theorem \ref{Thm:enhance}}

\begin{proof}
We first consider the case of $\delta<1$. Taking Fourier transform in $x$ to \eqref{eq:NS-L}, we obtain
\beno
\pa_t\widehat{\omega}+\mathcal{L}_{\nu}(\al,t)\widehat{\omega}=0,
\eeno
where
\beno
\mathcal{L}_{\nu}(\al,t)=\nu(-\pa_y^2+\al^2)-i\al a_0e^{-\nu t}\cos y\big(1+(\pa_y^2-\al^2)^{-1}\big).
\eeno
Let $-(\pa_y^2-\al^2)\widehat{\psi}=\widehat{\om}$.  Then we have
\begin{align*}
\partial_t\langle\widehat{\omega},\widehat{\omega}-\widehat{\psi}\rangle
&=\langle\widehat{\omega}_t,\widehat{\omega}\rangle
+\langle\widehat{\omega},\widehat{\omega}_t\rangle
-\langle\widehat{\psi},\widehat{\om}_t\rangle
-\langle\widehat{\om}_t, \widehat{\psi}\rangle\\
&=-2\nu\langle\pa_y\widehat{\omega},\pa_y\widehat{\omega}-\pa_y\widehat{\psi}\rangle
-2\nu\al^2\langle\widehat{\omega},\widehat{\omega}-\widehat{\psi}\rangle.
\end{align*}
This gives by Lemma \ref{lem:K-relation} that
\beno
\partial_t K_1=-2\nu K_2\leq -2\nu\al^2K_1,
\eeno
which implies that
\ben\label{eq:K_1K_2rough}
K_1(t,\al)\leq e^{-2\nu\al^2t}K_1(0,\al), \quad \int_{0}^T2\nu K_2(s,\al)ds=K_1(0,\al)-K_1(T,\al).
\een

Now fix $\al,$ consider the following cases.\smallskip

{\bf Case 1.} $\nu\geq a_0|\al|e^{-\tau}$ or $\nu|\al|^3\geq1. $\smallskip

Due to $|\al|>1,$ we have $ \nu^{\frac{1}{3}}+\nu^{\frac{1}{2}}\leq C\nu\al^2$. Then by Lemma \ref{lem:K-relation} and  \eqref{eq:K_1K_2rough}, we get
\beno
\al^2K_3(t,\al)\leq K_1(t,\al)\leq e^{-2\nu\al^2t}K_1(0,\al)\leq \f{Ce^{-\nu\al^2t}}{1+\nu t^3}K_1(0,\al)\le \f{Ce^{-c\sqrt{\nu}t}}{1+\nu t^3}K_1(0,\al).
\eeno

{\bf Case 2:} $\nu<a_0|\al|e^{-\tau}$ and $\nu|\al|^3<1. $ \smallskip

Take $ t_1=(\nu|\al|)^{-\frac{1}{2}},$ then $\nu t_1\al^2<1.$
Fix $0\leq t\leq \tau/\nu $, then $\big(e^{\nu\al^2s}\widehat{\omega}(t+s,\al,y),\ e^{\nu\al^2s}\widehat{\psi}(t+s,\al,y)\big) $ solves \eqref{eq:toymodel-3} with $a=a_0\al e^{-\nu t}.$
By our assumption, $0<\nu<|a|$. Then by Lemma \ref{lem: corss_term}, we obtain
\begin{align*}
&e^{2\nu\al^2t_1}K_3(t+t_1,\al)\\
&=\langle e^{\nu\al^2t_1}\widehat{\psi}(t+t_1,\al),e^{\nu\al^2t_1}\widehat{\omega}(t+t_1,\al)-e^{\nu\al^2s}\widehat{\psi}(t+t_1,\al)\rangle\\
&\leq \f{C}{\nu t_1^3a^2}\langle e^{\nu\al^2t_1}\widehat{\omega}(t,\al),e^{\nu\al^2t_1}\widehat{\omega}(t,\al)-e^{\nu\al^2t_1}\widehat{\psi}(t,\al)\rangle\\
&\leq \f{C}{\nu t_1^3a^2}K_1(t,\al)
\leq \f{C}{\nu t_1^3\al^2}K_1(t,\al)=CK_1(t,\al)t_1\nu.
\end{align*}
Thus, there exists an absolute constant $c_0\in(0,1)$  so that
\ben
c_0e^{c_0}K_3(t+t_1,\al)\leq K_1(t,\al)t_1\nu.\label{eq:K3-K1}
\een
Set $M=\max\limits_{0\leq t\leq \tau/\nu}e^{c_0t/t_1}K_1(t,\al)$. Then there exists $t_2\in[0,\tau/\nu]$ so that $e^{c_0t_2/t_1}K_1(t_2,\al)=M. $ Next we will show that $t_2\leq t_1$, otherwise $K_1=0$. If $t_2>t_1$, then $ \partial_t(e^{c_0t/t_1}K_1(t,\al))|_{t=t_2}\geq 0$.
Since
\begin{align*}
\partial_t(e^{c_0t/t_1}K_1(t,\al))
&=e^{c_0t/t_1}\big(c_0/t_1K_1(t,\al)+\partial_tK_1(t,\al)\big)\\
&=e^{c_0t/t_1}\big((c_0/t_1)K_1(t,\al)-2\nu K_2(t,\al)\big) ,
\end{align*}
we obtain
\ben
K_2(t_2,\al)\leq  \f{c_0}{2\nu t_1}K_1(t_2,\al).\label{eq:K2-K1}
\een
By \eqref{eq:K3-K1}, we get
\begin{align*}
c_0e^{c_0}K_3(t_2,\al)&\leq K_1(t_2-t_1,\al)t_1\nu\\
&\leq e^{-c_0(t_2-t_1)/t_1}M t_1\nu\\
&=e^{-c_0(t_2-t_1)/t_1}e^{c_0t_2/t_1}K_1(t_2,\al) t_1\nu=e^{c_0}K_1(t_2,\al) t_1\nu,
\end{align*}
which gives
\ben
K_3(t_2,\al)\leq K_1(t_2,\al) t_1\nu/c_0.\label{eq:K3-K1-1}
\een
It follows from \eqref{eq:K2-K1} and \eqref{eq:K3-K1-1} that
\beno
K_2(t_2,\al)K_3(t_2,\al)\leq \f12K_1(t_2,\al)^2.
\eeno
On the other hand,  $K_2K_3\geq K_1^2$ by Lemma \ref{lem:K-relation}. Then we conclude that $K_1(t_2,\al)=0,$ which implies $ M=0.$

Thus, we only need to deal with the case $t_2\leq t_1$. Then
\beno
M=e^{c_0t_2/t_1}K_1(t_2,\al)\leq e^{c_0} K_1(t_2,\al)\leq e^{c_0} K_1(0,\al).
\eeno
Therefore, for $0\leq t\leq \tau/\nu$,
\ben\label{eq:K_1est1}
K_1(t,\al)\leq e^{-c_0t/t_1}M \leq e^{c_0-c_0t/t_1}K_1(0,\al)=e^{c_0-c_0\sqrt{\nu|\al|}t}K_1(0,\al).
\een
This inequality is also true for {\bf Case 1} by \eqref{eq:K_1K_2rough}, since we have $\nu|\al|^3\geq C^{-1} $ or $\nu\al^2\geq C^{-1}\sqrt{\nu|\al|}$ in {\bf Case 1}.

As $\big(e^{\nu\al^2t}\widehat{\omega}(t,\al,y),\ e^{\nu\al^2t}\widehat{\psi}(t,\al,y)\big)$ solves \eqref{eq:toymodel-3} with $a=a_0\al$ and $|\al|>1$, we deduce from Lemma \ref{lem: corss_term} that
$0<t\leq (\nu|\al|)^{-\f12}$ and $\nu\al^2t<1$,
\beno
e^{2\nu\al^2t}K_3(t,\al)\leq \f{C}{\nu t^3a^2} K_1(0,\al)\leq \f{Ce^{-c_0t/(2t_1)}}{\nu t^3\al^2}  K_1(0,\al)
\eeno
for $0\leq t\leq t_1$. For $t_1\leq t\leq \tau/\nu $, by \eqref{eq:K3-K1}, we have
\beno
c_0e^{c_0}K_3(t,\al)\leq K_1(t-t_1,\al)t_1\nu\leq e^{c_0-c_0(t-t_1)/t_1}K_1(0,\al)t_1\nu,
\eeno
which gives
\begin{align*}
\al^2K_3(t,\al)
&\leq c_0^{-1}e^{c_0-c_0t/t_1}K_1(0,\al)t_1\nu\al^2\\
&=c_0^{-1}e^{c_0-c_0t/t_1}K_1(0,\al)t_1\nu/(t_1^2\nu)^2\\
&=c_0^{-1}e^{c_0-c_0t/t_1}K_1(0,\al)/(t_1^3\nu)
\leq Ce^{-c_0t/(2t_1)}K_1(0,\al)/(t^3\nu).
\end{align*}
Therefore, for $0\leq t\leq \tau/\nu$,
\beno
\al^2K_3(t,\al)\leq Ce^{-c_0t/(2t_1)} K_1(0,\al)/(\nu t^3).
\eeno
While by \eqref{eq:K_1est1},
\beno
\al^2K_3(t,\al)\leq K_1(t,\al)\leq e^{c_0-c_0t/t_1}K_1(0,\al).
\eeno
This shows that for $0\leq t\leq \tau/\nu$,
\begin{align*}
\al^2K_3(t,\al)\leq Ce^{-c_0t/(2t_1)} K_1(0,\al)/(1+\nu t^3)=\f{Ce^{-c_0\sqrt{\nu|\al|}t/2}}{1+\nu t^3} K_1(0,\al).
\end{align*}

Combining two cases, we deduce that there exists an absolute constant $c_1\in(0,1)$ so that for $0\leq t\leq \tau/\nu$,
\ben
&&K_1(t,\al)\leq Ce^{-2c_1\sqrt{\nu}t} K_1(0,\al),\\
\label{eq: K_3estK_1}
&&\al^2K_3(t,\al)\leq \f{Ce^{-2c_1\sqrt{\nu}t}}{1+\nu t^3} K_1(0,\al).
\een
By Proposition \ref{Prop: oper_D}, we have  for $0\leq t\leq \tau/\nu,$
\begin{align*}
\|\omega(t)\|_{L^2}^2
&=\sum_{\al\neq 0}\|\widehat{\omega}(t,\al)\|_{L^2}^2
\leq \sum_{\al\neq 0}\frac{\al^2}{\al^2-1}\|\bbD(\widehat{\omega})(t,\al)\|_{L^2}^2\\
&= \sum_{\al\neq 0}\frac{\al^2}{\al^2-1}K_1(t,\al)
\leq C\sum_{\al\neq 0}e^{-2c_1\sqrt{\nu}t}K_1(0,\al)\\
&\leq C\sum_{\al\neq 0}e^{-2c_1\sqrt{\nu}t}\|\widehat{\omega}(0,\al)\|_{L^2}^2\\
&=Ce^{-2c_1\sqrt{\nu}t}\|\omega_0\|_{L^2}^2.
\end{align*}
Using \eqref{eq: K_3}, we have for $0\leq t\leq \tau/\nu,$
\begin{align*}
\|V(t)\|_{\dot{H}_x^1L_y^2}^2
&=\sum_{\al\neq 0}\al^2(\al^2\|\widehat{\psi}(t,\al)\|_{L^2}^2+ \|\partial_y\widehat{\psi}(t,\al)\|_{L^2}^2)\\
&\leq \sum_{\al\neq 0}\frac{\al^4}{\al^2-1}K_3(t,\al)\\
&\leq C\sum_{\al\neq 0}\f{e^{-2c_1\sqrt{\nu}t}}{1+\nu t^3}K_1(0,\al)
\leq \f{Ce^{-2c_1\sqrt{\nu}t}}{1+\nu t^3}\|\omega_0\|_{L^2}^2.
\end{align*}
Therefore for $0\leq t\leq \tau/\nu,$
\beno
\|\omega(t)\|_{L^2}\leq Ce^{-c_1\sqrt{\nu}t}\|\omega_0\|_{L^2},\quad
\|V(t)\|_{\dot{H}_x^1L_y^2}\leq \frac{Ce^{-c_1\sqrt{\nu}t}}{\sqrt{1+\nu t^3}}\|\omega_0\|_{L^2}.
\eeno

For the case of $\d=1$, the proof is almost the same by making the orthogonal projection $P_{|\al|\geq 2}$, since $\mathcal{R}_{\al}$ has no eigenvalue and embedding eigenvalues for $|\al|\geq 2$ and Lemma \ref{lem: corss_term}-Lemma \ref{lem:K-relation}  hold for $|\al|\geq 2$.
\end{proof}

Let us conclude this section by introducing some space-time estimates of
the solution for the linearized Navier-Stokes equations, which will be used  in the proof of nonlinear enhanced dissipation.
\begin{lemma}\label{Lem: VL2Linfty}
Let $\delta<1$.
Under the same assumptions as in Theorem \ref{Thm:enhance}, we have
\begin{align*}
&\int_0^{\tau/\nu}\|\nabla\omega(t)\|_{L^2}^2dt\leq C\nu^{-1}\|\omega_0\|_{L^2}^2,\\
&\int_0^{\tau/\nu}\|\partial_x\omega(t)\|_{L^2}dt\leq C\nu^{-\f23}\|\omega_0\|_{L^2},\\
&\int_0^{\tau/\nu}\|V(t)\|_{L^{\infty}}^2dt\leq C(|\ln \nu|+1)\nu^{-\f13}\|\omega_0\|_{L^2}^2.
\end{align*}
\end{lemma}
\begin{proof}
By \eqref{eq:K_1K_2rough}, we have
\begin{align*}
\int_0^{\tau/\nu}\|\nabla\omega(t)\|_{L^2}^2dt
&=\int_0^{\tau/\nu}\sum_{\al\neq 0}\left(|\al|^2\|\widehat{\omega}(t,\al)\|_{L^2}^2+\|\partial_y\widehat{\omega}(t,\al)\|_{L^2}^2\right)dt
\\
&\leq \int_0^{\tau/\nu}\sum_{\al\neq 0}\frac{\al^2}{\al^2-1}K_2(t,\al)dt\\
&=\frac{1}{2\nu}\sum_{\al\neq 0}\frac{\al^2}{\al^2-1}(K_1(0,\al)-K_1(\tau/\nu,\al))\\
&\leq\frac{C}{\nu}\sum_{\al\neq 0}K_1(0,\al)\leq\frac{C}{\nu}\sum_{\al\neq 0}\|\widehat{\omega}(0,\al)\|_{L^2}^2=\frac{C}{\nu}\|\omega_0\|_{L^2}^2.
\end{align*}

Using the fact that $\pa_tK_1\leq -2\nu\al^2 K_1$, we have
\beno
|\al|^2K_1(t,\al)\leq |\al|^2e^{-2\nu\al^2t}K_1(0,\al)\leq CK_1(0,\al)/(\nu t)
\eeno
and \eqref{eq:K_1est1} implies that
\beno
|\al|^2K_1(t,\al)\leq |\al|^2Ce^{-2c_1\sqrt{\nu|\al|}t}K_1(0,\al)\leq CK_1(0,\al)/(\nu^2 t^4),
\eeno
from which, we infer that
\begin{align*}
\|\partial_x\omega(t)\|_{L^2}^2
&=\sum_{\al\neq 0}|\al|^2\|\widehat{\omega}(t,\al)\|_{L^2}^2
\leq \sum_{\al\neq 0}|\al|^2\frac{\al^2}{\al^2-1}K_1(t,\al)\\
&\leq \sum_{\al\neq 0}CK_1(0,\al)\min\{(\nu t)^{-1},(\nu^2 t^4)^{-1}\}\\
&\leq C\|\omega_0\|_{L^2}^2\min\{(\nu t)^{-1},(\nu^2 t^4)^{-1}\}.
\end{align*}
This gives
\begin{align*}
\int_0^{\tau/\nu}\|\partial_x\omega(t)\|_{L^2}dt
&\leq\int_0^{\tau/\nu}C\|\omega_0\|_{L^2}\min\{(\nu t)^{-\f12},(\nu t^2)^{-1}\}dt
\leq C\nu^{-\f23}\|\omega_0\|_{L^2}.
\end{align*}

By Minkowski inequality, we get
\begin{align*}
\int_0^{\tau/\nu}\|V(t)\|_{L^{\infty}}^2dt
&\leq\int_0^{\tau/\nu}\left(\sum_{\al\neq 0}\|\widehat{V}(t,\al,\cdot)\|_{L^{\infty}}\right)^2dt\leq
\left(\sum_{\al\neq 0}\left(\int_0^{\tau/\nu}\|\widehat{V}(t,\al,\cdot)\|_{L^{\infty}}^2dt\right)^{\f12}\right)^2,
\end{align*}
and by the interpolation,
\begin{align*}
\|\widehat{V}(t,\al,\cdot)\|_{L^{\infty}}^2
\leq& C\|\widehat{V}(t,\al,\cdot)\|_{L^{2}}\|\widehat{V}(t,\al,\cdot)\|_{H^{1}}\\
\leq& C\big(\al^2\|\widehat{\psi}(t,\al,\cdot)\|_{L^2}^2+ \|\partial_y\widehat{\psi}(t,\al,\cdot)\|_{L^2}^2\big)^{\f12}\\
&\quad\times\big(\al^4\|\widehat{\psi}(t,\al,\cdot)\|_{L^2}^2+2\al^2\|\partial_y\widehat{\psi}(t,\al,\cdot)\|_{L^2}^2+ \|\partial_y^2\widehat{\psi}(t,\al,\cdot)\|_{L^2}^2\big)^{\f12}\\
&\leq CK_3(t,\al)^{\f12}K_1(t,\al)^{\f12}\leq
CK_3(t,\al)^{\f34}K_2(t,\al)^{\f14},
\end{align*}
from which and \eqref{eq: K_3estK_1}, \eqref{eq:K_1K_2rough},  we infer that
\begin{align*}
\int_0^{\tau/\nu}\|\widehat{V}(t,\al)\|_{L^{\infty}}^2dt\leq&\int_0^{\tau/\nu}CK_3(t,\al)^{\f12}K_1(t,\al)^{\f12}dt\\
\leq&\int_0^{\tau/\nu}\frac{CK_1(0,\al)}{|\al|\sqrt{1+\nu t^3}}dt\leq\frac{CK_1(0,\al)}{|\al|\nu^{\f13}},
\end{align*}
and
\begin{align*}
\int_0^{\tau/\nu}\|\widehat{V}(t,\al)\|_{L^{\infty}}^2dt&\leq\int_0^{\tau/\nu}CK_3(t,\al)^{\f34}K_2(t,\al)^{\f14}dt
\\
&\leq C\left(\int_0^{\tau/\nu}K_3(t,\al)dt\right)^{\f34}\left(\int_0^{\tau/\nu}K_2(t,\al)dt\right)^{\f14}\\
&\leq C\left(\int_0^{\tau/\nu}\frac{K_1(0,\al)}{\al^2(1+\nu t^3)}dt\right)^{\f34}\left(\frac{K_1(0,\al)-K_1(\tau/\nu,\al)}{2\nu}\right)^{\f14}\\
&\leq C\left(\frac{K_1(0,\al)}{\al^2\nu^{\f13}}\right)^{\f34}\left(\frac{K_1(0,\al)}{2\nu}\right)^{\f14}
\leq\frac{CK_1(0,\al)}{|\al|^{\f32}\nu^{\f12}}.
\end{align*}
Therefore,
\begin{align*}
\int_0^{\tau/\nu}\|V(t)\|_{L^{\infty}}^2dt
&\leq
\left(\sum_{\al\neq 0}\left(\int_0^{\tau/\nu}\|\widehat{V}(t,\al)\|_{L^{\infty}}^2dt\right)^{\f12}\right)^2\\&\leq
\left(\sum_{0<|\al|\leq \nu^{-\f13}}\left(\frac{CK_1(0,\al)}{|\al|\nu^{\f13}}\right)^{\f12}+\sum_{|\al|> \nu^{-\f13}}\left(\frac{CK_1(0,\al)}{|\al|^{\f32}\nu^{\f12}}\right)^{\f12}\right)^2\\&\leq
C\left(\sum_{\al\neq 0}K_1(0,\al)\right)\left(\sum_{0<|\al|\leq \nu^{-\f13}}\frac{1}{|\al|\nu^{\f13}}+\sum_{|\al|> \nu^{-\f13}}\frac{1}{|\al|^{\f32}\nu^{\f12}}\right)\\&\leq
C\left(\sum_{\al\neq 0}\|\widehat{\omega}(0,\al)\|_{L^2}^2\right)\left(\frac{1+|\ln \nu|}{\nu^{\f13}}+\frac{\nu^{\f16}}{\nu^{\f12}}\right)\\ &\leq C\|\omega_0\|_{L^2}^2(1+|\ln \nu|)/\nu^{\f13},
\end{align*}
which gives the third inequality.
\end{proof}

\section{Nonlinear enhanced dissipation}
\label{Nonlinear}

In this section, we prove Theorem \ref{Thm:NS}. We define  $W_2=\text{span}\{\sin y,\cos y\}$ and the non-shear space
\beno
X=\Big\{\omega\in L^2(\mathbb{T}_{\d}^2):\omega(x,y)=\sum\limits_{\al\neq0}\widehat{\omega}(\al,y)e^{i\al x}\Big\}.
\eeno
We denote by $P_2$ the orthogonal projection to $W_2={\rm span}\{\sin y, \cos y\},$
and $P_{\neq0}$ the orthogonal projection to $X$, and $P_0=I-P_{\neq0}$ the orthogonal projection to the shear space $X_0=\{\omega\in L^2(\mathbb{T}):\partial_x\omega=0\}.$

\subsection{Semigroup estimates}

Let $S(t,s)=S(t,s,a_0)$ be the solution operator of linearized equation \eqref{eq:NS-L}.
That is, $\om(t,y)=S(t,s)f(y)$ solves \eqref{eq:NS-L}  with $\om(s,y)=f(y)$, here we assume $t\geq s$.
Now Theorem \ref{Thm:enhance} and Lemma \ref{Lem: VL2Linfty} can be restated as follows
\begin{align*}
&\|S(t,0)f\|_{L^2}\leq Ce^{-c_1\sqrt{\nu}t}\|f\|_{L^2}\ \ \text{for}\ \ 0<t<\tau/\nu,\\
&\int_0^{\tau/\nu}\|\nabla S(t,0)f\|_{L^2}^2dt\leq C\nu^{-1}\|f\|_{L^2}^2,\\
&\int_0^{\tau/\nu}\|\partial_x S(t,0)f\|_{L^2}dt\leq C\nu^{-\f23}\|f\|_{L^2},\\
&\int_0^{\tau/\nu}\|\nabla\Delta^{-1}S(t,0)f\|_{L^{\infty}}^2dt\leq C(|\ln \nu|+1)\nu^{-\f13}\|f\|_{L^2}^2.
\end{align*}
Notice that $S(t,s,a_0)=S(t-s,0,a_0e^{-\nu s})$ and $a_0e^{-\nu s}$ has uniform positive upper and lower bounds. Then we also have for $0\leq s<t<\tau/\nu$,
\begin{align}
\label{eq: om enhance}
&\|S(t,s)f\|_{L^2}\leq Ce^{-c_1\sqrt{\nu}(t-s)}\|f(s)\|_{L^2},\\ \label{eq: omL^2L^2}
&\int_s^{\tau/\nu}\|\nabla S(t,s)f\|_{L^2}^2dt\leq C\nu^{-1}\|f(s)\|_{L^2}^2,\\
\label{eq: pa_xomL^1L^2}
&\int_s^{\tau/\nu}\|\partial_x S(t,s)f\|_{L^2}dt\leq C\nu^{-\f23}\|f(s)\|_{L^2},\\
\label{eq: VL^2L^infty}
&\int_s^{\tau/\nu}\|\nabla\Delta^{-1}S(t,s)f\|_{L^{\infty}}^2dt\leq C\nu^{-\gamma_1}\|f(s)\|_{L^2}^2,
\end{align}
where $ \gamma_1$ is a constant such that $\f13<\gamma_1<2\gamma-1$.\smallskip

Let  $f$ be the solution of the inhomogeneous linearized Navier-Stokes equations:
\begin{align}\label{eq:inhomLNS}
\pa_tf+\mathcal{L}_{\nu}(t)f=g.
\end{align}
Then we have
\beno
f(t)=S(t,0)f(0)+\int_0^tS(t,s)g(s)ds.
\eeno
Using \eqref{eq: om enhance}-\eqref{eq: VL^2L^infty} and the Minkowski inequality, we deduce that  for $0<t<\tau/\nu, $
\begin{align}
\label{eq: ILNSL^inftyL^2}
&\|f(t)\|_{L^2}\leq Ce^{-c_1\sqrt{\nu}t}\|f(0)\|_{L^2}+C\int_0^t\|g(s)\|_{L^2}ds,\\
\nonumber
&\nu^{\f12}\left(\int_0^{t}\|\nabla f(s)\|_{L^{2}}^2ds\right)^{\f12}+\nu^{\f23}\int_0^{t}\|\pa_x f(s)\|_{L^{2}}ds
+\nu^{\f{\gamma_1}2}\left(\int_0^{t}\|\nabla\Delta^{-1} f(s)\|_{L^{\infty}}^2ds\right)^{\f12} \\
&\leq C\|f(0)\|_{L^2}+C\int_0^t\|g(s)\|_{L^2}ds,\label{eq: ILNAL_t^2}
\end{align}
where $C, c_1$ are constants independent of $\nu, t$.

\subsection{Proof of Theorem \ref{Thm:NS}}
First of all, we have the following energy dissipation law:
\beno
&&\pa_t\|V\|_{L^2}^2=-2\nu\|\na V\|_{L^2}^2,\quad \pa_t\|\om\|_{L^2}^2=-2\nu\|\na \om\|_{L^2}^2,
\eeno
which gives
\begin{align}\label{eq: energy1}
\partial_t(\|\om\|_{L^2}^2-\|\na\psi\|_{L^2}^2)=-2\nu(\|\na \om\|_{L^2}^2-\|\om\|_{L^2}^2),
\end{align}
where $\psi=-\Delta^{-1}\omega.$ The key point is that if $\delta<1$, it holds that for $ \omega\in\{1\}^{\bot},$
\begin{align}\label{eq: uplowI-P_2}
&\|(I-P_2)\omega\|_{L^2}^2\geq \|\omega\|_{L^2}^2-\|\nabla\psi\|_{L^2}^2\geq C_0\|(I-P_2)\omega\|_{L^2}^2,\\
\label{eq: up na(I-P_2)}
&\|\nabla\omega\|_{L^2}^2-\|\omega\|_{L^2}^2\geq C_0\|\nabla(I-P_2)\omega\|_{L^2}^2,
\end{align}
 which can be easily proved by using Fourier transform.\smallskip

By \eqref{eq: energy1}-\eqref{eq: up na(I-P_2)}, we have
\begin{align}\label{eq: I-P_2 om L^2}
\|(I-P_2)\omega(t)\|_{L^2}^2&\leq C(\|\omega(t)\|_{L^2}^2-\|\nabla\psi(t)\|_{L^2}^2)\\ \nonumber
&\leq C(\|\omega(0)\|_{L^2}^2-\|\nabla\psi(0)\|_{L^2}^2)\leq C\|(I-P_2)\omega(0)\|_{L^2}^2\leq C\nu^{2\gamma},
\end{align}
for $0<t<\tau/\nu, $ and
\begin{align}\label{eq: na(I-P_2)om L^2L^2}
\int_0^{\tau/\nu}\|\nabla(I-P_2)\omega(t)\|_{L^2}^2dt
&\leq C\int_0^{\tau/\nu}(\|\nabla\omega(t)\|_{L^2}^2-\|\omega(t)\|_{L^2}^2)dt\\ \nonumber
&\leq C(\|\omega(0)\|_{L^2}^2-\|\nabla\psi(0)\|_{L^2}^2)/(2\nu)\nonumber\\
&\leq C\|(I-P_2)\omega(0)\|_{L^2}^2/(2\nu)
\leq C\nu^{2\gamma-1}.\nonumber
\end{align}

Now, we decompose $\omega=\omega_s+\omega_n,$ where
$\om_s=P_0\omega$ is the shear part and
$\om_n=P_{\neq0}\omega$. Correspondingly, $V=V_s+V_n,\ \psi=\psi_s+\psi_n .$ We denote $V_s=(V^1(t,y),0),\, V_n=(V^1_n,V^2_n).$ Then the equation \eqref{eq:NS-vorticity} can be written as
\begin{align}\label{eq: om_s}
&\partial_t\omega_s=\nu\Delta\omega_s-P_0(V_n\cdot\nabla\omega_n)
=\nu\pa_y^2\omega_s-\partial_yP_0(V_n^2\omega_n),\\
\label{eq: om_n}
&\partial_t\omega_n=\nu\Delta\omega_n-V_s^1\pa_x\om_n-V_n\cdot\nabla\om_s-P_{\neq 0}(V_n\cdot\nabla\omega_n).
\end{align}
For $P_2\omega=P_2\omega_s$, we have
\beno
\partial_tP_2\om
=\nu\Delta P_2\omega-\pa_yP_2(V_n^2\omega_n)=-\nu P_2\omega-\partial_yP_2(V_n^2\omega_n),
\eeno
therefore,
\beno
P_2\omega(t)
=e^{-\nu t} P_2\omega(0)-\int_0^te^{-\nu (t-s)}\partial_yP_2(V_n^2\omega_n)(s)ds,
\eeno
from which and \eqref{eq: I-P_2 om L^2},  we deduce that for $0<t<\tau/\nu,$
\begin{align}\label{eq: P2om est}
\|P_2\omega(t)-e^{-\nu t} P_2\omega(0)\|_{L^2}
&\leq\int_0^t\|\partial_yP_2(V_n^2\omega_n)(s)\|_{L^2}ds\\ \nonumber
&\leq\int_0^tC\|(V_n^2\omega_n)(s)\|_{L^1}ds\leq\int_0^tC\|V_n^2(s)\|_{L^2}\|\omega_n(s)\|_{L^2}ds\\ \nonumber
&\leq \int_0^tC\|\omega_n(s)\|_{L^2}^2ds
\leq \int_0^tC\|(I-P_2)\omega(s)\|_{L^2}^2ds\\ \nonumber
&\leq \int_0^tC\nu^{2\gamma}ds=Ct\nu^{2\gamma}.
\end{align}
In particular, we have
\beno
\|P_2\omega(t)-e^{-\nu t} P_2\omega(0)\|_{L^2}\leq C\tau\nu^{2\gamma-1},
\eeno
Thanks to $C_1^{-1}\leq \|P_2\omega(0)\|_{L^2}\leq C_1$, there exist $c_{2}\in(0,1),C>1$ so that if $0<\nu<c_{2}$, then
\ben\label{eq:p2-uniform}
C^{-1}\leq \|P_2\omega(t)\|_{L^2}\leq C\quad \text{for}\,\, 0<t<\tau/\nu.
\een

Now we choose $t_0\sim \nu^{-\f12}$ so that $Ce^{-c_1\sqrt{\nu}t_0}=\f14, $. Then it suffices to show that for $0\leq t'<t<\tau/\nu$,
\begin{align}\label{eq: induction est1}
\begin{split}
&\|\om_n(t)\|_{L^2}\leq C\|\om_n(t')\|_{L^2}\quad \text{if}\quad t-t'< t_0,\\
&\|\omega_n(t)\|_{L^2}\leq \f12\|\omega_n(t')\|_{L^2}\quad \text{if}\quad t-t'= t_0.
\end{split}
\end{align}
Indeed, Theorem \ref{Thm:NS} follows from \eqref{eq: induction est1} by using the following iteration
\begin{align*}
\|\omega_n(t)\|_{L^2}\leq C\|\omega_n(mt_0)\|_{L^2}\leq C\f{1}{2^m}\|\om_n(0)\|_{L^2}\leq Ce^{-c\sqrt{\nu}t}\|\om_n(0)\|_{L^2},
\end{align*}
with $m=[t/t_0]$ and $t_0\leq C\nu^{-\f12}$.\smallskip

By time translation and  \eqref{eq:p2-uniform}, we only need to prove that\eqref{eq: induction est1} for $t'=0.$ Let $P_2\om(0)=-a_0\sin (y+\th)$, where $a_0>0,\, \th \in [0,2\pi)$. Then $\|P_2\omega(0)\|_{L^2}/\|\sin y\|_{L^2}=a_0. $ By translation, we may assume $\th=0$. Thus,
\beno
P_2\omega(0) =-a_0\sin y,\quad P_2V(0)=(-a_0\cos y,0).
\eeno
Let
\beno
\om(t)=e^{-\nu t} P_2\omega(0)+\widetilde{\omega}(t),\quad
V(t)=e^{-\nu t} P_2V(0)+\widetilde{V}(t).
\eeno
We decompose $\widetilde{\omega}=\widetilde{\omega}_s+{\omega}_n$ and $\widetilde{V}=\widetilde{V}_s+{V}_n,$  where $\widetilde{\omega}_s=P_0\widetilde{\omega}$ and  $\widetilde{V}_s=P_0\widetilde{V}$. Then \eqref{eq: om_n} can be rewritten as
\begin{align}\label{eq: om_n Lform}
\pa_t\om_n+\mathcal{L}_{\nu}(t)\om_n=-\big(\widetilde{V}_s^1\partial_x\om_n+V_n\cdot\nabla\widetilde{\om}_s+P_{\neq 0}(V_n\cdot\nabla\omega_n)\big).
\end{align}
For $0<t\leq \min\{t_0,\tau/\nu\}$, we infer from \eqref{eq: P2om est}  that \begin{align}\label{eq: P_2om_s}
&\|P_2\widetilde{\omega}_s(t)\|_{L^2}=
\|P_2\omega(t)-e^{-\nu t} P_2\omega(0)\|_{L^2}
\leq Ct\nu^{2\gamma}
\leq Ct_0\nu^{2\gamma}\leq C\nu^{2\gamma-\f12}.
\end{align}
Using \eqref{eq: I-P_2 om L^2}, \eqref{eq: P_2om_s} and the fact that $(I-P_2)\widetilde{\om}_s(t)= P_0(I-P_2){\om}(t)$, we deduce that
\begin{align}\label{eq: V_s^1Linfty}
\|\widetilde{V}_s^1(t)\|_{L^{\infty}}
&\leq C\|\widetilde{\om}_s(t)\|_{L^2}
\leq C\|P_2\widetilde{\om}_s(t)\|_{L^2}+C\|(I-P_2){\om}(t)\|_{L^2}\\ \nonumber
&\leq C\nu^{2\gamma-\f12}+C\nu^{\gamma}
\leq C\nu^{\gamma}.
\end{align}
Using \eqref{eq: na(I-P_2)om L^2L^2} and \eqref{eq: P_2om_s}, we deduce that
\begin{align}\label{eq: na omL^2L^2}
\int_0^{t}\|\nabla\widetilde{\omega}(s)\|_{L^2}^2ds
&=\int_0^{t}\|\nabla P_2\widetilde{\omega}_s(s)\|_{L^2}^2ds+\int_0^{t}\|\nabla (I-P_2)\widetilde{\omega}_s(s)\|_{L^2}^2ds\\ \nonumber
&\leq C\int_0^{t}\| P_2\widetilde{\omega}_s(s)\|_{L^2}^2ds+\int_0^{t}\|\nabla (I-P_2){\omega}(s)\|_{L^2}^2ds\\ \nonumber
&\leq C\int_0^{\tau/\nu}\nu^{4\gamma-1}ds+ C\nu^{2\gamma-1}
\\
&=C\tau\nu^{4\gamma-2}+ C\nu^{2\gamma-1}
\leq C\nu^{2\gamma-1}.\nonumber
\end{align}

Let $f=\omega_n$ and
\beno
g=-\big(\widetilde{V}_s^1\partial_x\om_n +V_n\cdot\nabla\widetilde{\om}_s
+P_{\neq 0}(V_n\cdot\nabla\omega_n)\big)=g_1+g_2+g_3,
\eeno
Then $(f,g)$ solves \eqref{eq:inhomLNS}. We denote
\beno
&&A_1(t)=\nu^{\f12}\left(\int_0^{t}\|\nabla \omega_n(s)\|_{L^{2}}^2ds\right)^{\f12}
+\nu^{\f23}\int_0^{t}\|\partial_x \omega_n(s)\|_{L^{2}}ds
+\nu^{\f{\gamma_1}2}\left(\int_0^{t}\|V_n(s)\|_{L^{\infty}}^2ds\right)^{\f12},\\
&&A_2(t)=\int_0^t\|g(s)\|_{L^2}ds\le \sum_{j=1}^3\int_0^t\|g_j(s)\|_{L^2}ds.
\eeno
Then \eqref{eq: ILNAL_t^2} implies that
\begin{align}\label{eq: A_1A_2}
A_1(t)\leq C\|\omega_n(0)\|_{L^2}+CA_2(t).
\end{align}
By \eqref{eq: V_s^1Linfty}, we have
\begin{align}\label{eq: g_1est}
\int_0^t\|g_1(s)\|_{L^2}ds
&=\int_0^t\|\widetilde{V}_s^1(s)\partial_x\omega_n(s)\|_{L^2}ds
\leq \int_0^t\|\widetilde{V}_s^1(s)\|_{L^{\infty}}\|\partial_x\omega_n(s)\|_{L^2}ds\\ \nonumber
&\leq C\nu^{\gamma}\int_0^t\|\partial_x\omega_n(s)\|_{L^2}ds
\leq C\nu^{\gamma-\f23}A_1(t),
\end{align}
and by \eqref{eq: na omL^2L^2},
\begin{align}\label{eq: g_2est}
\int_0^t\|g_2(s)\|_{L^2}ds
&=\int_0^t\|V_n(s)\cdot\nabla\widetilde{\om}_s(s)\|_{L^2}ds
\leq \int_0^t\|V_n(s)\|_{L^{\infty}}\|\nabla\widetilde{\om}_s(s)\|_{L^2}ds\\ \nonumber
&\leq \left(\int_0^{t}\|V_n(s)\|_{L^{\infty}}^2ds\right)^{\f12} \left(\int_0^{t}\|\nabla\widetilde{\omega}_s(s)\|_{L^2}^2ds\right)^{\f12} \\ \nonumber
&\leq \nu^{-\f{\gamma_1}{2}}A_1(t)(C\nu^{2\gamma-1})^{\f12}
\leq C\nu^{\gamma-\f{\gamma_1+1}{2}}A_1(t).
\end{align}
Using \eqref{eq: na(I-P_2)om L^2L^2} and the fact that $\nabla{\omega}_n(t)= P_{\neq 0}\nabla(I-P_2){\omega}(t)$, we infer that
\begin{align}\label{eq: g_3est}
\int_0^t\|g_3(s)\|_{L^2}ds
&=\int_0^t\|P_{\neq 0}(V_n\cdot\nabla\omega_n)\|_{L^2}ds
\leq \int_0^t\|V_n(s)\|_{L^{\infty}}\|\nabla{\omega}_n(s)\|_{L^2}ds\\ \nonumber
&\leq \left(\int_0^{t}\|V_n(s)\|_{L^{\infty}}^2ds\right)^{\f12} \left(\int_0^{t}\|\nabla(I-P_2){\omega}(s)\|_{L^2}^2ds\right)^{\f12} \\ \nonumber
&\leq \nu^{-\f{\gamma_1}{2}}A_1(t)(C\nu^{2\gamma-1})^{\f12}
\leq C\nu^{\gamma-\f{\gamma_1+1}{2}}A_1(t).
\end{align}
It follows from \eqref{eq: A_1A_2}-\eqref{eq: g_3est} that
\begin{align*}
A_1(t)\leq
C\|\omega_n(0)\|_{L^2}+C(\nu^{\gamma-\f23}+\nu^{\gamma-\f{\gamma_1+1}{2}})A_1(t).
\end{align*}

As $ \gamma>\f23,\ \gamma_1<2\gamma-1,$ we can find $c_{2} \in (0,1)$ so that
$C(\nu^{\gamma-\f23}+\nu^{\gamma-\f{\gamma_1+1}{2}})\leq \f12 $ for $0<\nu<c_{2}$. Then we obtain
\beno
 A_1(t)\leq C\|\omega_n(0)\|_{L^2},\quad
 A_2(t)\leq C(\nu^{\gamma-\f23}+\nu^{\gamma-\f{\gamma_1+1}{2}})\|\omega_n(0)\|_{L^2}.
 \eeno
While, by \eqref{eq: ILNSL^inftyL^2}, we have
\begin{align}\label{eq: om_n26}
\|\om_n(t)\|_{L^2}
\leq Ke^{-c_1\sqrt{\nu}t}\|\om_n(0)\|_{L^2}+CA_2(t).
\end{align}
Therefore,
$\|\om_n(t)\|_{L^2}\leq C\|\omega_n(0)\|_{L^2}$ for $0<t<\min\{t_0,\tau/\nu\}$. If $t=t_0<\tau/\nu$, then we have
\begin{align*}
\|\om_n(t)\|_{L^2}&\leq Ce^{-c_1\sqrt{\nu}t_0}\|\om_n(0)\|_{L^2}+CA_2(t)\\
&\leq \f14\|\om_n(0)\|_{L^2}+C_(\nu^{\gamma-\f23}+\nu^{\gamma-\f{\gamma_1+1}{2}})\|\om_n(0)\|_{L^2}.
\end{align*}
Take $c_2$ small enough(if necessary) so that $C(\nu^{\gamma-\f23}+\nu^{\gamma-\f{\gamma_1+1}{2}})\leq \f14 $ for $0<\nu<c_{2}$. Thus,
$\|\omega_n(t)\|_{L^2}\leq \f12\|\omega_n(0)\|_{L^2} $ for $t=t_0<\tau/\nu$.\smallskip

This completes the proof of Theorem \ref{Thm:NS}.

\section*{Acknowledgement}
The authors thank Zhiwu Lin for profitable discussions.
Z. Zhang is partially supported by NSF of China under Grant 11425103.
W. Zhao is partially supported by China Postdoctoral Science Foundation under Grant 2017M610007.

\end{CJK*}

\end{document}